\documentclass[11pt]{elsarticle}

\usepackage{hyperref}

\usepackage{amssymb, amsmath}
\usepackage[utf8]{inputenc}
\usepackage[english]{babel}
\usepackage{bm}
\usepackage{forest,rotating}
\usepackage{tikz}
\usetikzlibrary{shapes,arrows}
\usepackage{rotating}
\usepackage{geometry}
\usepackage[title]{appendix}
\usepackage{
  amsthm,amsmath,amssymb,esint,graphicx,color,psfrag,
  hyperref,placeins,caption,pgfplotstable,booktabs,rotating,
  units,bm,mathtools,epsfig,subfig,array,lscape
}
\usepackage{floatrow}

\mathtoolsset{showonlyrefs}
\tikzset{
  font={\fontsize{12pt}{12}\selectfont}}

\definecolor{chocolate}{rgb}{0, 0, 0}
\definecolor{forest}{rgb}{0, 0, 0}
\definecolor{blue}{rgb}{0, 0, 0}
\definecolor{red}{rgb}{0, 0, 0}
\definecolor{magenta}{rgb}{0, 0, 0}

\makeatletter
\def\Ddots{\mathinner{\mkern1mu\raise\p@
\vbox{\kern7\p@\hbox{.}}\mkern2mu
\raise4\p@\hbox{.}\mkern2mu\raise7\p@\hbox{.}\mkern1mu}}
\makeatother

\newcommand{\suchthat}{\;\ifnum\currentgrouptype=16 \middle\fi|\;}

\def\mcB{{\mathcal B}}
\def\mbbR{\mathbb{R}}

\newcommand{\beq}{\begin{equation}}
\newcommand{\eeq}{\end{equation}}
\newcommand{\beqn}{\begin{equation*}}
\newcommand{\eeqn}{\end{equation*}}

\def\x{\times}

\def\mcO{\mathcal{O}}

\def\mcN{\mathcal{N}}
\def\mcS{\mathcal{S}}

\def\mcL{\mathcal{L}}

\def\reals{{{\rm l} \kern -.15em {\rm R} }}

\def\mbbP{\mathbb{P}}
\def\mbbH{\mathbb{H}}

\usepackage{multirow}

\newtheorem{definition}{Definition}[section]
\newtheorem{theorem}{Theorem}[section]

\newtheorem{remark}[theorem]{Remark}
\newtheorem{exam}{Example}

\newcommand{\dd}{\; d}
\newcommand{\abs}[1]{\left|#1\right|}

\usepackage{xcolor,etoolbox}

\begin{document}

\begin{frontmatter}
\title{What Is the Fractional Laplacian? \\ {\large A Comparative Review with New Results}}

\author[brown]{Anna Lischke\corref{cor1}}
\author[brown]{Guofei Pang\corref{cor1}}
\author[brown]{Mamikon Gulian\corref{cor1}}
\author[brown]{Fangying Song}
\author[sandia]{Christian Glusa}
\author[brown]{Xiaoning Zheng}
\author[brown]{Zhiping Mao}
\author[smu]{Wei Cai}
\author[msu]{Mark M. Meerschaert}
\author[brown]{Mark Ainsworth}
\author[brown]{George Em Karniadakis\corref{cor2}}

\cortext[cor1]{\color{chocolate} These three authors contributed equally to this work.}
\cortext[cor2]{\color{chocolate} Corresponding author.}

\address[brown]{Division of Applied Mathematics, Brown University, Providence, RI 02912}
\address[sandia]{Center for Computing Research, Sandia National Laboratory, Albuquerque, NM 87123}
\address[smu]{Department of Mathematics, Southern Methodist University, Dallas, TX 75275}
\address[msu]{Department of Statistics and Probability, Michigan State University, East Lansing, MI 48824}

\begin{abstract}
{\color{chocolate} The fractional Laplacian in $\mbbR^d$, which we write as $(-\Delta)^{\alpha/2}$ with $\alpha \in (0,2)$, has multiple equivalent characterizations. Moreover, in bounded domains, boundary conditions must be incorporated in these characterizations in mathematically distinct ways, and there is currently no consensus in the literature as to which definition of the fractional Laplacian in bounded domains is most appropriate for a given application. The Riesz (or integral) definition, for example, admits a nonlocal boundary condition, where the value of a function must be prescribed on the entire exterior of the domain in order to compute its fractional Laplacian. In contrast, the spectral definition requires only the standard local boundary condition. These differences, among others, lead us to ask the question: ``What is the fractional Laplacian?" Beginning from first principles, we compare several commonly used definitions of the fractional Laplacian theoretically, through their stochastic interpretations as well as their analytical properties. Then, we present quantitative comparisons using a sample of state-of-the-art methods. We discuss recent advances on nonzero boundary conditions and present new methods to discretize such boundary value problems: radial basis function collocation (for the Riesz fractional Laplacian) and nonharmonic lifting (for the spectral fractional Laplacian).}

{\color{chocolate} In our numerical studies, we aim to compare different definitions on bounded domains using a collection of benchmark problems. We consider the fractional Poisson equation with both zero and nonzero boundary conditions, where the fractional Laplacian is defined according to the Riesz definition, the spectral definition, the directional definition, and the horizon-based nonlocal definition. We verify the accuracy of the numerical methods used in the approximations for each operator, and we focus on identifying differences in the boundary behaviors of solutions to equations posed with these different definitions. Through our efforts, we aim to further engage the research community in open problems and assist practitioners in identifying the most appropriate definition and computational approach to use for their mathematical models in addressing anomalous transport in diverse applications.}
\end{abstract}


\end{frontmatter}

\pagebreak

\tableofcontents

\vfill
\break

\begin{table}[ht!]
\centering
\begin{tabular}{| p{.97\textwidth} |}
\hline
\textbf{Guide to our notation and terminology:} 

\\

We define the fractional Laplacian to be $(-\Delta)^{\alpha/2}$, where $\Delta = \partial^2/\partial x_1^2 + 
... + \partial^2/\partial x_d^2$. We take the fractional power of $(-\Delta)$ to obtain a positive operator. As a result, our definition of the fractional Laplacian $(-\Delta)^{\alpha/2}$ is the \emph{negative} generator of the standard isotropic $\alpha$-stable L\'evy process, and reduces to $-\Delta = -\partial^2/\partial x_1^2 - 
... - \partial^2/\partial x_d^2$ when $\alpha = 2$.

\ \\
In this work, the symbol $\alpha$ is always used as the fractional order. In particular, the fractional Laplacian is represented as $(-\Delta)^{\alpha/2}$ and $\alpha \in (0,2)$.\\
\ \\
The symbol $s$ is always used in the representation of a real-ordered Sobolev space, $H^s$, and is often used in this work to describe the regularity of the source function $f$ of a fractional Poisson equation. The symbol $s$ should not be confused with the fractional order of the Laplacian, as often appears in the literature. All {\color{chocolate} fractional} Sobolev spaces mentioned in this work are defined in Appendix \ref{sobolev_spaces}, {\color{chocolate} where basic properties such as the fractional trace theorem are reviewed.}\\
\ \\
In Section \ref{background}, we do not make a notational distinction between the definitions of the fractional Laplacian, as the definition should be clear from the context or can be understood from the subsection heading. \\
\ \\
\emph{Homogeneous} fractional Laplacians are defined in the context of zero boundary conditions, and \emph{inhomogeneous} fractional Laplacians are defined with nonzero boundary conditions. The type of boundary condition (in this work, Dirichlet or Neumann) is specified in the text. \\
\ \\
In the sections following Section \ref{background}, multiple definitions appear together for the purpose of comparison, so we use the following notation: \\
\hspace{10pt} $(-\Delta_{\text{Riesz}})^{\alpha/2}$ represents the \emph{Riesz} (or \emph{integral}) definition (see Section \ref{sec:Riesz}), \\
\hspace{10pt} $(-\Delta_{\text{spectral}})^{\alpha/2}$ represents the \emph{spectral} definition (see Section \ref{spectral}), and \\
\hspace{10pt} $(-\Delta_{\text{M}})^{\alpha/2}$ represents the \emph{directional} definition (see Section \ref{directional}). \\
\hline
\end{tabular}
\end{table}

\section{Introduction}\label{intro}

\subsection{Overview}
During the past few decades, scientists have been exploring fractional calculus as a tool for developing more sophisticated mathematical models that can accurately describe complex anomalous systems \cite{pozrikidis_book, bucur2016nonlocal, meerschaert_sikorskii, vazquez2017mathematical}. In particular, the fractional Laplacian has been used in place of the integer-order Laplacian in many applications, including, for example, the fractional models listed in Table \ref{example-equations}. The fractional Laplacian can be defined in $\mbbR^d$ in many equivalent ways \cite{Kwasnicki2017}; however, when these definitions are restricted to bounded domains, the associated boundary conditions lead to distinct operators. 

\begin{table}[ht!]
\centering
\begin{tabular}{l | c | c }
	\multicolumn{2}{c|}{Fractional Equation} & Domain \\
	\hline
	Diffusion-Reaction \cite{Yamamoto2012} & $\partial_t u + (-\Delta)^{\alpha/2} u + c(t,x) u = 0$ & $(0,+\infty) \times \mbbR^d$ \\ 
	& & \\
	Quasi-geostrophic \cite{Constantin1999} & $\partial_t \theta + u \cdot \nabla \theta + \kappa (-\Delta)^{\alpha/2} \theta = f$ & $[0,T] \times \mbbR^2$ \\ 
	& & \\
	Cahn-Hilliard \cite{Akagi2015,AinsworthMao2017,MR3672018} & $\partial_t u + (-\Delta)^{\alpha/2} (-\varepsilon^2 \Delta u + f(u)) = 0$ & $(0,T] \times (0,2\pi)^2$ \\
	& &\\
	 Porous Medium \cite{Akagi2015,dePablo2011} & $\partial_tu + (-\Delta)^{\alpha/2}(|u|^{m-1}\text{sign}u) = 0$ & $(0,+\infty) \times \mbbR^d$ \\
	 & & \\
	 Schr\"odinger \cite{Laskin2000} & $i\hbar\partial_t \psi = D_\alpha(-\hbar^2 \Delta)^{\alpha/2} \psi + V(r,t) \psi$ & $(r,t) \in \mbbR^3 \times (0,+\infty)$ \\
	 & & \\
	 Ultrasound \cite{treebycox2010,ChenHolm2004} & $\frac{1}{c_0^2}\partial_t^2 p = \nabla^2 p - \left\{ \tau \partial_t (-\Delta)^{\alpha/2} + \eta (-\Delta)^{(\alpha+1)/2}\right\}p$ & $(-\infty,+\infty) \times \mbbR^d$ 
\end{tabular}
\caption{Important equations involving the fractional Laplacian. \label{example-equations}}
\end{table}

The purpose of this work is two-fold: (i) to give a comprehensive report of the commonly used definitions of the fractional Laplacian and examine their differences in bounded domains, and (ii) {\color{chocolate} to quantitatively explore these differences in benchmark problems using a sample of state-of-the-art numerical methods that are described in a detailed and reasonably self-contained way. 
As research on numerical methods for approximating the fractional Laplacian is extensive and ongoing, we do not attempt to include a comprehensive description of all available numerical methods. Instead, the sample of methods chosen for the comparisons in this work reflects the expertise of the authors.
Of significance is the emphasis on boundary value problems with nonzero boundary conditions, including stochastic methods, and the inclusion of new methods for discretizing such problems}. This work {\color{chocolate} will} be of use to practitioners looking to gain insight into which fractional Laplacian definition and associated numerical methods may be appropriate for their application.

{\color{forest}
A number of articles which include comparisons of the different fractional Laplacians on bounded domains have appeared recently, such as those of Bonito et. al. \cite{nochetto_three}, Duo et. al. \cite{duo_wang_zhang}., and \v{C}iegis et. al. \cite{margenov}.} 
{\color{chocolate}The present article differs from these in that there is a focus on recent advances in boundary value problems with \emph{nonzero} boundary conditions and the \emph{inhomogeneous} fractional operators that such problems entail. In addition to a review of the theoretical advances in this area, we include a number of new results and methods, which we now summarize. In Section \ref{RBFM} a novel radial basis function collocation method is presented for the Riesz fractional Poisson problem with nonzero Dirichlet boundary conditions based on discretizing the directional representation with the vector Gr\"unwald-Letnikov (GL) formula. This method offers advantages in complex domains and high dimensions, and has a clear extension to more general, non-symmetric operators corresponding to non-isotropic L\'evy motion.  Moreover, the method is applicable in the case of nonzero boundary conditions, which is significant due to the relative scarcity of solvers for such boundary value problems.
In Section \ref{sec:inhomogeneous_spectral}, we show the equivalence of recently proposed definitions of \cite{AntilPfeffererRogovs} and \cite{Cusimano2017} for the inhomogeneous spectral fractional Laplacian, and we provide a new equivalent characterization via the \emph{inverse} Laplacian. The equivalence of these approaches allows us to conclude that the problem of defining the inhomogeneous \emph{spectral} fractional Laplacian, and posing boundary value problems with it, has largely been solved. In addition, in Section \ref{sec:nonharmonic_lifting}, we introduce a new method of nonharmonic lifting for the Poisson problem with nonzero boundary conditions for the spectral fractional Laplacian. 
}

{\color{blue}
Another goal of this article is to illuminate the physical meaning of the different definitions of the fractional Laplacian in bounded domains through their associated stochastic processes. In particular, we discuss the fact that these differing definitions can be interpreted through different ways of applying boundary conditions to $\alpha$-stable L\'evy processes. We discuss this in \ref{sec:relation_to_levy}, \ref{killed_levy_motion}, and \ref{sec:subordinate_BM}, and compare the resulting stochastic processes and their operators in \ref{Summary}. This is most easily summarized for Dirichlet boundary conditions, where the stochastic picture involves two successive modifications of Brownian motion: stopping (when the motion exits the domain $\Omega$) and subordination (a stochastic time change by the standard $\alpha$-stable subordinator, a strictly increasing jump process). These modifications do not commute, leading to two distinct stochastic processes depending on the order in which the modifications are performed \cite{Song, GuanMa2006}. Each corresponds to a distinct fractional Laplacian operator. 
The spectral fractional Laplacian is the negative infinitesimal generator of subordinate stopped Brownian motion (Section \ref{sec:subordinate_BM}). Since paths of Brownian motion are continuous, stopped Brownian paths stop at $\partial \Omega$, and therefore subordinate stopped Brownian paths also stop at $\partial \Omega$, despite being discontinuous in the interior of $\Omega$. Thus, a local boundary condition prescribed only on $\partial \Omega$ is sufficient for a spectral Laplacian model of anomalous diffusion. In contrast, the Riesz fractional Laplacian is the negative infinitesimal generator of stopped subordinate Brownian motion (i.e., stopped $\alpha$-stable L\'evy motion; Sections \ref{sec:relation_to_levy} and \ref{killed_levy_motion}), which represents particles that are stopped upon exiting the domain via a jump over the boundary. Hence, conditions prescribed merely on the boundary of $\Omega$ are not sufficient to describe the behavior of particles that are exiting the domain, and instead an \emph{exterior condition} on the behavior of the process within {\color{magenta}$\mathbb{R}^d \setminus \Omega$} must be given to {\color{magenta}obtain} a physically meaningful model.
The relation to L\'evy processes is more than just conceptual; we use a recent stochastic solution method, the walk-on-spheres algorithm of \cite{kyprianou2016unbiasedwalk}, to solve the Riesz fractional Poisson equation, and we use the resulting data in our comparisons. 
}

{
\color{chocolate}
After considering a one-dimensional example below to illustrate some significant differences between the Riesz and spectral fractional Laplacians on bounded domains, the remainder of the article is organized as follows. In Section \ref{background}, we provide theoretical background on the fractional Laplacian definitions studied in this work, first in $\mbbR^d$ and then in bounded domains. In Section \ref{sec:num_meth}, we present the numerical methods used in this work, followed by comparisons of the solutions of two-dimensional benchmark problems with zero Dirichlet boundary conditions in Section \ref{sec:num_comp}. In Section \ref{sec:nonzerobcs}, we present some comparisons for nonzero Dirichlet boundary conditions. Section \ref{conclusion} contains a summary of the numerical methods discussed in this work, along with a discussion of our results and observations.
}

\subsection{\color{chocolate} Motivating Example}
{\color{chocolate}
To motivate the present study, we consider some one-dimensional benchmark problems involving different definitions of the fractional Laplacian, which are defined and discussed in detail in Section \ref{background}. Later in this work, we will return to this problem in higher dimensions and in different domains.
Consider the one-dimensional fractional Poisson problem on an interval $\Omega = (-L,L)$:
\begin{align}
\label{fracPoisson}
	(-\Delta)^{\alpha/2} u(x) &= f(x), \hspace{15pt} x \in \Omega,
\end{align}
with zero Dirichlet boundary conditions and $\alpha \in (0,2)$. Importantly, we consider two cases for the operator $(-\Delta)^{\alpha/2}$: the Riesz fractional Laplacian (introduced in Section \ref{sec:Riesz}) and the spectral fractional Laplacian (introduced in Section \ref{spectral}). The formulation of the zero Dirichlet boundary conditions depends on the definition of the fractional Laplacian. For the Riesz fractional Laplacian \eqref{E:riesz_definition_omega}, the boundary condition is formulated as $u(x) = 0$ in $\mbbR \setminus (-L,L)$, and for the spectral fractional Laplacian \eqref{E:spectral_omega}, the boundary condition is $u(\pm L) = 0$. The reasons for these formulations are discussed in Sections \ref{RieszDef} and \ref{SpectralDef}. The benchmark problems are posed with source functions $f = 1$ and $f = \sin(\pi x)$ on $\Omega$. Below, we observe that the spectral and Riesz solutions evolve in different ways as the fractional order $\alpha$ is changed, and that these evolutions are dependent on the size of the interval $\Omega$. Additionally, we discuss the differing behaviors of the solutions to the benchmark problems near the endpoints of the interval $\Omega$.}

\textbf{Case 1: $f(x) = 1$.} 
To discretize the spectral definition, we use the discrete eigenfunction method described in Sec. \ref{SEM}, and the Riesz fractional Poisson equation is solved numerically using the one-dimensional spectral method of Ref. \cite{MCS2016APPNM}. We plot numerical solutions of the fractional Poisson equation for both the spectral and Riesz definitions in Figure \ref{1dcomparisons} with various values of the fractional order $\alpha$. We observe from Figure \ref{1dcomparisons} that the maximum value of the Riesz solution at $x = 0$ does not vary in a monotone fashion as $\alpha$ ranges from $1.99$ to $0.01.$ Indeed, from $\alpha = 1.99$ to $\alpha = 0.5$, the maximum value increases, and from $\alpha = 0.5$ to $\alpha = 0.1$, the maximum value decreases. In contrast, the maximum values of the solutions corresponding with the spectral definition increase in a monotone fashion with $\alpha$.

If we instead defined Eq. \eqref{fracPoisson} on $\mbbR$, we would expect the solutions for different fractional Laplacian definitions to be the same, as these definitions are equivalent on $\mbbR$. This observation leads to the following question: is this monotonicity property of the solutions to Eq. \eqref{fracPoisson} affected by changing the size of the computational domain? To investigate this, we solved the same fractional Poisson problems (using both the spectral and Riesz definitions) but changed the length of the interval $\Omega$. {\color{blue} In this example, we denote the solutions to Eq. \eqref{fracPoisson} by $u_L$, to make the dependence on $L$ explicit. 
Figure \ref{figmax} includes plots of the trajectories $M_L(\alpha) = \max(u_L)$ of the maximum values of the solutions $u_L$ for $\alpha \in (0,2]$ and for different lengths $L$ of the interval $\Omega$ (each curve corresponds to a different value of $L$).  
Using a change of variables, one can show that the solution $u_L(x)$, for both the spectral and the Riesz fractional Laplacian in Eq. \eqref{fracPoisson}, has the property $u_L(x) = L^\alpha u_1 (x/L)$. However, as $u_1$ itself depends on $\alpha$, the interaction beteen the factors $L^\alpha$ and $u_1(x/L)$ leads to a switch in the trajectories $M_L(\alpha)$ from monotonically decreasing behavior, to non-monotonic behavior, and finally to monotonically increasing behavior as $L$ increases. As is clear in Figure \ref{figmax}, this occurs at different values of $\alpha$ for the different fractional Laplacians (Riesz and spectral). 
Thus, we observe that the size of the computational domain affects the behavior of the solutions in relation to the fractional order $\alpha$, and it does so in different ways for the different fractional Laplacian definitions.
}

In Figure \ref{fig:1d_diff1}, we plot the differences $u_{\text{Riesz}} - u_{\text{spectral}}$ of solutions to \eqref{fracPoisson} on $\Omega = (-1,1)$ with $f =1$ for different values of $\alpha$. We observe that boundary layers in the differences start to form as $\alpha$ drops below $1$, {\color{chocolate} becoming particularly sharp and developing singularities in their derivatives as} $\alpha$ approaches zero. This behavior can be understood by examining boundary regularity of solutions arising from the two different fractional Laplacian definitions, which we discuss in detail in Section \ref{sec:regularity}. Futhermore, we notice that the differences are nonnegative in all of $\Omega$, indicating that the Riesz solutions, for any value of $\alpha \in (0,2)$, 
{\color{blue} lie above} the spectral solutions. {\color{blue} This is consistent with the theoretical result that the inverse Riesz fractional Laplacian minus the inverse spectral fractional Laplacian (for zero Dirichlet boundary conditions) is positivity preserving \cite{MUSINA20161667}.} When the problem has nonzero boundary conditions, this {\color{blue}property need not hold}, as discussed in Section \ref{NonzeroComparison}. We also observe this property in the two-dimensional {\color{blue} zero Dirichlet boundary condition} comparisons in Section \ref{sec:num_comp}.

\begin{figure}[ht!]
\centering
 \subfloat[]{
\begin{minipage}[]{.5\textwidth}\centering
\includegraphics[scale=.36]{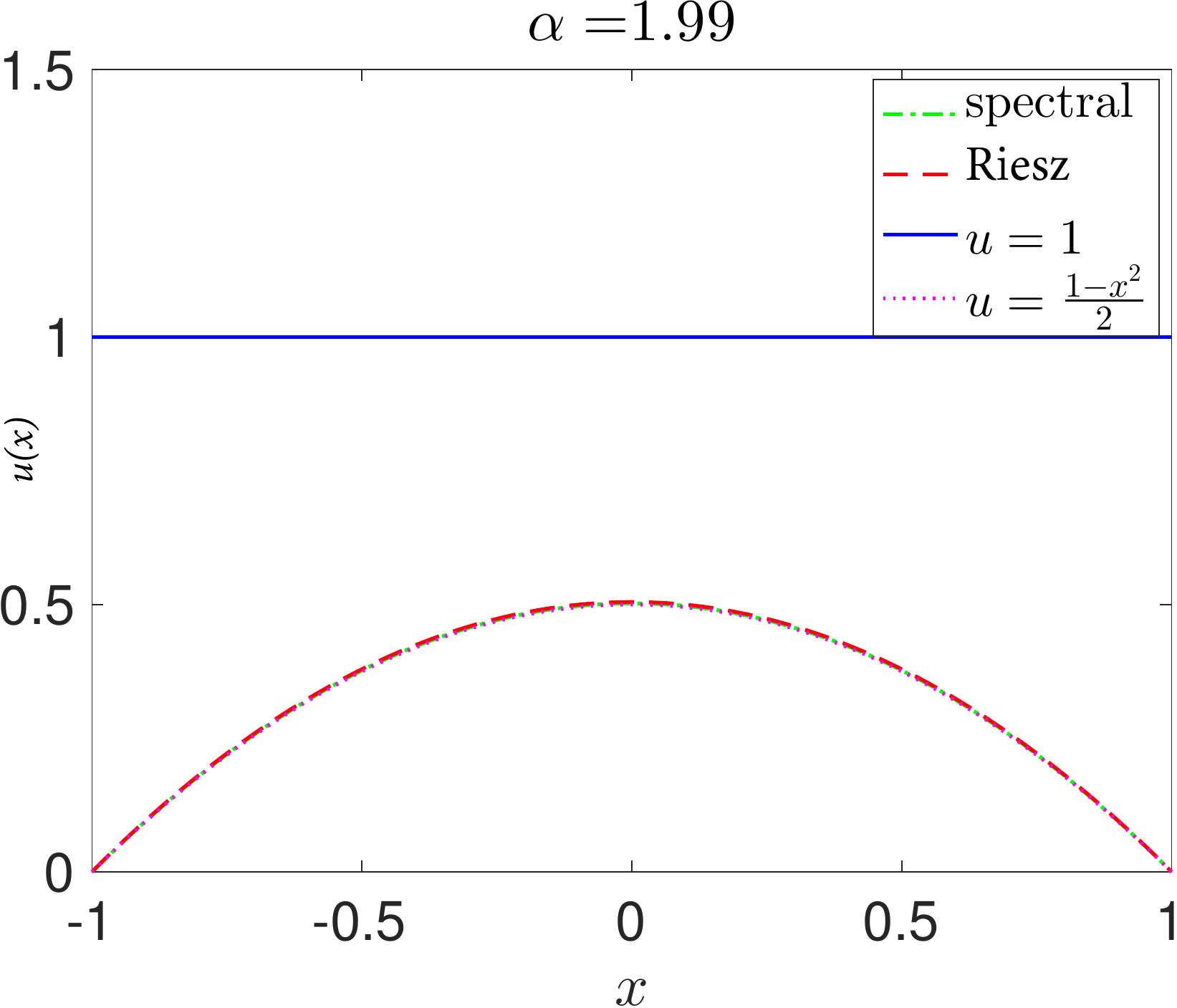}
\end{minipage}}
 \subfloat[]{
\begin{minipage}[]{.5\textwidth}\centering
\includegraphics[scale=.36]{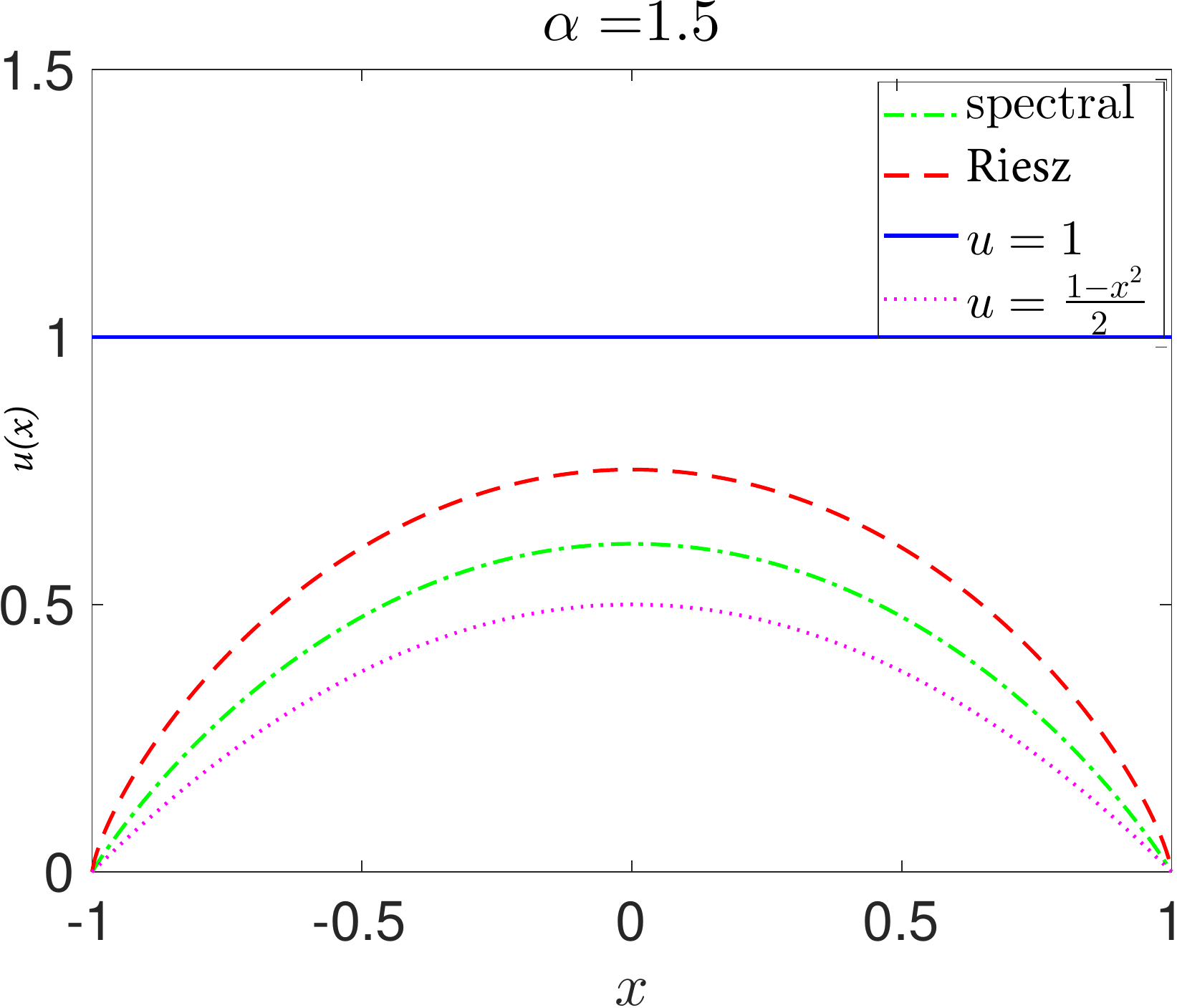}
\end{minipage}}\\
 \subfloat[]{
\begin{minipage}[]{.5\textwidth}\centering
\includegraphics[scale=.36]{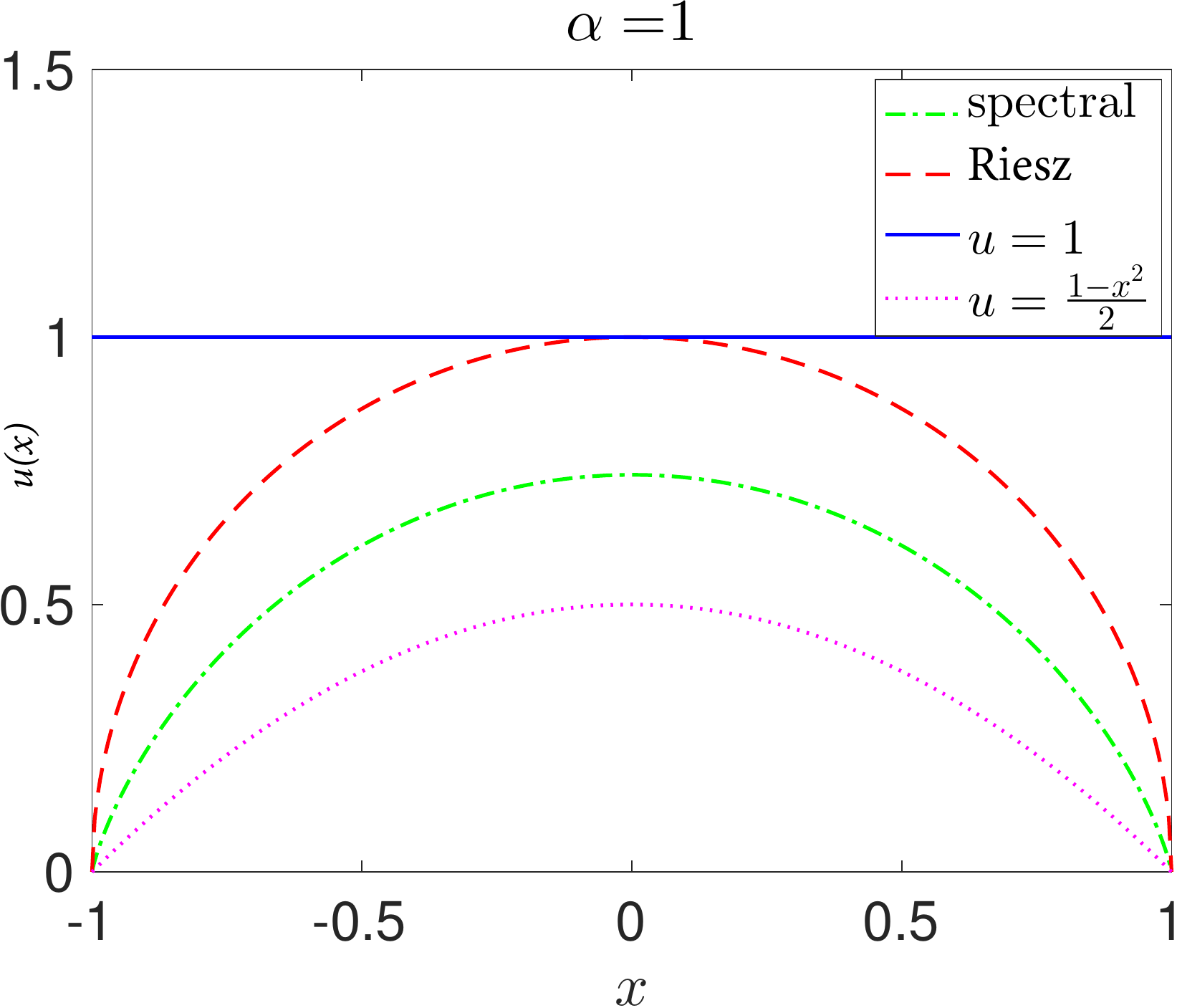}
\end{minipage}}
 \subfloat[]{
\begin{minipage}[]{.5\textwidth}\centering
\includegraphics[scale=.36]{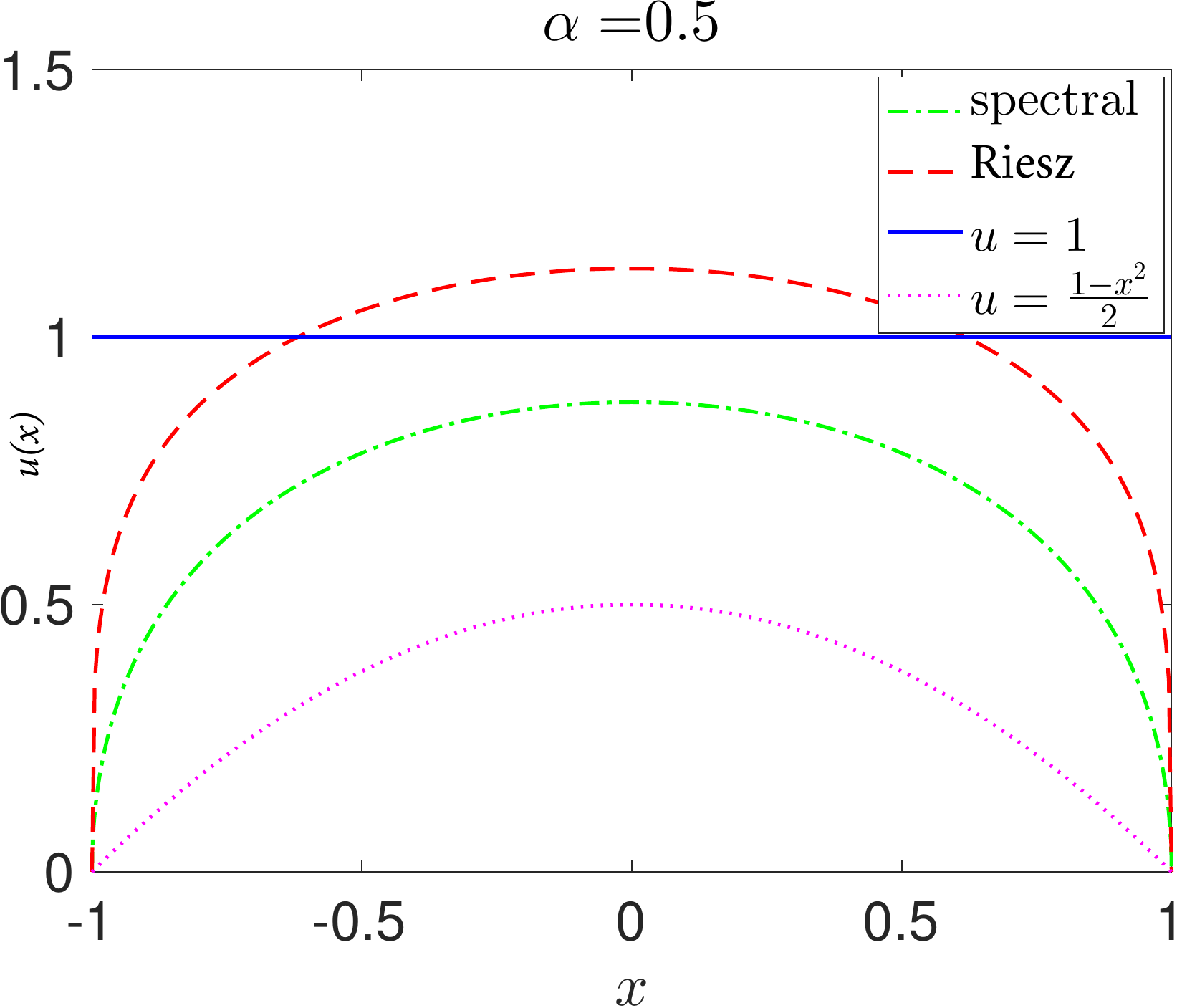}
\end{minipage}}\\
 \subfloat[]{
\begin{minipage}[]{.5\textwidth}\centering
\includegraphics[scale=.36]{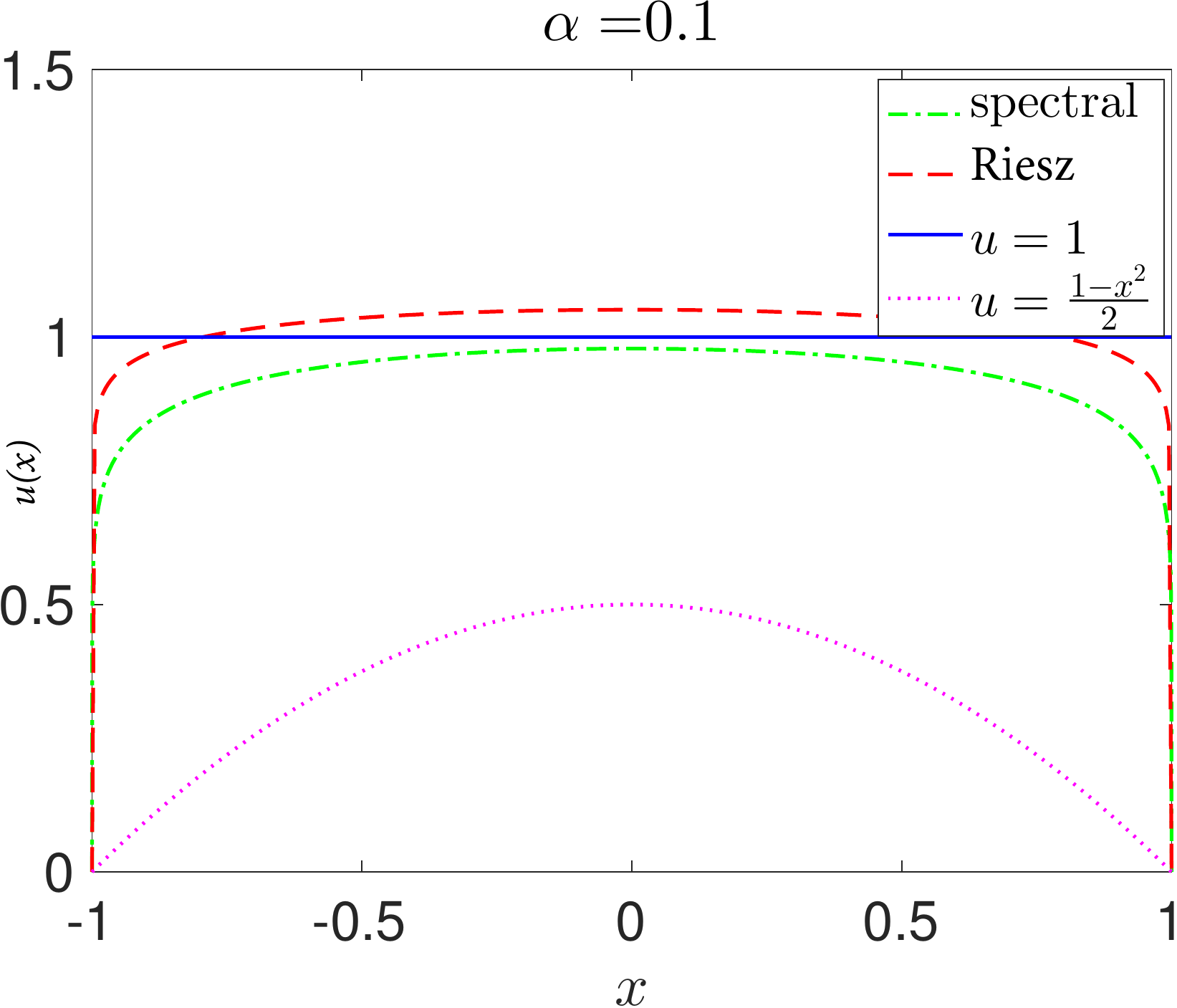}
\end{minipage}}
 \subfloat[]{
\begin{minipage}[]{.5\textwidth}\centering
\includegraphics[scale=.36]{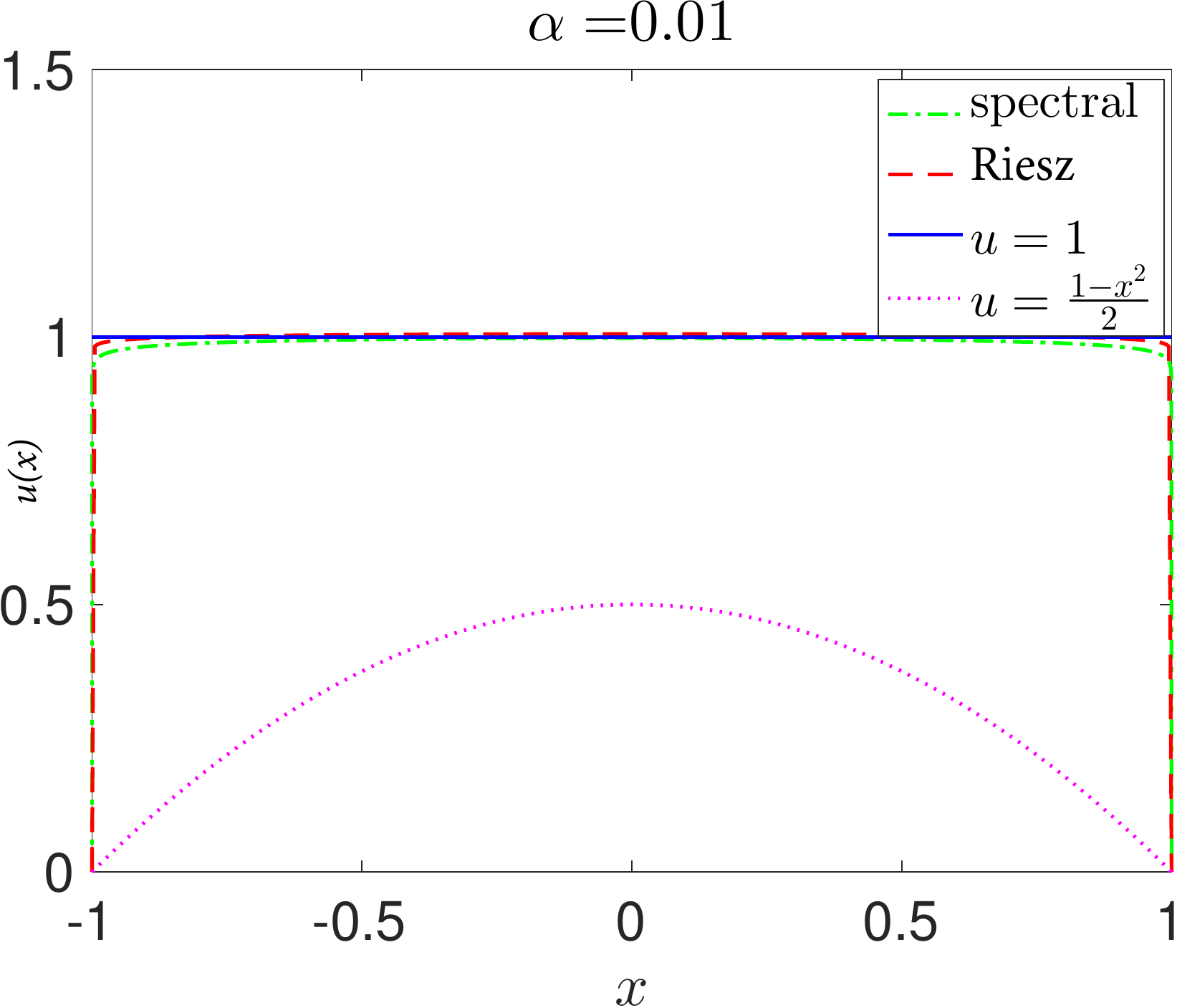}
\end{minipage}}
\caption{{One-dimensional study:} The profiles of the numerical solutions of {\color{chocolate} \eqref{fracPoisson}} for different fractional orders $\alpha$: (a)-(e) for $\alpha=1.99,1.5,1.0,0.5,0.1,0.01$. The pink dotted curve corresponds to the solution in the case $\alpha = 2$, and the blue line represents the (discontinuous) solution for $\alpha = 0$, and are included for reference. Note that the Riesz solution has greater amplitude than the spectral solution, and this amplitude increases above $u = 1$ as $\alpha$ goes to zero before decreasing to one in the plot for $\alpha = 0.01$. \label{1dcomparisons}}
\end{figure}

\begin{figure}[ht!]
 \centering
 \includegraphics[height=0.30\textheight]{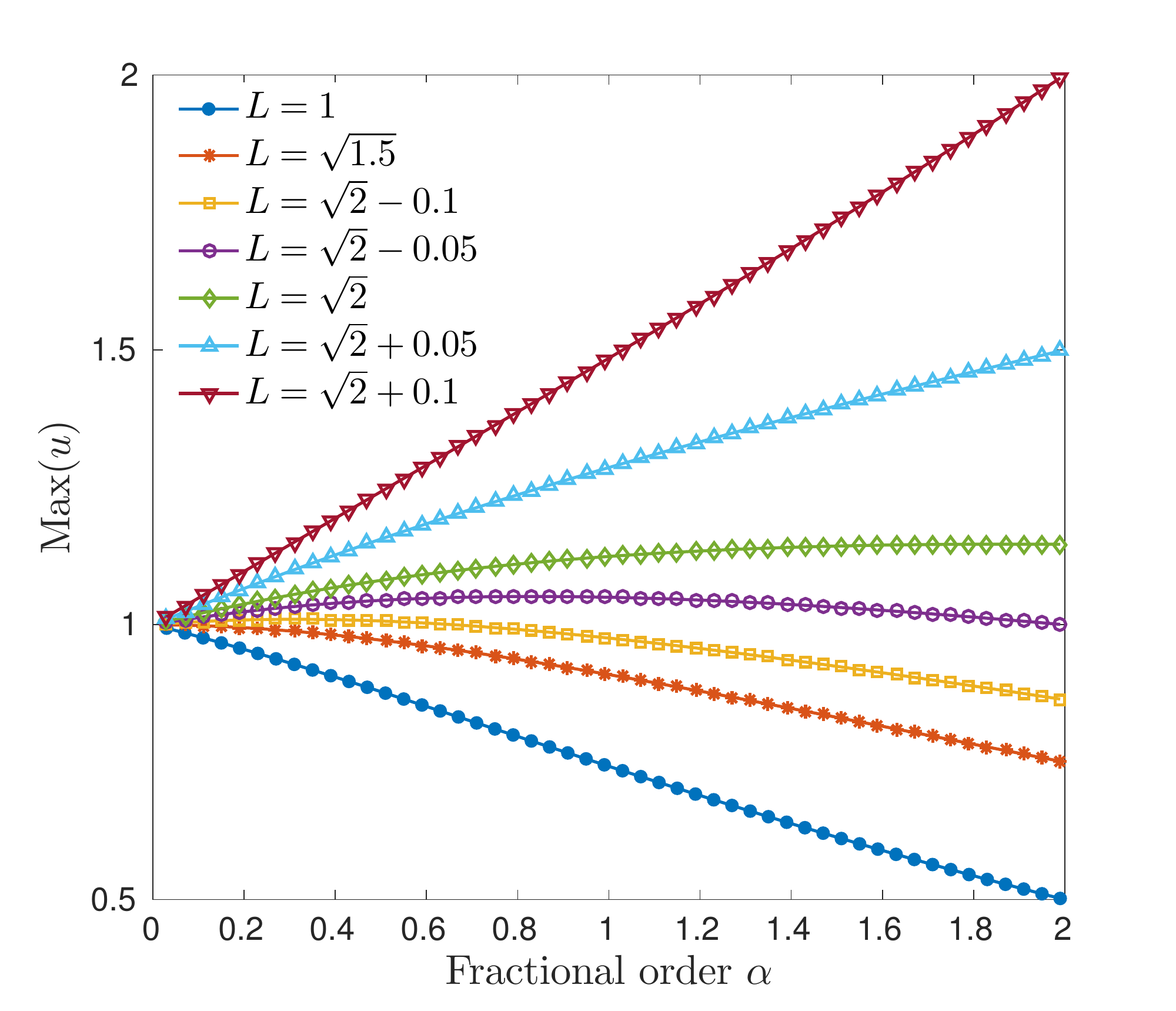}
 \includegraphics[height=0.30\textheight]{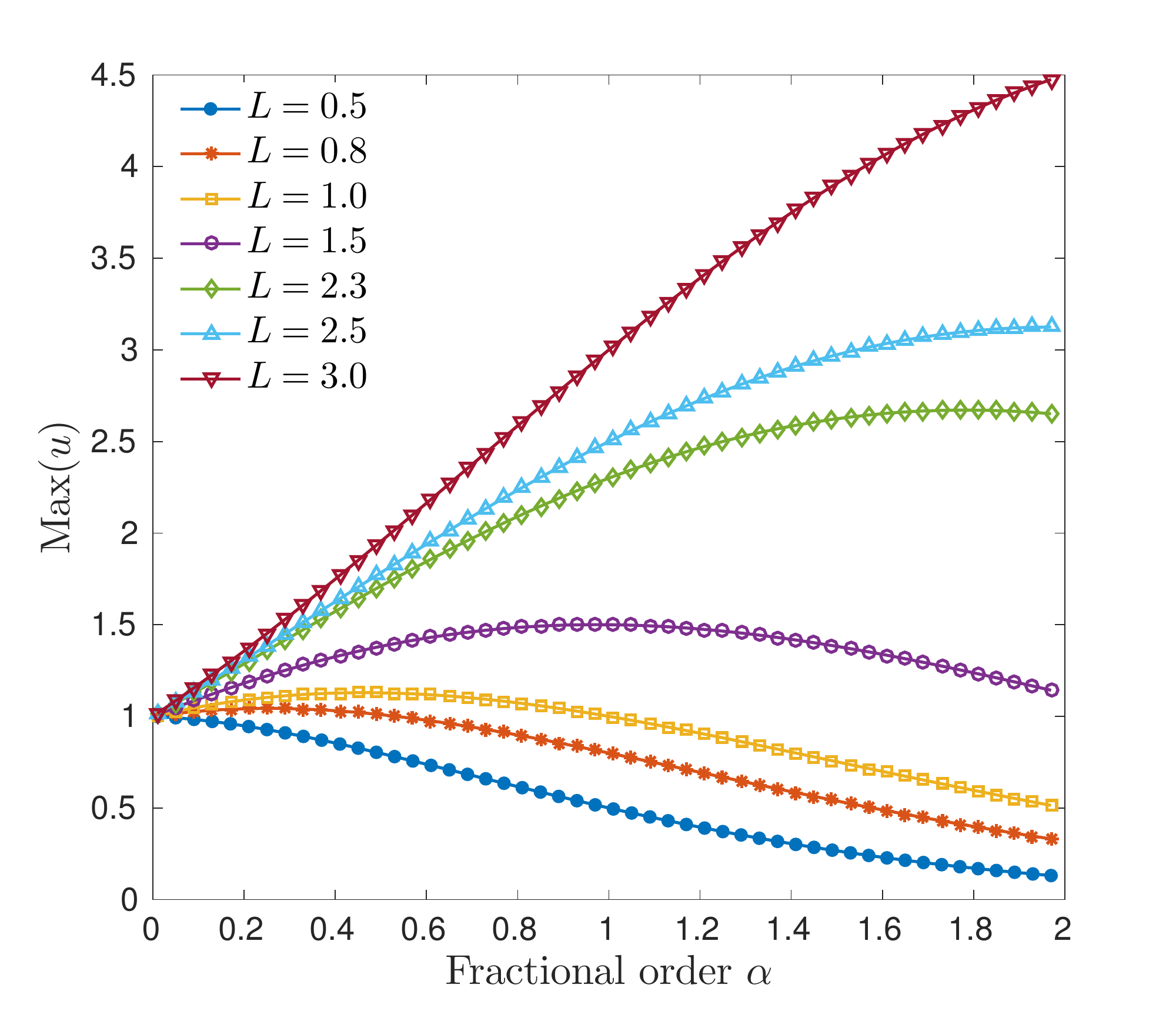}
 \caption{{Monotonicity Study}: (\emph{left}) The maximum value of the numerical solutions for the spectral fractional Poisson equation {\color{chocolate} \eqref{fracPoisson}} in the interval {\color{chocolate} $(-L,L)$} for $\alpha\in(0,2)$ with $f(x) = 1$. (\emph{right}) The maximum value of the numerical solutions for the Riesz fractional Poisson equation in the interval $(-L,L)$ for $\alpha \in (0, 2)$ with $f = 1$.}
 \label{figmax}
\end{figure}

\begin{figure}[ht!]
\centering
\includegraphics[width=.5\textwidth]{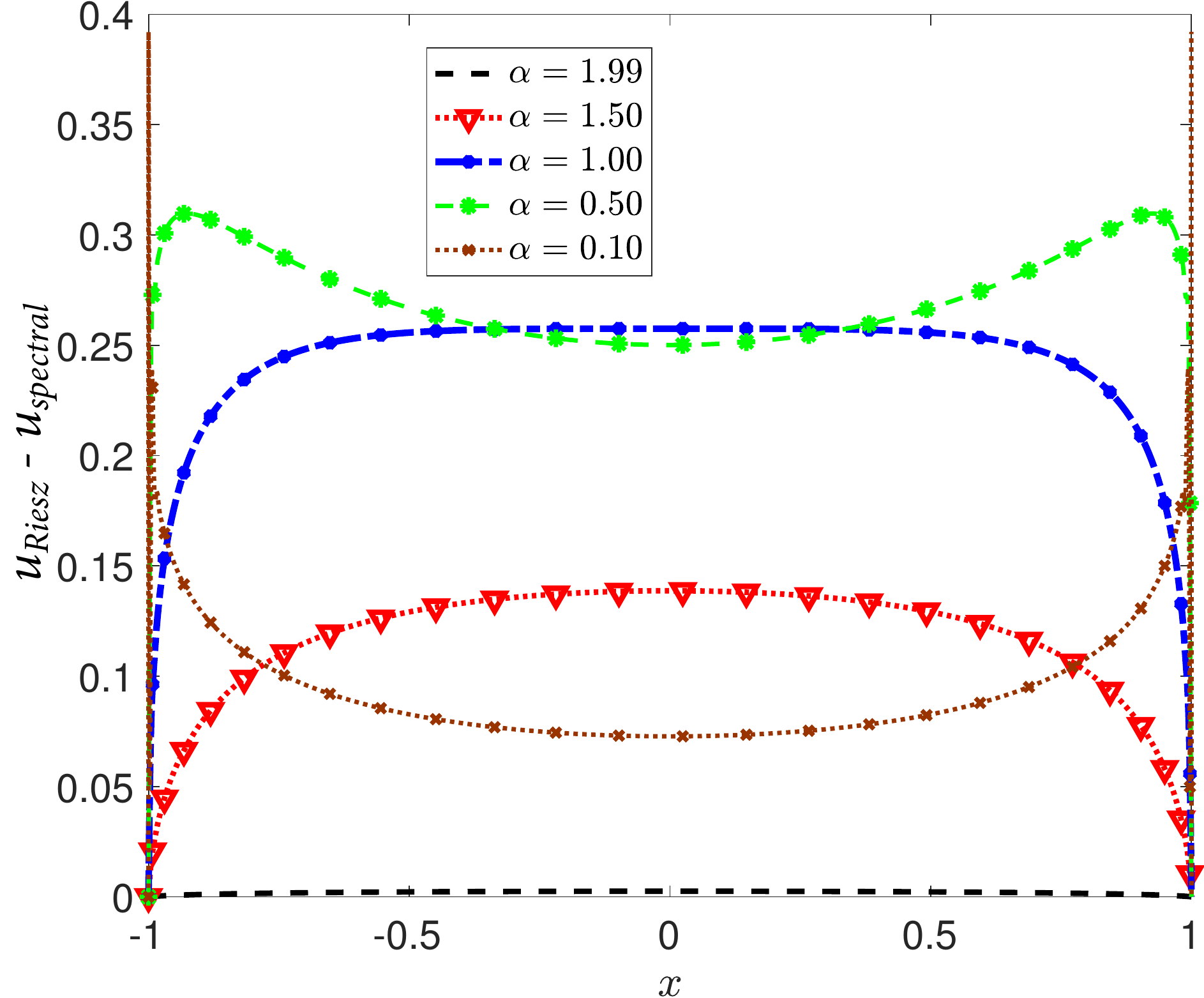}
\caption{{Plot of differences} between the spectral solutions and the {\color{chocolate} Riesz} solutions of {\color{chocolate} \eqref{fracPoisson}} for the right-hand-side $f = 1$ on the interval $(-1,1)$. For values of $\alpha$ between $1$ and $2$, the difference is greatest in the interior of the interval. When $\alpha < 1,$ a boundary layer forms and sharpens as $\alpha \rightarrow 0$.\label{fig:1d_diff1}}
\end{figure}

\textbf{Case 2: $f(x) = \sin(\pi x)$.} 
The solution to the Riesz fractional Poisson equation is computed using a spectral method \cite{MCS2016APPNM}, and the solution to the spectral Poisson equation is computed using the discrete eigenfunction method discussed in {\color{chocolate} Section} \ref{SEM}. The solutions are shown in Figure \ref{fsincomp}. The interesting feature to note is the boundary layer in the Riesz solutions that sharpens as $\alpha \rightarrow 0$ in comparison with the smooth behavior of the spectral solutions at the boundaries. Since the source function in this example, $f = \sin(\pi x)$, is an eigenfunction of the spectral Laplacian, the spectral solution is analytic in $\Omega$ and no boundary layer forms. For the Riesz solution, however, {\color{chocolate} the boundary regularity decreases with $\alpha$}, resulting in the boundary singularities observed in Fig. \ref{fsincomp}. In fact, for smooth source functions that satisfy the zero boundary conditions, we can always expect this difference in behaviors near the boundary.  On the other hand, it is possible to achieve a singular (at the boundaries) solution using the spectral definition if the source function itself is sufficiently singular; see Section \ref{sec:spectral_regularity}. This is a useful observation for modeling anomalous diffusion systems, since the Riesz definition may be a better choice to model data that exhibits such a boundary layer, given a smooth forcing function.

\FloatBarrier

\begin{figure}[htbp]
  \centering
  \includegraphics[height=0.25\textheight]{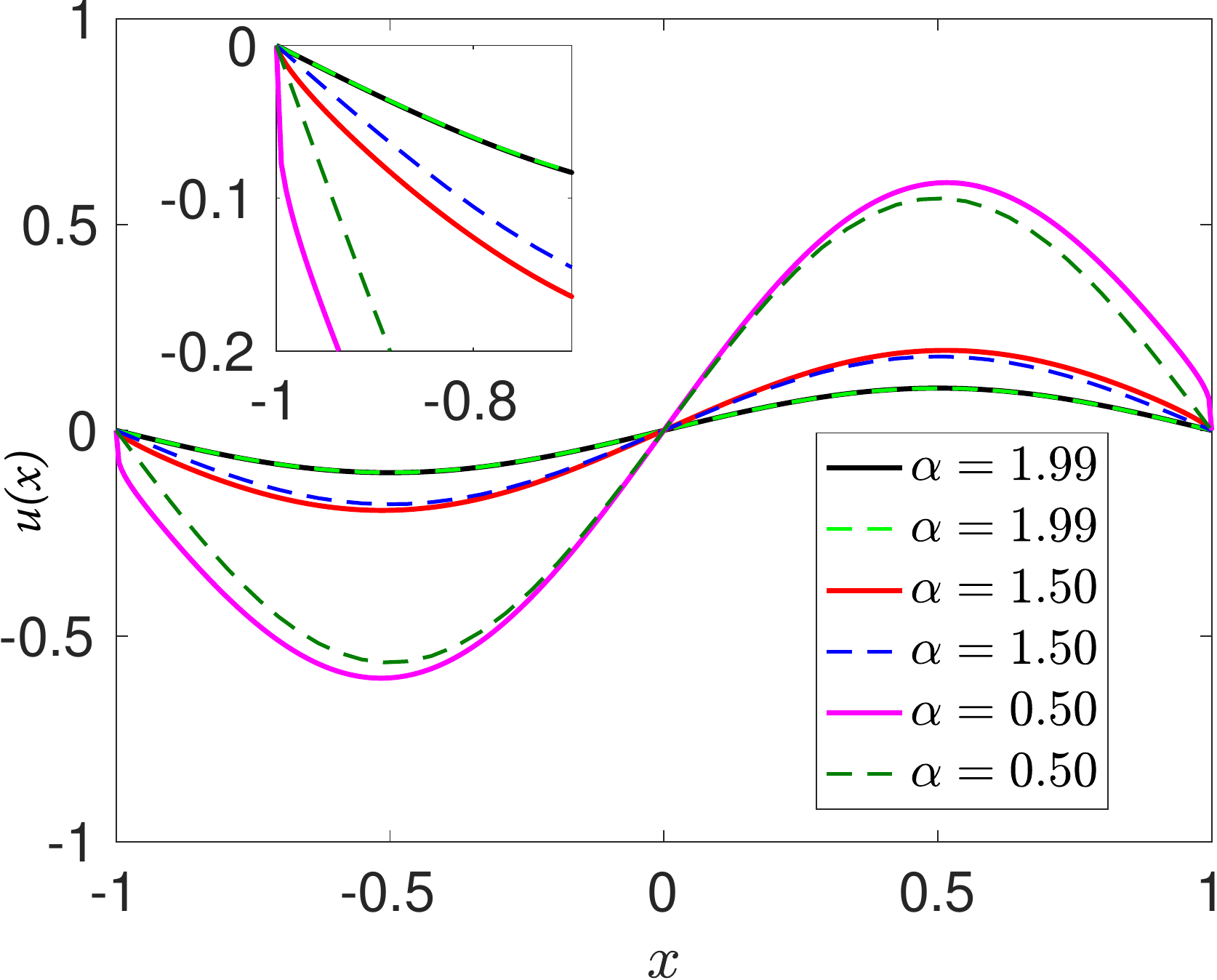}
   \includegraphics[height=0.25\textheight]{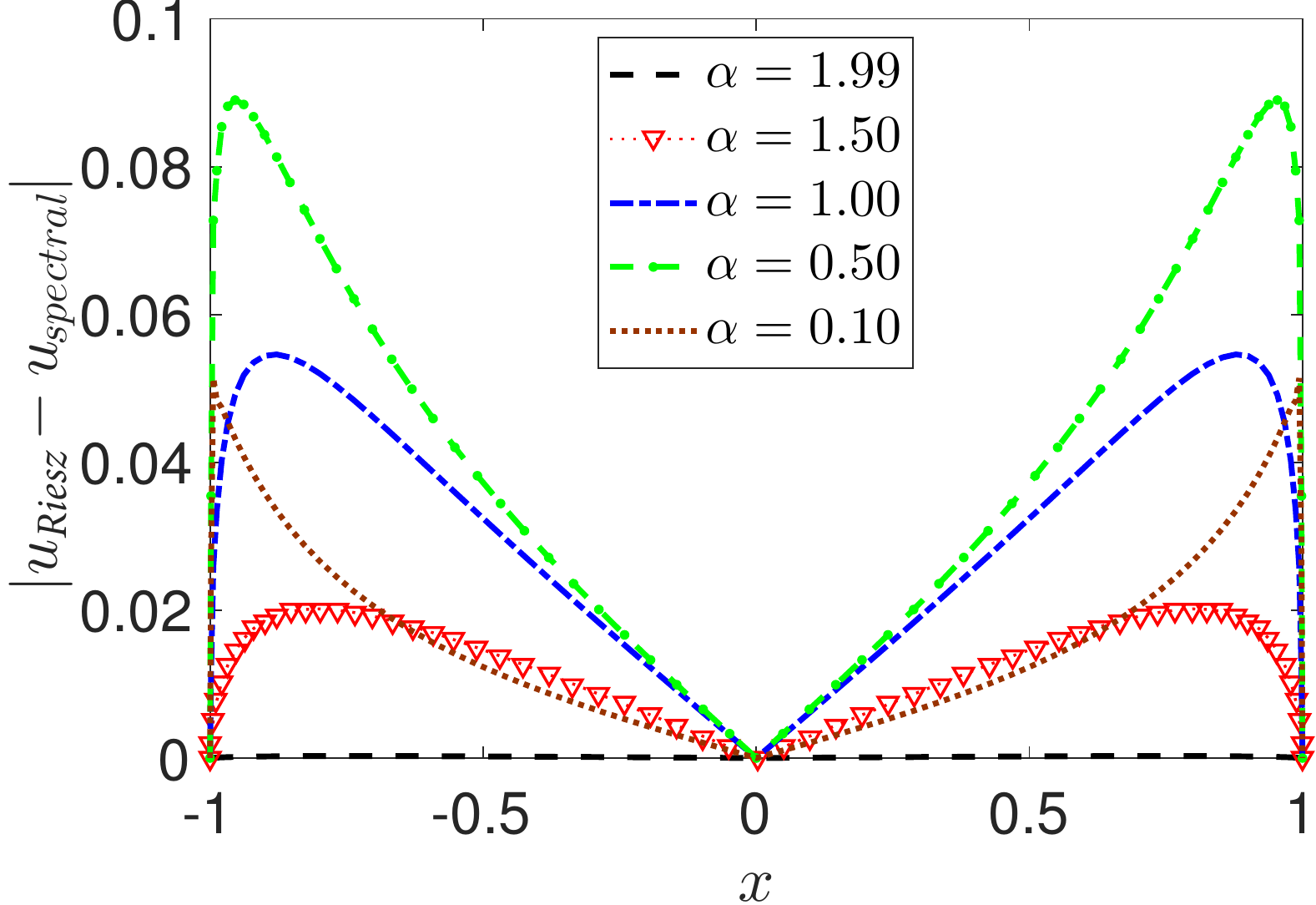}
  \caption{{One-dimensional study:} (\emph{left}) The \emph{numerical solutions} of both the spectral and Riesz fractional Poisson equations with $f=\sin(\pi x)$. The solid curves are solutions of the spectral Poisson equation computed with the spectral element method where $\mathcal{N}=64$. The dashed curves represent the solutions of the Riesz fractional Poisson equation which are computed using a finite element approximation.  (\emph{right}) The \emph{absolute differences} between the spectral solutions $u_{\text{spectral}}$ and the Riesz solutions $u_{\text{Riesz}}$ of their respective fractional Poisson equations for the smooth right-hand-side $f=\sin(\pi x)$ on the interval $(-1,1)$.}
  \label{fsincomp}
\end{figure}

\section{Definitions and Properties of Fractional Laplacians\label{background}}

\hrule
\vspace{1em}
\noindent \textbf{Section Overview}\\[0.2em]
\indent To describe and compare the different definitions of the fractional Laplacian considered in this work from a theoretical perspective, we review the derivations, regularity properties, and stochastic connections for the Riesz, spectral, directional, and regional definitions of the fractional Laplacian. We first discuss these different characterizations in $\mbbR^d$, including their derivations, equivalence, and relations to L\'evy processes. Next, we describe the breakdown in the equivalence when these characterizations are restricted to a bounded domain with either zero or nonzero boundary conditions. We provide a summary of the stochastic processes for which each fractional Laplacian (with the associated implicit boundary conditions) is the infinitesimal generator. This information is useful for developing Feynman-Kac type formulas for solving fractional elliptic/parabolic problems, {\color{chocolate} as in the} walk-on-spheres algorithm {\color{chocolate} discussed in Section \ref{sec:wos}}. Furthermore, we describe extension methods that have been used to reformulate the spectral fractional Laplacian, which make more conventional computational techniques useful in discretizing the differential equations involving these operators. We also summarize {\color{chocolate} well-posedness and }regularity properties of each operator.
\vspace{1em}
\hrule
\vspace{2em}

\subsection{The Fractional Laplacian on $\mathbb{R}^d$}

Although the fractional Laplacian on $\mathbb{R}^d$ is not considered in the numerical experiments contained in this work, the characterizations discussed in this section will lead to the definitions of the most commonly used fractional Laplacians (Riesz, directional, and spectral) in a bounded domain. Thus, this section serves as a basis for the entire article. For {\color{chocolate} proofs} that the characterizations we consider in this article, as well as other characterizations, are equivalent in $\mbbR^d$, see \cite{Kwasnicki2017}.

\subsubsection{Spectral/Fourier Definition}\label{Spectral/Fourier Definition}
We wish to construct a fractional power of the Laplacian, $(-\Delta)^{\alpha/2}$, for $0 < \alpha < 2$. A general approach to define the positive real powers $L^\rho$, with $\rho \in [-1,1]$, of a positive self-adjoint linear
operator $L$, such as $L = -\Delta$, is facilitated by the spectral {\color{chocolate} theorem. This result states that} for a self-adjoint, densely defined linear operator $L$ (not necessarily bounded) on a Hilbert space $\mathcal{H}$,
\begin{equation}\label{E:dense_subspace}
L: \mathcal{D}(L) \rightarrow \mathcal{\mathcal{H}} \text{,} \quad \mathcal{D}(L)\text{ a dense subspace of } \mathcal{H},
\end{equation}
there is a projection-valued measure $E_\lambda$ such that
\begin{equation}\label{E:spectral_theorem}
L = \int_{\lambda \in \sigma(L)} \lambda\, dE_\lambda \text{\quad on $\mathcal{D}(L) \subset \mathcal{H}$.}
\end{equation}
Here, $E_\lambda$ is the unique operator-valued spectral measure (resolution of the identity) associated to $L$, and ${\sigma(L) \subset \mathbb{R}}$ is the spectrum of ${L}$, which is the support of $E_\lambda$. 
See Reed \& Simon (\cite{simon}, p. 263) or Rudin (\cite{rudin}, p. 368)  for a full discussion of the spectral theorem for self-adjoint operators. 
{\color{chocolate}The dense domain $\mathcal{D}(L)$ is often referred to as a 
\emph{core} for the operator $L$ \cite{ethier2009markov}, and may be chosen to be smaller and more convenient space than a ``maximal'' domain of definition of $L$}. 
Using the above spectral representation, powers of the operator $L$ of order 
 $-1 \le \alpha/2 \le 1$
can be defined as the self-adjoint operator
\begin{equation}\label{E:spectral_power}
L^{\alpha/2} = \int_{\lambda \in \sigma(L)} \lambda^{\alpha/2} dE_\lambda.
\end{equation}
The domain of the operator $L^{\alpha/2}$ can then be extended to $\mathcal{H}$ by continuity. 

We wish to consider $L = -\Delta$ on $\mathbb{R}^d$ on the Sobolev space $\mathcal{H} = H^2(\mathbb{R}^d)$, {\color{chocolate} with domain the dense subspace $\mathcal{D}(-\Delta) = {\color{chocolate}C_0^\infty(\mathbb{R}^d) \subset \mathcal{H}}$ taken with the $\mathcal{H}$-norm. }
This gives
${-\Delta  = \int_{\sigma(-\Delta)} \lambda dE_\lambda}$ from \eqref{E:dense_subspace} and \eqref{E:spectral_theorem}. For a regular domain $\Omega$, 
the spectrum $\sigma(-\Delta)$ is a point spectrum, i.e., consisting
entirely of eigenvalues \cite{simon}, so one can think of ${dE_\lambda}$ as being, for each $\lambda$, 
a projection operator onto the eigenspace of ${\lambda}$.
Then, from \eqref{E:spectral_power}, the fractional Laplacian on $\mathbb{R}^d$ is defined by
\begin{equation}\label{E:spectral_general}
(-\Delta)^{\alpha/2} := \int_{\sigma(-\Delta)} \lambda^{\alpha/2} dE_\lambda.
\end{equation}
Let us briefly discuss the use of $-\Delta$ instead of $\Delta$ above. 
Since ${\Delta = \frac{\partial^2}{\partial x_1^2} + ... + \frac{\partial^2}{\partial x_d^2}}$ has negative 
eigenvalues, if we were to push such an operator through the above machinery, the resulting fractional operator would have complex eigenvalues. For this reason, the spectral theorem is usually applied to positive-definite operators. 

To make the definition \eqref{E:spectral_general} more explicit, the spectrum ${\sigma(-\Delta)}$ must be known exactly. 
On ${\mathbb{R}^d}$, this spectrum
consists of eigenvalues ${|\xi|^2}$, where $\xi \in \mbbR^d$, with corresponding generalized eigenfunctions
${e^{-i\xi \cdot x}}$. 
Thus, the projection valued measure is given on {\color{blue} $\mathcal{D}(-\Delta)$} by
\begin{align}
dE &= \frac{1}{(2\pi)^d}( \cdot, e^{-i \xi \cdot x} ) e^{i \xi \cdot x} d\xi,
\end{align}
where $( u, v ) = \int uv dx$ denotes the $L^2$ inner product on $\mathbb{R}^d$. The scale factor $1/(2\pi)^d$ is required so that $\int dE_\lambda = I$ (where $I$ is the identity).
The fractional Laplacian in ${\mathbb{R}^d}$ can therefore be written
as\footnote{
\color{blue}
A alternate statement of the spectral theorem can be made which involves representations of operators as multiplication operators. From that point of view, this result is not surprising, as $\mathcal{F}$ is precisely the unitary transformation $\mathcal{H} \rightarrow L^2$ specified in that theorem that diagonalizes the Laplacian, turning it into a multiplication operator (\cite{simon}, p. 260). 
}
\begin{equation}\label{E:definition_spectral}
(-\Delta)^{\alpha/2} u(x) = \frac{1}{(2\pi)^d} \int_{\color{chocolate}\mathbb{R}^d} |\xi|^{\alpha}  ( u , e^{-i \xi \cdot x} ) e^{i \xi \cdot x} d\xi
= \mathcal{F}^{-1} \left\{ |\xi|^{\alpha} \mathcal{F}\{u\}(\xi) \right\} (x),
\end{equation}
where $\mathcal{F}$ and $\mathcal{F}^{-1}$ denote the Fourier and inverse Fourier transforms, respectively\footnote{
Here, we use the convention $\mathcal{F}\{u\}(\xi)=\frac{1}{(2\pi)^{d/2}}\int_{\mathbb{R}^d} u(x) e^{- i \xi \cdot x} dx$, and 
$\mathcal{F}^{-1}\{\hat{u}\}(x)= \frac{1}{(2\pi)^{d/2}}\int_{\mathbb{R}^d} \hat{u}(\xi) e^{i \xi \cdot x} d\xi$.
}. 
Thus, we see that $(-\Delta)^{\alpha/2}$, defined by \eqref{E:spectral_general}, is a \emph{Fourier multiplier operator} with symbol ${|\xi|^{\alpha}}$, i.e., 
\begin{equation}\label{multiplier}
\mathcal{F} \left\{ (-\Delta)^{\alpha/2}u \right\} (\xi) = |\xi|^{\alpha} \mathcal{F}\{u\}(\xi),
\end{equation}
which generalizes the well-known Fourier multiplier property of ${-\Delta}$. 
Many authors use this relation to define the fractional Laplacian as a pseudodifferential operator \cite{stein}. The drawback {\color{chocolate} of taking this as a starting point} is that
the Fourier transform is no longer available for bounded domains, although the functional calculus approach \eqref{E:spectral_general} using the spectral theorem is applicable for the case of zero boundary conditions. Of course, in that setting, the Hilbert space $\mathcal{H}$, spectrum $\sigma(-\Delta)$, and measure ${dE_\lambda}$ must be taken accordingly{\color{chocolate}, as discussed in Section \ref{SpectralDef}}.

\subsubsection{Singular Integral Representation}\label{sec:singular_integral}

The fractional Laplacian can be expressed directly as a singular integral in real space ${\mathbb{R}^d}$, rather than as a ${2d}$-integral in both real and frequency space, as in Equation \eqref{E:definition_spectral}. This article focuses on positive powers $0 \le \alpha \le 2$, but the negative $\alpha$ case bears mentioning in connection with this goal. 

For  ${-d < \alpha < 0}$, i.e., for fractional \emph{inverse} Laplacians,
the multiplier ${ |\xi|^{\alpha}}$ is decaying and has a Fourier inverse in the sense of distributions:
$\mathcal{F}^{-1} \{  |\xi|^{\alpha} \} = C(d,\alpha)|x|^{-d-\alpha}$ (see Stein \cite{stein} or Landkof \cite{landkof}). The constant $C(d,\alpha)$ is given by 
\begin{equation}\label{E:constant}
C(d,\alpha) = \frac{2^{\alpha} \Gamma\left(\frac{\alpha}{2}+\frac{d}{2}\right)}{\pi^{d/2}
{\color{chocolate} |}
\Gamma\left(-\frac{\alpha}{2}
\right)
{\color{chocolate} |}
}.
\end{equation}
Then, \eqref{multiplier} and the convolution property of the Fourier transform 
imply that the fractional inverse Laplacian is given, for $-2 < \alpha < 0$, by
\begin{align}
\begin{split}
\label{E:riesz_potential}
(-\Delta)^{\alpha/2} u(x) &= C(d,\alpha) |x|^{-d-\alpha} * u (x) = C(d,\alpha) \int_{\mathbb{R}^d}\frac{u(y)}{|x-y|^{d+\alpha}} dy =: I_{-\alpha}u(x).
\end{split}
\end{align}
This results in a {\color{chocolate} well-defined} function if $u(x)$ is, say, a smooth function with {\color{chocolate} sufficient decay} (Joshi \& Freidlander \cite{joshi} or Reed \& Simon \cite{simon}). 
This operator ${I_{-\alpha}}$ is known as the Riesz potential, the properties of which
(such as $L^p$ boundedness) are discussed at length
in Stein \cite{stein}. The Riesz potential {\color{chocolate} is an important tool} in harmonic analysis and the analysis of linear PDEs \cite{hormander}.
 
For $0 < \alpha < 2$ (the fractional Laplacians in which we are interested), the above derivation fails because the inverse Fourier transform of the symbol 
$ |\xi|^{\alpha} $
no longer exists, even as a distribution. In addition, the representation \eqref{E:riesz_potential} does not continue for $\alpha>0$, since the singularity would no longer be
integrable. However, starting from the negative $\alpha$ case
\eqref{E:riesz_potential}, a nice argument that can be found in Landkof (\cite{landkof}, p. 45) 
based on analytic continuation in $\alpha$ of $(-\Delta)^{\alpha/2} u(x)$, for fixed $u$ and $x$, yields the following real-space formula for the fractional Laplacian:
\begin{align}
\label{E:riesz_Laplacian_rn}
(-\Delta)^{\alpha/2} u &= C(d,\alpha) \text{   p.v.}\int_{\mathbb{R}^d} \frac{u(x) - u(y)}{|x-y|^{d+\alpha}}dy.
\end{align}
The constant $C(d,\alpha)$ is the same as in Eq. \eqref{E:constant}, and
 ``p.v.'' denotes the principal value of the integral:
\begin{equation*}
\text{   p.v.}\int_{\mathbb{R}^d} \frac{u(x) - u(y)}{|x-y|^{d+\alpha}}dy
=
\lim_{\epsilon \rightarrow 0} \int_{\mathbb{R}^d \setminus {\color{chocolate} B_\epsilon(x)}} \frac{u(x) - u(y)}{|x-y|^{d+\alpha}}dy,
\end{equation*}
where $B_\epsilon{\color{chocolate}(x)}$ is a ball of radius $\epsilon$ {\color{chocolate} centered at $x$}. The difference $u(x)-u(y)$ in the numerator of \eqref{E:riesz_Laplacian_rn}, which vanishes at the singularity, provides a regularization, which together with averaging of positive and negative parts allows the principal value to exist, e.g., for smooth $u$ with sufficient decay. 

The relation between the Riesz potential and the fractional Laplacian is discussed in detail in \cite{samko1993fractional}, Sections 5.25 and 5.26, {\color{chocolate} where it is shown that}
\begin{align}
	(-\Delta)^{\alpha/2} I_{\alpha} u &= u.
\end{align}
This identity leads to the representation of the fractional Laplacian directly in terms of the Riesz potential:
\begin{equation}\label{e:laplacian_in_terms_of_riesz_potential}
(-\Delta)^{\alpha/2}u(x) = 
-\Delta {I_{2-\alpha}} u(x).
\end{equation}

\subsubsection{Via the standard Laplacian: Elliptic Extension, Heat Semigroup, \& Balakrishnan Formula}\label{s:via_standard_laplacian}

Next, we describe three representations of the fractional Laplacian 
$(-\Delta)^{\alpha/2}$ of a function $u(x)$ on $\mathbb{R}^d$ that require the solution of equations involving the standard Laplacian, albeit in the $(d+1)$-dimensional half-plane. The first is the \emph{extension method}, or the {\color{chocolate} \emph{Dirichlet-to-Neumann map}}, which requires the solution of a degenerate \emph{elliptic} equation in the half-plane using $u(x)$ as the Dirichlet boundary data, followed by a type-of normal derivative of the solution. The second is the \emph{heat semigroup method}, which requires the solution of a \emph{parabolic} equation -- the simple heat equation -- in the half-plane with $u(x)$ as the intial condition, followed by long-time integration. The third is the Balakrishnan formula, which expresses the fractional Laplacian in terms of the resolvent $(sI - \Delta)^{-1}$. 

The \emph{extension method} is based on the following result.
Given a function $u(x)$ on $\mathbb{R}^d$, consider the extension $\tilde{u}(x,y)$ on ${ \mathbb{R}^d \times [0, \infty) }$ that solves
\begin{align}
\begin{split}\label{E:degenerate_elliptic}
\Delta_x \tilde{u} + \frac{1-\alpha}{y} \partial_y \tilde{u} + \partial^2_y \tilde{u} &= 0 \\
\tilde{u}(x,0) &= u(x).
\end{split}
\end{align}
Then
\begin{equation}
(-\Delta)^{\alpha/2} u(x)  = c \lim_{y \rightarrow 0}\frac{u(x,y) - u(x,0)}{y^{\alpha}}
\end{equation}
for a certain constant $c$ that depends on $d$ and $\alpha$. 

The above statement is taken directly from Caffarelli and Silvestre \cite{CaffarelliSilvestre2007_ExtensionProblemRelatedToFractionalLaplacian}, which is the most widely read source for the extension method. The extension method was reported as early as 1968 by Molchanov and Ostrovskii, in their studies of symmetric stable processes \cite{Molchanov_Ostrovskii}. The result was also used by other authors (e.g., \cite{deblassie_exit_time}), before it was systematically addressed by Caffarelli and Silvestre.  

The \emph{heat semigroup representation} of the the fractional Laplacian uses the solution of 
the heat equation in ${\mathbb{R}^d \times [0,\infty)}$: 
\begin{equation}\label{heatsg_integral}
(-\Delta)^{\alpha/2} u(x) = \frac{1}{\Gamma(-\alpha)} \int_0^\infty  \left( e^{t\Delta} u(x) - u(x) \right) 
\frac{dt}{t^{1+\alpha/2}}.
\end{equation}
Here, ${e^{t\Delta}}$ is the propagator of the heat equation, i.e., ${w(x,t) = e^{t\Delta}u(x)}$ is the 
solution of the {\color{chocolate} problem} 
{\color{blue}
\begin{align}
\begin{split}\label{heatsg_heateqn}
\partial_t w - \Delta w &= 0  \text{  on $\mathbb{R}^d \times [0,\infty) $}  \\
w(x,t = 0) &= u. 
\end{split}
\end{align}
}
The family $\{e^{t\Delta}\}$ is called the heat semigroup. See \cite{stinga_thesis} and \cite{StingaTorrea2010_ExtensionProblemHarnacksInequalityFractionalOperators} for a full discussion. 
{ \color{chocolate} An implementation of the formula \eqref{heatsg_integral} on $\mathbb{R}^d$
was studied in \cite{stinga_5_nonlocal}. }

The \emph{Balakrishnan formula}, {\color{blue} introduced in \cite{balakrishnan1960}, is a result from spectral theory and the theory of semigroups. This formula for 
the fractional Laplacian is}
\begin{equation}\label{E:balakrishnan_rn}
(-\Delta)^{\alpha/2} u(x) = \frac{\sin(\alpha \pi/2)}{\pi} 
\int_0^\infty  \Delta (sI - \Delta)^{-1} u (x) s^{\alpha/2 - 1}ds. 
\end{equation}

For a further discussion of these characterizations in $\mathbb{R}^d$, we refer to \cite{Kwasnicki2017} and references therein. {\color{chocolate} Although fractional Laplacians in bounded domains will be introduced in the next subsection, we preface that discussion by pointing out that analogues of these methods hold in bounded domains, depending on the fractional Laplacian being considered.
In the case of the spectral fractional Laplacian for functions that satisfy zero Dirichlet boundary conditions, the characterization via an elliptic extension holds with the half-plane being replaced by a cylinder over the original domain \cite{stinga_thesis,StingaTorrea2010_ExtensionProblemHarnacksInequalityFractionalOperators}. 
The Balakrishnan formula and the closely related Dunford-Taylor formula (for the inverse spectral fractional Laplacian) are valid and have efficient numerical implementations \cite{bonito2015,bonito2017}. 
The heat kernel formula has been found to be valid even for nonzero boundary conditions \cite{Cusimano2017}. 
See sections \ref{homo_representations} and \ref{sec:heat_semigroup} for further discussion and references to proofs and implementations of these methods. 
}
\subsubsection{Relation to L\'evy processes}\label{sec:relation_to_levy}

The fractional Laplacian is connected to \emph{anomalous} diffusion, which accounts for 
much of the interest in modeling with fractional equations. Just as the Laplacian 
is the negative generator of Brownian motion 
{(\color{chocolate} scaled by $\sqrt{2}$)}, the fractional Laplacian is the 
infinitesimal generator of a standard isotropic $\alpha$-stable L\'evy motion $X_t^\alpha$, {\color{chocolate} which can be expressed by}
\begin{equation}\label{e:generator_levy}
-(-\Delta)^{\alpha/2} f(x) = \lim_{h \rightarrow 0} \frac{\mathbb{E} \left[ f(x-X^\alpha_h) - f(x) \right]}{h}.  
\end{equation}
{\color{blue} This connection is explained more fully in the last two paragraphs of this section.}
This process $X^\alpha_t$ is a L\'evy process (\cite{meerschaert_sikorskii}, p.\ 100) in which the increments are drawn from a spherically-symmetric $\alpha$-stable distribution (\cite{meerschaert_sikorskii}, Ex. 6.24). This process can be viewed as the long-time scaling limit of a random walk with power law jumps (\cite{meerschaert_sikorskii}, Theorem 6.17). For ${\alpha = 2}$, ${X^\alpha_t}$ reduces to {\color{chocolate} scaled }Brownian motion {\color{chocolate} $X^2_t = \sqrt{2} B_t$,}
with the $2$-stable distribution being the normal distribution 
{\color{chocolate} $\mathcal{N}(0,\sigma^2 = 2)$.}
For $\alpha < 2$, the $\alpha-$stable distribution exhibits heavy tails and infinite variance. Moreover, the mean is finite if and only if $\alpha>1$. The L\'evy process $X_t^\alpha$ has superdiffusive scaling $X_{ct}^\alpha \sim c^{1/\alpha}X_t^\alpha$, and exhibits long, infinite-variance jumps when $\alpha < 2$. As an example, a superdiffusing cloud of particles would spread in space like $t^{1/\alpha}$. In many ways, $\alpha$-stable L\'evy flights are the simplest generalization of Brownian motion.
As a result of \eqref{e:generator_levy}, the fractional Laplacian naturally appears in macroscopic governing equations of systems of particles undergoing $\alpha$-stable L\'evy motion, making it a powerful and useful generalization \cite{meerschaert_sikorskii}. In Figure \ref{fig:2dLevyBM_comparison}, we include examples of 2D isotropic stable L\'evy motion and standard Brownian motion.

\FloatBarrier

\begin{figure}[ht!]
  \centering
  \includegraphics[width=.45\textwidth]{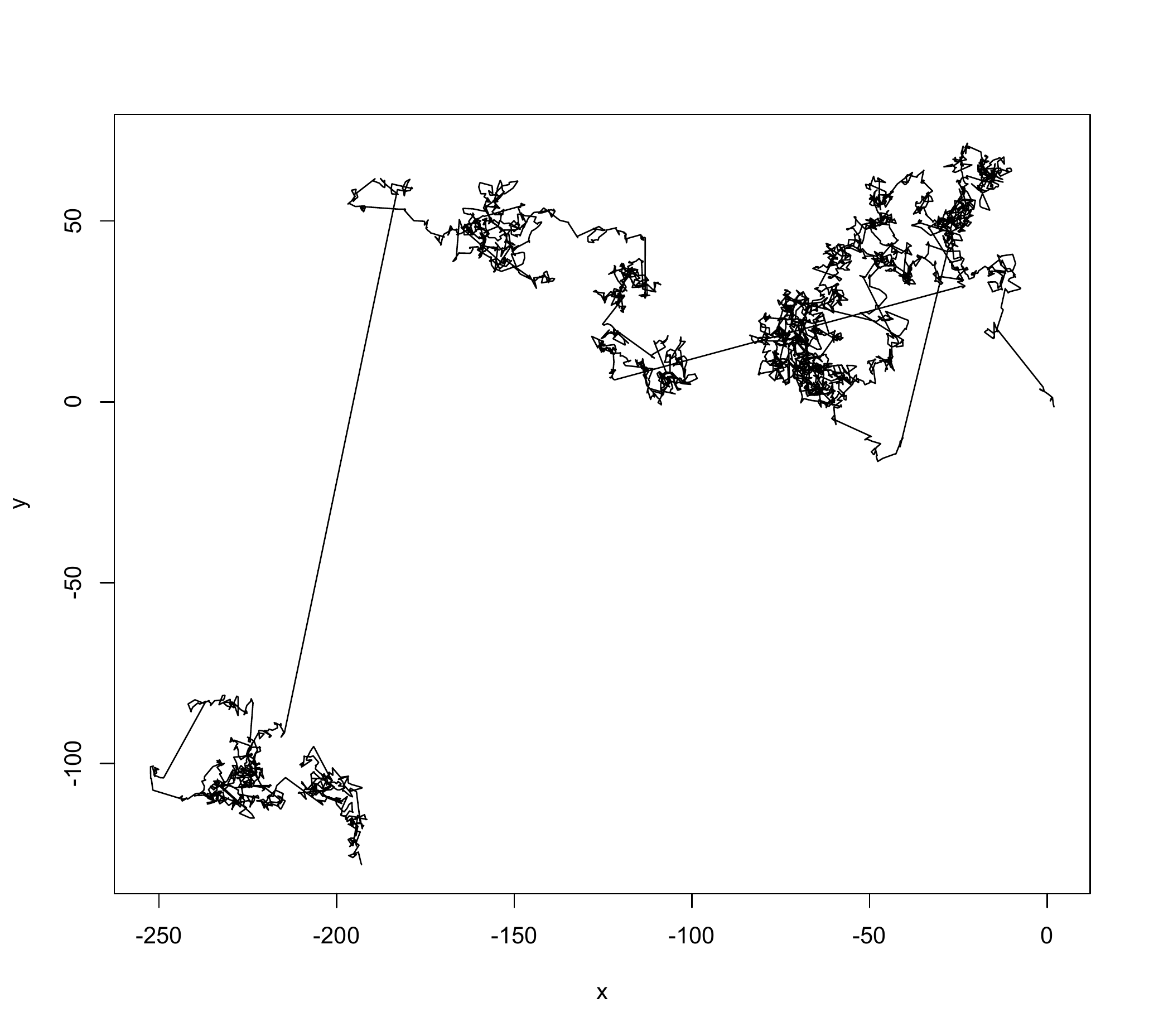}
  \includegraphics[width=.45\textwidth]{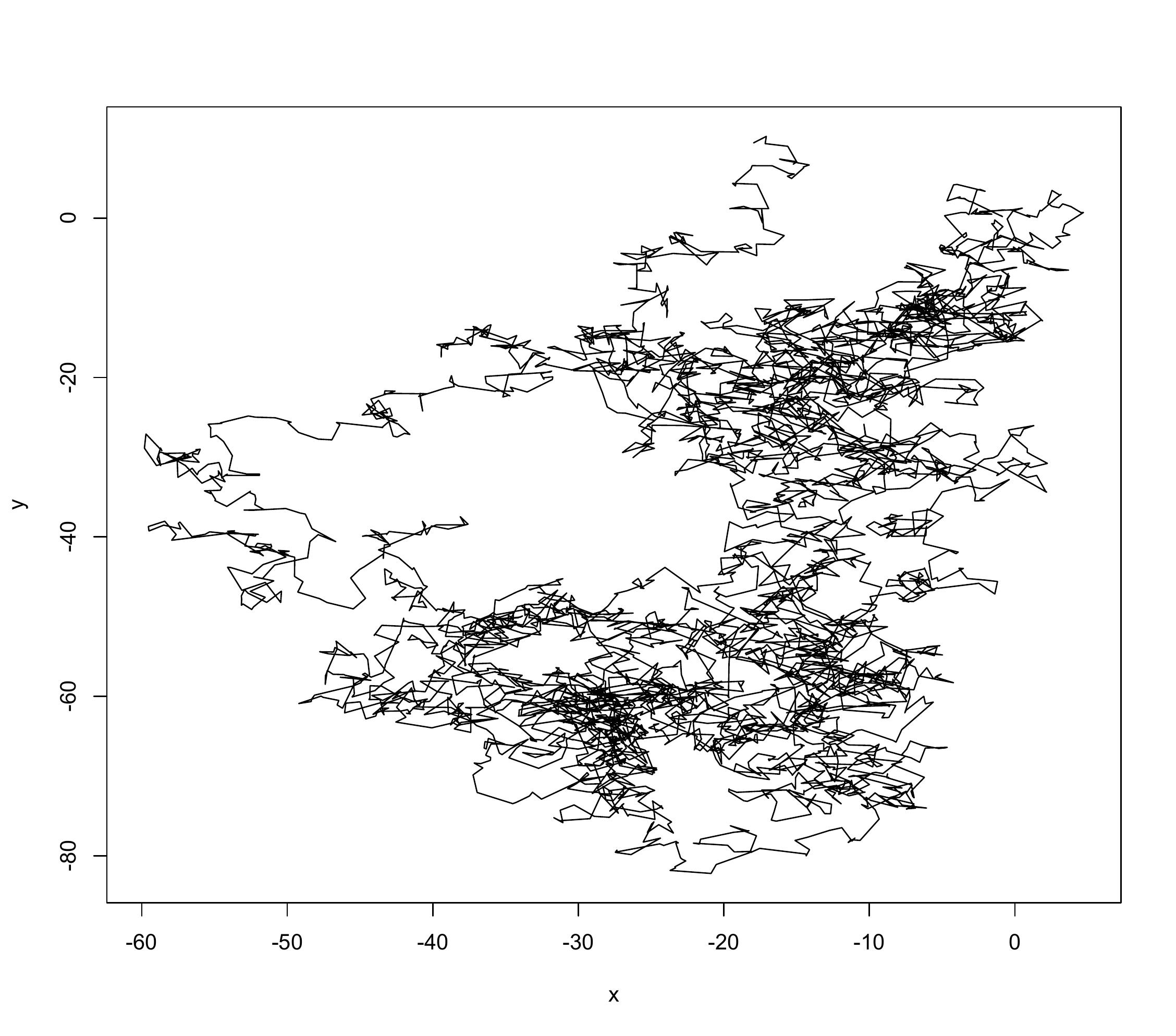}
\caption{Comparison of L\'evy motion and Brownian motion: \emph{(left)} A 2D standard isotropic stable L\'evy motion with $\alpha = 1.8$.  \emph{(right)} A 2D standard isotropic stable L\'evy motion with $\alpha = 2.0$ (Brownian motion).}
\label{fig:2dLevyBM_comparison}
\end{figure}

The heavy-tailed behavior of $\alpha$-stable random variables is very much in demand for modeling, although the infinite variance property may sometimes be undesirable, for physical or numerical reasons. Recently, \emph{tempered fractional calculus} has been developed to avoid this issue \cite{meerschaert_tempered}. {\color{chocolate} Tempered fractional diffusion equations model} particles that undergo \emph{tempered $\alpha$-stable L\'evy motion}, which is based on an exponentially tempered $\alpha$-stable density \cite{Koponen1995}. Roughly, this density exhibits the power-law decay up to a certain argument, then decays exponentially. The generator of tempered $\alpha$-stable L\'evy motion, i.e., the tempered fractional Laplacian, was discussed in \cite{deng_tempered_Laplacian}, where a Riesz basis Galerkin method was proposed to solve the Poisson problem with this operator. 

This characterization as an infinitesimal generator of ${X^\alpha_t}$ also lends itself to the stochastic (Monte Carlo) solution of boundary value problems involving the various fractional Laplacians in bounded domains. These formulas are typically referred to as Feynman-Kac (Dirichlet {\color{chocolate}boundary conditions}, \cite{feynman_kac_original}) or Brosamler (Neumann {\color{chocolate}boundary conditions}, after \cite{brosamler_original, lions_sznitman, hsu_reflecting_bm}) formulas. One such solution method is discussed in Section \ref{killed_levy_motion}, and its implementation is described in Section \ref{sec:wos}.

{\color{blue} We now give a brief outline of the probabilistic theory behind the relation of the fractional Laplacian to L\'evy processes. This will also result in the \emph{generator form} of the fractional Laplacian, which is related to the directional representation in the next Section \ref{sec:directional_unbounded}.
Following the text \cite{meerschaert_sikorskii}, we review the classical L\'evy-{\color{magenta}Khintchine} formula for generators of L\'evy processes and sketch how it results in the fractional Laplacian for the special case of isotropic $\alpha$-stable processes.} Given a L\'evy process $\{Z_t : t > 0\}$,  define the family of linear operators
\begin{align}
	T_t f(x) = \mathbb{E}[f(x-Z_t)], \hspace{10pt} t \geq 0,
\end{align}
for suitable functions $f$. Suppose that the characteristic function of the random variable $Z_1$ is $\mathbb{E}[e^{ik \cdot Z_1}] = e^{\psi(k)}$, where
\begin{align}
	\psi(k) &= ik \cdot a - \frac{1}{2}k \cdot B k + \int\left(e^{ik \cdot y} - 1 - \frac{ik \cdot y}{1+|y|^2}\right)\phi(dy).
\end{align}
Then $T_tf(x)$ defines a $C_0$-semigroup on $C_0(\mathbb{R}^d)$ with generator (\cite{meerschaert_sikorskii}, Theorem 6.26)
\begin{align}
\label{gen_form}
	Lf(x) = -a\cdot \nabla f(x) + \frac{1}{2} \nabla \cdot Q \nabla f(x) + \int\left(f(x-y)-f(x) + \frac{y \cdot \nabla f(x)}{1 + |y|^2}\right) \phi(dy),
\end{align}
where $\phi(dy)$ is some L\'evy measure (\cite{meerschaert_sikorskii}, Eq.\ (6.20)).  The domain of $L$ contains $\{f : f, f', f'' \in C_0(\mathbb{R}^d)\}$ (\cite{meerschaert_sikorskii}, Theorem 6.26). {\color{blue}This is the L\'evy-{\color{magenta}Khintchine} formula.} The first order differential term in this operator can be understood as the ``drift" term, the second order differential term can be understood as the {\color{chocolate} standard} diffusion, and the L\'evy measure describes the jumps of the L\'evy process (\cite{meerschaert_sikorskii}, p.\ 159).

{\color{blue} As we will now see,} if we choose $Z_t$ to be the standard isotropic $\alpha$-stable L\'evy motion, then \eqref{gen_form} reduces to the negative of the fractional Laplacian. The details of the following discussion can be found in \cite{meerschaert_sikorskii}, Examples 6.28 and 6.29. If  {\color{chocolate} $0 < \alpha < 1$} and $Z_t$ is a L\'evy process such that $Z_1$ has the characteristic function
\begin{align}
\label{characteristic_function}
	\mathbb{E}[e^{ik \cdot Z_1}] = \exp\left[-C\Gamma(1-\alpha) \int_{|\bm\theta|=1} (-ik \cdot \bm{\theta})^\alpha M(d\bm{\theta})\right],
\end{align}
then the generator of the corresponding stable semigroup can be written in the form 
\begin{align}
\label{stable_semigroup}
	Lf(x) = \int (f(x-y)-f(x))\phi(dy),
\end{align}
where 
\begin{equation}
\label{levy_phi}
\phi(dy) = \alpha Cr^{-\alpha-1}dr M(d\bm{\theta}) \text{ with $y = r\bm{\theta},$ $r > 0$, and $|\bm{\theta}|=1$.  }
\end{equation}
If $M(d{\bm \theta})$ is a uniform measure on the unit sphere, and if 
\begin{align}\label{char_constants}
	C^{-1} = B\Gamma(1-\alpha), \ \text{and} \ B = \cos(\pi \alpha/2)\int_{|{\bm{\theta}}|=1}| \theta_1|^\alpha M(d{\bm{\theta}}),
\end{align}
then the right-hand side of \eqref{characteristic_function} reduces to $\exp[-|k|^\alpha]$, so that $Z_t = X_t^\alpha$, the standard isotropic $\alpha$-stable L\'evy motion\footnote{
	The general isotropic $\alpha$-stable random variable has characteristic function $\exp[-c^\alpha|k|^\alpha]$, where $c$ is the scale parameter. The standard $\alpha$-stable random variable has scale parameter $c = 1$. Note that the standard isotropic 2-stable random variable corresponds to a normal random variable with characteristic function $\exp\left[\frac{-\sigma^2|k|^2}{2}\right]$ and variance 
{\color{chocolate} $\sigma^2 = 2$}.
}, 
and $L = -(-\Delta)^{\alpha/2}$, the negative of the fractional Laplacian (\cite{meerschaert_sikorskii}, Example 6.24). Indeed, \eqref{stable_semigroup} reduces to \eqref{E:riesz_Laplacian_rn} after a simple change of variables (see \cite{defterli2015}, Section 4). Now if $1 < \alpha < 2$, and $Z_1$ has the characteristic function
\begin{align}
\label{charfun}
	\mathbb{E}[e^{ik\cdot Z_1}] = \exp\left[C\frac{\Gamma(2-\alpha)}{\alpha-1}\int_{|\bm\theta|=1}(-ik\cdot\bm{\theta})^\alpha M(d\bm{\theta})\right],
\end{align}
then the generator of the corresponding stable semigroup can be written in the form
\begin{align}
\label{stable_semigroup2}
	Lf(x) = \int(f(x-y)-f(x) + y\cdot \nabla f(x))\phi(dy),
\end{align}
where $\phi$ is defined as in \eqref{levy_phi}, and the constants $B$ and $C$ are defined as in \eqref{char_constants}. {\color{chocolate}In this case, if $M(d\bm{\theta})$ is a uniform measure on the unit sphere, then $L = (-\Delta)^{\alpha/2}$, which has a different sign than in the case $0 < \alpha < 1$. The sign change is due to the fact that $B > 0$ for $0 < \alpha < 1$, and $B < 0$ for $1 < \alpha < 2$. }{\color{blue} Equations \eqref{stable_semigroup} and \eqref{stable_semigroup2} yield alternate forms of the fractional Laplacian \eqref{E:riesz_Laplacian_rn}, known as the generator forms, without using the principal value. The generator form \eqref{gen_form} of the fractional Laplacian is useful for probabilists working in this area. }

\subsubsection{Directional Representation}\label{sec:directional_unbounded}

Another integral characterization of the fractional Laplacian, which we refer to as the directional representation, can be found in \cite{samko1993fractional} (Eq. (26.24)). The operator is written as
\begin{equation}\label{def:directional_uniform}
(-\Delta)^{\alpha/2}u(x)=C_{\alpha,d}\int_{|\boldsymbol{\theta}|=1}D^{\alpha}_{\boldsymbol{\theta}}u(x)d\boldsymbol{\theta}, \quad x,\boldsymbol{\theta}\in \mathbb{R}^d, \alpha\in (0,1) \cup (1,2],
\end{equation}
where the scaling constant before the integral is \cite{samko1993fractional,pang2013gauss}
\begin{equation}\label{const}
C_{\alpha,d}=\frac{\Gamma(\frac{1-\alpha}{2})\Gamma(\frac{d+\alpha}{2})}{2{\pi}^\frac{1+d}{2}}.
\end{equation}
The Fourier transform of \eqref{def:directional_uniform} {\color{chocolate} corresponds to multiplication of the Fourier transform of $u$ by} \eqref{characteristic_function} when $\alpha < 1$ and by \eqref{charfun} when $\alpha > 1$. 
Here, $D_{\boldsymbol{\theta}}^{\alpha}(\cdot)$ is the Riemann-Liouville fractional directional derivative \cite{MEERSCHAERT2006181} given by
{\color{forest}
\begin{equation}\label{directional_integral}
D_{\boldsymbol{\theta}}^{\alpha}(\cdot)= ( \boldsymbol{\theta} \cdot \nabla ) I_{\boldsymbol{\theta}}^{1-\alpha}(\cdot)
\quad
\text{ for $0 < \alpha < 1$;}
\quad
D_{\boldsymbol{\theta}}^{\alpha}(\cdot)= ( \boldsymbol{\theta} \cdot \nabla )^2 I_{\boldsymbol{\theta}}^{2-\alpha}(\cdot)
\quad
\text{ for $1 < \alpha < 2$; }
\end{equation}
where $\boldsymbol{\theta} \cdot \nabla = \sum_{i = 1}^d \theta_i \frac{\partial}{\partial x_i}$ is the directional derivative, thus $(\boldsymbol{\theta} \cdot \nabla)^2 = \sum_{j=1}^d \sum_{i = 1}^d \theta_i \theta_j \frac{\partial^2}{\partial x_j \partial x_i}$}, and the fractional directional integral $I_{\boldsymbol{\theta}}^{\beta}(\cdot)$ is defined by (for $\beta\in(0,1)$)
\begin{equation}\label{directional_integral0}
I_{\boldsymbol{\theta}}^{\beta}u(x)=\frac{1}{\Gamma(1-\beta)}\int_0^{+\infty}\varsigma^{-\beta}u(x-\varsigma\boldsymbol{\theta})d\varsigma.
\end{equation}
{\color{forest}
Note that \eqref{def:directional_uniform} excludes the case $\alpha = 1$. 
When $\alpha = 1$, the directional representation is more complicated; for the one-dimensional case, see \cite{kelly2018anomalous}. 
}

{\color{chocolate}
We now explicitly write out the directional representation \eqref{def:directional_uniform} for one and two dimensions; these formulas will be 
used later in this article.  
In one dimension, using the identity $\Gamma(x)\Gamma(1-x) = \pi/\sin(\pi x)$, we have
\begin{equation}
\Gamma\left(\frac{1-\alpha}{2}\right)
\Gamma\left(\frac{1+\alpha}{2}\right)
=
\Gamma\left(\frac{1-\alpha}{2}\right)
\Gamma\left(1-\frac{1-\alpha}{2}\right)
=
\frac{\pi}{\sin\left(\pi\frac{1-\alpha}{2} \right)}
=
\frac{\pi}{\cos\left(\frac{\pi\alpha}{2}\right)},
\end{equation}
and so
\begin{equation}
\begin{split}
(-\Delta)^{\alpha/2}u(x)
&= 
\frac{1}{2\cos\left( \frac{\pi\alpha}{2}\right)}
\left[ D_{\theta = -1}^{\alpha}u(x) + D_{\theta = 1}^{\alpha} u(x) \right]\\
&= \frac{1}{2\cos\left( \frac{\pi\alpha}{2}\right)}
\left[_{-\infty}^{RL}D_x^{\alpha} u(x) + {}^{RL}_{\ \ x}D^\alpha_\infty u(x) \right],
\end{split}
\end{equation}
where $_{-\infty}^{RL}D_x^{\alpha}$ and ${}^{RL}_{\ \ x}D^\alpha_\infty$ are the well-known one-dimensional left- and right-sided Riemann-Liouville fractional derivatives, respectively \cite{yang2010numerical}. 
Note that, using the Fourier transforms of these operators, 
$
\mathcal{F}[(-\Delta)^{\alpha/2} u] =   \frac{1}{\cos(\alpha\pi/2)}
\left[ \frac{1}{2} (-i\xi)^{\alpha} 
+ \frac{1}{2} (i \xi)^{\alpha} \right] \hat{u} 
= |\xi|^{\alpha} \hat{u}$ for $\alpha \neq 1$. 
In two dimensions and for $\alpha \neq 1$, the fractional Laplacian can be written in polar form as
\begin{equation}
(-\Delta)^{\alpha/2}u(x)=C_{\alpha,2}\int_0^{2\pi}D_{{\theta}}^{\alpha} u(x)d{\theta},
\end{equation}
where $D_{{\theta}}^{\alpha}$ is a derivative along the direction $[\cos(\theta),\sin(\theta)]$.
}

Meerschaert et al. \cite{meerschaert1999multidimensional} extended the above operator \eqref{def:directional_uniform} to an anisotropic version, which they called the general (asymmetric) fractional derivative operator $\nabla^{\alpha}_M$,
\begin{equation}\label{general_directional_laplacian}
\nabla^{\alpha}_M u(x)=\int_{|\boldsymbol{\theta}|=1}D^{\alpha}_{\boldsymbol{\theta}}u(x)M(d\boldsymbol{\theta}), \quad x,\boldsymbol{\theta}\in \mathbb{R}^d, \alpha\in (0,1) \cup (1,2], 
\end{equation}
where $M(d\boldsymbol{\theta})$ is {\color{chocolate}an} arbitrary probability measure on the unit sphere $\{|\boldsymbol{\theta}|=1\}$. The motivation for the definition extension is to generate ``the full range of L\'evy-stable motions". For more details, see Equation \eqref{characteristic_function} and the corresponding discussion. {\color{chocolate} The study of this more general operator is beyond the scope of the current article. However, it is an interesting and pertinent topic for future extensions of many of the results and methods discussed in this work.}
It should be noted that the {\color{magenta}Riesz} fractional Laplacian is recovered if the measure $M(d\boldsymbol{\theta})$ is uniform, namely, $M(d\boldsymbol{\theta})=d\boldsymbol{\theta}/S_d$ where $S_d=2\pi^{d/2}/\Gamma(d/2)$ is the surface area of the unit sphere in $\mathbb{R}^d$. In this case, the two types of operators are related by
\begin{equation}
(-\Delta)^{\alpha/2}u(x)=\frac{\Gamma\left(\frac{d+\alpha}{2}\right)\Gamma\left(\frac{1-\alpha}{2}\right)}{\Gamma\left(\frac{d}{2}\right)\sqrt{\pi}}\nabla^{\alpha}_M u(x).
\end{equation}

\subsection{Fractional Laplacians on Bounded Domains}

We have discussed several characterizations and formulas for the 
fractional Lapacian in ${\mathbb{R}^d}$, which are all equivalent
in that setting.  If we {\color{chocolate}apply} these different formulas
{\color{chocolate} on} a bounded domain ${\Omega \subset \mathbb{R}^d}$, 
certain equivalences break down, which leads to the definition of several distinct
fractional Laplacians in a bounded domain. Some researchers have defined other fractional Laplacians on a bounded domain directly, instead of restricting the definition for $\mathbb{R}^d$ \cite{chen_pang_2016}. Here, we survey the \emph{Riesz} fractional Laplacian (sometimes called the \emph{integral} fractional Laplacian), the \emph{directional} fractional Laplacian, and the \emph{spectral} fractional Laplacian. We focus on the following topics.
\begin{itemize}
\item
Appropriate boundary conditions for the Poisson problem
\begin{align}\label{E:intro_poisson_problem}
\begin{split}
(-\Delta)^{\alpha/2} u &= f \text{ in } \Omega.
\end{split}
\end{align}
{\color{chocolate}
The spectral fractional Laplacian admits  
local boundary conditions on $\partial \Omega$; the regional fractional Laplacian has also been considered with local boundary conditions. On the other hand,} the Riesz (and directional) fractional Laplacians require \emph{exterior} boundary conditions on $\mathbb{R}^d \setminus \Omega$. 
\item
Connection of each fractional Laplacian with given boundary conditions to a stochastic process. The standard negative Laplacian $-\Delta$ in a bounded domain is the negative generator of \emph{stopped} \cite{kolokoltsov_book} Brownian motion if taken with Dirichlet BCs \cite{privault_book,Freidlin}, and \emph{reflected} Brownian motion if taken with Neumann BCs \cite{brosamler_original,hsu_reflecting_bm, Bencherif2009}.
Similarly,
the fractional Laplacians are negative generators of certain stopped\footnote{\color{chocolate}
Many results in the probability literature (e.g., \cite{Song}) are stated and proved for \emph{killed} processes rather than \emph{stopped} processes. Without going into details here, we note that the resulting semigroups are equivalent when applied to functions $f$ that vanish outside the domain, in the sense that 
$\mathbb{E}_x \left[ f(X_t) \right]$ is the same whether $X_t$ is the killed or the stopped process. Here, we are assuming that $f(\partial)=0$ where $\partial$ is the cemetery point specified for the killed process. 
}
or reflected L\'evy motions. 
This is useful for physical interpretation as well as stochastic solution methods.
\item
{\color{chocolate} 
Well-posedness and regularity properties of the fractional} Poisson problem \eqref{E:intro_poisson_problem} with appropriate boundary conditions.
\end{itemize}

{
\color{chocolate}
We do not consider the Laplacian with periodic boundary conditions (i.e., on the torus $\mathbb{T}^d$) in this article. The construction of such an operator is unambiguous,  and follows the construction of the fractional Laplacian in $\mathbb{R}^d$ in Section \ref{Spectral/Fourier Definition} except that the spectrum is discrete, so the operator is given by a Fourier series. For a basic discussion, see, e.g., \cite{hermann_physics}, and for discussions of extension characterizations as well as well-posedness and regularity, see \cite{stinga_1_torus, stinga_2_transference}. 
}

\subsection{Riesz Fractional Laplacian}\label{sec:Riesz}
\subsubsection{Definition}\label{RieszDef}
One approach to defining the fractional Laplacian on a bounded domain ${\Omega}$ is 
to apply the real space formula \eqref{E:riesz_Laplacian_rn} to functions on ${\Omega}$. 
This leads to the \emph{Riesz} fractional Laplacian in $\Omega$.
Let us first consider Dirichlet boundary conditions. 
Since formula \eqref{E:riesz_Laplacian_rn} requires values of ${u}$ on all of 
${\mathbb{R}^d}$, an \emph{exterior} boundary condition 
\begin{equation}\label{E:dirichlet_exterior_bc}
u = g \text{ in } 
\mathbb{R}^d \setminus \Omega
\end{equation}
is required, even to define the Riesz Laplacian within $\Omega$. 
For functions $u$ that satisfy \eqref{E:dirichlet_exterior_bc}, the \emph{Riesz} fractional Laplacian is defined for $x \in \Omega$ by
\begin{align}
\label{E:riesz_definition_omega}
\begin{split}
(-\Delta)^{\alpha/2} u(x) &= C(d,\alpha) 
\text{   p.v.}\int_{\color{forest} \mathbb{R}^d} \frac{u(x) - u(y)}
{|x-y|^{d+\alpha}}dy \\
&= 
C(d,\alpha)
\left[
\text{   p.v.}\int_{\Omega} \frac{u(x) - u(y)}
{|x-y|^{d+\alpha}}dy
+
\int_{\mathbb{R}^d \setminus \Omega} \frac{u(x) - g(y)}
{|x-y|^{d+\alpha}}dy
\right].
\end{split}
\end{align} 
The Riesz fractional Laplacian depends directly on $\Omega$ and the exterior boundary values $g$.
{\color{magenta} It also referred to in the literature as the \emph{integral} fractional Laplacian \cite{nochetto_three} or as the \emph{restricted} (Dirichlet) fractional Laplacian \cite{vazquez2014recent, grubb_spectral}; for consistency, we do not use these terms in this article.}

\subsubsection{Boundary Conditions: Dirichlet vs. Neumann}

The Neumann condition for the Riesz Laplacian is, at the time of this writing, an area of active development. A type of exterior fractional normal derivative is needed to specify $u$ on $\mathbb{R}^d \setminus \Omega$, but there is no widely studied or accepted definition. Recently, a fractional Neumann operator was proposed in \cite{gunzburger_neumann}, and the properties were studied in detail in \cite{Dipierro14}. Another approach to defining reflecting boundary conditions based on mass conservation for diffusion in one dimension was explored in \cite{meerschaert_bcs, Kelly2sided}.

The Poisson problem with Dirichlet boundary conditions
\begin{align}\label{E:dirichlet_poisson_riesz}
\begin{split}
(-\Delta)^{\alpha/2} u &= f \text{ in $\Omega$ }, \\
u &= g \text{ in $\mathbb{R}^d \setminus \Omega$ }
\end{split}
\end{align}
has been extensively studied in the $g \equiv 0$ case \cite{Acosta, rosoton_regularity, AinsworthGlusa2017_TowardsEfficientFiniteElement, AinsworthGlusa2017_AspectsAdaptiveFiniteElement, Grubb2015_FractionalLaplaciansDomainsDevelopment}, but literature based on the nonzero case is more recent and limited \cite{Grubb2015_FractionalLaplaciansDomainsDevelopment, kyprianou2016unbiasedwalk, borthagaray2018finite, antil2019external}.
For the case of zero 
exterior Dirichlet condition $g \equiv 0$ in \eqref{E:dirichlet_poisson_riesz}, finite element algorithms have been developed in 
\cite{Acosta},
\cite{AinsworthGlusa2017_TowardsEfficientFiniteElement}, and in particular the
adaptive finite element scheme of \cite{AinsworthGlusa2017_AspectsAdaptiveFiniteElement} has been used for the computations of this paper. {\color{magenta}For the case of nonzero exterior Dirichlet condition}, in \cite{kyprianou2016unbiasedwalk} a Monte Carlo
method based on the Feynman-Kac formula for the problem \eqref{E:dirichlet_poisson_riesz} 
was developed and studied. This method is described in more detail in Section \ref{sec:wos}. {\color{magenta}The radial basis function collocation method introduced in Section \ref{RBFM} (based on discretizing the directional representation of the Riesz fractional Laplacian) readily admits nonzero exterior Dirichlet condition as well. In \cite{acosta2018finite}, a finite element method for nonzero exterior Dirichlet conditions was introduced in which the exterior condition was enforced through a Lagrange multiplier formulation. Another approach for the case of nonzero exterior condition, with application to exterior control problems, is that of \cite{antil2019external}, which implemented nonzero exterior Robin  conditions and used this to approximate the solution to the nonzero exterior Dirichlet problem.}

\subsubsection{Stopped L\'evy Motion}\label{killed_levy_motion}
The Riesz fractional Laplacian with Dirichlet boundary conditions is the generator of the stopped $\alpha$-stable L\'evy motion. The stopped L\'evy motion is defined as follows \cite{kolokoltsov_book, privault_book, Freidlin}. Given a stable L\'evy motion $X_t^\alpha$ and a bounded domain $\Omega$, we define the stopping time (or exit time) 
\begin{align}\label{stopping_time}
	\sigma_\Omega = \inf \{t : X_t^\alpha \notin \Omega\}.
\end{align}
Then, the stopped L\'evy process is defined as $X^\alpha_{t \wedge \sigma_\Omega}$.
Unlike stopped Brownian motion, which has almost surely continuous paths that are stopped at the boundary, the paths of $\alpha$-stable L\'evy motion are discontinuous and almost surely exit the domain by a jump. Thus, the paths of stopped $\alpha$-stable L\'evy motion pass into the exterior $\mathbb{R}^d \setminus \Omega$ where they are stopped immediately.

In \cite{kyprianou2016unbiasedwalk}, this stochastic connection was exploited to prove a Feynman-Kac for the 
Poisson problem.
\begin{equation}\label{fra_Poisson}
\begin{split}
(-\Delta)^{\alpha/2}u(x) & =  f(x), \quad x\in \Omega\subset \mathbb{R}^d,\\
 u(x)  & =  g(x), \quad  x \in \mathbb{R}^d \setminus \Omega.
\end{split}
\end{equation}
Under mild conditions on $f$ and $g$, the following Feynman-Kac formula was proven:
\begin{equation}\label{FK-formula}
u(x)=\mathbb{E}_x\left[g(X_{\sigma_\Omega}^\alpha)\right]+\mathbb{E}_x\left[\int_0^{\sigma_\Omega}f(X_s^\alpha)ds\right], \ x \in \Omega,
\end{equation}
where $X^\alpha_t$ is the symmetric $\alpha$-stable process, $\mathbb{E}_x(\cdot)$ denotes the expectation with respect to all the sample paths with the initial location $x$, and $\sigma_\Omega=\inf\{t>0: X_t\notin \Omega\}$ is the first exit time of the sample path. 

In addition to a proof, in \cite{kyprianou2016unbiasedwalk} a walk-on-spheres (WOS) method for much faster computation of \eqref{FK-formula} was introduced and analyzed. This speeds up the implementation of \eqref{FK-formula} by replacing the simulation of exact stopped paths in $\Omega$ by a series of maximal inscribed spheres. Conceptually, the WOS method is based on the isotropy of the problem \eqref{fra_Poisson} as well as the the analytic formula for mean-exit time of $X^\alpha_t$ on a sphere \cite{getoor_mean_exit}. In the Brownian motion case, the center of each sphere is chosen by sampling a uniform distribution on the boundary of the previous sphere, and the procedure terminates when the process comes within some tolerance $\varepsilon$ of the boundary of $\Omega$. In the $\alpha$-stable case, the center of each sphere is chosen by sampling a distribution (which is described in detail in Section \ref{sec:wos}) on the exterior of the previous sphere, and the procedure terminates when the process jumps outside the domain $\Omega$.
This is because, in the $\alpha$-stable case, the sample path may exit the sphere by a jump, instead of by passing through the boundary; convergence in finite steps was proven in \cite{kyprianou2016unbiasedwalk}.

\subsubsection{{\color{red}Well-posedness and Regularity}}\label{sec:riesz_regularity}
In this section, we consider the problem 
\begin{align}\label{riesz_regularity_problem}
\begin{split}
(-\Delta)^{\alpha/2} u &= f \text{ in $\Omega$,}\\
u &= 0 \text{ in $\mathbb{R}^d\setminus\Omega$. }
\end{split}
\end{align}
{\color{red}Using the Lax-Milgram Lemma, {\color{forest}for a bounded Lipschitz domain $\Omega$,} it is easy to show that this problem is well-posed if $f(x)\in \mathbb{H}^{-\alpha/2}(\Omega)$ and the resulting solution $u(x)\in \mathbb{H}^{\alpha/2}(\Omega)$. See, for example, \cite{AinsworthGlusa2017_TowardsEfficientFiniteElement,nochetto_three,Acosta}. 
{\color{blue}For the definitions of the Sobolev space $\mathbb{H}^{s}$, which will also feature in the following regularity results, see Appendix \ref{sobolev_spaces}.}

As for regularity,}
our intention is not to give a complete survey of all the known regularity results for the fractional Laplacian. Instead, we mention two results that give an indication of the regularity for the Riesz fractional Laplacian, for the case of the zero boundary condition. These properties should be compared with the regularity of the spectral fractional Laplacian (see Section \ref{sec:spectral_regularity}). The regularity up to the boundary is a key difference between the two definitions. 

The H\"older regularity for the solution of the problem \ref{riesz_regularity_problem}
was studied in \cite{rosoton_regularity}. The authors were motivated by the exact solution for $f \equiv 1$ in the ball 
$B_r(x_0)$ centered at $x_0$ of radius  $r$:
\begin{align}\label{E:riesz_counterexample}
\begin{split}
(-\Delta)^{\alpha/2} u &= 1 \text{ in $B_r(x_0)$, } \\
u &= 0 \text{ in $\mathbb{R}^d\setminus B_r(x_0)$, }
\end{split}
\end{align}
which has solution
\begin{equation*}
u(x) = \frac{2^{-\alpha}\Gamma(n/2)}{\Gamma\left(\frac{n+\alpha}{2}\right)\Gamma(1+\alpha/2)}
\left( r^2 - |x-x_0|^2 \right)^{\alpha/2} \text{ in $B_r(x_0)$}.
\end{equation*}
This solution exhibits strict $C^{\alpha/2}$ regularity up the boundary. The authors proved the following result. 
{\color{forest} Let $\Omega$ be a bounded $C^{1,1}$ domain} and $\delta(x) = d(x,\partial\Omega)$ denote the distance to the boundary,
Then $g \in L^\infty(\Omega)$ implies that $u/\delta^{\alpha/2}$ can be extended to a $C^{r}(\overline{\Omega})$ function for some 
$r < \text{min}(\alpha/2, 1-\alpha/2)$, and 
\begin{equation*}
\|u/\delta^{\alpha/2} \|_{C^{r}(\overline{\Omega})} \le C \| g\|_{L^\infty(\Omega)}.
\end{equation*} 
This result is sharp for the problem \eqref{E:riesz_counterexample}. 
{\color{magenta}Boundary value problems in which weighted boundary conditions (on the trace of $u/\delta^{\alpha/2}$ rather than $u$) are specified on $\partial\Omega$ are studied in 
\cite{Grubb2015_FractionalLaplaciansDomainsDevelopment}.} 

The following theorem relates to the Sobolev regularity of \eqref{riesz_regularity_problem}. This form is presented nicely in \cite{AinsworthGlusa2017_AspectsAdaptiveFiniteElement} and is based on the results of \cite{Grubb2015_FractionalLaplaciansDomainsDevelopment} and \cite{acosta_bersetche_borthagaray}. {\color{magenta} We mention that although the formulation in \cite{AinsworthGlusa2017_AspectsAdaptiveFiniteElement} considers $s \geq -\alpha/2$, \cite{Grubb2015_FractionalLaplaciansDomainsDevelopment} establishes the first case below for $-1/2 < \alpha/2 + s < 1/2$. Regularity for the inhomogeneous problem \eqref{E:dirichlet_poisson_riesz} is also studied in [72], although we restrict ourselves to the case below.}
\begin{theorem}{\cite{AinsworthGlusa2017_AspectsAdaptiveFiniteElement,Grubb2015_FractionalLaplaciansDomainsDevelopment,acosta_bersetche_borthagaray}}
\label{Riesz_regularity}
{\color{forest}Let $\partial \Omega \in C^\infty$,} $f \in H^s(\Omega)$ for $s \geq -\alpha/2$ and $u \in \mathbb{H}^{\alpha/2}(\Omega)$ be the solution of the fractional Poisson problem \eqref{riesz_regularity_problem}. Then
\begin{align}
	u \in \begin{cases} H^{\alpha + s}(\Omega) & \text{if} \ 0 < \alpha/2 + s < 1/2, \\
	H^{\alpha/2+1/2-\varepsilon}(\Omega) \ \forall \varepsilon > 0 & \text{if} \ 1/2 \leq \alpha/2 + s.\end{cases}
\end{align}
\end{theorem}

An interesting feature of the Sobolev regularity for the Riesz fractional Laplacian is that if $\alpha/2 + s \geq 1/2$, the global regularity of $u$ need not improve if $s$, the regularity of $f$, is increased. Rather, $s$ merely improves the interior regularity of $u$. This was proven in \cite{Grubb2015_FractionalLaplaciansDomainsDevelopment}, where it was shown that for a smooth boundary $\partial\Omega$, $f \in H^s(\Omega)$ implies that $u \in H^{ \alpha +s}_\text{loc}(\Omega)$. Therefore, the global regularity (i.e., regularity up to the boundary) for the Riesz fractional Laplacian is in contrast to the global regularity for the spectral fractional Laplacian (see Section \ref{sec:spectral_regularity}). In the case when $f \in L^2(\Omega)$ for the spectral Laplacian, $u \in \mathbb{H}^\alpha(\Omega)$ for all $\alpha \in (0,2)$, while this is only true if $\alpha < 1$ for the Riesz fractional Laplacian. In the case when $f \in H^s(\Omega)$ and $s > 1/2$, then $u \in \mathbb{H}^{s+\alpha}(\Omega)$ for the spectral fractional Laplacian provided $f \equiv 0$ on the boundary $\partial \Omega$, while for the Riesz definition, $u \in H^{\alpha/2 + 1/2 - \varepsilon}(\Omega)$, independent of $s$. Even if $f \not\equiv 0$ on $\partial \Omega$, the regularity of $u$ for the spectral definition improves as $\alpha +1/2 - \varepsilon$, while in the Riesz case, it improves as $\alpha/2 + 1/2 - \varepsilon$.

\subsection{Directional Fractional Laplacian}\label{directional}

We can also apply the directional fractional Laplacian $(-\Delta)^{\alpha/2}_M$ to functions on $\Omega$ with an exterior boundary condition, which is applied in the same way as for the Riesz definition. Although the directional characterization was motivated in \cite{meerschaert1999multidimensional} by the desire to capture anisotropic anomalous diffusion, in this work, we always choose $M(d\boldsymbol{\theta})=d\boldsymbol{\theta}/S_d$ where $S_d=2\pi^{d/2}/\Gamma(d/2)$ is the surface area of the unit sphere in $\mathbb{R}^d$, so that our computational results for this definition can be compared to those of the Riesz definition. With this choice of measure, the associated L\'evy process is symmetric stable motion that is stopped upon exiting the domain $\Omega$, which is the same process as the one associated with the Riesz fractional Laplacian. A nice proof {\color{chocolate} of this equivalence} can be  found in \cite{chen_meerschaert_2012}, Lemma 4.1. Hence the discussions of Section \ref{killed_levy_motion} and \ref{sec:riesz_regularity} also apply to the directional fractional Laplacian with our choice of integration measure.

\subsection{Spectral Fractional Laplacian}\label{spectral}

In this article, we have split the discussion of the spectral fractional Laplacian into two: the case of zero boundary condition and the case of nonzero boundary conditions. This is due to the vast body of work that considers only the homogeneous spectral fractional Laplacian (i.e., the zero boundary condition case), as only recently was the inhomogeneous version considered. In deference to the fact that many results in the literature apply only to the homogeneous case, we have taken this approach to avoid misleading the reader. The results reported in Sections \ref{SpectralDef} and \ref{sec:spectral_regularity} on the homogeneous spectral fractional Laplacian should not be taken to apply in the case of nonzero boundary conditions without modification.

The inhomogeneous spectral fractional Laplacian has been considered by Antil et al. \cite{AntilPfeffererRogovs} (see Section \ref{inhom_series}) and Cusimano et al. \cite{Cusimano2017} (see Section \ref{sec:heat_semigroup}). In Section \ref{sec:heat_semigroup}, we point out that the definitions considered in \cite{AntilPfeffererRogovs} and \cite{Cusimano2017} are the same, and are essentially a harmonic lifting to the homogeneous spectral fractional Laplacian. Moreover, we show that this amounts to another natural definition in terms of the inverse spectral fractional Laplacian. 

\subsubsection{Zero Boundary Condition}\label{SpectralDef}
The spectral approach to defining the factional Laplacian in $\Omega$ is to start with the standard negative Laplacian $-\Delta$ on that domain and take the spectral power defined by \eqref{E:spectral_general}. Let us first consider the case of zero boundary conditions. {\color{blue} We will take $\mathcal{D}(-\Delta)$ to be the subspace of $H^2(\Omega)$ consisting of smooth functions with zero Dirichlet or zero Neumann boundary values\footnote{\color{chocolate}
We remind the reader that, for using the spectral theorem as in Section \ref{Spectral/Fourier Definition}, $\mathcal{D}(-\Delta)$ need not be all of $\mathcal{H}$, or even a ``maximal'' domain of definition of $-\Delta$ within $\mathcal{H}$, but rather a more convenient, dense subspace of $\mathcal{H}$ (a \emph{core} for $-\Delta$) . 
}, depending on the choice of boundary condition.}

The spectrum $\sigma(-\Delta)$, which is now discrete, depends on the domain 
$\Omega$ and on whether Dirichlet or Neumann conditions are used. 
The spectrum of the Laplacian in the Dirichlet case is defined by the problem, 
\begin{align}
\begin{split}\label{E:eig_dirichlet_zero}
-\Delta e_k &= \lambda_k e_k \text{ in $\Omega$} \\
e_k &= 0 \text{ on $\partial\Omega$},
\end{split}
\end{align}
and in the Neumann case by 
\begin{align}\label{E:eign_neumann_zero}
\begin{split}
-\Delta e_k &= \lambda_k e_k \text{ in $\Omega$} \\
\frac{\partial e_k}{\partial n} &= 0 \text{ on $\partial\Omega$}.
\end{split}
\end{align}
The $\lambda_k$ are the eigenvalues, and $e_k$ the eigenfunctions, 
of ${-\Delta}$ with zero Dirichlet or Neumann boundary condition, respectively.  
The spectral decomposition \eqref{E:spectral_theorem} then reads 
\begin{equation}\label{standard_laplacian_series}
-\Delta u(x) = \sum_{k=1}^\infty \lambda_k 
(
u, e_k
)_{L^2_\Omega}
e_k(x).
\end{equation}
We remark that this identity is valid on $\mathcal{D}(\Omega)$, and can be extended by continuity to $H_0^2(\Omega)$ in the Dirichlet case, or {\color{blue}$H_{\partial u/\partial n = 0}^2(\Omega) = \left\{ u \in H^2(\Omega) \ : \ \frac{\partial u}{\partial n}\big|_{\partial \Omega} = 0\right\}$} in the Neumann case. It is not in general true for functions $u$ with nonzero boundary values (Dirichlet or Neumann). For example, in one dimension with $-\Delta = -\frac{\partial^2}{\partial x^2}$, and $(\lambda_k, e_k)$ from \eqref{E:eig_dirichlet_zero} on $[0,2\pi]$ and $u=\cos(x)$, the equation \eqref{standard_laplacian_series} does not hold. In fact, the series on the right-hand side diverges. 

Thus, on $\mathcal{D}(-\Delta)$, in accordance with \eqref{E:spectral_general}, the \emph{spectral fractional Laplacian} is defined by
\begin{equation}\label{E:spectral_omega}
(-\Delta)^{\alpha/2} u (x) := \sum_{k=1}^\infty \lambda_k^{\alpha/2}
( 
u, e_k
)_{L^2_\Omega}
e_k(x).
\end{equation}
As usual, by continuity, this can be extended to an operator on $\mathbb{H}^\alpha(\Omega)$ (see Appendix \ref{sobolev_spaces}).
The spectral fractional Laplacian is nonlocal on the interior of $\Omega$ for noninteger $\alpha \in (0,2)$. 
We see that to compute the inner product $( u, e_k )_{L^2_\Omega}$, it suffices for 
$u$ to be defined on the interior of $\Omega$. 
{\color{chocolate}{ Unlike the Riesz fractional Laplacian, the spectral fractional Laplacian requires}} no information about $u$ on the exterior $\mathbb{R}^d \setminus \Omega$. 
Thus, from a conceptual viewpoint, in boundary value problems the spectral fractional Laplacian
{\color{chocolate} admits} the same type of boundary conditions as the standard, local Laplacian $-\Delta$. 
{\color{chocolate}
Precise conditions and references for rigorous proofs for well-posedness of boundary value problems for the spectral fractional Laplacian with such local boundary conditions (Dirichlet and Neumann) are discussed in Section \ref{sec:spectral_regularity}.
}

\subsubsection{Other representations}\label{homo_representations}
In this section, we merely point out the analogues of the characterizations in Section 
\ref{s:via_standard_laplacian} which have been proven (and implemented) for the spectral fractional Laplacian. 
Unless otherwise stated, these are not {\color{chocolate} applicable without modification} in the case of nonzero boundary conditions. We do not use these {\color{chocolate} characterizations} in our numerical comparisons, but each has its advantages and disadvantages. 

The analogue of the extension method \eqref{E:degenerate_elliptic} for the homogeneous spectral fractional Laplacian on a bounded domain was derived in \cite{StingaTorrea2010_ExtensionProblemHarnacksInequalityFractionalOperators}. In this characterization, a degenerate elliptic equation is solved in the extruded domain (``cylinder'') over $\Omega$, followed by a similar Neumann-type trace. Further studies include \cite{gale2013extension, stinga_4_caccioppoli,stinga_thesis} As this higher-dimensional formulation involves only integer-order operators, standard discretization approaches may be applied, as in \cite{Nochetto2015, Gatto2015, AinsworthGlusa2017_HybridFiniteElementSpectral}. 

 {\color{chocolate}  The heat semigroup formula on a bounded domain $\Omega$ was studied and implemented} in \cite{Cusimano2017}, where the heat semigroup was used to define the spectral fractional Laplacian. This approach is robust for nonzero boundary conditions, and is discussed at length in Section \ref{sec:heat_semigroup}.

The Balakrishnan formula \label{E:balakrishnan_rn} was implemented in, e.g., \cite{bonito2015}. More recently, in \cite{bonito2017}, a rapidly convergent sinc quadrature was developed for this formula, which is used together with a finite-element approximation to the resolvent $(sI - \Delta)^{-1}$. 

Yet another approach in \cite{VABISHCHEVICH2015289} considered a reformulation of the spectral fractional Poisson equation with zero Robin boundary conditions as an integer-order (time-dependent) pseudo-parabolic equation that could be solved with standard FEMs in space and finite difference methods in time. The solution of the pseudo-parabolic equation at time $t = 1$ turns out to be the inverse spectral fractional Laplacian applied to the source function, which yields the solution to the fractional Poisson problem.

\subsubsection{Subordinate Stopped/Reflected Brownian Motion}\label{sec:subordinate_BM}
It is well known that the Laplacian in a bounded domain is the generator of 
stopped Brownian motion for Dirichlet boundary conditions, and reflected Brownian motion for Neumann boundary conditions \cite{yosida1980}. The construction of stopped Brownian motion uses the same procedure as in \eqref{stopping_time}, while the construction of reflected Brownian motion involves the solution of the Skorokhod problem \cite{hsu_reflecting_bm}. The corresponding processes for the spectral fractional Laplacian can be obtained by means of subordination. 
Subordination of a process $X_t$ results in a process $X_{T(t)}$, in which time $t$ is replaced by 
``operational time'' ${T(t)}$, itself a stochastic process {\color{chocolate} -- more specifically, an increasing L\'evy process.} 
In $\mathbb{R}^d$, an isotropic $\alpha$-stable L\'evy motion can be constructed by subordinating the isotropic Brownian motion 
with the stable subordinator (see \cite{Cont2004}, Section 4.4). 
{\color{chocolate}The same time change} gives a way of converting Brownian paths that are stopped at the boundary into $\alpha$-stable L\'evy paths that are stopped at the boundary. 

The spectral fractional Laplacian with Dirichlet boundary conditions is the generator of subordinate stopped Brownian motion \cite{Song}, i.e., stopped Brownian motion that is then subordinated by the standard stable subordinator. The spectral fractional Laplacian with Neumann boundary conditions is the generator of subordinate reflected Brownian motion. 
These are both results from the semigroup theory of Markov processes (see \cite{yosida1980}, Chapter IX; 
a detailed discussion may be found in {\color{magenta}\cite{mamikon_thesis, gulian2018stochastic})}. Thus, one imposes the boundary condition on the process first, before subordinating and turning the stopped/reflected Brownian motion into a L\'evy motion. 
It is important to note the order of the modifications;
the reverse order corresponds to the \emph{Riesz} fractional
Laplacian in $\Omega$, as discussed in Section \ref{killed_levy_motion}.

\subsection{Inhomogeneous Spectral Fractional Laplacian}\label{sec:inhomogeneous_spectral}
The construction of the spectral power \eqref{E:spectral_omega} cannot simply be repeated for the case of nonzero boundary conditions. This is because subsets of $C^\infty(\Omega)$ that satisfy a fixed nonzero boundary condition are no longer linear spaces, which prohibits the use of the spectral theorem. However, in the case of the standard Laplacian $-\Delta$ on a bounded domain, using a lifting technique, a spectral series can be derived ({\color{chocolate} see eq. \eqref{E:series_inhomo_standard_laplacian}}). The generalization of this approach can be used to derive a suitable series representation for the inhomogeneous spectral fractional Laplacian, namely
\begin{align}
	(-\Delta_{\Omega,g})^{\alpha/2}u = \sum_{i=1}^\infty \left(\lambda_i^{\alpha/2} (u,e_i)_{L^2(\Omega)} - \lambda_i^{\alpha/2-1} \left(u,\frac{\partial e_i}{\partial n}\right)_{L^2(\partial\Omega)} \right)e_i,
\end{align}
where $g$ is the nonzero Dirichlet boundary condition for $u$. This approach was taken by Antil et al. \cite{AntilPfeffererRogovs} and is discussed in Section \ref{inhom_series}. The same authors show that under suitable regularity conditions, the inhomogeneous spectral fractional Laplacian can be written as
\begin{align}\label{lifting_characterization_for_harbir}
	(-\Delta_{\Omega,g})^{\alpha/2}u &= (-\Delta_{\Omega,0})^{\alpha/2}[{\color{magenta}u}  - v],
\end{align}
where $-\Delta v = 0$ in the weak sense, and $v \big|_{\partial \Omega} = g$. This essentially reduces the inhomogeneous operator to the well-studied homogeneous spectral fractional Laplacian. The details of the harmonic lifting are discussed in Section \ref{sec:harmonic_lifting}. Nonharmonic lifting functions may also be used; this is discussed in Section \ref{sec:nonharmonic_lifting}. {\color{magenta} We note here a consequence of \eqref{lifting_characterization_for_harbir} that may be surprising at first: as $\alpha \rightarrow 0$, $(-\Delta_{\Omega,g})^{\alpha/2}u \rightarrow u - v$; in other words, for $g \not\equiv 0$, $(-\Delta_{\Omega,g})^{\alpha/2}$ does not reduce to the identity operator as $\alpha \rightarrow 0$.}

A different approach was taken by Cusimano et al. \cite{Cusimano2017} where the heat semigroup was used to define the {\color{chocolate} inhomogeneous} spectral fractional Laplacian, for which they used the notation {\color{chocolate} $\mcL_{\Omega,g}^{\alpha/2}$:}
\begin{align}
\label{cusimano_def1}
	\mcL_{\Omega,g}^{\alpha/2}u &= -\frac{1}{\Gamma(-\alpha/2)}\int_0^\infty \left(e^{t\Delta_{\Omega,g}}u(x) - u(x)\right) \frac{dt}{t^{1+\alpha/2}}.
\end{align}
This approach is discussed in Section \ref{sec:heat_semigroup}. A lifting characterization equivalent to \eqref{lifting_characterization_for_harbir} was obtained for this operator $\mcL_{\Omega,g}^{\alpha/2}$; thus both the operator $(-\Delta_{\Omega,g})^{\alpha/2}$ defined my Antil et al. \cite{AntilPfeffererRogovs} and the operator $\mcL_{\Omega,g}^{\alpha/2}$ defined by Cusimano et al. \cite{Cusimano2017} are the same. Hence, we use the symbol $(-\Delta_{\Omega,g})^{\alpha/2}$ in the remainder of this work. We compare numerical methods that can be applied to both approaches in Section \ref{lifting_weak_compare}.

\subsubsection{Series Representation \label{inhom_series}}

We motivate the approach in \cite{AntilPfeffererRogovs} by discussing the spectral representation of the standard Laplacian $-\Delta u= -\frac{\partial^2 u}{\partial x_1^2} - ... - \frac{\partial^2 u}{\partial x_d^2}$ on a bounded domain $\Omega$ for arguments $u$ with nonzero boundary values.
Given a function $u$ such that
$ u|_{\partial \Omega} = g$, where $g$ need not be zero, subtract from $u$ a harmonic function $v$ inside $\Omega$ with the same boundary values:
\begin{equation*}
-\Delta v = 0,\quad v|_{\partial \Omega} = g.
\end{equation*}
We refer to this function $v$ as a \emph{harmonic lifting} of $g$. Then
\begin{equation*}
-\Delta u = -\Delta(u - v), \quad (u-v)|_{\partial \Omega} = 0.
\end{equation*}
In other words, harmonic lifting of the boundary values does not change the Laplacian $u$, and since the boundary value for $u-v$ is zero, 
we can use it in formula \eqref{standard_laplacian_series}:
\begin{align*}
-\Delta u &= -\Delta (u - v)\\ 
&= \sum \lambda_k ( u - v, e_k )_{\color{blue} L^2(\Omega) } e_k \\
&= \sum \lambda_k ( u, e_k )_{\color{blue} L^2(\Omega) } e_k - \lambda_k ( v, e_k )_{\color{blue} L^2(\Omega) } e_k.
\end{align*} 
Let us rewrite the second inner product in the sum, which involves the harmonic function $v$: 
\begin{align*}
( v, e_k )
&= ( v,  -\Delta (-\Delta)^{-1} e_k )_{\color{blue} L^2(\Omega) }  \\
&= ( v,  -\Delta {\color{magenta}\lambda_k}^{-1} e_k )_{\color{blue} L^2(\Omega) }  \text{ \hspace{0.5in} (by the eigenfunction property) } \\
&= {\color{magenta}\lambda_k}^{-1} \int_{\Omega} -\Delta v e_k + {\color{magenta}\lambda_k}^{-1} \int_{\partial\Omega} v \frac{\partial e_k}{\partial n}  - 
{\color{magenta}\lambda_k}^{-1} \int_{\partial\Omega} e_k \frac{\partial v}{\partial n} \text{ \hspace{0.5in} (Green's second identity) } \\
&=  {\color{magenta}\lambda_k}^{-1} \int_{\partial\Omega} v \frac{\partial e_k}{\partial n} 
\text{  \hspace{0.5in} (since $-\Delta v =0$ on $\Omega$ and $e_k = 0$ on $\partial \Omega$)} \\
&=  {\color{magenta}\lambda_k}^{-1} \int_{\partial\Omega} u \frac{\partial e_k}{\partial n} \text{ \hspace{0.5in}
 (since $u-v = 0$ on $\partial\Omega$)}.
\end{align*}
This gives us the spectral expansion of $-\Delta$ which is now valid for \emph{any} smooth function $u$ on $\Omega$, regardless of boundary values:
\begin{equation}\label{E:series_inhomo_standard_laplacian}
-\Delta u = \sum  \left( \lambda_k \int_\Omega u e_k - \int_{\partial\Omega} u \frac{\partial e_k}{\partial n}  \right)e_k. 
\end{equation} 

This is a key result for defining the inhomogeneous spectral fractional Laplacian. However, before we can proceed, the operator $(-\Delta)^{-1}$ that appeared in the derivation above requires clarification. \emph{A priori}, it is a multiple-valued operator, since $(-\Delta)^{-1} f$ is only specified up to a harmonic function in $\Omega$. However, owing to the uniqueness of the (standard) Poisson problem, this arbitrariness may be removed by requiring $(-\Delta)^{-1} f$ to have fixed boundary value $\tilde{g}$ on $\partial\Omega$. Thus, we define $u = (-\Delta_{\tilde{g}})^{-1} f$ as the function such that $-\Delta u = f$ and $u = \tilde{g}$ on $\partial \Omega$. Regardless of the fixed boundary condition, this will result in a single-valued operator $(-\Delta_{\tilde{g}})^{-1}$ such that 
\begin{equation}\label{e:inverse_forward_first}
-\Delta(-\Delta_{\tilde{g}})^{-1} = \text{Id}_{L^{2}},
\end{equation}
{\color{blue} and}
\begin{equation}\label{e:inverse_forward_second}
(-\Delta_{\tilde{g}})^{-1} (-\Delta) = \text{Id}_{ \{ u \in {\color{blue} H^2} \text{ such that } u|_{\partial{\Omega}} = \tilde{g} \} }.
\end{equation}
For definiteness, and to obtain an identity operator on a linear subset of {\color{blue} $H^2$} in \eqref{e:inverse_forward_second}, the zero boundary value $\tilde{g} \equiv 0$ is chosen. We sometimes refer to this as the \emph{zero gauge} inverse Laplacian, and simply write $(-\Delta)^{-1}$ rather than $(-\Delta_0)^{-1}$. By the spectral theorem, 
the series expansion of the inverse Laplacian is then the spectral power
\begin{equation*}
(-\Delta)^{-1} f = \sum_{i=1}^{{\color{magenta}\infty}} {\color{magenta}\lambda_i}^{-1} (f,e_i)_{\color{blue} L^2(\Omega) } e_i,
\end{equation*}
which defines an operator from $H^{-2}(\Omega)$ to $L^2(\Omega)$. We remark that, even with the choice of a zero gauge, \eqref{e:inverse_forward_first} is valid on functions with arbitrary boundary values. 

Using the result \eqref{E:series_inhomo_standard_laplacian}, a natural approach to defining the spectral fractional Laplacian with nonzero boundary conditions is to write 
\begin{equation*}
(-{\Delta})^{\alpha/2}u := (-\Delta)^{\alpha/2-1} (-\Delta) u
\end{equation*}
and to combine a definition of $(-\Delta)^{\alpha/2-1}$ with the series \eqref{E:series_inhomo_standard_laplacian} above. 
At first glance, this merely shifts the problem to defining the \emph{inverse} spectral fractional Laplacian with nonzero boundary conditions. However, defining this inverse inhomogeneous operator is considerably easier. By the above discussion of $(-\Delta)^{-1}$, the spectral power
\begin{equation*}
(-\Delta)^{\alpha/2-1} :=
\sum_{i=1}^\infty \lambda_i^{\alpha/2-1} (\cdot, e_i)_{\color{blue} L^2(\Omega) } e_i
\end{equation*}
is a convergent series defining a single valued operator on $H^{2(\alpha/2-1)}(\Omega)$, regardless of boundary conditions. For $\alpha=0$, we understand the operator $(-\Delta)^{-1}u$ for $u \in L^2$ to be the projection onto the zero-boundary value functions $H^2_0(\Omega)$ of all functions $v$ such that $-\Delta v = u$.  Thus, the zero-boundary condition is used to eliminate multi-valuedness of $(-\Delta)^{-1}$, but the operator $(-\Delta)^{\alpha/2-1}$ may take as argument a function with arbitrary boundary values, in particular $(-\Delta)u$.  

As a result of these definitions, using the spectral expansion \eqref{E:series_inhomo_standard_laplacian} of the inhomogeneous (integer-order) Laplacian, we obtain
\begin{align}
\label{fractional_power_series}
\begin{split}
	(-\Delta)^{\alpha/2}u = (-\Delta)^{\alpha/2-1} (-\Delta) u &= \sum_{i=1}^\infty \lambda_i^{\alpha/2-1} (-\Delta u, e_i) _{\color{blue} L^2(\Omega) } e_i \\
	&= \sum_{i=1}^\infty \lambda_i^{\alpha/2-1} \left(\sum_{j=1}^\infty \lambda_j (u,e_j)_{\color{blue} L^2(\Omega) } e_j - \left(u, \frac{\partial e_j}{\partial n}\right)_{\color{blue} L^2(\partial \Omega) } e_j, e_i \right) e_i \\
	&= \sum_{i=1}^\infty \lambda_i^{\alpha/2-1} \left(\lambda_i (u,e_i)_{\color{blue} L^2(\Omega) } - \left(u, \frac{\partial e_i}{\partial n}\right)_{\color{blue} L^2(\partial \Omega) } \right)e_i \\
	&= \sum_{i=1}^\infty \lambda_i^{\alpha/2} (u,e_i)_{\color{blue} L^2(\Omega) } e_i - \lambda_i^{\alpha/2-1} \left(u,\frac{\partial e_i}{\partial {n}}\right)_{\color{blue} L^2(\partial \Omega) } e_i.
\end{split}
\end{align}
This is the definition in \cite{AntilPfeffererRogovs} of a general inhomogeneous spectral fractional Laplacian. {\color{forest} This can be considered the fractional analogue of \eqref{E:series_inhomo_standard_laplacian}}. In their paper, Antil et al. also proved a similar formula for nonzero Neumann boundary conditions. Furthermore, they developed an integration-by-parts formula for these formulations, proved regularity results, introduced finite element discretizations, derived the associated error estimates, and included numerical experiments.

\subsubsection{Harmonic Lifting \label{sec:harmonic_lifting}}

Taking the above series representation of the operator $(-\Delta_{\Omega,g})^{\alpha/2}$ as a definition, Antil et al. \cite{AntilPfeffererRogovs} considered the inhomogeneous spectral fractional Poisson problem:
\begin{align}
\begin{split}
	(-\Delta_{\Omega,g})^{\alpha/2} u(x) &= f(x), \hspace{15pt} x \in \Omega, \ \ \ \alpha \in (0,2), \\
	u(x) &= g(x), \hspace{15pt} x \in \partial \Omega.
\end{split}
\end{align}
By linearity, the solution $u$ can be written as
\begin{align}
\label{harbir_splitting}
u(x) = a(x) + b(x),
\end{align}
where $a$ solves
\begin{align}
\label{harbir_splitting_zero_bc}
\begin{split}
	(-\Delta_{\Omega,g})^{\alpha/2} a = (-\Delta_{\Omega,0})^{\alpha/2} a &= f \hspace{20pt} \text{in} \ \Omega, \\
	a \big|_{\partial \Omega} &= 0 \hspace{20pt} \text{on} \ \partial \Omega,
\end{split}
\end{align}
and the component $b$ solves the equation
\begin{align}
\label{harbir_splitting_homogeneous}
\begin{split}
	(-\Delta_{\Omega,g})^{\alpha/2} b &= 0 \hspace{20pt} \text{in} \ \Omega, \\
	b \big|_{\partial \Omega} &= g \hspace{20pt} \text{on} \ \partial \Omega.
\end{split}
\end{align}
However, this ``fractional harmonic" function $b$, under certain conditions on the regularity of $g$, is simply the solution to the standard Laplace equation
\begin{align}
\label{harbir_splitting_very_weak}
\begin{split}
	-\Delta b &= 0 \hspace{20pt} \text{in} \ \Omega, \\
	b \big|_{\partial \Omega} &= g \hspace{20pt} \text{on} \ \partial \Omega.
\end{split}
\end{align}
The simple intuition behind this equivalence involves the characterization of the operator $(-\Delta_{\Omega,g})^{\alpha/2}$ as \\ $(-\Delta)^{\alpha/2-1}(-\Delta)$ and by our choice of the inverse Laplacian with $(-\Delta)^{\alpha/2-1} 0 = 0$, which yield
\begin{align}
	-\Delta b = 0 \implies (-\Delta_{\Omega,g})^{\alpha/2}b = 0.
\end{align}
Indeed, the authors of \cite{AntilPfeffererRogovs} show that Equation \eqref{harbir_splitting_homogeneous} {\color{chocolate} is, for boundary functions $g \in L^2(\partial \Omega)$, equivalent to the \emph{very weak variational form} of Equation \eqref{harbir_splitting_very_weak}:}
\begin{equation}
\label{variational_very_weak}
\int_\Omega b (-\Delta) \phi = \int_{\partial \Omega} g \frac{\partial \phi}{\partial n}, \text{ for all } \phi \in H^1_0(\Omega) \cap H^2(\Omega).
\end{equation} 
{\color{chocolate} The phrase ``very weak''} refers to the transfer of all derivatives of $v$ to the test function $\phi$ via fractional integration-by-parts, a result of the same work \cite{AntilPfeffererRogovs}. {\color{chocolate} Of course, if the boundary data $g$ is sufficiently regular, then $b$ may be sought as a solution to the weak (rather than very weak) variational form of Equation \eqref{harbir_splitting_very_weak}:
\begin{equation}
\int_\Omega \nabla b \cdot \nabla \phi = \int_{\partial \Omega} g \frac{\partial \phi}{\partial n}, \text{ for all } \phi \in H^1_0(\Omega).
\end{equation}
For the precise regularity estimate for the problem \eqref{variational_very_weak} in terms of the boundary function $g$, we refer to Lemma 4.1 and the surrounding discussion in} the article \cite{AntilPfeffererRogovs}. We point out the simplest case, in which $g \in H^{1/2}(\partial \Omega)$ implies that $b \in H^1(\Omega)$ and $b$ satisfies \eqref{harbir_splitting_very_weak} in the weak sense.
{\color{chocolate}
This approach is implemented in Section \ref{lifting_weak_compare}.
}
\subsubsection{Nonharmonic Lifting \label{sec:nonharmonic_lifting}}

While Antil et al. \cite{AntilPfeffererRogovs} used a fractional harmonic lifting in their approach, it is possible to obtain a variational form with an arbitrary {\color{blue}(i.e., nonharmonic)} lifting function. In this section, we describe a lifting approach in which the lifting function $v \in H^{1}(\Omega)$ need only satisfy the boundary condition $v \big|_{\partial \Omega} = u \big|_{\partial \Omega}$. Again, we wish to solve the spectral fractional Poisson problem
\begin{align}
\label{SpectralPoisson}
\begin{split}
	(-\Delta)^{\alpha/2} u &= f, \hspace{10pt} \text{in} \ \Omega\\
	u \big|_{\partial \Omega} &= g, \hspace{10pt} \text{on} \ \partial \Omega.
\end{split}
\end{align}
The unknown function $u$ can be decomposed as $u = w - v$, where $v$ is the lifting function. This function $v$ need not be unique, and the {\color{chocolate} fractional harmonic} function $v$ described above is also admissible.

To derive the variational form of Equation \eqref{SpectralPoisson}, we need the following integration-by-parts formula for the inverse spectral fractional Laplacian $(-\Delta)^{\alpha/2-1}$. 
\begin{theorem}\label{intbyparts}
%Let $\phi \in H^{\alpha/2}(\Omega)$, and  $f \in L^2(\Omega)$. The integration-by-parts formula for the inverse spectral fractional Laplacian $(-\Delta)^{\alpha/2-1}$ is given by
{\color{blue}
Let $0 \le \alpha \le 2$ and $f, \phi \in L^2(\Omega)$. Then}
\begin{align}
	((-\Delta)^{\alpha/2-1}f,\phi) &= (f, (-\Delta)^{\alpha/2-1}\phi).
\end{align}
\end{theorem}
\begin{proof}
We use the notation $\hat{f}_i = (f,e_i)$ and $\hat{\phi}_i = (\phi,e_i)$.
\begin{multline}
((-\Delta)^{\alpha/2-1}f, \phi) = \left(\sum_{i=1}^\infty \lambda_i^{\alpha/2-1} \hat{f}_i e_i, \sum_{j=1}^\infty \hat{\phi}_j e_j\right) 
	= \sum_{i=1}^\infty \sum_{j=1}^\infty \lambda_i^{\alpha/2-1} \hat{f}_i \hat{\phi}_j (e_i,e_j) 
	= \sum_{i=1}^\infty \lambda_i^{\alpha/2-1} \hat{f}_i \hat{\phi}_i \\
	= \sum_{i=1}^\infty \sum_{j=1}^\infty \lambda_j^{\alpha/2-1} \hat{f}_i \hat{\phi}_j (e_i,e_j) 
	= \left(\sum_{i=1}^\infty \hat{f}_i e_i, \sum_{j=1}^\infty \lambda_j^{\alpha/2-1} \hat{\phi}_j e_j\right) 
	= (f, (-\Delta)^{\alpha/2-1} \phi).
\end{multline}
\end{proof}
This formula is valid regardless of the boundary values of the functions $f$ and $\phi$. This is due to our choice of zero gauge for the definition of $(-\Delta)^{\alpha/2-1}$. Then, the variational form of the problem is written as follows: find the function $w \in \mathbb{H}^{\alpha}$ such that for any $\phi \in L^2(\Omega)$, 
\begin{align}
	((-\Delta)^{\alpha/2}(w+v), \phi) &= (f,\phi).
\end{align}
Now we can define the inhomogeneous spectral fractional Laplacian of $v$ in weak form by integrating-by-parts in the $v$ term on the left-hand-side. 
\begin{align}
	((-\Delta)^{\alpha/2}v,\phi) &= (-\Delta v, (-\Delta)^{\alpha/2-1}\phi) \\
	&= (\nabla v, \nabla (-\Delta)^{\alpha/2-1}\phi),
\end{align}
which is our (weak sense) definition of the inhomogeneous spectral fractional Laplacian.

Finally, we can solve for $w$ by solving the variational equation
\begin{align}
\label{weak_nonharmonic_lifting_form}
	((-\Delta)^{\alpha/2} w, \phi) &= (f,\phi) - (\nabla v, \nabla((-\Delta)^{\alpha/2-1}\phi)),
\end{align}
where all the fractional operators that appear are now applied to functions satisfying zero boundary conditions. {\color{blue} Note that this (standard) weak form \eqref{weak_nonharmonic_lifting_form} requires that $v \in H^1(\Omega)$, which by $v|_{\partial \Omega} = g$ and the trace theorem \ref{trace_theorem}, requires the boundary data $g \in H^{1/2}(\Omega)$. While other weak variational forms leading to less regular solution spaces may be studied to treat rougher boundary data, for the purposes of this article,  \eqref{weak_nonharmonic_lifting_form} will suffice.} We recover the solution to Equation \eqref{SpectralPoisson} using the relation $u = w + v$.

These lifting approaches are compared using numerical examples in Section \ref{lifting_weak_compare}.

\subsubsection{Heat Semigroup \label{sec:heat_semigroup}}

The authors of \cite{Cusimano2017} proposed an extension of the heat semigroup method to define an inhomogeneous spectral fractional Laplacian. In their work, only discretizations of their newly-defined operator are considered, not the solution of the Poisson problem using this operator. We compare these discretizations with the approach of Antil et al. \cite{AntilPfeffererRogovs} numerically in Section \ref{comparison_APR_HeatSG}.

Cusimano et al. define their inhomogeneous spectral fractional Laplacian as in Equation \eqref{cusimano_def1}, in which
$e^{t\Delta_{\Omega,g}}$ is the propagator of the heat equation. 
Next, the authors use a splitting for $w(x,t)$, which allows one to relate $\mcL_{\Omega,g}^{\alpha/2}$ to $(-\Delta)_{\Omega,0}^{\alpha/2}$, the  homogeneous spectral fractional Laplacian. The unknown function $w(x,t)$ can be expressed as $w(x,t) = v(x,t) + z(x)$, where $z(x)$ is a harmonic function, i.e.,
\begin{align}
\label{harmonic_fn}
\begin{split}
	-\Delta z &= 0 \hspace{10pt} \text{in} \ \Omega, \\
	z & = g \hspace{10pt} \text{on} \ \partial \Omega,
\end{split}
\end{align}
and $v(x,t)$ solves a zero boundary value heat equation:
\begin{align}
\begin{split}
	\partial_t v - \Delta v &= 0 \hspace{10pt} \text{in} \ \Omega \times [0,+\infty), \\
	v(x,t=0) &= u(x) - z(x), \\
	v(x,t) &= 0 \hspace{10pt} \text{on} \ \partial \Omega \times [0,+\infty).
\end{split}
\end{align}
We notice that $w(x,0) = v(x,0) + z(x) = u(x) - z(x) + z(x) = u(x)$. Then, using the definition of $w(x,t)$ and Equation \eqref{cusimano_def1}, we see that
\begin{align}
\begin{split}
	\mcL_{\Omega,g}^{\alpha/2}u &= \frac{1}{\Gamma(-\alpha/2)} \int_0^\infty (w(x,t) - w(x,0)) \frac{dt}{t^{1+\alpha/2}} \\
	&= \frac{1}{\Gamma(-\alpha/2)} \int_0^\infty ([v(x,t) + z(x)] - [v(x,0) + z(x)]\frac{dt}{t^{1+\alpha/2}} \\
	&= \frac{1}{\Gamma(-\alpha/2)} \int_0^\infty (v(x,t) - v(x,0)) \frac{dt}{t^{1+\alpha/2}} \\
	&= \frac{1}{\Gamma(-\alpha/2)} \int_0^\infty (e^{t\Delta_{\Omega,g}} [u-z](x) - [u-z](x)) \frac{dt}{t^{1+\alpha/2}} \\
	&= (-\Delta_{\Omega,0})^{\alpha/2}[u-z](x).
\end{split}
\end{align}

Given this information, we can see that the two formulations of the spectral fractional Laplacians in works \cite{AntilPfeffererRogovs} and \cite{Cusimano2017} are equivalent. We know that the heat semigroup approach leads to the relation
\begin{align}
	\mcL_{\Omega,g}^{\alpha/2}u &= (-\Delta_{\Omega,0})^{\alpha/2} [u-z],
\end{align}
where $z$ is the harmonic function given by Equation \eqref{harmonic_fn}. We can also obtain a similar formula for the operator of Antil et al. From Equations \eqref{harbir_splitting_homogeneous} and \eqref{harbir_splitting_zero_bc}, we know 
\begin{align}
	(-\Delta_{\Omega,g})^{\alpha/2}u &= (-\Delta_{\Omega,0})^{\alpha/2} {\color{magenta}a}.
\end{align}
Inserting ${\color{magenta}a} = u-{\color{magenta}b}$ from the splitting \eqref{harbir_splitting}, we have
\begin{align}
	(-\Delta_{\Omega,g})^{\alpha/2} u &= (-\Delta_{\Omega,0})^{\alpha/2} [u-{\color{magenta}b}].
\end{align}
This shows that the inhomogeneous spectral fractional Laplacian $(-\Delta_{\Omega,g})^{\alpha/2}$ is just the standard homogeneous operator applied to the lifting $u-{\color{magenta}b}$. Since ${\color{magenta}b}$ is also a harmonic function in $\Omega$ with ${\color{magenta}b}\big|_{\partial \Omega} = g$, by uniqueness, ${\color{magenta}b} = z$. Therefore
\begin{align}
	(-\Delta_{\Omega,g})^{\alpha/2}u &= (-\Delta_{\Omega,0})^{\alpha/2}[u-z] = \mcL_{\Omega,g}^{\alpha/2} u.
\end{align}
This fact is demonstrated numerically in Section \ref{comparison_APR_HeatSG}.

%\pagebreak
\subsubsection{\color{red}Well-posedness and Regularity}\label{sec:spectral_regularity}
{
\color{red}
We begin by discussing well-posedness (existence and uniqueness) of the spectral fractional Poisson problem. We consider the general case of nonzero boundary conditions. Thus, the fractional Laplacian that is used in the two problems below is the inhomogeneous operator that has been reviewed in this section. We have transcribed the existence and uniqueness theorems proven in \cite{AntilPfeffererRogovs}. 
These results have been abbreviated by considering $\alpha/2 \ge 1/2$, resulting in minimal regularity of boundary data $g \in L^2(\partial \Omega)$ in the Dirichlet case, and $g \in H^{-1}(\partial \Omega)$ for the Neumann case; the full theorems in \cite{AntilPfeffererRogovs} actually allow for $\alpha/2 < 1/2$ and rougher boundary data. The Sobolev space notations used in this section are defined in Appendix \ref{sobolev_spaces}.

For the Dirichlet boundary condition, {\color{forest} let $\Omega$ be a bounded quasi-convex domain.} If $f \in \mathbb{H}^{-{\frac{\alpha}{2}}}(\Omega)$ and 
$g \in H^{{\frac{\alpha}{2}}-\frac{1}{2}}(\partial\Omega)$ for ${\alpha/2} \ge 1/2$, then there exists a unique solution $u \in H^{\frac{\alpha}{2}}(\Omega)$ to
\begin{equation}\label{regularity_spectral_dirichlet_inhomo}
(-\Delta)^{\alpha/2} u = f \text{ in $\Omega$}, \quad { u|_{\partial\Omega} = g},
\end{equation}
which satisfies
\begin{equation}
\| u \|_{H^{\frac{\alpha}{2}}(\Omega)} \le C 
\left(
\| f \|_{\mathbb{H}^{-{\frac{\alpha}{2}}}(\Omega)}
+
\| g \|_{{H}^{{\frac{\alpha}{2}}-\frac{1}{2}}(\partial\Omega)}
\right).
\end{equation}
For ${\alpha/2} = 1/2$, the convention is $H^{0}(\partial\Omega) = L^2(\partial\Omega)$. This statement is a special case of Theorem 4.5 in \cite{AntilPfeffererRogovs}. 

For the Neumann boundary condition, {\color{forest} again let $\Omega$ be a bounded quasi-convex domain.} Let $1/2 \le {\alpha/2} \le 1$. If $f \in {H}^{{\frac{\alpha}{2}}}(\Omega)^*$ (the dual of 
${H}^{{\frac{\alpha}{2}}}(\Omega)$; for ${\alpha/2} > 1/2$ this is distinct from $\mathbb{H}^{-{\frac{\alpha}{2}}}(\Omega)$) and $g \in H^{{\frac{\alpha}{2}}-\frac{3}{2}}(\partial\Omega)$ satisfy the compatibility condition
\begin{equation}\label{spectral_compatibility_condition}
\int_\Omega f + \int_{\partial\Omega} g = 0,
\end{equation}
then there exists a unique solution $u \in H^{\frac{\alpha}{2}}(\Omega)$ such that $\int_{\Omega} u = 0$ to  
\begin{equation}\label{regularity_spectral_neumann_inhomo}
(-\Delta)^{\alpha/2} u = f \text{ in $\Omega$}, \quad { \partial u/\partial n|_{\partial\Omega} = g},
\end{equation}
which satisfies
\begin{equation}
| u |_{H^{\frac{\alpha}{2}}(\Omega)} \le C 
\left(
\| f \|_{{H}^{{\frac{\alpha}{2}}}(\Omega)^*}
+
\| g \|_{\mathbb{H}^{{\frac{\alpha}{2}}-\frac{3}{2}}(\partial\Omega)}
\right).
\end{equation}
This statement is a special case of Theorem 5.4 in \cite{AntilPfeffererRogovs}. 
}

{
\color{red}
Next, we discuss regularity for the spectral fractional Poisson equation. Unlike the well-posedness results discussed above, we now focus exclusively on the case of  \emph{zero} boundary conditions. This is because almost all regularity results that have been published at the time of this writing have been presented in this form. Various extensions to nonzero boundary conditions can be made by considering these results in combination with the harmonic lifting property discussed in \ref{sec:harmonic_lifting} and \ref{sec:heat_semigroup}.
}

For the standard homogeneous spectral fractional Laplacian, we state only the simplest case of the Sobolev regularity results in \cite{grubb_spectral}, and direct the reader to that article for a full discussion. First, we consider the Dirichlet problem, using the zero Dirichlet boundary value fractional Sobolev space $\mbbH^{\alpha/2}(\Omega)$. 
{\color{forest}Let $\Omega$ be a bounded, $C^\infty$-smooth subset of $\mathbb{R}^d$.} Let $u$ satisfy the equation
\begin{equation}\label{regularity_spectral_dirichlet}
(-\Delta)^{\alpha/2} u = f \text{ in $\Omega$}, \quad {\color{chocolate} u|_{\partial\Omega} = 0} \\
\end{equation}
for $0 < \alpha < 2$. 
Then
\begin{enumerate}
\item
If $s < 1/2$, then $f \in H^{s}(\Omega)$ implies $u \in \mbbH^{s+\alpha}(\Omega)$. 
\item
If $s = 1/2$, then $f \in H_{00}^{s}(\Omega)$ implies $u \in \mbbH^{s+\alpha}(\Omega)$. 
\item
If $1/2 < s < 2 + 1/2$, then $f \in H^{s}(\Omega)$ implies only that $u \in \mbbH^{1/2-\epsilon + \alpha}(\Omega)$, and we have the stronger result $u \in \mbbH^{s + \alpha}(\Omega)$ if and only if $f = 0$ on 
$\partial\Omega$. 
\end{enumerate}

% \subsubsection{Example in 1D}
{\color{blue} We illustrate that this last result is sharp with a simple example. Note that the Sobolev norm in $H^s(I)$ of a Fourier sine series is given by 
\begin{equation*}
\left\| \sum a_k\sin(kx) \right\|_{H^s(I)} = \sum (1+k^2)^s | a_k^2 |. 
\end{equation*}
Consider the interval $[0,2\pi]$, where the eigenvalues of the (Dirichlet) Laplacian are $\lambda_k = k^2$, with corresponding eigenfunctions $e_{\lambda_k}=\sin(kx)$. Let us begin by constructing a low regularity $f \in L^2$:
\begin{equation}\label{E:regularity_f_definition}
f = \sum_{k=1}^\infty \frac{1}{\sqrt{k}\log(k+1)} \sin(kx). 
\end{equation}
Then, $f \in L^2$ since 
\begin{align*}
\sum_{k=1}^\infty \frac{1}{k\log^2(k+1)} &\leq \int_1^\infty \frac{dx}{x \log^2(x+1)}  = \int_1^\infty \frac{dx}{(x+1)\left(\log^2(x+1) - \frac{\log^2(x+1)}{x+1}\right)} = \int_{\log 2}^\infty \frac{dx}{x^2(1-e^{-x})} \leq 2.
\end{align*}
But ${f \not\in H^s}$ for
any ${s>0}$ since  
\begin{align*}
	\|f\|_{H^s} \sim \sum_{k=1}^\infty \frac{(1+k^2)^s}{k(\log^2(k+1))} &\geq \sum_{k=1}^\infty \frac{k^{2s}}{k \log^2(k+1)} \geq \sum_{k=1}^\infty \frac{1}{k} = \infty.
\end{align*}
Now, the solution to the fractional Poisson problem with zero Dirichlet boundary conditions and $f$ given by 
\eqref{E:regularity_f_definition} is 
\begin{equation}
\label{summation_u}
u = \sum_{k=1}^\infty \frac{k^{-\alpha}}{\sqrt{k}\log(k+1)} \sin(kx).
\end{equation}
We see $\|u\|_{H^{s+\alpha}} \sim \|f\|_{H^s}$, so $u$ has strict  $H^{\alpha}$ regularity: $u \in H^{\alpha}$ and $u \not\in H^{\alpha+\epsilon}$. \emph{Thus, the $\alpha$-gain in regularity as stated in the result above is sharp}.}

Next, we consider the Neumann problem, and for $s > 3/2$ we define the zero Neumann boundary value fractional Sobolev space
\begin{align*}
{\color{blue} H^{s}_{\partial u/\partial n = 0}(\Omega)} = \left\{ u \in H^s(\Omega) \text{ such that } \frac{\partial u}{\partial n}\suchthat_{\partial\Omega}  = 0 
\right\}.
\end{align*}
If $s<3/2$, we define 
${\color{blue} H^{s}_{\partial u/\partial n = 0}(\Omega)} = H^{s}(\Omega)$.
{\color{forest}Let $\Omega$ be a bounded, $C^\infty$-smooth subset of $\mathbb{R}^d$.}
{\color{blue} Let $u$ satisfy the equation} 
\begin{equation}\label{regularity_spectral_neumann}
(-\Delta)^{\alpha/2} u = f \text{ in $\Omega$}, \quad {\color{chocolate} \partial u/\partial n|_{\partial\Omega} = 0.}  \\
\end{equation}
{\color{blue} Then,}
\begin{enumerate}
\item
If $s < 3/2$, then {\color{chocolate} $f \in H^{s}(\Omega)$ implies } $u \in {\color{blue} H^{s+\alpha}_{\partial u/\partial n = 0}(\Omega)}$. 
\item
If $3/2< s < 7/2$, then {\color{chocolate} $f \in H^{s}(\Omega)$ implies } $u \in {\color{blue} H^{s+\alpha}_{\partial u/\partial n = 0}(\Omega)}$ if and only if $\partial f/\partial n = 0$ on 
$\partial\Omega$. 
\end{enumerate}
Among additional results, \cite{grubb_spectral} discusses the extension by induction of the above points to higher $H^s(\Omega)$ regularity of $f$.

{
\color{chocolate}
Regularity in the H\"older spaces $C^{k,r}$ for the spectral fractional Poisson problem was studied extensively in \cite{stinga_4_caccioppoli}. In that work, an array of results was obtained, for both interior and boundary regularity of the solution $u$ in the Dirichlet problem \eqref{regularity_spectral_dirichlet} and in the Neumann problem \eqref{regularity_spectral_neumann}, with conditions of the form $f\in C^{0,r}(\Omega)$ or of the form $f \in L^p(\Omega)$, under fairly general conditions on the domain $\Omega$. The results also allow for powers of more general, variable-coefficient elliptic operators. We transcribe just one of these results for the fractional Laplacian, namely, the interior regularity for  $f\in C^{0,r}(\Omega)$ of the Poisson problem, which is the same for both zero Dirichlet or zero Neumann boundary condition:

Assume that $\Omega$ is a bounded Lipschitz domain and that $f\in C^{0,r}(\Omega)$,
for some $0<r<1$. Let $u$ be a solution to \eqref{regularity_spectral_dirichlet} or \eqref{regularity_spectral_neumann}.
 \begin{enumerate}[$(1)$]
  \item If $0<r+\alpha<1$, then $u\in C^{0,r+\alpha}({\Omega})$ and
  $$[u]_{C^{0,r+\alpha}({\Omega})}\leq C\big(\|u\|_{L^2(\Omega)}+[u]_{H^{\alpha/2}(\Omega)}
  +\|f\|_{C^{0,r}(\Omega)}\big).$$
   \item If $1<r+\alpha<2$, then $u\in C^{1,r+\alpha-1}({\Omega})$ and
  $$[u]_{C^{1,r+\alpha-1}({\Omega})}\leq C\big(\|u\|_{L^2(\Omega)}+[u]_{H^{\alpha/2}(\Omega)}
  +\|f\|_{C^{0,r}(\Omega)}\big).$$
 \end{enumerate}
  The constants $C$ depend only on $d$, $\Omega$, $r$, and $\alpha$. For the proof of this the other regularity results, see \cite{stinga_4_caccioppoli}. Similar results have also been obtained for the equation $(-\epsilon \Delta)^{1/2}u+u=f$ with zero Neumann boundary condition in \cite{stinga_3_keller}. 
}    
%\pagebreak

\subsection{Summary}\label{Summary}
\begin{center}
\begin{table}[ht!]

\caption {\label{table:summary}Comparison of the Riesz and spectral fractional Laplacians in a bounded domain ${\Omega \subset \mathbb{R}^d}$.
\label{comparison_table}}
\begin{tabular}{|p{0.6in}|p{1.375in}|p{0.75in}|p{0.75in}|p{1.25in}|p{1.25in}|}
\hline
$(-\Delta)^{\alpha/2} u$ & Definition & Domain & BC type & Stopped process for Dirichlet BC &
Reflected process for zero Neumann BC \\
\hline
Spectral & 
$ \sum \lambda^{\alpha/2} ( u, e_\lambda ) e_\lambda$,
where $(\lambda, e_\lambda)$ is the spectrum
of ${-\Delta}$ on $\Omega$. 
&
Functions on ${\Omega}$
& Boundary ($\partial \Omega$) 
& Subordinate stopped Brownian Motion {\color{magenta}\cite{yosida1980,Song,gulian2018stochastic}}
& Subordinate reflected Brownian Motion \cite{yosida1980,mamikon_thesis}\\
\hline
Riesz & 
$
C \ \text{p.v.} \int_{\mathbb{R}^d} \frac{u(x) - u(y)}
{|x-y|^{d+\alpha}}dy,
$
{\color{chocolate} 
for
$C =
\frac{2^{\alpha} \Gamma\left(\frac{\alpha}{2}+\frac{d}{2}\right)}{\pi^{d/2}
{|}
\Gamma\left(-\frac{\alpha}{2}
\right)
{|}
}.
$
}
&
Functions on ${\mathbb{R}^d}$.
& Exterior ({\color{magenta}$\mathbb{R}^d \setminus \Omega$})& Stopped $\alpha$-stable motion \cite{kyprianou2016unbiasedwalk}. 
& Various conditions/processes proposed \cite{gunzburger_neumann,Dipierro14} \\
\hline
\end{tabular}
\end{table}
\end{center}

In Table \ref{table:summary}, we have summarized the fundamental properties and stochastic interpretations of the Riesz and spectral fractional Laplacians, as they were discussed in this section. From the stochastic perspective, these operators differ in that the Riesz Laplacian is associated to processes that leave the closure of the domain, while the spectral Laplacian is associated to processes that are confined to the closure the domain. 

\begin{remark}\normalfont
	The Riesz and the spectral fractional Laplacians are merely the two most commonly used definitions. Another fractional Laplacian that has been studied is the \emph{regional} definition:
	\begin{align}
		(-\Delta_{\text{regional}})^{\alpha/2} u(x) &= C(d,\alpha) \ \text{p.v.} \int_{\Omega} \frac{u(x) - u(y)}{|x-y|^{d+\alpha}}dy,
	\end{align}
	where $C(d,\alpha)$ is given by \eqref{E:constant}. Note that the domain of this operator consists of functions defined on $\Omega$, rather than $\mbbR^d$.
	The regional definition differs from the Riesz definition, even if $u(x) \equiv 0$ for $x \in \mbbR^d \setminus \Omega$:
	\begin{align}
	\begin{split}
		(-\Delta_{\text{Riesz}})^{\alpha/2}u(x) &= C(d,\alpha) \ \text{p.v.} \int_{\mbbR^d} \frac{u(x)-u(y)}{|x-y|^{d+\alpha}} \ dy \\
		&= C(d,\alpha) \ \text{p.v.} \int_\Omega \frac{u(x)-u(y)}{|x-y|^{d+\alpha}} \ dy - u(x) \int_{\mbbR^d \setminus \Omega} \frac{C(d,\alpha)  }{|x-y|^{d+\alpha}} \ dy \\
		&= (-\Delta_{\text{regional}})^{\alpha/2}u(x) - u(x)  \int_{\mbbR^d \setminus \Omega} \frac{C(d,\alpha) }{|x-y|^{d+\alpha}} \ dy.
	\end{split}
	\end{align}
	The well-posedness of the fractional Poisson problem involving the regional Laplacian was studied using the Feynman-Kac formula \cite{guan_ma2005}. For a further discussion of the regional Laplacian and the relation to reflected and censored $\alpha$-stable processes, see \cite{GuanMa2006}. Neumann and Robin boundary conditions for the regional Laplacian have been discussed in \cite{Warma2015}.

Other notions of fractional Laplacians that arise from related processes are discussed in \cite{kim2017boundary}. In general, the probabilistic literature on stable-type processes in bounded domains and related notions of fractional Laplacians is very rich \cite{bogdan2009potential}. For example, estimates for eigenvalues of the spectral Laplacian were derived using probabilistic techniques in \cite{chen2005two}.	
	
\end{remark}

\section{Numerical Methods} \label{sec:num_meth}

\hrule
\vspace{1em}
\noindent \textbf{Section Overview}\\[0.2em]
\indent One of our primary goals in this work is to compare numerical solutions for the fractional Poisson problem using different definitions of the fractional Laplacian. To this end, we develop new or implement existing methods to discretize each definition. All methods that we use to compute the solutions of the benchmark problems are described in this {\color{chocolate} section}. To discretize the Riesz fractional Laplacian, we use the adaptive finite element method (AFEM) of \cite{AinsworthGlusa2017_TowardsEfficientFiniteElement} and the Walk-on-Spheres (WOS) method of \cite{kyprianou2016unbiasedwalk}. We discretize the spectral fractional Laplacian directly using the spectral element method (SEM) of \cite{SongXuKarniadakis2017}, and the heat semigroup approach \cite{stinga_thesis, Cusimano2017, StingaTorrea2010_ExtensionProblemHarnacksInequalityFractionalOperators}, which is used in Section \ref{sec:nonzerobcs}, and elliptic extension approaches \cite{CaffarelliSilvestre2007_ExtensionProblemRelatedToFractionalLaplacian, 
StingaTorrea2010_ExtensionProblemHarnacksInequalityFractionalOperators,
stinga_thesis, gale2013extension, stinga_4_caccioppoli} are also discussed for completeness. We develop a new approximation method for the directional definition using a radial basis function (RBF) collocation method, which also makes use of the vector Gr\"unwald scheme of \cite{meerschaert2004vector}. This is also the first work in which numerical results have been produced using the vector Gr\"unwald scheme, as no other work (to our knowledge) has implemented the method of \cite{meerschaert2004vector}.

Furthermore, we examine another nonlocal operator, the ``horizon-based nonlocal" definition \cite{dusiamreview}, which can be seen as an approximation to the Riesz fractional Laplacian. We develop a finite volume method to compute the numerical results of one-dimensional fractional Poisson equations posed with different horizon parameters. We compare these results to the numerical solution of the Riesz fractional Poisson equation. The horizon-based nonlocal definition is equivalent to the Riesz fractional Laplacian in the limit as the horizon parameter approaches infinity, as is demonstrated by our numerical examples. We can also see in these examples approximately how large the horizon parameter should be to result in a reasonably close approximation of the solution of the Riesz fractional Poisson equation.
\vspace{1em}
\hrule
\vspace{2em}

\subsection{Riesz Definition}\label{Riesz}

Using the Riesz definition presented above in Equations \eqref{E:constant} and \eqref{E:riesz_Laplacian_rn}, we consider the fractional Poisson equation \eqref{fracPoisson} with zero Dirichlet boundary conditions $u(x) = 0$ for $x \in \mbbR^d \setminus \Omega,$ where $\Omega \subset \mbbR^d$ is a bounded Lipschitz domain, and we define the Riesz fractional Laplacian to be the \textit{restriction} of the operator to functions with compact support in $\Omega.$ The boundary condition for this definition is considered ``nonlocal" and is also called a ``volume constraint", as it is defined on the exterior of the domain $\Omega.$

\subsubsection{Adaptive Finite Element Method (AFEM) \label{sec:afem}}

The Riesz fractional Poisson problem takes the variational form \cite{AinsworthGlusa2017_TowardsEfficientFiniteElement}
\begin{align}
  \text{Find } u\in \mathbb{H}^{\alpha/2}\left(\Omega\right) : \quad a\left(u,v\right)=\left( f,v\right) \quad \forall v\in \mathbb{H}^{\alpha/2}\left(\Omega\right), \label{eq:fracPoissonVariational}
\end{align}
where
\begin{align}\label{afem_varform}
  a(u,v)
  &=\frac{C(d,\alpha)}{2} \int_{\Omega} \dd{\color{magenta}x} \int_{\Omega}\dd{\color{magenta}y}  \frac{\left(u\left({\color{magenta}x}\right)-u\left({\color{magenta}y}\right)\right)\left(v\left({\color{magenta}x}\right)-v\left({\color{magenta}y}\right)\right)}{\abs{{\color{magenta}x}-{\color{magenta}y}}^{d+\alpha}} \\
  &\quad+ \frac{C(d,\alpha)}{\alpha} \int_{\Omega} \dd {\color{magenta}x} \int_{\partial\Omega} \dd {\color{magenta}y} \frac{u\left({\color{magenta}x}\right) v\left({\color{magenta}x}\right) ~ {\color{magenta}n}_{{\color{magenta}y}}\cdot\left({\color{magenta}x}-{\color{magenta}y}\right)}{\abs{{\color{magenta}x}-{\color{magenta}y}}^{d+\alpha}},\label{eq:bilinearForm}
\end{align}
and where \({\color{magenta}n}_{y}\) is the \emph{inward} normal to \(\partial\Omega\) at \({\color{magenta}y}\). The space $\mathbb{H}^{\alpha/2}(\Omega)$ is defined in Appendix \ref{sobolev_spaces}, Definition \ref{sobdef:boldh}.

A straightforward finite element discretisation of \eqref{eq:fracPoissonVariational} encounters several difficulties:
\begin{enumerate}
\item
  The element contributions for adjacent or identical element pairs of \eqref{eq:bilinearForm} are given by singular integrals.
  Special Gaussian quadrature methods are given in \cite{AinsworthGlusa2017_TowardsEfficientFiniteElement}, \cite{AinsworthGlusa2017_AspectsAdaptiveFiniteElement}.
\item
  The resulting linear algebraic system is dense.
  In \cite{AinsworthGlusa2017_TowardsEfficientFiniteElement}, the panel clustering technique known from the boundary element literature is adapted for the fractional Laplacian.
  This reduced the complexity of the matrix-vector product from \(\mathcal{O}\left(n^{2}\right)\) to \(\mathcal{O}\left(n \log^{2d} n\right)\), where \(d=1\) or \(d=2\) is the spatial dimension.
\item
  As shown in \cite{Grubb2015_FractionalLaplaciansDomainsDevelopment}, solutions to the Riesz fractional Laplacian display generally low regularity close to the boundary of the domain. {\color{magenta}This behavior is noticeably different from the classical integer-order Laplacian for which higher regularity of the domain and right-hand side imply higher regularity of the solution.}
  Therefore, globally quasi-uniform meshes are ill-suited for the discretisation of the integral fractional Laplacian.
  {\color{magenta}In \cite{AinsworthGlusa2017_AspectsAdaptiveFiniteElement}, posteriori error estimates of residual type and a D\"orfler
marking strategy are employed to adaptively refine the discretisation,
resulting in meshes that are highly refined close to the boundary.}
  It was shown that, using piecewise linear finite elements, optimal rates of convergence are obtained:
  \begin{center}
    \begin{tabular}{l|cc}
      &\(\mathbb{H}^{\alpha/2}\)& \(L^{2}\)\\
      \hline
      \(d=1\)& \(n^{\alpha/2-2}\) & \(n^{-2}\)  \\
      \(d=2\)& \(n^{-1/2}\) & \(n^{-1/2-\alpha/4}\)
    \end{tabular}
  \end{center}
  Again, a clustering approach leads to \(\mathcal{O}\left(n \log^{2d} n\right)\) complexity for the computation of the error indicators.
\item
  The efficient solution of the arising linear systems of equations can be achieved by a standard multigrid solver \cite{AinsworthGlusa2017_TowardsEfficientFiniteElement}, \cite{AinsworthGlusa2017_AspectsAdaptiveFiniteElement}.
  Quasi-optimal complexity is shown to be obtained.
 
\end{enumerate}

\begin{remark}\label{AFEM_regularity}\normalfont
Notice that in the variational form \eqref{eq:fracPoissonVariational}, we only require $u \in \mathbb{H}^{\alpha/2}(\Omega)$, and in the case where $0 < \alpha < 1$, this variational form accepts functions $u$ that may not admit a trace (see Theorem \ref{trace_theorem}). Although the weak solution of the Riesz fractional Poisson problem has the regularity described in Theorem \ref{Riesz_regularity}, and therefore admits a trace if $\alpha + s > 1/2$, the AFEM approximation space is $\mathbb{H}^{\alpha/2}(\Omega)$, which is only contained in the trace space $\mathbb{H}^{1/2 + \varepsilon}(\Omega)$ in the case $\alpha > 1$. This is the reason that the examples in Sec. \ref{sec:num_comp} (namely, Figures \ref{square12}, \ref{disk12}, and \ref{diskslicesf1y0}) where $\alpha < 1$ show solutions to the homogeneous Riesz fractional Poisson equation in which the zero boundary condition is not strongly enforced on the approximation, despite the solution belonging to $H^{1/2+\varepsilon}(\Omega)$ and possessing a zero trace. 
As the finite element mesh is refined and the approximant converges to the true weak solution, the oscillations at the boundary are diminished.
\end{remark}

\FloatBarrier

\subsubsection{Walk-on-Spheres (WOS) Method \label{sec:wos}}

The \emph{walk-on-spheres} method is a type of Monte Carlo method for simulating solutions to the Dirichlet fractional Poisson problem with both zero and nonzero boundary conditions. 
{\color{chocolate}
It was originally proposed by Muller \cite{muller1956}  in 1956 for solving Laplace equations with Dirichlet boundary conditions (see also \cite{cai2013}), and has been used for Neumann boundary conditions \cite{cai2017}, and Robin boundary conditions \cite{Zhou2016} as well.}
This approach {\color{chocolate} has recently been extended \cite{kyprianou2016unbiasedwalk}} to the following Riesz fractional Poisson problem:
\begin{align}\label{riesz_poisson}
\begin{split}
	(-\Delta_{\text{Riesz}})^{\alpha/2} u(x) &= f(x), \hspace{10pt} \text{in} \ \Omega, \\
	u(x) &= g(x), \hspace{10pt} \text{in} \ {\color{magenta}\mathbb{R}^d \setminus \Omega},
\end{split}
\end{align}
where $g$ and $f$ are suitably regular functions, $f : \Omega \rightarrow \mathbb{R}$, and $g : {\color{magenta}\mathbb{R} \setminus \Omega} \rightarrow \mathbb{R}^d$. In order to formulate the walk-on-spheres method, one must first identify the stochastic process that is generated by the operator $(-\Delta_{\text{Riesz}})^{\alpha/2}$. In this case, the process is a {\color{chocolate} killed} isotropic $\alpha$-stable process with $\alpha \in (0,2)$ \cite{kyprianou2016unbiasedwalk}. One significant difference in this setting is that, in contrast with the Brownian Motion setting when $\alpha = 2$, the stable process exits $\Omega$ by a jump rather than by passing through the boundary. A consequence of this is that disconnected domains can be considered, and the walk-on-spheres algorithm will terminate in finite steps, whereas the walk-on-spheres algorithm for Brownian Motion will not terminate and must be truncated.

Kyprianou et al. \cite{kyprianou2016unbiasedwalk} proved the following Feynman-Kac formula for the solution of Equation \eqref{riesz_poisson} (see Theorem 6.1). {\color{blue} Given a Borel set $S$, define the space $L_\alpha^1(S)$ to be all real-valued, measurable functions $\phi$ that satisfy
\begin{align*}
	\int_{S} \frac{|\phi(x)|}{1 + |x|^{\alpha+d}} \, dx < \infty.
\end{align*}
}Let $g \in C(\Omega) \cap L^1_\alpha({\color{magenta}\mathbb{R}^d \setminus \Omega}),$ $f \in C^{\alpha+\varepsilon}(\overline{\Omega})$ with some $\varepsilon > 0$. Then there exists a continuous solution $u(x) \in L_\alpha^1(\mathbb{R}^d)$ to Equation \eqref{riesz_poisson}, where
\begin{align}\label{feynman_kac_shardlow}
	u(x) &= \mathbb{E}_x\left[g(X_{\sigma_\Omega})\right] + \mathbb{E}_x \left[\int_0^{\sigma_\Omega} f(X_s) ds\right],
\end{align}
and $\sigma_\Omega = \inf\{t > 0 : X_t \notin \Omega\}$ is the first exit time of the process $X_t$ from $\Omega$.
{\color{magenta}This equation is known as a Feynman-Kac formula and allows for stochastic solution of the associated boundary value problem; in principle, the derivation of such formulas can make possible stochastic solution methods for other definitions of fractional Laplacian and other types of boundary conditions as well \cite{gulian2018stochastic}.}

In order to use the Feynman-Kac formula \eqref{feynman_kac_shardlow}, we must simulate a {\color{chocolate} very large number} of paths of the $\alpha$-stable process $X_t$, beginning from the point $x \in \Omega$ at which we want to compute the solution.
This approach is embarrassingly parallel, as one could assign each point $x$ in the domain to a different compute node, and no information needs to be shared between processors. Further, while one could generate an exact simulation of each path, the idea of the walk-on-spheres approach is to ``speed up" these simulations by only computing a few points along each path until the process jumps out of the domain. 
{\color{chocolate} This is especially effective when the right-hand side $f$ in the problem \ref{riesz_poisson} is zero, since the Feynman-Kac formula \eqref{feynman_kac_shardlow}
reduces to an expectation over the distribution of exit points of the process.}
We will describe the walk-on-spheres algorithm, first presented in \cite{kyprianou2016unbiasedwalk}, below.

First, we discuss some of the key ingredients of the paper by Kyprianou et al. \cite{kyprianou2016unbiasedwalk} that are used to formulate the walk-on-spheres algorithm. The following result gives the distribution of a stable process that begins from the origin when it first exits the unit sphere:
\begin{theorem}{\cite{kyprianou2016unbiasedwalk}}
	Suppose that $B(0,1)$ is the unit ball centered at the origin, and write $\sigma_{B(0,1)} = \inf \{t > 0 : X_t \notin B(0,1)\}.$ Then
	\begin{align}\label{first_exit}
		\mathbb{P}_0(X_{\sigma_{B(0,1)}} \in dy) = \pi^{-(d/2+1)}\Gamma(d/2) \sin(\pi \alpha/2) |1-|y|^2|^{-\alpha/2} |y|^{-d} dy, \hspace{10pt} |y| > 1.
	\end{align}
\end{theorem}

Using this theorem, we can construct a series of points along the sample paths of stable processes. One must choose a tolerance $\varepsilon$ and an initial point $x = X_0$. Then $x$ is circumscribed by a sphere of radius $\varepsilon$.  Let $E_1$ represent a sampling from the distribution of Equation \eqref{first_exit}, which gives the exit from a ball of radius one centered at the origin. Using a scaling property and the fact that $X$ has stationary and independent increments, $x + \varepsilon E_1$ gives the exit position from the ball $B(X_0=x,\varepsilon)$. Then, we define $X_1 = x + \varepsilon E_1$ and proceed inductively, generating $X_{n+1}$ as the exit point of the ball centered at $X_n$ with radius $\varepsilon$, i.e., $X_{n+1} = X_n + \varepsilon E_{n+1}$, where $E_{n+1}$ is an i.i.d. copy of $E_1$.

Now we describe the walk-on-spheres solution method for the fractional Poisson equation \cite{kyprianou2016unbiasedwalk}. Suppose that $\Omega$ is a convex domain in $\mathbb{R}^d$, $d \geq 2$.  $\Omega$ can be bounded or unbounded, as long as the measure of ${\color{magenta}\mathbb{R}^d \setminus \Omega}$ is not zero. Given a starting position $x \in \Omega$, we inscribe the largest sphere that fits inside $\Omega$ and is centered at $\rho_0 := x$. The radius of this sphere is denoted by $r_1$. We continue inductively to generate the ``walk-on-spheres" $(\rho_n, \ n\geq 0)$. Given $\rho_{n-1}$, we select the distribution of $\rho_n$ according to an independent copy of $X_{\sigma_{B_n}}$ under $\mathbb{P}_{\rho_n -1}$ (the shifted version of $\mathbb{P}_0$ from Equation \eqref{first_exit}), where
\begin{align}
	B_n = \{ x \in \mathbb{R}^d : |x-\rho_{n-1}| < r_n \} \ \ \ \text{and} \ \ \ \sigma_{B_n} = \inf\{t > 0 : X_t \notin B_n\}.
\end{align}
The algorithm terminates at the random index $N = \min\{ n \geq 0 : \rho_n \notin \Omega\}$, i.e., when the walk-on-spheres jumps out of the domain $\Omega$. We refer the reader to \cite{kyprianou2016unbiasedwalk} for details of the implementation of this procedure.

{\color{blue} For two different classes of domain $\Omega$, Kyprianou et al. also proved that for all $x \in \Omega$, the index $N$ will always be finite and at most geometrically distributed. The first class is convex (possibly unbounded) domains, and the second class is non-convex bounded domains that satisfy the uniform exterior-cone condition See \cite{kyprianou2016unbiasedwalk} for proofs in either case and a discussion of how convexity relates to the proof.} Furthermore, the convergence of the walk-on-spheres method is proved, and the code for the numerical examples implemented in \cite{kyprianou2016unbiasedwalk} can be found at \texttt{https://bitbucket.org/wos\_paper/wos\_repo}.

%%%%%%%%%%%%%%%%%%%%%%
{\color{chocolate}
Using the sequences of spheres, the Feynman-Kac formula for the solution $u$ can be replaced by an expectation over the boundary condition $g$ evaluated at the terminal centers $\rho_N$, and the right-hand side $f$ integrated over the expected occupation measure of the stable process prior to exiting each sphere. The final representation obtained by \cite{kyprianou2016unbiasedwalk} is,
for $x\in D$, $g\in L^1_\alpha(D^\mathrm{c})$ and $f \in C^{\alpha +\varepsilon}(\overline{D})$,
\begin{equation}
u(x) =\mathbb{E}_x[g(\rho_{N})] 
+
\mathbb{E}_x \left[
\sum_{n=0}^{N-1} r_n^{\alpha} 
V_{1}\big(0, f(\rho_n + r_n  \cdot )\big)
\right], \text{ where }
\quad
V_1(x,f(\cdot)) = \int_{|y-x|<1}f(y)\,V_1(x,dy).
\end{equation}

The expected occupation measure $V_1(x,dy)$ of the stable process prior to exiting a unit ball centered at the origin is given \cite{boggio1905sulle, blumenthal1961distribution}, for $|y|<1$, by 
\begin{align}
V_1(0,dy)
=
2^{-\alpha}\,\pi^{-d/2}\, \frac{\Gamma(d/2)}{\Gamma(\alpha/2)^{2}}\,|y|^{\alpha -d}\, \left(
\int_0^{|y|^{-2}-1}
(u+1)^{-d/2}
u^{\alpha/2-1}
du 
\right)
dy.
\end{align}
The implementation of this formula, using Monte Carlo integration for the integral 
$V_{1}\big(0, f(\rho_n + r_n  \cdot )\big)$, is discussed in \cite{kyprianou2016unbiasedwalk}. We use the WOS algorithm in Figure \ref{cmp-RBF-WOS} and throughout Sections \ref{sec:num_comp} and \ref{sec:nonzerobcs}.}
%%%%%%%%%%%%%

\subsection{Spectral Definition}\label{Spectral}

Several approaches have been proposed for discretizing the spectral definition, which is defined in Equation \eqref{E:spectral_omega}. {\color{chocolate}Stinga and Torrea \cite{StingaTorrea2010_ExtensionProblemHarnacksInequalityFractionalOperators} showed } that the fractional Poisson problem
\begin{align}
  \left\{
  \begin{array}{rlrl}
    \left(-\Delta_{\text{spectral}}\right)^{\alpha/2}u\left({\color{magenta}x}\right) &= f\left({\color{magenta}x}\right), && {\color{magenta}x}\in\Omega,\\
    u({\color{magenta}x})&=0, && {\color{magenta}x}\in\partial\Omega
  \end{array}\right. \label{eq:specFracLapl}
\end{align}
can be recast as a problem over the extruded domain \(\mathcal{C}=\Omega\times[0,\infty)\):
\begin{align}
  \left\{
  \begin{array}{rlrl}
    -\nabla \cdot y^{\beta} \nabla U\left({\color{magenta}x},y\right) &= 0, && \left({\color{magenta}x},y\right)\in\mathcal{C}, \\
    U\left({\color{magenta}x},y\right) &= 0, && \left({\color{magenta}x},y\right)\in\partial_{L}\mathcal{C} := \partial\Omega\times[0,\infty), \\
    \frac{\partial U}{\partial n^{\beta}}\left({\color{magenta}x}\right) &= d_{\alpha}f\left({\color{magenta}x}\right), && {\color{magenta}x}\in \Omega,
  \end{array} \right. \label{eq:extendedProblem}
\end{align}
where \(\beta = 1-\alpha/2\), \(d_{\alpha} = 2^{1-\alpha}\frac{\Gamma\left(1-\alpha/2\right)}{\Gamma\left(\alpha/2\right)}\), and
\begin{align*}
  \frac{\partial U}{\partial n^{\beta}}\left({\color{magenta}x}\right)&= -\lim_{y\rightarrow 0^{+}}y^{\beta}\frac{\partial U}{\partial y}\left({\color{magenta}x},y\right).
\end{align*}
The solution to \(\eqref{eq:specFracLapl}\) can then be recovered by taking the trace of \(U\) on \(\Omega\), i.e. \(u=\operatorname{tr}_{\Omega}U\). As this higher-dimensional formulation involves only integer-order operators, standard discretization approaches may be applied, as in \cite{Nochetto2015, Gatto2015, AinsworthGlusa2017_HybridFiniteElementSpectral}.

The work of Nochetto et al. \cite{Nochetto2015} used the Dirichlet-to-Neumann map for a singular elliptic problem posed on a semi-infinite cylinder in order to study solution techniques for problems on bounded domains with Dirichlet, Neumann, and Robin boundary conditions. A truncation was proposed based on the rapid decay of the solution to the problem on the cylinder, and a priori error estimates were derived in weighted Sobolev spaces. Also along these lines, a hybrid finite element-spectral method was recently introduced by Ainsworth and Glusa \cite{AinsworthGlusa2017_HybridFiniteElementSpectral}, where the discretization along the direction of the problem domain $\Omega$ was done using a finite element method, and the direction along the cylinder was discretized with a spectral method. We emphasize that these methods are all for the fractional Poisson problem with zero Dirichlet boundary conditions. For the computations in this article, we prefer to use the spectral element method of Song et al. \cite{SongXuKarniadakis2017}, described below, due to its accuracy and ease of implementation for the considered examples.

\subsubsection{Discrete Eigenfunction and Spectral Element Methods}\label{SEM}

We consider the {\color{black} following} eigenvalue problem (EVP) for the spectral Laplacian operator:
\begin{align}
\label{eq:Eig}
\begin{split}
-\Delta \phi - \lambda \phi &= 0, \hspace{10pt} x \in \Omega, \\
\phi \big|_{\partial \Omega} &= 0,
\end{split}
\end{align}
where $\Omega\subset\mbbR^d,$ with $d=1,2,3$, is a bounded domain. When $d=1$, we use a Galerkin expansion to discretize Equation \eqref{eq:Eig}. For a nonnegative integer $N$, let $\mbbP_N(\Omega)$ be the space of polynomials on $\Omega$ up to order $N$. The Galerkin basis functions $\{p_n(x)\}_{n=1}^N$ are chosen from the space $\mcS_N(\Omega) := \mbbP_N(\Omega) \cap H_0^1(\Omega)$, so that the eigenfunctions $\phi$ can be approximated as
\begin{align}\label{galerkin_expansion}
	\phi &\approx \sum_{n=1}^N \hat{\phi}_n p_n(x),
\end{align}
where $\hat{\phi}_n = (\phi, p_n)$ and $(\cdot,\cdot)$ represents the $L^2(\Omega)$-inner product. The weak form of Equation \eqref{eq:Eig} is then written as
\begin{align}\color{forest}
	(-\Delta \phi, p_k) - \lambda(\phi, p_k) = (\nabla \phi, \nabla p_k) {\color{magenta}- \lambda(\phi, p_k)} &= 0.
\end{align}
The inner product {\color{forest}of the gradients} is discretized according to
\begin{align}
	(\nabla \phi, \nabla p_k) &= \sum_{n=1}^N \hat{\phi}_n (\nabla p_n, \nabla p_k) = A_N \bm{\phi},
\end{align}
where $(A_N)_{kn} = (\nabla p_n, \nabla p_k)$ and $\bm{\phi} = (\hat{\phi}_1, \hat{\phi}_2, \dots, \hat{\phi}_{N})$. {\color{forest}Again using the expansion \eqref{galerkin_expansion} of $\phi$ in $(\phi,p_k)$ yields}
\begin{align}
	(\phi, p_k) &= \sum_{n=1}^N \hat{\phi}_n ( p_n, p_k) = (M_N) \bm{\phi},
\end{align}
{\color{forest}where $(M_N)_{kn} = (p_n, p_k).$} These discretizations result in the discrete eigenproblem
\begin{align}
\label{disceigenp} A_N \bm{\phi} - \lambda M_N \bm{\phi} &= 0.
\end{align}
We left-multiply Equation \eqref{disceigenp} by $M_N^{-1}$, so that the discrete eigenproblem becomes
\begin{align}
	(M_N^{-1} A_N - \lambda) \bm{\phi} &= 0.
\end{align}

We define the matrix $K_N := M_N^{-1}A_N$ and compute its discrete eigenpairs $\{(\lambda_i, \phi_i)\}_{i=1}^N$ using the QR algorithm, {\color{forest}following \cite{SongXuKarniadakis2017}.} The eigenpairs are used to approximate the fractional Laplacian of a function $u$ according to
\begin{align}
	(-\Delta_{\text{spectral}})^{\alpha/2} u &\approx \sum_{i=1}^N \lambda_i^{\alpha/2} (u, \phi_i)_{L^2(\Omega)} \phi_i.
\end{align}
{\color{forest}Furthermore, the solution of the spectral fractional Poisson equation with zero Dirichlet boundary conditions may be approximated using the formula}
\begin{align}
	u &\approx \sum_{k=1}^N u_k \phi_k, \ \ \ \text{where} \ \ \ u_k := \frac{(f,\phi_k)}{\lambda_k^{\alpha/2}}.
\end{align}
We call this method the \emph{Discrete Eigenfunction Method}.

When $d > 1$, we use a spectral element method (SEM) to discretize Equation \eqref{eq:Eig}. First, a grid with $\ell$ elements is generated to discretize the domain $\Omega$. The SEM is developed using nodal Lagrangian polynomials of degree $N$ on each element. The number of degrees of freedom for the SEM is denoted by $\mcN := \ell \cdot N$. $A_\mcN$ represents the $\mcN \x \mcN$ stiffness matrix associated with the integer Laplacian, where each element of $A_\mcN$ is the $L^2$-inner product of the gradients of the Lagrangian polynomials. Similarly, the mass matrix $M_\mcN$ of size $\mcN \x \mcN$ is computed as the $L^2$-inner product of the Lagrangian polynomials. {\color{forest}The eigenpairs $\{\lambda_i, \phi_i\}_{i=1}^\mcN$ are computed by solving the eigenproblem \eqref{disceigenp}.
Then, a weighted Gram-Schmidt procedure is applied to transform the basis $\{\phi_i\}$ into the orthonormal basis $\{\tilde{\phi}_i\}_{i=1}^\mcN$ (see \cite{SongXuKarniadakis2017} for the details of this procedure). Finally, we use the equation
\begin{align}
	\sum_{i=1}^\mcN \lambda_i^{\alpha/2} (u, \tilde{\phi}_i) \tilde{\phi}_i &\approx \sum_{i=1}^\infty (f,\tilde{\phi}_i)\tilde{\phi}_i \approx \sum_{i=1}^\mcN (f, \tilde{\phi}_i) \tilde{\phi}_i,
\end{align}
from which we infer that $\hat{u}_i = (u,\tilde{\phi}_i)= (f,\tilde{\phi}_i)\lambda_i^{-\alpha/2}.$
Then, the solution is written as $u \approx \sum_{i=1}^\mcN \hat{u}_i \tilde{\phi}_i.$}
Further details of this method can be found in \cite{SongXuKarniadakis2017}, {\color{blue} where it is demonstrated in numerical examples that the method is stable and accurate for the same domains that we consider in the following sections. Rigorous proofs of these properties are under development and have not been published at the time of this writing.}

\subsubsection{\color{blue} Boundary Regularity of Solutions using the Discrete Eigenfunction Method}\label{sec:regularity}

{\color{blue}Now we address an important issue of regularity near the boundary and the enforcement of boundary conditions of solutions to the spectral fractional Poisson equation using the discrete eigenfunction method. Recall the regularity results of Grubb \cite{grubb_spectral} for the Dirichlet problem for the spectral fractional Laplacian for $\alpha<1$, reviewed in Section \ref{sec:spectral_regularity}: 
 if $(-\Delta)^{\alpha/2}u = f$, and
if $f \in H^s$, then $u \in H^{s+\alpha}$. This means that if $s+\alpha \leq 1/2$, then $u \not \in H^{1/2+\varepsilon}$. However, {\color{magenta} numerical solutions obtained by} the discrete eigenfunction method even with $s+\alpha \leq 1/2$ will always satisfy the zero boundary condition exactly, even when the true solution, by the regularity result just stated, does not admit a trace. This is an artifact of the discrete eigenfunction method, in which the approximant is always a finite-dimensional projection of $u$ onto the eigenfunctions, which are zero on the boundary. 
\color{magenta}{As a simple example in one dimension, one can plot the partial sums of the solution \eqref{summation_u} to the homogeneous spectral fractional Poisson with the strictly $L^2$ right-hand side \eqref{E:regularity_f_definition}; for $\alpha \leq 0.5$, the lack of trace of the true solution will manifest as Gibbs oscillations as the number of discrete eigenfunctions is increased.}}

\FloatBarrier
{
\color{chocolate}
\subsection{Directional {\color{chocolate}Representation}}\label{direction_definition}
The directional fractional Laplacian $\nabla^{\alpha}_M$ in $\mathbb{R}^d$ \eqref{general_directional_laplacian}
was reviewed in Section \ref{sec:directional_unbounded}. 
It was pointed out that for a uniform measure $M$, the directional fractional Laplacian reduces to the standard fractional Laplacian $(-\Delta)^{\alpha/2}$ that is the focus of this paper; otherwise, it represents the generator of general multivariate $\alpha$-stable Levy motions. Thus, when applied to functions in a bounded domain satisfying an exterior boundary condition, this (uniform) directional fractional Laplacian agrees with the Riesz fractional Laplacian, and is an advantageous representation in certain geometries. 

In this section, we present a new radial basis function collocation method based on the directional representation and a modified vector Gr\"unwald-Letnikov formula to discretize the Poisson equation in bounded domains for the Riesz fractional Laplacian. The method has a clear extension to directional fractional Laplacians (with non-uniform measures), but to maintain the scope of the article, we focus only on the uniform case resulting in the Riesz fractional Laplacian. The use of this method for numerical studies of general directional fractional Laplacians in bounded domains is an interesting topic for future work.}

{\color{magenta}To introduce the method and illustrate the main idea, we split the fractional directional integral \eqref{directional_integral0} appearing in the directional representation \eqref{general_directional_laplacian} of the fractional Laplacian as the sum of two integrals corresponding to integration of $u$ inside and outside the bounded domain $\Omega$: 
\begin{equation}\label{splitting_of_directional_integral}
    I^{\beta}_{\boldsymbol{\theta}}u(x)=\frac{1}{\Gamma(1-\beta)}\left(\int_{0}^{\delta(x,\boldsymbol{\theta},\Omega)}\varsigma^{-\beta}u(x-\varsigma\boldsymbol{\theta})d\varsigma+\int_{\delta(x,\boldsymbol{\theta},\Omega)}^{+\infty}\varsigma^{-\beta}g(x-\varsigma\boldsymbol{\theta})d\varsigma\right).
\end{equation}
Here, $\delta(x,\boldsymbol{\theta},\Omega)$ is the distance from $x$ to the domain boundary $\partial\Omega$ along the direction $-\boldsymbol{\theta}$. 
In \cite{pang2013gauss,pang2015space}, $\delta(\cdot,\cdot,\cdot)$ is termed the \textit{backward distance}. The nonlocal exterior Dirichlet condition $u(x)=g(x)$ for $x\in \mathbb{R}^d\setminus \Omega$ is assumed. Note that for $g \equiv 0$, the second integral vanishes. For general (non-zero) $g$, we introduce a finite difference (Gr\"unwald-Letnikov type) scheme to approximate the fractional directional derivative $D^{\alpha}_{\boldsymbol{\theta}}u(x)$ following the above representation for $I^{\beta}_{\boldsymbol{\theta}}u(x)$, which is then used in a collocation method.}

{\color{magenta}
\subsubsection{A Radial Basis Function Collocation Method for the Riesz fractional Laplacian}\label{RBFM}

The radial basis function (RBF) collocation method has been used to solve equations involving directional fractional Laplacians \cite{pang2015space}. Here, we focus on the fractional Poisson problem
\begin{equation}\label{frac-Poisson}
 \begin{split}
  (-\Delta_{\text{Riesz}})^{\alpha/2}u(x) & = f(x), \quad x\in\Omega \subset \mathbb{R}^d,\\
  u(x) & = g(x), \quad x\in \mathbb{R}^d\setminus\Omega.
 \end{split} 
\end{equation}
The first step of the RBF collocation method is to approximate $u(x)$ inside $\Omega$ as the weighted sum of RBFs:
\begin{equation}\label{RBF-expansion}
    u(x) \approx \sum_{j=1}^{M+N}\lambda_j\phi(|x-x_j|),\quad x\in \Omega,
\end{equation}
where $\lambda_j$'s are unknown coefficients. The RBF $\phi(\cdot)$ is taken as the multiquadric function with the shape parameter $c$,
\begin{equation*}
    \phi(r)=\sqrt{r^2+c^2}.
\end{equation*}
It is straightforward to consider other RBFs \cite{chen2014recent}, but here we restrict our attention to the multiquadric RBF.
The collection of points $\{x_j\}_{j=1}^{j=M+N}$ are called collocation points. The first $M$ collocation points are located inside the domain $\Omega$, and the last $N$ collocation points are located on the boundary $\partial \Omega$ (to obtain higher accuracy near $\partial \Omega$). Enforcing the approximate solution \eqref{RBF-expansion} to satisfy both the equation inside $ \Omega$ and the exterior condition on $\partial \Omega$ from \eqref{frac-Poisson} produces the following equations:
\begin{equation}\label{discrete-sym-0}
    \begin{split}
       C_{\alpha,d}\int_{|\boldsymbol{\theta}|=1}D^{\alpha}_{\boldsymbol{\theta}}u(x_i)d\boldsymbol{\theta} & = f(x_i), \quad x_i\in \Omega, \quad i=1,2,\cdots, M, \\
        \sum_{j=1}^{M+N}\lambda_j\phi(|x_i-x_j|)&=g(x_i), \quad x_i \in \partial\Omega, \quad i=M+1, M+2, \cdots, M+N.
    \end{split}
\end{equation}
The full exterior condition $u = g$ on $\mathbb{R}^d \setminus \Omega$ in \eqref{frac-Poisson} is enforced directly through the discretization of the operator $D^{\alpha}_{\boldsymbol{\theta}}$ in the first equation above as alluded to by \eqref{splitting_of_directional_integral}. 
We discretize $D^{\alpha}_{\boldsymbol{\theta}}$ through a finite difference scheme known as the vector Gr\"unwald-Letnikov (GL) formula that was proposed in \cite{meerschaert2004vector}, splitting the terms into interior and exterior contributions. The original vector GL formula takes the form
\begin{equation}\label{original_gl_formula}
    D^{\alpha}_{\boldsymbol{\theta}}u(x_i)  = h^{-\alpha}\sum_{k=0}^{\infty}(-1)^k\binom{\alpha}{k}u(x_i-kh\boldsymbol{\theta})+O(h). 
\end{equation}
This formula has first-order accuracy with respect to the spatial step $h$. The coefficients $c_k = (-1)^k\binom{\alpha}{k}$ can be alternatively calculated by the iterative formula $c_0 =1, c_k =\left(1-\frac{\alpha+1}{k}\right)c_{k-1},k\ge 1$. For our scheme, we truncate the vector GL formula and combine it with the RBF approximation \eqref{RBF-expansion}:
\begin{equation}\label{modified-GL}
  \begin{split}
    D^{\alpha}_{\boldsymbol{\theta}}u(x_i) 
    & \approx h^{-\alpha}\sum_{k=0}^{K_1}c_k u (x_i-kh\boldsymbol{\theta}) + h^{-\alpha}\sum_{k=K_1+1}^{K_2}c_k g(x_i-kh\boldsymbol{\theta}) \\
    & = h^{-\alpha}\sum_{k=0}^{K_1}c_k\sum_{j=1}^{M+N}\lambda_j\phi(|x_i-kh\boldsymbol{\theta}-x_j|) + h^{-\alpha}\sum_{k=K_1+1}^{K_2}c_k g(x_i-kh\boldsymbol{\theta}).
    \end{split}    
\end{equation}
The integer $K_1$ depends on $x_i$, $\boldsymbol{\theta}$ and $\Omega$ -- it is determined by the value of the backward distance $\delta(x,\boldsymbol{\theta},\Omega)$, and is an integer taken such that $x_i - K_1 h \boldsymbol{\theta} \in \Omega $ and $x_i - (K_1+1)h\boldsymbol{\theta} \in \mathbb{R}^d\setminus \Omega$. It is easy to see that $\delta(x_i,\boldsymbol{\theta},\Omega)\approx K_1 h$ by noting that $|\boldsymbol{\theta}|=1$. We refer to formula \eqref{modified-GL} as the modified GL formula.

Before discussing how \eqref{modified-GL} is used in \eqref{discrete-sym-0} to obtain a numerical scheme, we discuss in more detail the truncation parameter $K_2$. 
Unlike $K_1$, the integer $K_2$ is user-selected and effects the accuracy of the modified GL formula \eqref{modified-GL}, introducing the truncation error 
\begin{equation}\label{modified-GL-error}
    R(x_i)= h^{-\alpha}\sum_{k=K_2+1}^{\infty}c_kg(x_i-kh\boldsymbol{\theta})
\end{equation}
into the calculation of $D^{\alpha}_{\boldsymbol{\theta}} u(x_i)$.
In the case that $g$ has compact support, choosing $K_2 h \ge \text{diam}(\text{supp } g)$ ensures no truncation error. In the case that $g$ does not have compact support, but decays at infinity, the following are two strategies to truncate:
\begin{enumerate}
\item Replace $g$ in the problem \eqref{frac-Poisson} by $g \chi_\mathcal{D}$, where $\chi_\mathcal{D}$ denotes the indicator function of an appropriate truncation domain $\mathcal{D}$  outside of which $g$ is sufficiently small:
\begin{equation}\label{trunc-Poisson}
 \begin{split}
  (-\Delta_{\text{Riesz}})^{\alpha/2} \tilde{u} & = f(x), \quad x\in\Omega \subset \mathbb{R}^d,\\
  \tilde{u}(x) & = g(x) \chi_\mathcal{D}(x), \quad x\in \mathbb{R}^d\setminus\Omega.
 \end{split} 
\end{equation}
One can then use the modified GL formula \eqref{modified-GL} with $g \chi_\mathcal{D}$ as the exterior condition and $K_2 h \ge \text{diam}(\mathcal{\mathcal{D}})$. Theoretically, the total error in solving \eqref{frac-Poisson} is controlled by the numerical error in solving the truncated problem \eqref{trunc-Poisson} plus the truncation error, i.e., the difference of exact solutions $u - \tilde{u}$. This truncation error has been studied by Acosta, Borthagaray, and Heuer in  \cite{acosta2018finite}, where the same truncation strategy was used to implement a finite element solver for the Riesz fractional Poisson problem with nonzero exterior condition. It was shown that $g \in H^s(\mathbb{R}^d \setminus \Omega)$ implies that as $\text{diam}(\mathcal{D}) \rightarrow \infty$, $\|u - \tilde{u}\|_{H^{\alpha/2}(\Omega)} \rightarrow 0$. Practically, this strategy corresponds to choosing $K_2h \approx \delta(x_i, \bm{\theta},\mathcal{D})$. 
\item
A more direct way to exploit the decay of the coefficients in \eqref{original_gl_formula} is to choose $K_2$ such that $R(x_i)$ is below a threshold value. In the simplest case where $K_2$ is chosen to be a constant independent of $x_i$, theoretically this corresponds to truncating the integral \eqref{splitting_of_directional_integral} or equivalently \eqref{directional_integral0} in the definition of the operator. In effect, this introduces a truncating parameter or \emph{horizon} into the directional fractional Laplacian \eqref{def:directional_uniform} (more generally \eqref{general_directional_laplacian}). Truncatated forms of the Riesz fractional Laplacian are discussed Section \ref{sec:horizon}, where references are given that establish convergence of the solutions of the associated Poisson problem to those of \eqref{frac-Poisson} in the case of zero exterior condition. However, we are not aware of such convergence results for the inhomogeneous problem, or for works that establish a connection to truncating the directional representations \eqref{general_directional_laplacian}.
\end{enumerate}

We adopt the second strategy; in practice, regardless of the strategy, the choice of $K_2$ should take into account the rate of decay of the exterior condition $g$ away from $\Omega$. The truncation error should be studied both at the level of approximating $D^{\alpha}_{\boldsymbol{\theta}}u(x)$ (i.e., by computing \eqref{modified-GL-error}) at various points $x$ and at the level of the final numerical solution using the RBF collocation scheme outlined below; the numerical solution should be converged with respect to $K_2$ for a fixed set of collocation points. 

We demonstrate this study for the case of $\Omega = \{(x_1, x_2) \in \mathbb{R}^2 \ | \ x_1^2 + x_2^2 < 1\}$ and $g(x)=\exp(-|x|^2)$ for $x\in \mathbb{R}^2\setminus \Omega$. In Fig. \ref{truncation-error}, the remainder $|R(x_i)|$ in \eqref{modified-GL-error} is compared for $K_2 = 2,000$ and $K_2 = 6,000$ with the spatial step fixed to be $h=0.001$  at two ``extreme cases'' of collocation points $x_1=(0,0)$ and $x_2=(0.98,0)$. The summation in $R(x_i)$ is taken from $K_2+1$ to $20,000$. 
Comparing the subplots indicates that for the collocation point $x_1=(0,0)$ at the center of the disk, the truncation error is isotropic with respect to $\theta$ due to the symmetry of $g(x)=\exp(-|x|^2)$ around $x_1$, but for the point $x_2=(0.98,0)$ near the boundary, the truncation error is direction-dependent. However, the choice of $K_2 = 6,000$ yields $|R(x_i)|$ sufficiently close to machine precision for use in the RBF collocation scheme. 

We also verify that the solution obtained by solving the RBF collocation scheme below is converged at this value of $K_2$ in Fig. \ref{solu-conv-K2}. In that figure, we show the relative $L^2$-error of the final numerical solution from the reference solution with respect to $K_2$ (for full details, see the discussion of \eqref{exact-u-ihm} below). Not surprisingly, the final numerical solution $u$ is less sensitive to $K_2$ than is $D^{\alpha}_{\boldsymbol{\theta}}u$, and is actually converged near $K_2 = 2,000$ (at which point the error becomes dominated by the numerical error of solving the RBF collocation system below). This justifies the use of $K_2 = 6,000$ with $h = 0.001$ for this choice of $\Omega$ and $g$, which are used as test cases below.

\begin{figure}[H]
\begin{minipage}{.49\textwidth}
\begin{center}
\includegraphics[width=\textwidth]{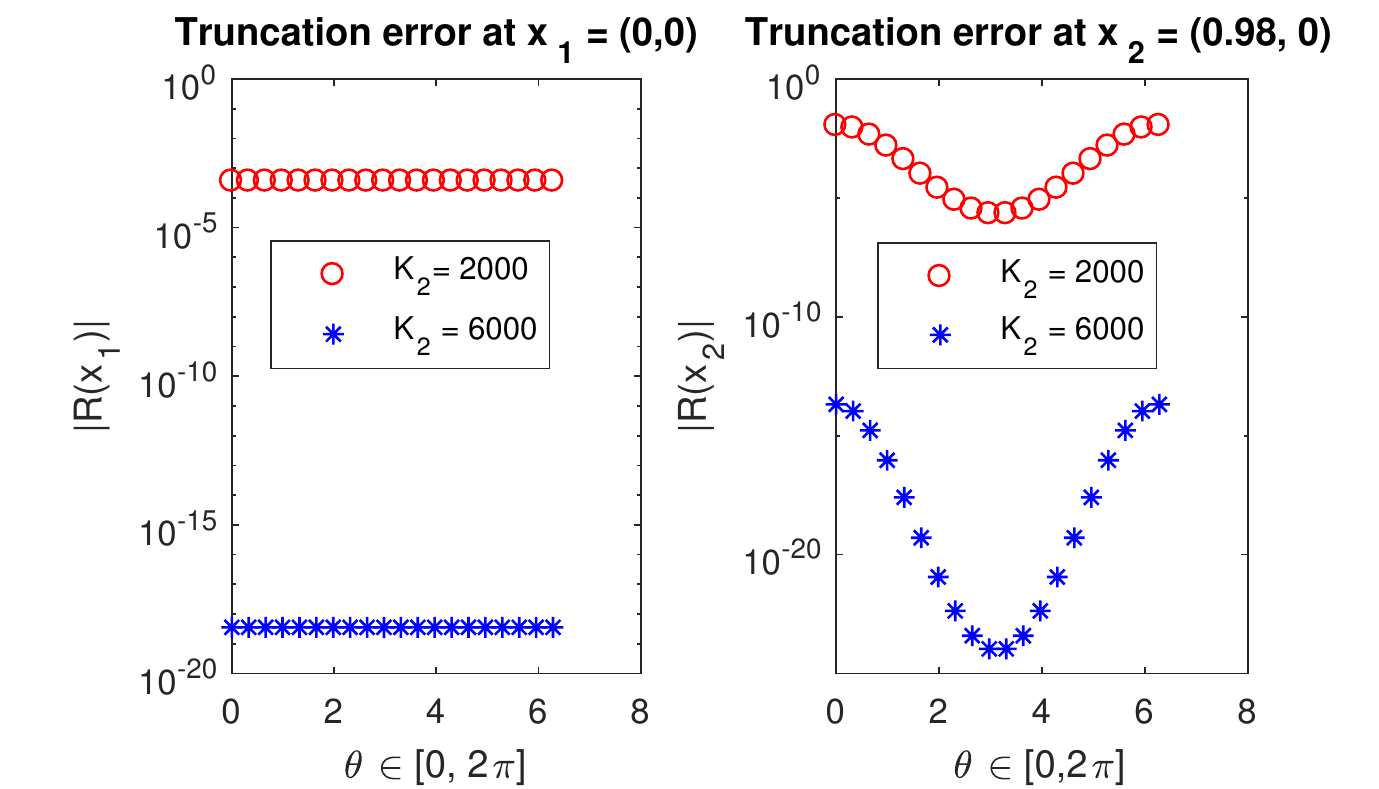}
\small (a) $\alpha = 0.5$
\end{center}
\end{minipage}
\begin{minipage}{.49\textwidth}
\begin{center}
\includegraphics[width=\textwidth]{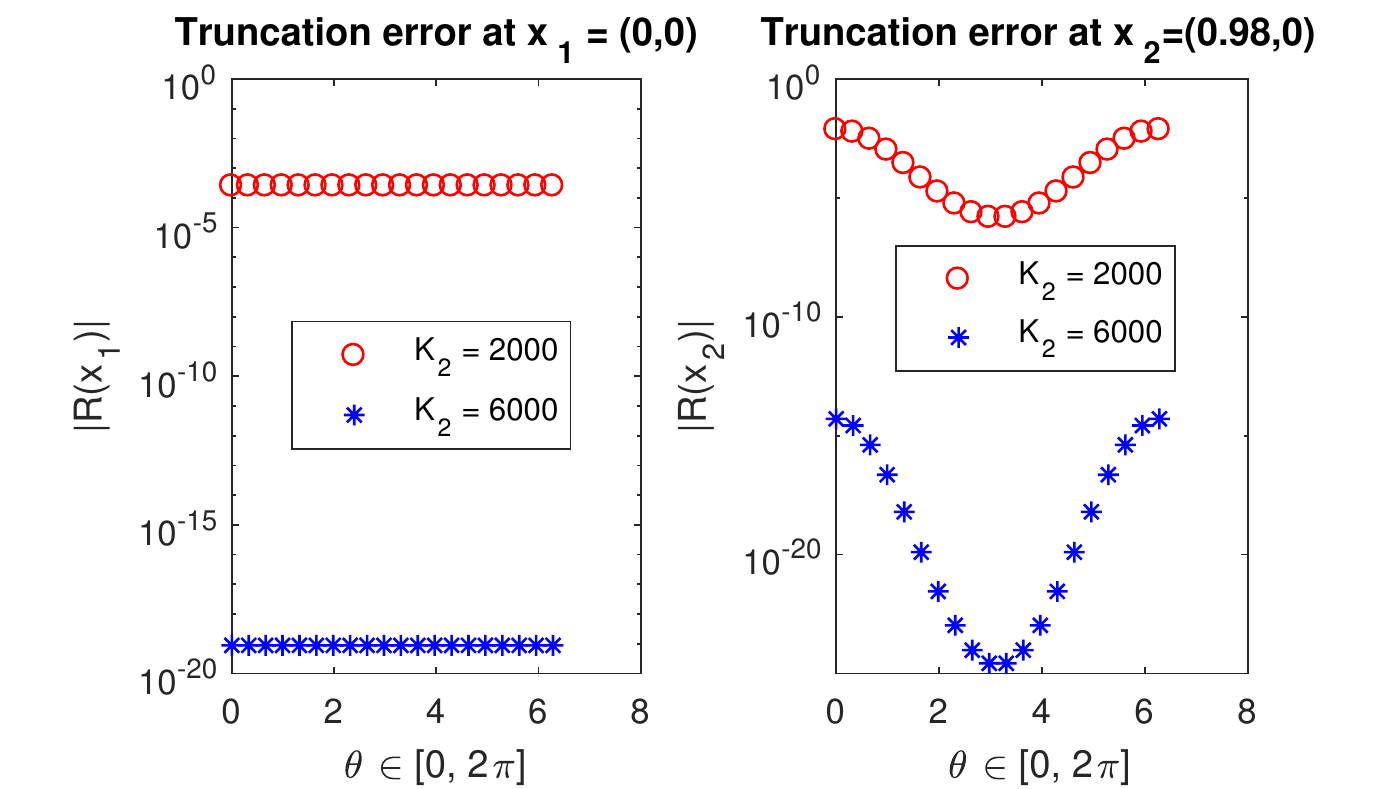}
\small (b) $\alpha = 1.5$
\end{center}
\end{minipage}
\caption{\label{truncation-error}\color{magenta}Absolute values of truncation errors $|R(x_i)|$ \eqref{modified-GL-error} for the modified GL formula \eqref{modified-GL} evaluated at the collocation points $x_1=(0,0)$ and $x_2=(0.98,0)$. Two truncation parameters $K_2=2000$ and $K_2=6000$ are considered. Subplots (a) and (b) correspond to $\alpha=0.5$ and $\alpha=1.5$, respectively. }
\end{figure}

{\color{magenta}
To complete the description of the RBF collocation scheme, next we discuss the discretization of the integral of $D^{\alpha}_{\boldsymbol{\theta}}(\cdot)$ with respect to 
$\boldsymbol{\theta}$, as called for by \eqref{discrete-sym-0}. This can be done, e.g., with the trapezoidal rule or Gauss-Legendre quadrature; 
we adopt the latter.
%%%
Of course, in one dimension, the desired integral reduces to 
\begin{equation*}
    \begin{split}
  \int_{|\boldsymbol{\theta}|=1}D^{\alpha}_{\boldsymbol{\theta}}u(x_i)d\boldsymbol{\theta} & = D^{\alpha}_{\boldsymbol{\theta}=(1,0)}u(x_i) + D^{\alpha}_{\boldsymbol{\theta}=(-1,0)}u(x_i) 
%\\
%       & =D^{\alpha}_{a+}u(x_i)+D^{\alpha}_{b-}u(x_i), \quad x_i \in (a,b) \subset \mathbb{R},
     \end{split}  
\end{equation*} 
and no quadrature rule is needed.  
%%%
In two dimensions, the integral over $\boldsymbol{\theta}$ is approximated by 
\begin{equation}\label{GL-rule}
    \int_{|\boldsymbol{\theta}|=1}D^{\alpha}_{\boldsymbol{\theta}}u(x_i)d\boldsymbol{\theta} = \int_0^{2\pi}D^{\alpha}_{\boldsymbol{\theta}=(\cos\theta,\sin\theta)}u(x_i)d\theta \approx \sum_{l=1}^P w_l D^{\alpha}_{\boldsymbol{\theta}_l=(\cos\theta_l,\sin\theta_l)}u(x_i).
\end{equation}
$w_l$ and $\theta_l$ are the Gauss-Legendre weights and points, respectively. 
%%%
In three dimensions case, Gauss-Legendre quadrature yields
\begin{equation*}
    \begin{split}
         \int_{|\boldsymbol{\theta}|=1}D^{\alpha}_{\boldsymbol{\theta}}u(x_i)d\boldsymbol{\theta} & = \int_0^{2\pi}\int_0^{\pi}D^{\alpha}_{\boldsymbol{\theta}=(\sin\phi\cos\theta,\sin\phi\sin\theta,\cos\phi)}u(x_i)\sin\phi d\phi d\theta \\
         & \approx \sum_{l=1}^P\sum_{m=1}^Q w_l v_m\sin\phi_m D^{\alpha}_{\boldsymbol{\theta}_l=(\sin\phi_m\cos\theta_l,\sin\phi_m\sin\theta_l,\cos\phi_m)}u(x_i),
    \end{split}
\end{equation*}
where $(\phi_m,v_m)$ and $(\theta_l,w_l)$ are Gauss-Legendre quadrature point and weight pairs, respectively. The $\sin\phi$ appearing in the integrand comes from the determinant of Jacobian matrix in the transformation from Cartesian coordinates to spherical coordinates. Once the integral of $D^{\alpha}_{\boldsymbol{\theta}} u$ has been discretized as above, it can be combined with the modified GL formula \eqref{modified-GL} to obtain a linear system for the coefficients $\lambda_j$ in the RBF expansion 
\eqref{RBF-expansion}.

For simplicity, we consider two dimensions. 
According to the modified GL formula \eqref{modified-GL} and the quadrature rule \eqref{GL-rule}, we can rearrange the discretized system \eqref{discrete-sym-0} as a system of $M+N$ equations in $M+N$ unknown coefficients:
\begin{align}
\begin{split}
\label{discrete-sym-1}
& C_{\alpha,d}h^{-\alpha}\sum_{j=1}^{M+N}\sum_{l=1}^P w_l\sum_{k=0}^{K_1}c_k \phi(|x_i-kh\boldsymbol{\theta}_l-x_j|)\lambda_j = 
\\
&\qquad \qquad \qquad \qquad \qquad f(x_i) - C_{\alpha,d}h^{-\alpha}\sum_{l=1}^P w_l \sum_{k=K_1+1}^{K_2}c_k g(x_i - kh\boldsymbol{\theta}_l), \quad x_i\in \Omega, \quad i=1,2,\cdots,M;
\\
&\sum_{j=1}^{M+N}\phi(|x_i-x_j|)\lambda_j = g(x_i), \quad x_i \in \partial\Omega, \quad i=M+1, M+2, \cdots, M+N.
\end{split}
\end{align}
This linear system can also be abbreviated in matrix-vector form as
\begin{equation}\label{matrix_RBF}
    \left[ 
    \begin{array}{c}
       (-\Delta)^{\alpha/2}\boldsymbol{\Phi}_{M\times (M+N)} \\
       \boldsymbol{\Phi}_{N\times(M+N)}
    \end{array} \right] 
    \left[
    \begin{array}{c}
        \boldsymbol{\lambda}_{M\times1} \\
        \boldsymbol{\lambda}_{N\times1}
  \end{array} \right] =
   \left[
    \begin{array}{c}
       \mathbf{F}-\mathbf{R}(K_2) \\
        \mathbf{G}
   \end{array} \right].        
\end{equation}
}

\subsubsection{\color{chocolate} Validation Using {\color{magenta}} Reference Solutions and Walk-on-Spheres\label{RBF_validation}}
{\color{magenta}Next, we demonstrate the convergence of the numerical solutions produced by the RBF collocation method by considering a circular domain 
$\{(x_1, x_2) \in \mathbb{R}^2 \ | \ x_1^2 + x_2^2 < 1\}$. 
In principle, the collocation points $\{x_i\}$ can have arbitrary distribution in $\Omega$ provided they do not overlap. 
We choose a total of $I^2+1$ collocation points in $\Omega \cup \partial \Omega$ on a uniform grid in polar coordinates:
\begin{equation}
\label{generating_points}
x_i = \big(
i_1\Delta r\cos(i_2\Delta\theta), i_1\Delta r\sin(i_2\Delta\theta) 
\big), 
\quad i = (i_1-1) \times I + i_2, 
\quad i_1,i_2 = 1,2,\cdots I, 
\end{equation}
where $\Delta r = 1/I$, and $\Delta \theta=2\pi/I$. 
The subscript pair $(i_1, i_2)$ labeling a collocation point in two dimensions is used to define the single subscript $i$ to index the vector of collocation points/weights. 
There are $M = (I-1) \times I + 1$ collocation points inside $\Omega$, consisting of $I-1$ concentric rings of $I$ collocation points together the center point $(0,0)$, and $N = I$ collocation points on $\partial \Omega$. The number of collocation points varies for convergence studies, but the the numerical errors are evaluated at a fixed collection of test points 
$\mathbf{x}_{\text{test}}$ which are generated via \eqref{generating_points} using $I=39$ throughout.
The shape parameter $c$ in the RBF $\phi(\cdot)$ is selected as $c=3/I$. For the Gauss-Legendre quadrature \eqref{GL-rule}, 16 quadrature points are taken, i.e, $P=16$. 
The spatial step in the modified GL formula is $h = 0.001$. Standard $LU$ decomposition in MATLAB is used to solve the linear system \eqref{matrix_RBF}.

Below, we test both the case of zero exterior condition, in which case the term 
$\sum_{k=K_1+1}^{K2} c_k g(x_i-kh\boldsymbol{\theta})$ in the modified GL formula (\ref{modified-GL}) is simply zero, and the non-zero exterior condition case, in which case we take $K_2=6000$, i.e., $K_2 h = 6$, three times the diameter of the unit disk. 
}
}

\textbf{Convergence of the GL formula}. \textcolor{chocolate}{Before demonstrating the convergence of RBF solutions, we first demonstrate the convergence of the vector GL formula \eqref{modified-GL} with respect to the parameter $h$}. {\color{magenta}In two dimensions}, we consider the following two functions:
\begin{equation}
\label{two_exact_us}
\text{
(i) $u_1(x)=(1-|x|)_+^{\alpha/2}$ and (ii) $u_2(x)=(1-|x|)_+^{1+\alpha/2}$,
}
\end{equation} where $(x)_+ = x$ if $x \ge 0$ and $(x)_+ = 0$ if $x < 0$. 
{\color{magenta}The domain is taken to be a unit disc: $\Omega=\{(x_1,x_2) \ | \  x_1^2+x_2^2 < 1\}$. }
The Riesz fractional Laplacians of these functions are \cite{dyda2012fractional}
\begin{equation}\label{f1_direction_test}
f_1(x)=(-\Delta_{\text{Riesz}})^{\alpha/2}u_1(x)=2^{\alpha}\Gamma\left(\frac{\alpha}{2}+1\right)^2,
\end{equation}
and
\begin{equation}\label{f2_direction_test}
f_2(x) = (-\Delta_{\text{Riesz}})^{\alpha/2}u_2(x) = 2^{\alpha}\Gamma\left(\frac{\alpha}{2}+2\right)\Gamma\left(\frac{\alpha}{2}+1\right)\left(1-\left(1+\frac{\alpha}{2}|x|^2\right)\right),
\end{equation}
respectively. 
\textcolor{chocolate}{
To measure the error, we use the $L^\infty$-error
\begin{equation}
\begin{split}
\epsilon_1 & = \|(-\Delta)^{\alpha/2}\mathbf{u}_1(\mathbf{x}_{\text{test}})-\mathbf{f}_1(\mathbf{x}_{\text{test}})\|_{L^\infty}, \\
\epsilon_2 & = \|(-\Delta)^{\alpha/2}\mathbf{u}_2(\mathbf{x}_{\text{test}})-\mathbf{f}_2(\mathbf{x}_{\text{test}})\|_{L^\infty},
\end{split}
\end{equation}
as well as the relative $L^2$-error
\begin{equation}\label{relative-error}
\begin{split}
E_1 & = \frac{\|(-\Delta)^{\alpha/2}\mathbf{u}_1(\mathbf{x}_{\text{test}})-\mathbf{f}_1(\mathbf{x}_{\text{test}})\|_{L^2}}{\|\mathbf{f}_1(\mathbf{x}_{\text{test}})\|_{L^2}}, \\
E_2 & = \frac{\|(-\Delta)^{\alpha/2}\mathbf{u}_2(\mathbf{x}_{\text{test}})-\mathbf{f}_2(\mathbf{x}_{\text{test}})\|_{L^2}}{\|\mathbf{f}_2(\mathbf{x}_{\text{test}})\|_{L^2}},
\end{split}
\end{equation}
where $(-\Delta)^{\alpha/2}\mathbf{u}(\mathbf{x}_{\text{test}})$ and $\mathbf{f}_i(\mathbf{x}_{\text{test}})$ are the vectors formed by the directional Laplacian using \eqref{modified-GL} and the functions  \eqref{f1_direction_test}, \eqref{f2_direction_test} evaluated at a collection of test points, respectively.} 
{\color{magenta} Fig. \ref{err-int-dir-hm} demonstrates first-order convergence of the errors with respect to $h$}, which agrees with the theoretical convergence rate given in \cite{meerschaert2004vector}.
%Unless stated otherwise, from now on the step size of the GL formula is fixed to be $h=0.001$ (equivalently, $\kappa=1000$).

\begin{figure}[ht!]
\centering
\includegraphics[width=0.5\textwidth]{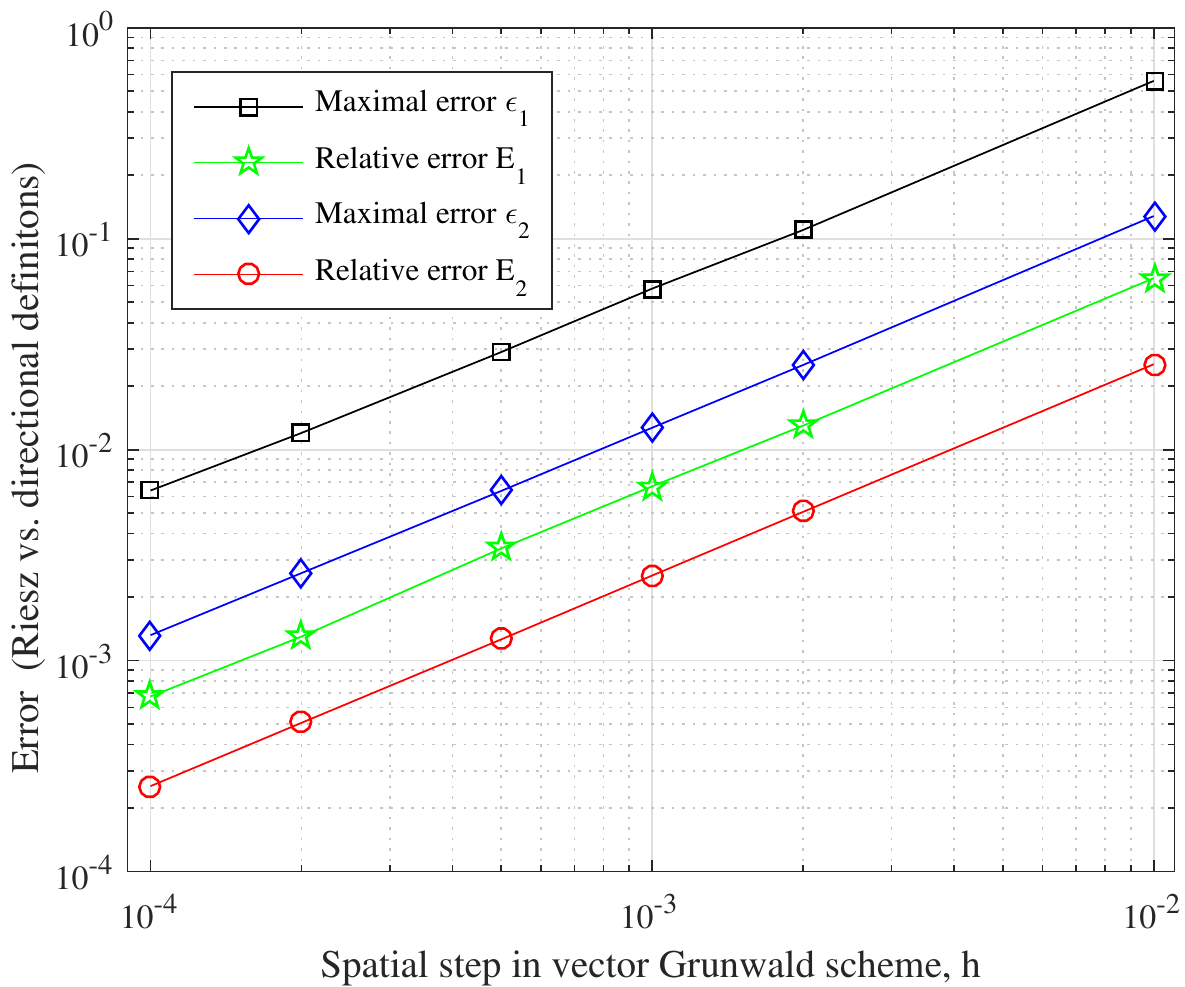}
\caption{\label{err-int-dir-hm}\textit{
Convergence of the {\color{magenta} modified} GL formula with respect to the step size $h$ for the 2D problem with $g(x)=0$ on the unit disk}. 
We plot the error between the modified vector GL formula \eqref{modified-GL} approximation of the fractional Laplacian of $u_1(x)$ and $u_2(x)$ and the source functions $f_1(x)$ and $f_2(x)$, respectively. The slope of each line is 1, demonstrating that the convergence rate of the vector Gr\"unwald scheme is 
$\mcO(h)$.}
\end{figure}

{\color{chocolate}\textbf{Convergence of the RBF solution for $\mathbf{g(x)=0}$}. We demonstrate the convergence of the RBF collocation method with respect to the number of collocation points (namely, $M+N$ in RBF expansion \eqref{RBF-expansion}) for solving the fractional Poisson problem with a zero Dirichlet boundary condition $g(x) = 0$ for $x \in \mathbb{R}^d\backslash\Omega$.
{\color{magenta}
We plot in Fig. \ref{conv_RBF} the convergence of the RBF method for two fabricated solutions $u_1$ and $u_2$ and four different values of $\alpha = 0.1, 0.5, 1.5$ and $1.9$. The convergence is measured in the relative $L_2$ error.} 
It is observed from Fig. \ref{conv_RBF} that the RBF method is convergent, although the convergence rate loses is reduced at larger numbers of collocation points due to the rapidly increasing condition number of the collocation matrix in the linear system \eqref{discrete-sym-1}; see \cite{chen2014recent} for a discussion of this well-known issue. 
{\color{magenta}Also, we see that the RBF method achieves lower accuracy for $u_1$ compared to that for $u_2$; this can be expected since $u_1$ has much higher gradients than $u_2$ near the boundary. We study the local behavior of the error in Fig. \ref{RBF-exa-hom}, in which exact and RBF solutions are compared. We see that the error is larger near the boundary because the collocation points are sparser near the boundary than around the domain center.}

\FloatBarrier

\begin{figure}[H] 
\centering
\includegraphics[width=.8\textwidth]{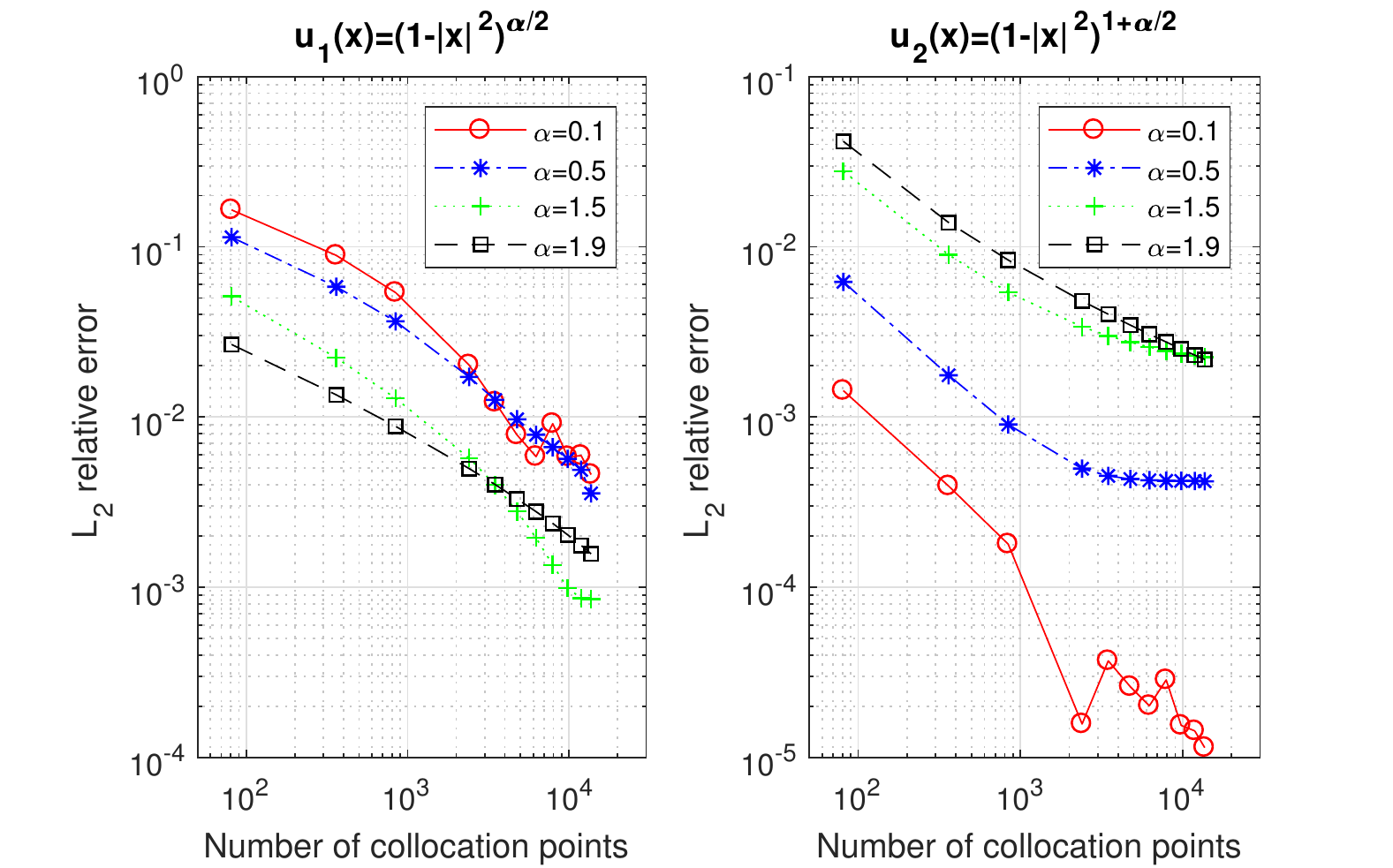}
\caption{\label{conv_RBF}\color{magenta}Convergence of the RBF method for the 2D fractional Poisson problem with zero exterior values on a unit disk. Left: fabricated solution $u_1$ with $\alpha=0.1,0.5,1.5$, and 1.9. Right: fabricated solution $u_2$ with $\alpha=0.1,0.5,1.5$, and 1.9.}
\end{figure}
 
\begin{figure}[ht!] 
\centering
\includegraphics[width=.8\textwidth]{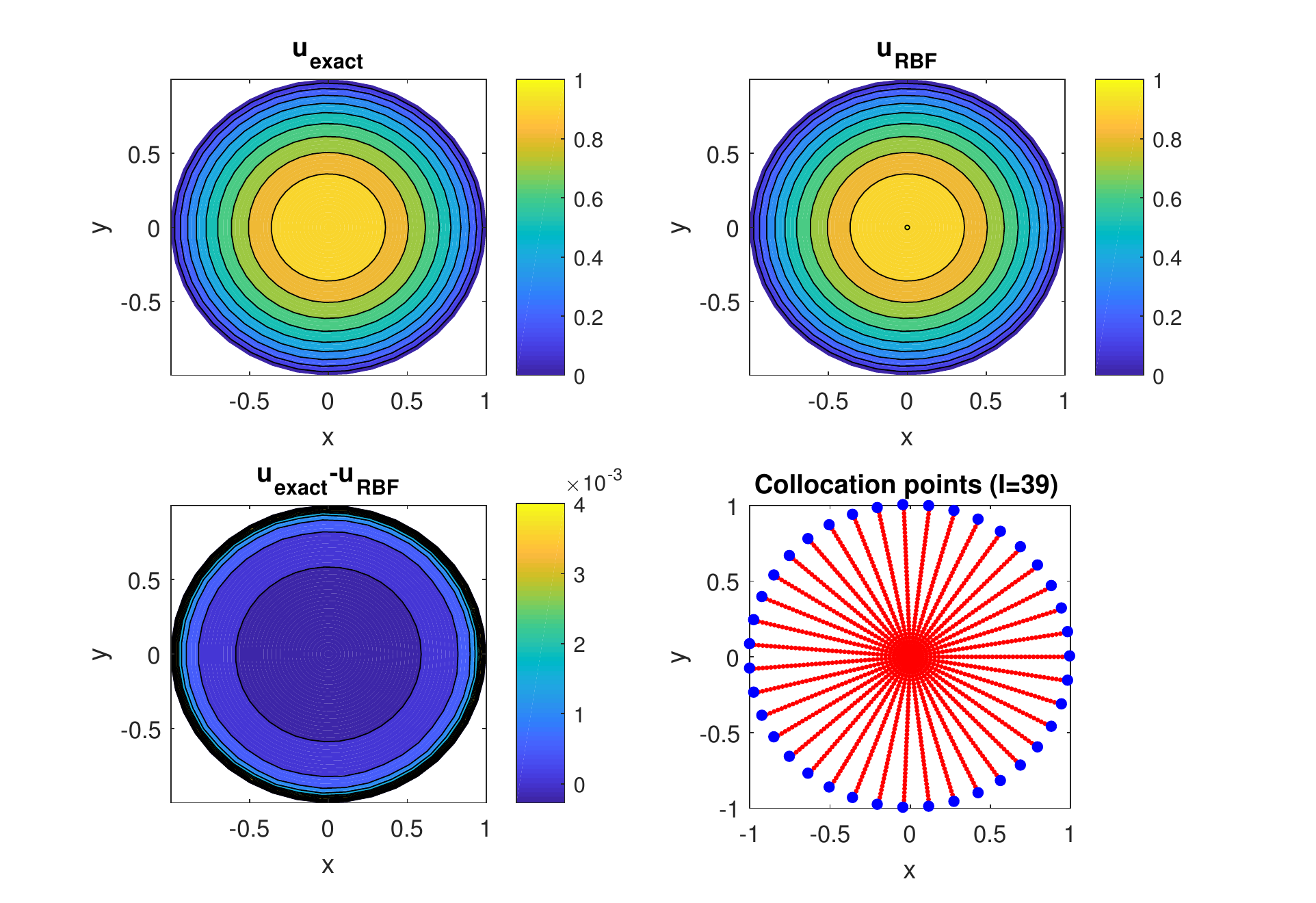}
\caption{\color{magenta}\label{RBF-exa-hom}Comparison of reference and RBF solutions for fractional Poisson problem with zero exterior values on a unit disk. Upper left: Reference solution $u_1$ with $\alpha=1.5$. Upper right: RBF solution with roughly 10000 collocation points (namely $I=99$). Lower left: Pointwise error between exact and RBF solutions. Lower right: Illustration of collocation points with $I=39$ where blue and red points are boundary and domain collocation points, respectively.  }
\end{figure} 

\FloatBarrier

\textbf{Convergence of the RBF solution for $\mathbf{g(x)\ne 0}$ {\color{magenta}and comparison with Walk-on-Spheres}}. 
{\color{chocolate}
From Theorem 2.11 of \cite{bucur2015some}, a solution formula for {\color{magenta}\eqref{frac-Poisson} with $f \equiv 0$} in the case that $\Omega$ is a ball $B_r$ of radius $r$ is
}
\begin{equation} \label{exact-u-ihm}
u(x)=\left\{\begin{array}{ll}
\int_{\mathbb{R}^d\setminus\Omega}P_r(y,x)g(y)dy & x\in B_r,\\
g(x) & x \in \mathbb{R}^d\setminus B_r.
\end{array}\right.
\end{equation}
The {\color{magenta}fractional} Poisson kernel $P_r(y,x)$ is given by
\begin{equation}
P_r(y,x)=\frac{\Gamma(\frac{d}{2})\sin{\frac{\pi\alpha}{2}}}{\pi^{\frac{d}{2}+1}}\left( \frac{r^2-|x|^2}{|y|^2-r^2} \right)^{\alpha/2}\frac{1}{|x-y|^d}.
\end{equation}
Using the RBF collocation method, we solve this problem with $g(x)=\exp(-|x|^2)$. To obtain the reference solution, the integral in \eqref{exact-u-ihm} is computed using a costly direct Monte Carlo (MC) integration with $5\times 10^7$ samples. 
{\color{magenta}To obtain the RBF solution, as discussed in section \ref{RBFM}, we use the truncation parameter $K_2=6,000$ and step size $h=0.001$ in \eqref{modified-GL}. Recall that we verified in Fig. \ref{truncation-error} that these values yielded sufficiently high accuracy of the discretization \eqref{modified-GL}, and at the same time in Fig. \ref{solu-conv-K2} that the final RBF solution was converged with respect to these truncation parameters. Using these parameters, we compute and plot the pointwise errors of RBF and WOS solutions with respect to the reference solution in Fig. \ref{cmp-RBF-WOS}(a); we see that both RBF and WOS methods can achieve accurate results.}

{\color{magenta}Next we compare errors and compute (wall) times of the RBF collocation and WOS methods in Fig. \ref{cmp-RBF-WOS}(b) and Fig. \ref{cmp-RBF-WOS}(c), respectively. Both methods are parallelized over 24 cores. We observe that the error of WOS solution decays with the rate $O(NT^{-0.5})$ where $NT$ is the number of trajectories, as expected for a Monte Carlo method. The error of the RBF method does not exhibit consistent convergence rate. Clearly, the convergence of the RBF method is stymied due to the increasing condition number of the collocation matrix. In terms of the computational time, WOS (which is embarassingly parallel) has time complexity $O(NT)$, whereas the time complexity of the RBF method is $O(NC^2)$ where $NC$ is the number of RBF collocation points. The assembly time of the collocation matrix \eqref{matrix_RBF} is $O(NC^2)$, which dominates the solution time and accounts for the quadratic time complexity here. However, if we solve \eqref{matrix_RBF} by direct methods, cubic complexity will be observed as we increase the number of collocation points. For this reason, iterative methods should be adopted. However, the increasing condition number of the collocation matrix prevents successful application of iterative methods, and preconditioning is required. It is still an open problem to develop preconditioners for the fractional Laplacian collocation matrix. 

For these simulation parameters, the two methods are roughly comparable and WOS method has a slight edge over the RBF method. For example, to achieve a $L_2$ relative error of $10^{-3}$, the WOS and the RBF methods take around 900 and 1100 seconds, respectively. Nevertheless, the RBF method has the potential to be more flexible and converge faster than the WOS method. The WOS method requires establishing Feynman-Kac formula for the associated boundary value problem, but for many fractional PDEs and boundary conditions, the corresponding Feynman-Kac formulas have not yet been found. Moreover, efficient sampling algorithms for WOS require characterization of exit distributions of the associated stochastic process.
}

\begin{figure}[ht!] 
\centering
\includegraphics[width=.5\textwidth]{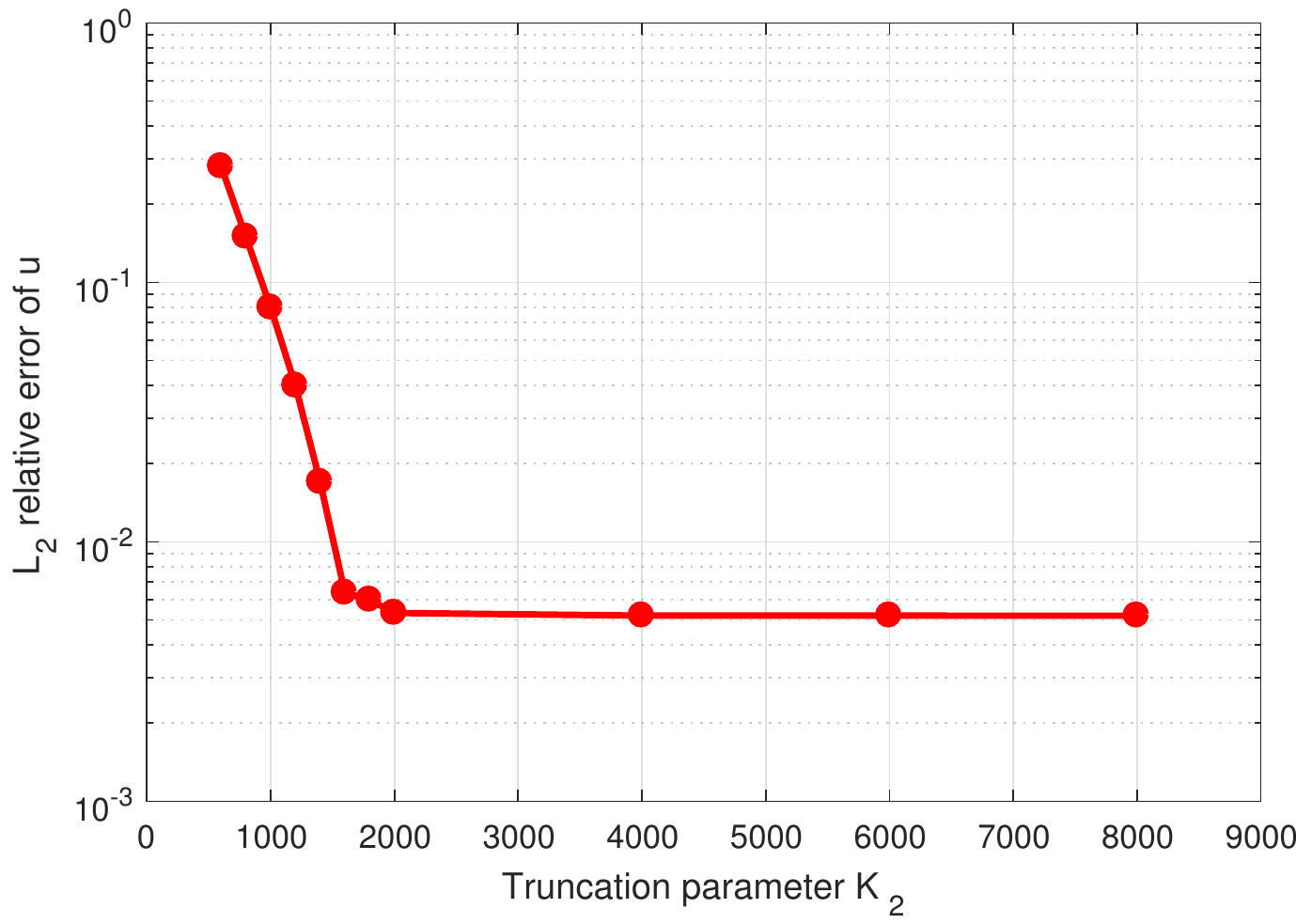}
\caption{\label{solu-conv-K2}\color{magenta} Relative $L^2$ error between the reference solution and the RBF solution to the problem \eqref{frac-Poisson} on the unit disk with $g(x)=\exp(-|x|^2)$ and zero source term $f(x)=0$. This indicates convergence to a consistent numerical solution with respect to the truncation parameter $K_2$ in the modified GL formula \eqref{modified-GL} for $\alpha=1.5$ at a fixed set of 2500 collocation points ($I=39$); for $K_2 > 2000$, the error is dominated by the numerical error in solving the linear system \eqref{matrix_RBF} at these collocation points.}
\end{figure} 

\begin{figure}[ht!]
\centering
\subfloat[Exact vs RBF vs WOS solutions]{
\includegraphics[width=.99\textwidth]{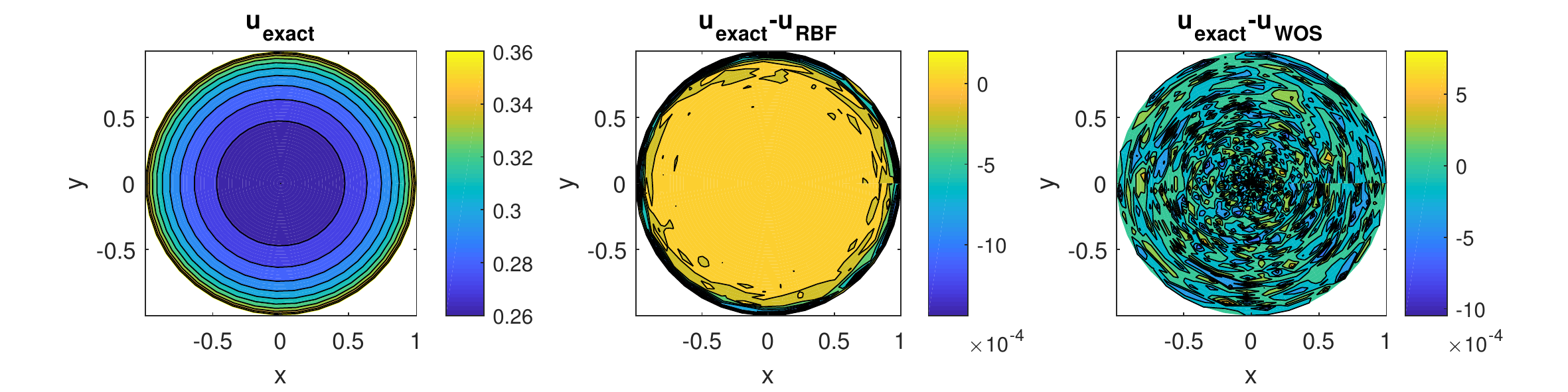}}\vfill
\subfloat[Convergence]{
\includegraphics[width=.45\textwidth]{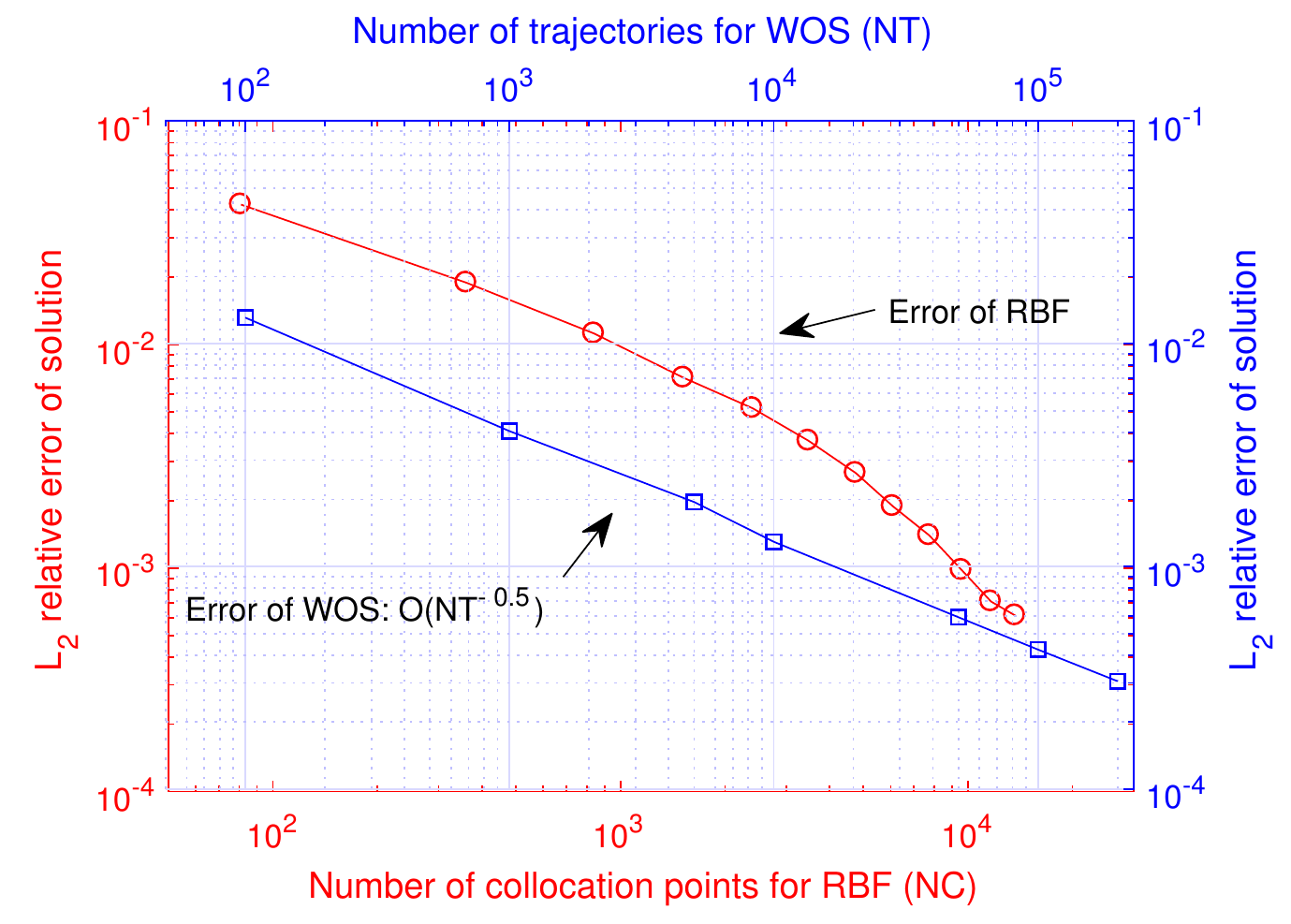}} \hfill
\subfloat[CPU time]{
\includegraphics[width=.45\textwidth]{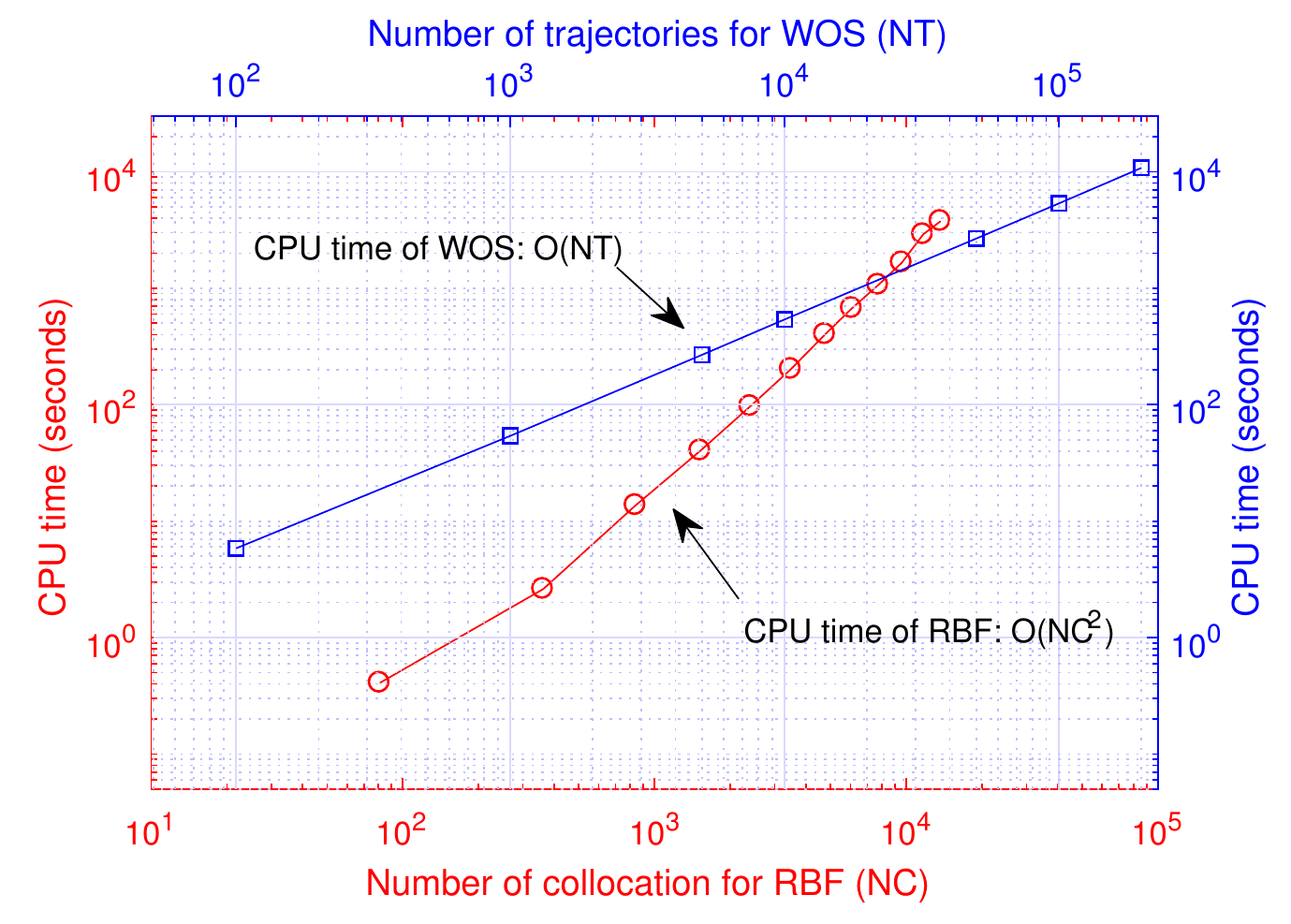}}
\caption{\label{cmp-RBF-WOS} \color{magenta}Comparison of the RBF and the WOS methods in terms of solution error and CPU time for the problem \eqref{frac-Poisson} on the unit disk with non-zero exterior condition $g(x)=\exp(-|x|^2)$ and zero source term $f(x)=0$. (a) Pointwise errors of RBF and WOS methods with respect to the reference solution (b) Comparison of $L_2$ relative errors of RBF and WOS methods. (c) Comparison of CPU times of RBF and WOS methods; both methods are parallelized over 24 CPU cores.}
\end{figure}

{\color{magenta}
\subsubsection{Future Directions}
The RBF collocation method introduced here for the Riesz fractional Laplacian takes advantage of the directional representation \eqref{def:directional_uniform} and extends readily to general nonisotropic directional fractional Laplacians \eqref{general_directional_laplacian}. We established two problems for further work: (1) Development of effective preconditioners for the RBF collocation matrices. This will improve convergence rates as the number of collocation points increases.  (2) Study of \emph{truncated} fractional directional Laplacians and error bounds for the difference between solutions of the associated boundary value problems and those of \eqref{frac-Poisson}. While in our example we drove the truncation error in \eqref{modified-GL} close to machine precision by choosing $K_2$ sufficiently large, theoretical justification for truncation in the modified GL formula should be studied.  To have full impact, these questions should be studied in the context of the most general directional fractional Laplacians \eqref{general_directional_laplacian}  connected to the class of multivariate stable processes, which goes beyond the scope of the current article. 
}

\FloatBarrier

\subsection{Horizon-based Nonlocal Definition}\label{sec:horizon}

Even though we consider fractional equations involving bounded domains, the region of dependence for the Riesz and directional fractional Poisson equations is still $\mbbR^d$. This is in contrast with the spectral fractional Laplacian, which admits only the usual locally-defined boundary conditions, and does not utilize information about the solution outside the bounded domain. However, in the Riesz or directional fractional Poisson problems, it can be noticed through computational experiments that the dependence of the solution on points far away from the domain decreases with distance.
Du et al. \cite{dusiamreview} proposed a type of nonlocal problem admitting a ``horizon parameter", which is used to describe an ``interaction domain", a proper subset of $\mbbR^d,$ in which all computational considerations take place. In other words, data from outside the interaction domain is ignored. Thus, the operator used in the formulation of these equations is distinct from operators we have presented above, although the nonlocal operator does approximate the Riesz fractional Laplacian as the interaction domain increases in size. Indeed, taking the horizon parameter to $+\infty$, we recover the Riesz definition. In this section, we describe this nonlocal formulation and demonstrate how the value of the horizon parameter affects the behavior of the solution to the following problem. A similar study of the nonlocal diffusion operator was carried out by D'Elia et al. in \cite{DElia2013}. Another noteworthy paper by Duo et al. \cite{duo_wang_zhang} compared numerical results of some one-dimensional fractional diffusion problems for different definitions of the fractional Laplacian and the horizon-based nonlocal operator, which they call the ``peridynamic" operator.

\subsubsection{Nonlocal Problem}
Let $\Omega:=(-1,1) \subset \mathbb{R}$, {\color{magenta}$L = |\Omega| = 2$}, and $B_\delta(x) = \{y\in \mathbb{R}:|y-x|<\delta \}$ denotes the interval centered at $x$ having length $2\delta$, where $\delta$ is the horizon parameter. Define the interaction domain as:
 $$\Omega_{\mathcal{I}} = \{y\in \mathbb{R}\setminus \Omega :|y-x|<\delta \text{ for }x\in\Omega\} = (-1-\delta,-1] \cup [1,1+\delta),$$
i.e., $\Omega_{\mathcal{I}}$ consists of those points outside of $\Omega$ that interact with points in $\Omega$. Consider the one-dimensional nonlocal problem  with a Dirichlet constraint,
\begin{equation}\label{nonlocalProblem}
\begin{aligned}
- \mathcal{L}_{\delta}u(x) &= f(x)   &\text{ in } \Omega, \\
u  &= 0   &\text{ in } \Omega_{\mathcal{I}},
\end{aligned}
\end{equation}
where
\begin{equation}\label{NL:FL}
    -\mathcal{L}_\delta u =  C_{\alpha} \, \text{p.v.}\, \int_{B_\delta(x)} \frac{u({ x}) - u({ y})} {|x-y|^{1+\alpha}} d{ y}, \quad 0<\alpha<2,
\end{equation}
with
$$C_\alpha = \frac{2^{\alpha}\alpha \Gamma((\alpha+1)/2)}{2\pi^{1/2} \Gamma(1-\alpha/2)}.$$

\subsubsection{Finite Volume Discretization}

In order to solve the nonlocal problem \eqref{nonlocalProblem}, we consider a finite volume discretization. We divide the domain $\bar{\Omega}= [-1,1]$ into $N$ sub-domains, $h = 2/N$ is the space step size, and we divide the interaction domain $\Omega_{\mathcal{I}}$ into $2K$ sub-domains where $K=\delta/h$. Then, we have a partition $x_k = -1 + kh,\,k=-K,\cdots,0,1,\cdots,N,N+1,\cdots,N+K$, and we denote the partition by $I_j = [x_{j-1},x_{j}], j = 1,\cdots,N$. The finite volume formulation of \eqref{nonlocalProblem} is written as:
\begin{equation}\label{FVM}
    \frac{C_\alpha}{h}\, \text{p.v.} \, \int_{I_j}\int_{\color{magenta}B_\delta(x)} \frac{u({x}) -u({y}) }{|{ y - { x}}|^{1+\alpha}} d{y} d{x} = \frac{1}{h} \int_{I_j} f  d{x}, \quad j = 1, \cdots,N.
\end{equation}
Let $\bar{v}_j: = \frac{1}{h} \int_{I_j} v(x)  d{x}$  be the mean value of $v(x)$ in the interval $I_j$. Note that
{\color{magenta}
\begin{align}
\begin{split}
\label{FVMComp}
\frac{C_\alpha}{h} \int_{I_j}\int_{B_\delta(x)} \frac{u(x) -u(y) }{|y-x|^{1+\alpha}} dydx
&= \frac{C_\alpha}{h} \int_{I_j}\int_{B_\delta(x)} \frac{u(x) -u(x+z) }{|z|^{1+\alpha}}dzdx\\
&= \frac{C_\alpha}{h} \int_{I_j} \int_{-\delta}^{\delta}  \frac{u({x}) -u({ x+z}) }{|z|^{1+\alpha}} dz dx.
\end{split}
\end{align}}
Using the right rectangle rule to discretize the inner integral of the above equation, the above equation is equal to
{\color{magenta}
\begin{equation*}
    \frac{C_\alpha}{h} \int_{I_j} \int_{0<|z|<Kh}  \frac{u({x}) - u({x+z}) }{|z|^{1+\alpha}} dz dx
    \approx  C_\alpha \sum_{0<k\le K}  \frac{-\bar{u}_{j+k} - \bar{u}_{j-k} + 2\bar{u}_{j}}{|kh|^{1+\alpha}}.
\end{equation*}}
Thus, \eqref{FVM} becomes 
\begin{equation}\label{FVM1}
   C_\alpha \sum_{0<k\le K}  \frac{-\bar{u}_{j+k} - \bar{u}_{j-k} + 2\bar{u}_{j}}{|kh|^{1+\alpha}} = \bar{ f }_j, \quad j = 1, \cdots,N.
\end{equation}

Writing the discretized equation in matrix form, we have
\begin{equation}\label{FVM-Matrix}
    SU = F,
\end{equation}
where
$U = [\bar{u}_{1},\bar{u}_{1},\cdots,\bar{u}_{N}]^T$ and $F = [\bar{f}_{1},\bar{f}_{1},\cdots,\bar{f}_{N}]^T$. $S$ is the stiffness matrix with its elements:
\begin{equation*}
    \begin{aligned}
      &S_{jj} = \sum_{k=1}^K \frac{2}{k^{1+\alpha}},\\
      &S_{j,j+|k|} = -\frac{1}{k^{1+\alpha}}, \quad   |k| = 1, 2, \cdots, \min(K,N-|k|), j = 1, \cdots,N. \\
    \end{aligned}
\end{equation*}
If $K<N$, then $S_{j,j+|k|} = 0$ for $|k| > K$, from which, we can see that $S$ is a symmetric, diagonally dominant Toeplitz matrix.

The diagonal dominance property ensures that the system \eqref{FVM-Matrix} admits a \emph{unique} solution. The third property of $S$ suggests that we can efficiently solve the linear system \eqref{FVM-Matrix} in $O(N\log N)$ operations and $O(N)$ storage by using a preconditioned iterative method, such as GMRES.

\subsubsection{Numerical Examples}

We first test the convergence of the finite volume scheme.
\begin{exam}
  Let $f(x) = 1$. Recall that $L = 2$ is the length of the domain $\Omega$. We test two different values of horizon parameter $\delta = 2L, 64L$ for different values of fractional order $\alpha = 0.5, 1.5$. Since there is no exact solution, we use the numerical solution with $N = 2^{15}$ as the reference solution. The convergence results of the $L^2$- and $L^\infty$-error are shown in Figure \ref{fig:convergence}.
\end{exam}
%Observe that all errors converge to zero which means that the finite volume scheme is effective.
{\color{blue}
We observe in Figure \ref{fig:convergence} that the convergence rates of the $L^2$-error for $\alpha=0.5$ and $\alpha = 1.5$ are about $O(h^{1.2})$ and $O(h^{1.5})$, respectively, while the convergence rates of the $L^\infty$-error for $\alpha=0.5$ and $\alpha = 1.5$ are about $O(h^{0.5})$ and $O(h^{0.7})$, respectively. We point out here that theoretical estimates for the convergence rate {\color{magenta} for finite-volume schemes} are not known at the time of this writing. {\color{magenta} However, we mention that $L^2$ convergence rates for finite element schemes for such problems have been established in \cite{borthagaray2018finite, burkovska2018regularity}. 
}

\begin{figure}[htp]
\begin{minipage}{.45\textwidth}%{0.46\linewidth}
\begin{center}
\includegraphics[scale=0.6,angle=0]{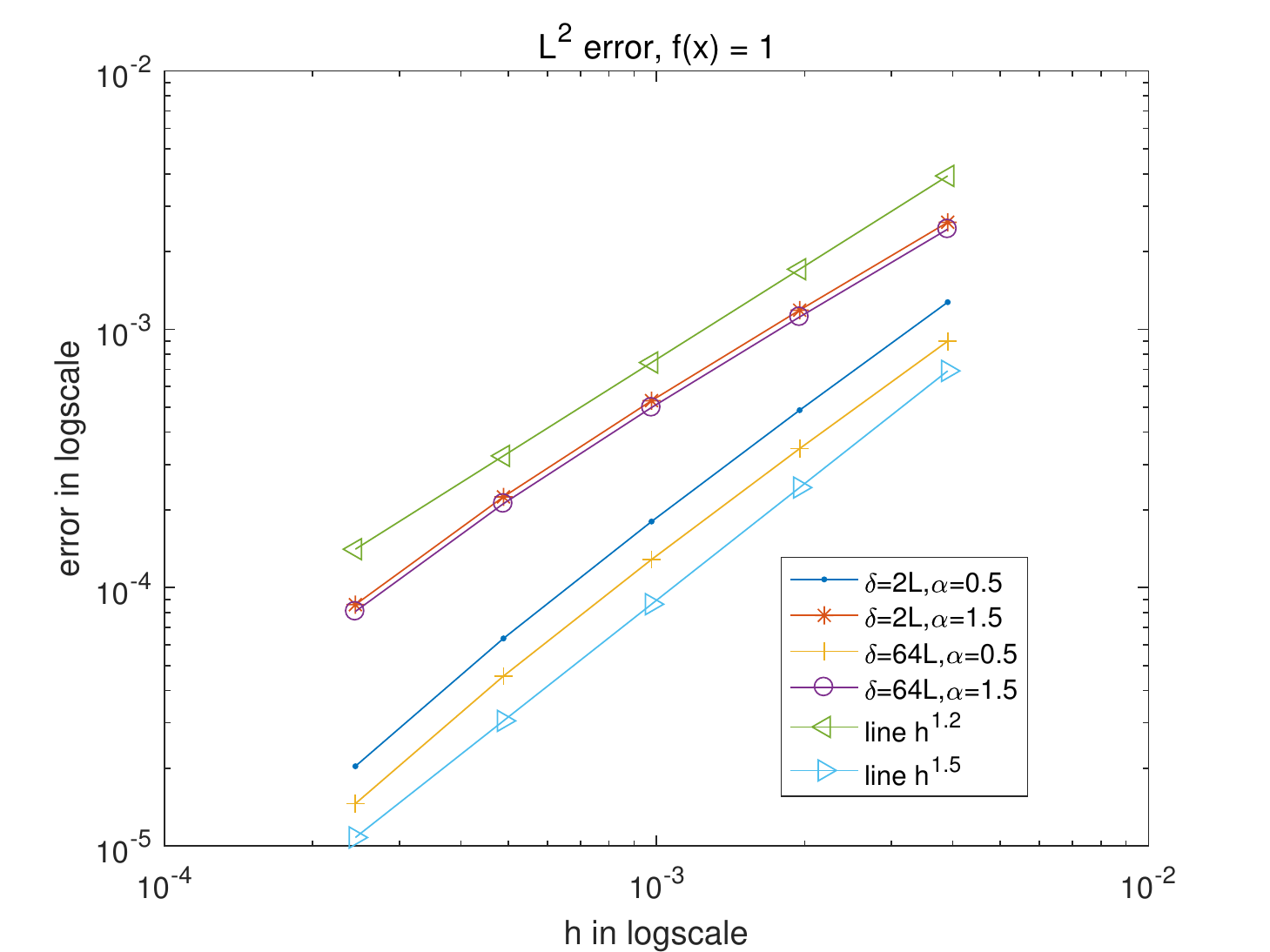}\end{center}
\end{minipage}
\begin{minipage}{.45\textwidth}%{0.46\linewidth}
\begin{center}
\includegraphics[scale=0.6,angle=0]{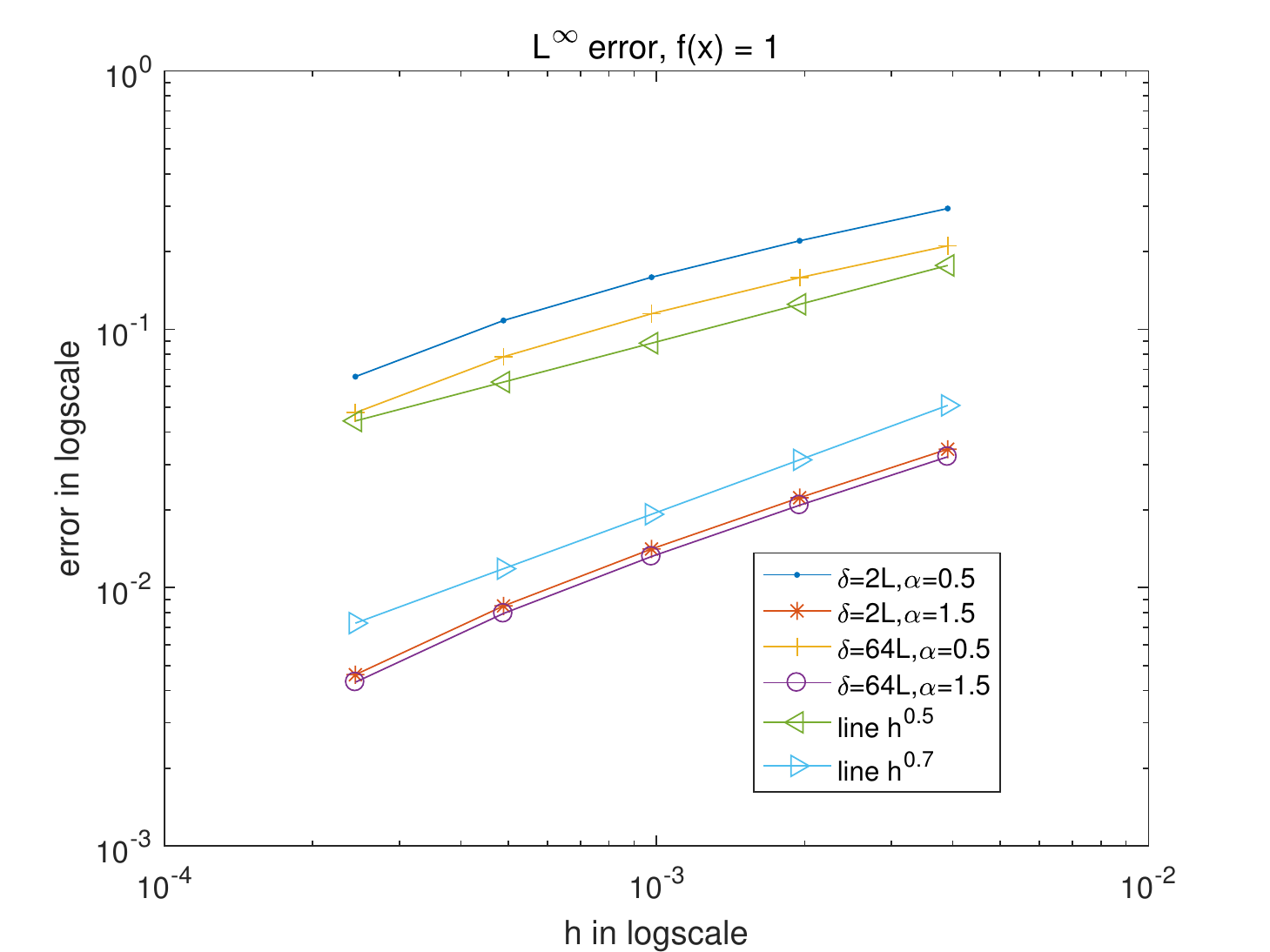}\end{center}
\end{minipage}
\caption{Convergence of the errors of the finite volume discretization, left: $L^2$-error, right: $L^\infty$-error.
}\label{fig:convergence}
\end{figure}

We now study how the value of the horizon parameter affects the behavior of the solution of the nonlocal problem. We also compare it with the solution of the nonlocal diffusion problem 
\begin{equation}\label{fracProblem1}
\begin{aligned}
C_{\alpha} \,\text{p.v.}\,  \int_{\mathbb{R}} \frac{u({ x}) - u({ y})} {|x-y|^{(1+\alpha)}} d{ y} &= f(x)   &\text{ in } \Omega, \\
u  &= 0   &\text{ in } \Omega^c: = \mathbb{R}\setminus \Omega.
\end{aligned}
\end{equation}
The above problem is equivalent to the Riesz fractional diffusion problem 
\begin{align}
\begin{split}
\label{fracProblem}
(-\Delta)^{\alpha/2} u(x) &= f(x)   \text{ in } \Omega, \\
{\color{magenta}u  }&= 0 {\color{magenta}\text{ in } \mathbb{R} \setminus \Omega.} 
\end{split}
\end{align}
The above fractional diffusion problem can be  solved by using the Petrov-Galerkin method proposed in \cite{MCS2016APPNM} (see also \cite{MK2017SINUM}).

\begin{exam}
  Let $N = 2^{14}$. We vary the value of horizon parameter $\delta$ for $f(x) = 1$ and $f(x) = \sin(\pi x)$, and consider different values of the fractional order $\alpha = 0.5, 1.5$. The results are shown in Figure \ref{fig:f1N14}.
\end{exam}

We observe that the numerical solution converges to the solution of the fractional diffusion problem \eqref{fracProblem1} as the horizon parameter becomes infinite. This phenomenon has also been observed by D'Elia and Gunzburger \cite{DElia2013}. 
{\color{blue} 
Moreover, we study the convergence rate in $L^\infty$ in the sense of the asymptotic behavior, in particular, we plot $\|u - u_\delta\|_{\infty}$ against the horizon parameter $\delta$ in log-log scale in Figure \ref{fig:f1a05N16gap}. 
We observe that the convergence rate of $\|u-u_{\delta}\|_{\infty}$ is 0.5 for $\alpha=0.5$ while it is 1.5 for $\alpha = 1.5$. The convergence rate can also be analyzed theoretically for $\alpha > 1$. We now show this as follows: Let $u$ and $u_{\delta}$ be the solutions of \eqref{fracProblem1} and \eqref{nonlocalProblem}, respectively, it has been shown that (see \cite[Theorem 3.1]{DElia2013}) 
\begin{equation*}
    \|u - u_\delta\|_{H^{\alpha/2}}\le C_1\delta^{-\alpha}, \quad  \|u - u_\delta\|_{L^{2}}\le C_2\delta^{-\alpha}. 
\end{equation*}
Thus, by using the interpolation theory and noting that $\|v\|_{\infty} \le c \|v\|_{H^s}$ for $s>1/2$, we have 
\begin{equation}\label{est:delta:inf}
    \|u - u_\delta\|_{\infty}\le C\delta^{-\alpha}, \quad \text{ for } \alpha > 1. 
\end{equation}
}
{\color{blue}
We point out in the following that, for solving the equation \eqref{fracProblem1} using the FVM, the corresponding integral \eqref{FVMComp} can be partitioned as}
\begin{multline}\label{FVMComp1}
\frac{C_\alpha}{h} \int_{I_j}\int_{\mathbb{R}} \frac{u(x) -u(y) }{|y-x|^{1+\alpha}} dydx
=
{\color{magenta}
\frac{C_\alpha}{h} \int_{I_j}\int_{\mathbb{R}} \frac{u(x) - u(x+z) }{|z|^{1+\alpha}} dzdx}\\
    = \frac{C_\alpha}{h} \int_{I_j} \int_{-L}^{L}  \frac{u({x}) -u({ x+z}) }{|z|^{1+\alpha}} dz  
      + \frac{C_\alpha}{h} \int_{I_j} \left\{\int_{-\infty}^{-L}  \frac{u({x}) -u({ x+z}) }{|z|^{1+\alpha}} dz +  \int_{L}^{\infty}  \frac{u({x}) -u({ x+z}) }{|z|^{1+\alpha}} dz  \right \} dx.
\end{multline}
The first term of the above integral is computed as before by using a rectangle rule discretization, while the last two terms are computed exactly. Thus, using the fact that {\color{magenta}$u= 0$ outside $\Omega$ and $x+z$ lies outside $\Omega$ for $x \in I_j \subset \bar{\Omega}$ and $|z| \ge L = |\Omega|$}, the above integral is approximated by 
$$
C_\alpha \sum_{0<k\le L/h}  \frac{-\bar{u}_{j+k} - \bar{u}_{j-k} + 2\bar{u}_{j}}{(kh)^{1+\alpha}}
+ \frac{2C_\alpha}{\alpha} L^{-\alpha} \bar{u}_{j}.
$$

\begin{figure}[htp]
\begin{minipage}{0.45\textwidth}%{0.46\linewidth}
\begin{center}
\includegraphics[scale=0.6,angle=0]{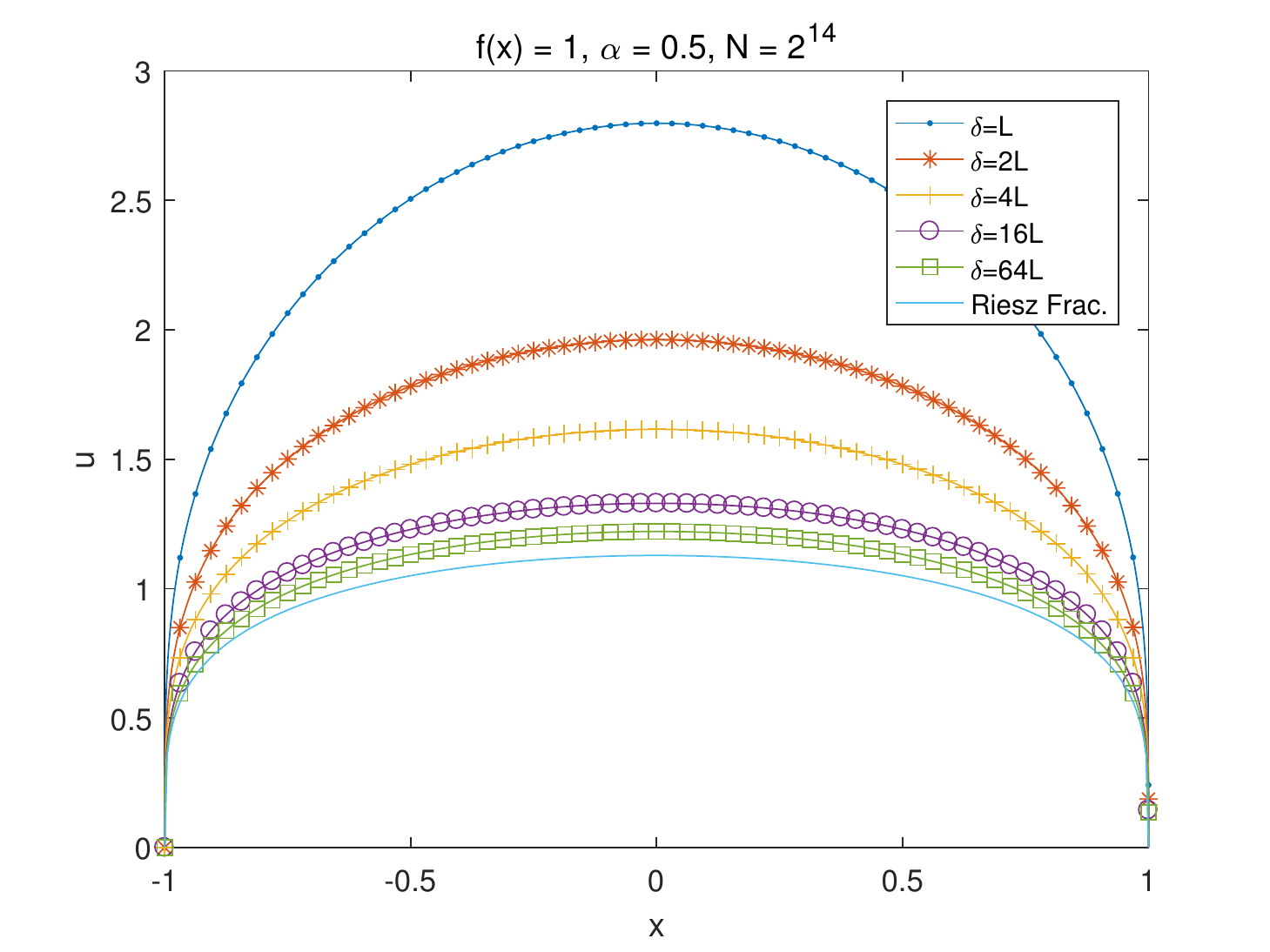}\end{center}
\end{minipage}
\begin{minipage}{0.45\textwidth}%{0.46\linewidth}
\begin{center}
\includegraphics[scale=0.6,angle=0]{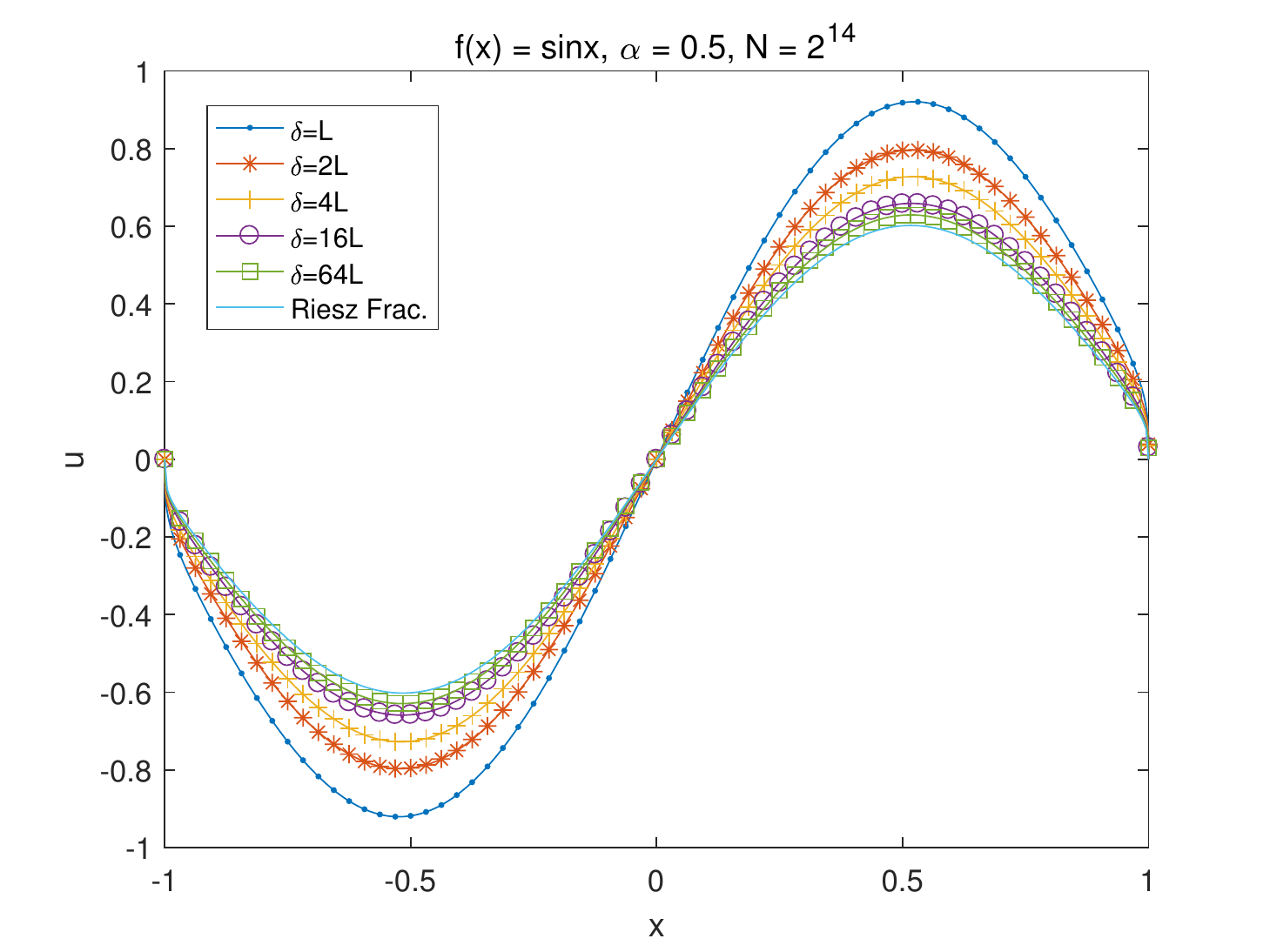}\end{center}
\end{minipage}
\caption{{Numerical solutions with different values of horizon parameter}, $N = 2^{14}$, $\alpha = 0.5$, (\emph{left}) $f = 1$, (\emph{right}) $f = \sin(\pi x)$.
}\label{fig:f1N14}
\end{figure}

\begin{figure}[http]
\begin{center}
\includegraphics[scale=0.65,angle=0]{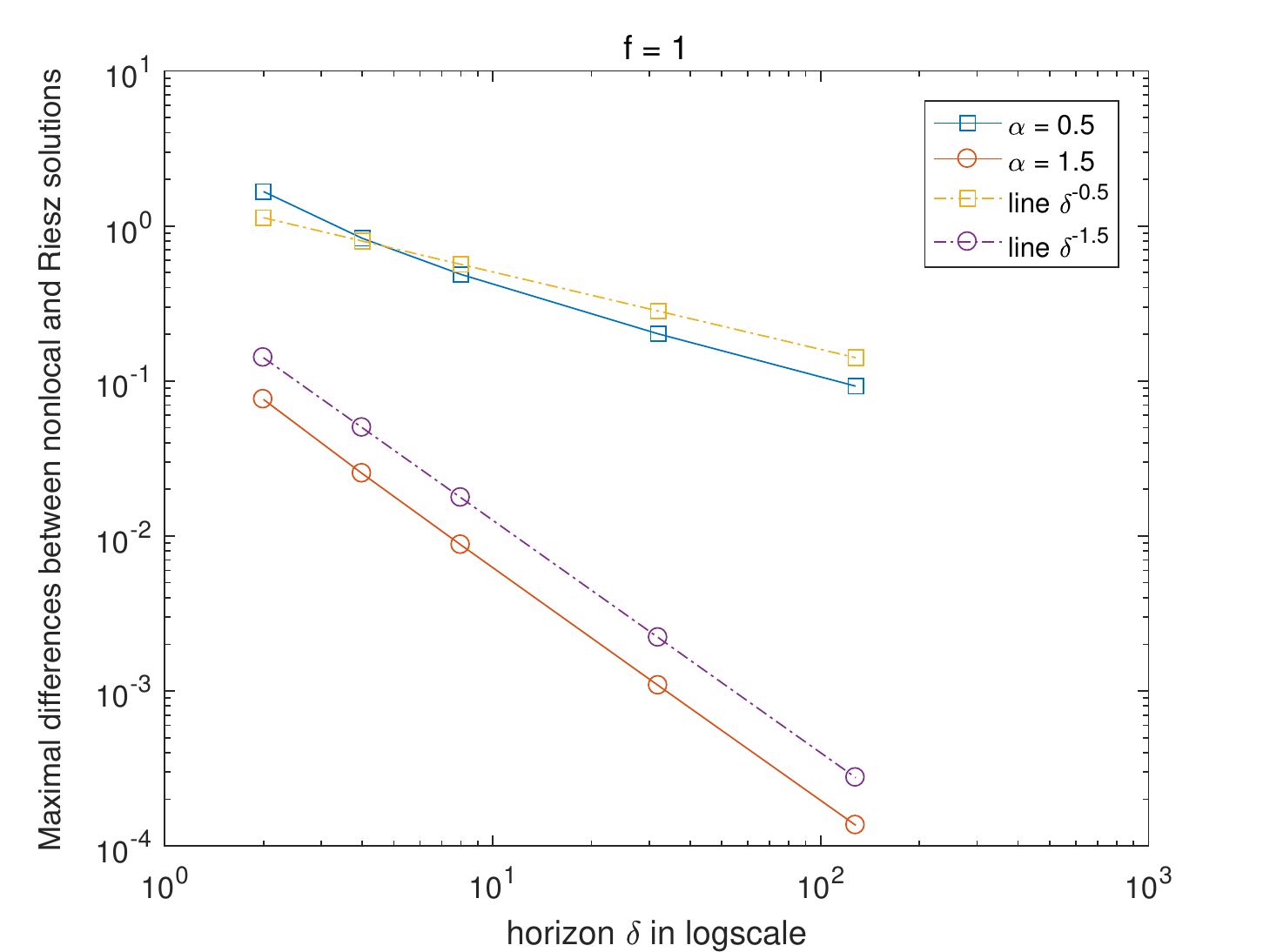}
\end{center}
\caption{
{\color{blue}
Convergence of  $\|u-u_{\delta}\|_{\infty}$ with respect to the horizon $\delta$ with different values of fractional order $\alpha$, $L = 2$, and $f(x) = 1$.
}
}\label{fig:f1a05N16gap}
\end{figure}

\section{Numerical Solutions: Comparisons} \label{sec:num_comp}

\hrule
\vspace{1em}
\noindent \textbf{Section Overview}\\[0.2em]
\indent {\color{magenta}In the preceding sections, we have thoroughly discussed the theoretical properties of the Riesz and spectral fractional Laplacians, and have established significant differences in the fractional Poisson problem associated to each operator. In this section, we leverage some of the numerical methods described in Section \ref{sec:num_meth} to demonstrate the practical differences between these different operators in benchmark problems. As expected from the preliminary numerical studies in Section \ref{intro} and the theoretical discussion in Section \ref{background}, these differences in the solutions of the respective Poisson problems are highly dependent on the fractional order $\alpha$ and the properties of the forcing term $f$.} We focus on numerical methods that are generalizable to higher dimensions, and we use these methods to compute benchmark solutions in two dimensions on the square, L-shaped, and unit disk domains. We solve fractional Poisson problems on each of these domains using the spectral, Riesz, and directional definitions with $\alpha = 0.5$ and $\alpha = 1.5$ with zero Dirichlet boundary conditions. We make observations about the differences between these numerical solutions, including regularity issues near the boundary, equivalence of the directional and Riesz definitions, and boundary oscillations in the case $\alpha = 0.5$. We also provide plots of the spectral and Riesz solutions compared along one-dimensional slices of the two-dimensional domains. Finally, we summarize the computational advantages and disadvantages of each method so that practitioners are better able to choose which method is best suited for a given equation.
\vspace{1em}
\hrule

\vspace{1.5em}

We consider four fractional Poisson problems in three $2$-dimensional domains of different shapes. In particular, we compare the solutions of Equation \eqref{fracPoisson} in a bounded domain $\Omega \subset \mbbR^2$ with both the Riesz and spectral fractional Laplacians and zero Dirichlet boundary conditions,
where the choices of $\Omega$, $f$, and $\alpha$ are described in Secs. \ref{2d:square} - \ref{2d:Lshape} and outlined in Table \ref{2dbenchmarks}.

To solve the Riesz fractional Poisson equation, we use the AFEM (see Sec. \ref{sec:afem}), the WOS method (see Sec. \ref{sec:wos}), and the RBF collocation method (see Sec. \ref{RBFM}).  
The shape parameters of the RBF collocation method for each example are reported in the figure captions for the corresponding grid figures in Appendix \ref{grids}. Sixteen-point Gauss-Legendre quadrature is used to approximate the integral of the directional derivative in Eq. \eqref{GL-rule}. We use multiple methods to verify the accuracy of the solutions and to ensure that the observed differences in the numerical solutions reflect differences in the fractional operators instead of inaccuracies in the numerical results. The meshes or collocation points used in each example below are displayed in Appendix \ref{grids}.

\begin{table}
\centering
\begin{tabular}{|c | c | c |}
	\multicolumn{3}{c}{\textbf{Benchmark Problems with Zero Dirichlet BCs}} \\
	\hline
	Sec. \ref{2d:square}, Square: $[-1,1]^2$ & Order $\alpha$ & Source function $f$ \\
	\hline
	Case 1 \& 2 & $0.5$ \& $1.5$ & $f = 1$ \\
	Case 3 \& 4 & $0.5$ \& $1.5$ & $f = \sin(\pi x)\sin(\pi y)$ \\
	\hline
	\hline
	Sec. \ref{2d:disk}, Disk: $\{(x,y) : x^2 + y^2 \leq 1\}$ & Order $\alpha$ & Source function $f$ \\
	\hline
	Case 1 \& 2 & $0.5$ \& $1.5$ & $f = 1$ \\
	Case 3 \& 4 & $0.5$ \& $1.5$ & $f = \sin(\pi r^2)$ \\
	\hline
	\hline
	Sec. \ref{2d:Lshape}, L-shape: $[-1,1]^2 \setminus [0,1)^2$ & Order $\alpha$ & Source function $f$ \\
	\hline
	Case 1 \& 2 & $0.5$ \& $1.5$ & $f = 1$ \\
	Case 3 \& 4 & $0.5$ \& $1.5$ & $f = \sin(\pi x)\sin(\pi y)$ \\
	\hline
\end{tabular}
\caption{\label{2dbenchmarks} Guide to benchmark problems formulated with zero Dirichlet boundary conditions. The test cases were used to make comparisons on the square, disk, and L-shaped domains using the Riesz and spectral definitions as discussed in Secs. \ref{2d:square} - \ref{2d:Lshape}.}
\end{table}

\subsection{Square Domain}\label{2d:square}

In this section, we solve the benchmark problems on the domain $[-1,1]^2$. In Figure \ref{square12}, we show cases 1 and 2, where $f = 1$ and $\alpha = 0.5$ and $1.5$, respectively. In each case, the spectral solution $u_{\text{spectral}}$ {\color{blue} lies below} the Riesz solution $u_{\text{Riesz}}$, {\color{blue} just as for}  the one-dimensional comparisons in Figure \ref{1dcomparisons}. {\color{blue} Since $f \ge 0$ in the domain, this is a special case of the theoretical results of \cite{MUSINA20161667}}. 
{\color{blue}
We remark that the apparent failure in the plotted solution $u_\text{Riesz}$ at enforcing the zero boundary condition for $\alpha = 0.5$ is an artifact of the AFEM method  \cite{AinsworthGlusa2017_TowardsEfficientFiniteElement}; the true solution does in fact possess a zero trace. This is discussed in detail in Remark \ref{AFEM_regularity}. The solution $u_\text{spectral}$ for $\alpha = 0.5$ also possesses a trace (as discussed in Sec. \ref{sec:regularity}), and in contrast, the numerical method used to compute it (SEM) is able to enforce the zero boundary condition.
}

In Figure \ref{square34}, we consider Cases 3 and 4, where $f = \sin(\pi x)\sin(\pi y)$, and $\alpha = 0.5$ and $1.5$, respectively. In this case, the issue of satisfying the boundary condition in the Riesz solution when $\alpha = 0.5$ is not as pronounced, but the oscillations are visible in the plot of $u_{\text{Riesz}}-u_{\text{spectral}}$.

Figure \ref{squareslicesf1y0} displays the solutions plotted along the lines $y = 0$ for $f = 1$ and $f = \sin(\pi x)\sin(\pi y)$, respectively. These figures highlight the difference between the boundary layers for the spectral and Riesz solutions, as we discussed in relation to Figure \ref{1dcomparisons}. Indeed, the one- and two-dimensional profiles are qualitatively similar for all four cases, and the boundary layer in the Riesz solution for Case 4 (see Figure \ref{squareslicesf1y0}) due to the singularity in the Riesz definition \eqref{E:riesz_potential} is particularly noteworthy, as the spectral solution always will have smooth behavior near the boundary given smooth $f$.

We also computed the solution to the four benchmark problems using the directional definition \eqref{def:directional_uniform}, and this solution is plotted in the top left panel of Figure \ref{cmp-squ-const}. This figure also shows the Riesz solution computed in {\color{blue}three} ways: (i) using the AFEM method \cite{AinsworthGlusa2017_TowardsEfficientFiniteElement}, (ii) using the walk-on-spheres method of Section \ref{sec:wos}, {\color{blue} and (iii) using the RBF collocation method of Section \ref{RBFM}}. The similarity of the solutions (on the square domains and on the domains discussed in the following sections) verifies our numerical results, and demonstrates that the directional solution is equivalent to the Riesz solution, up to some numerical error. The spectral solution in the top right panel is significantly different from the Riesz and directional solutions, as indicated by their maximal values reported in the titles of the plots.
 
\newpage

\begin{figure}[ht!]
 \centering
\subfloat[Solutions $u$ associated with $f=1$ and $\alpha = 0.5$ in the square domain using the spectral definition (using SEM) (\emph{left}) and the Riesz definition (using AFEM) (\emph{center}), and the difference between $u_{\text{Riesz}}$ and $u_{\text{spectral}}$ for this case (\emph{right}).]{
 \begin{minipage}[]{\textwidth}\centering
 \includegraphics[width=0.25\textwidth]{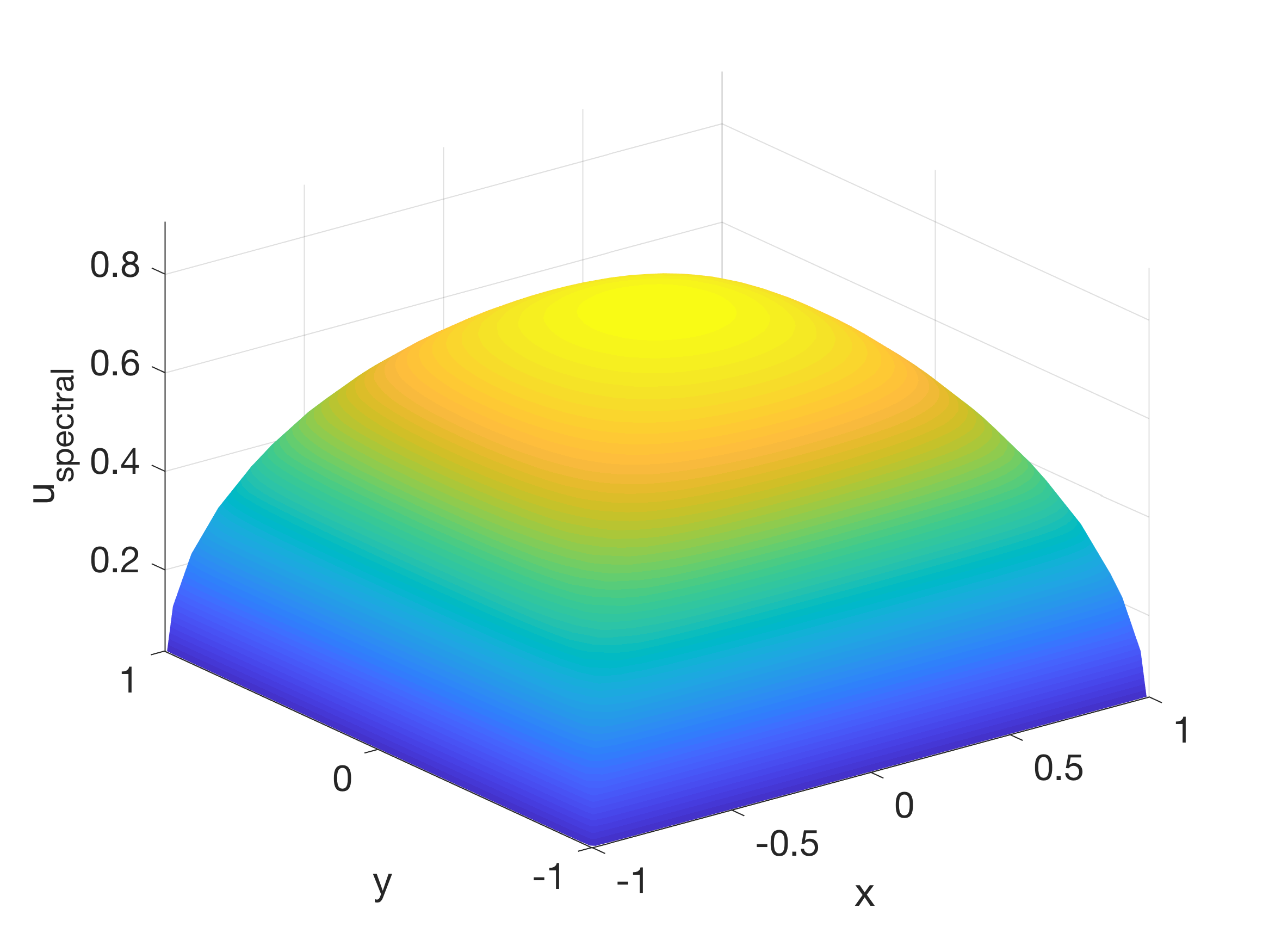}
  \includegraphics[width=0.25\textwidth]{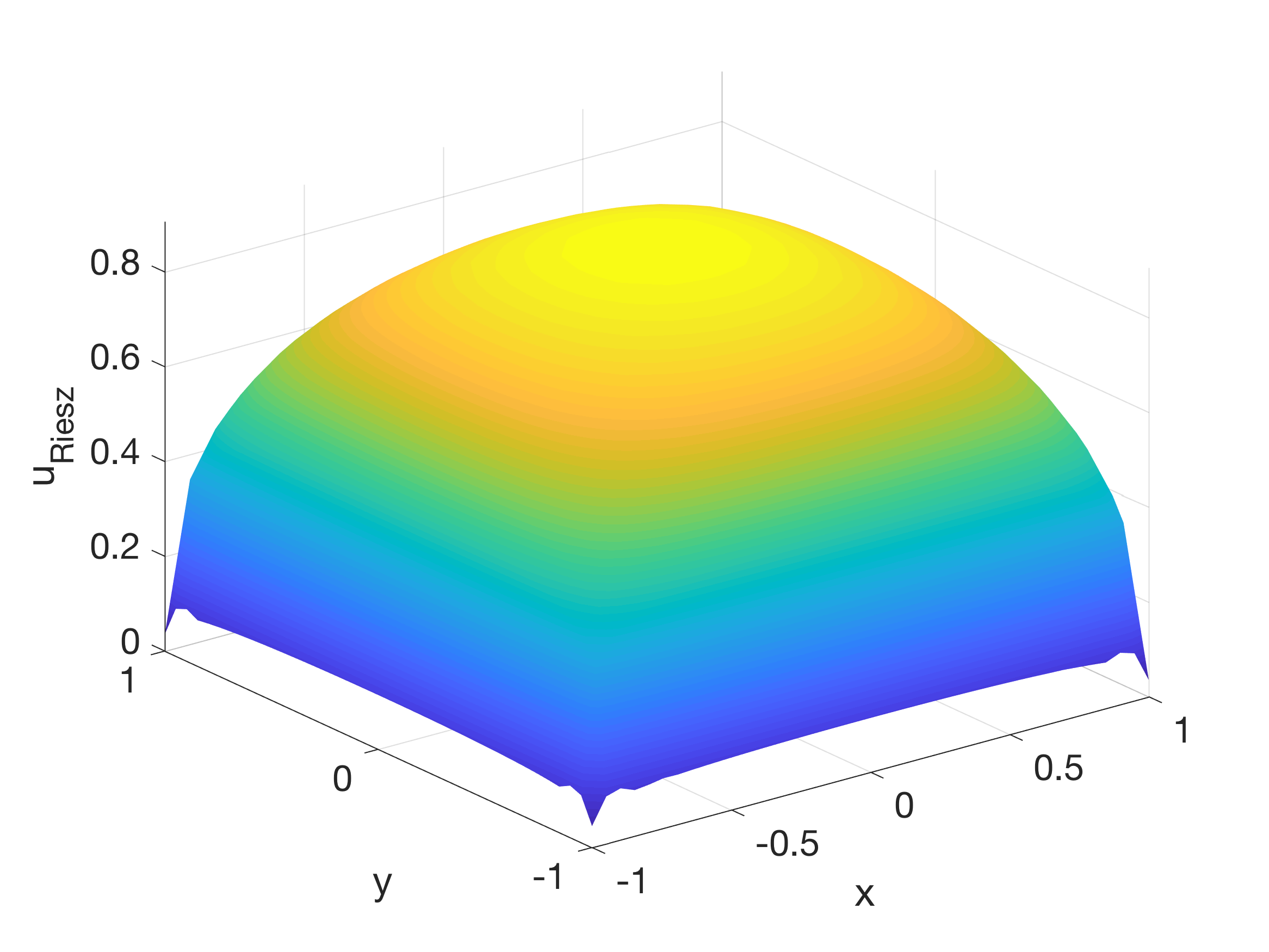}
  \includegraphics[width=0.25\textwidth]{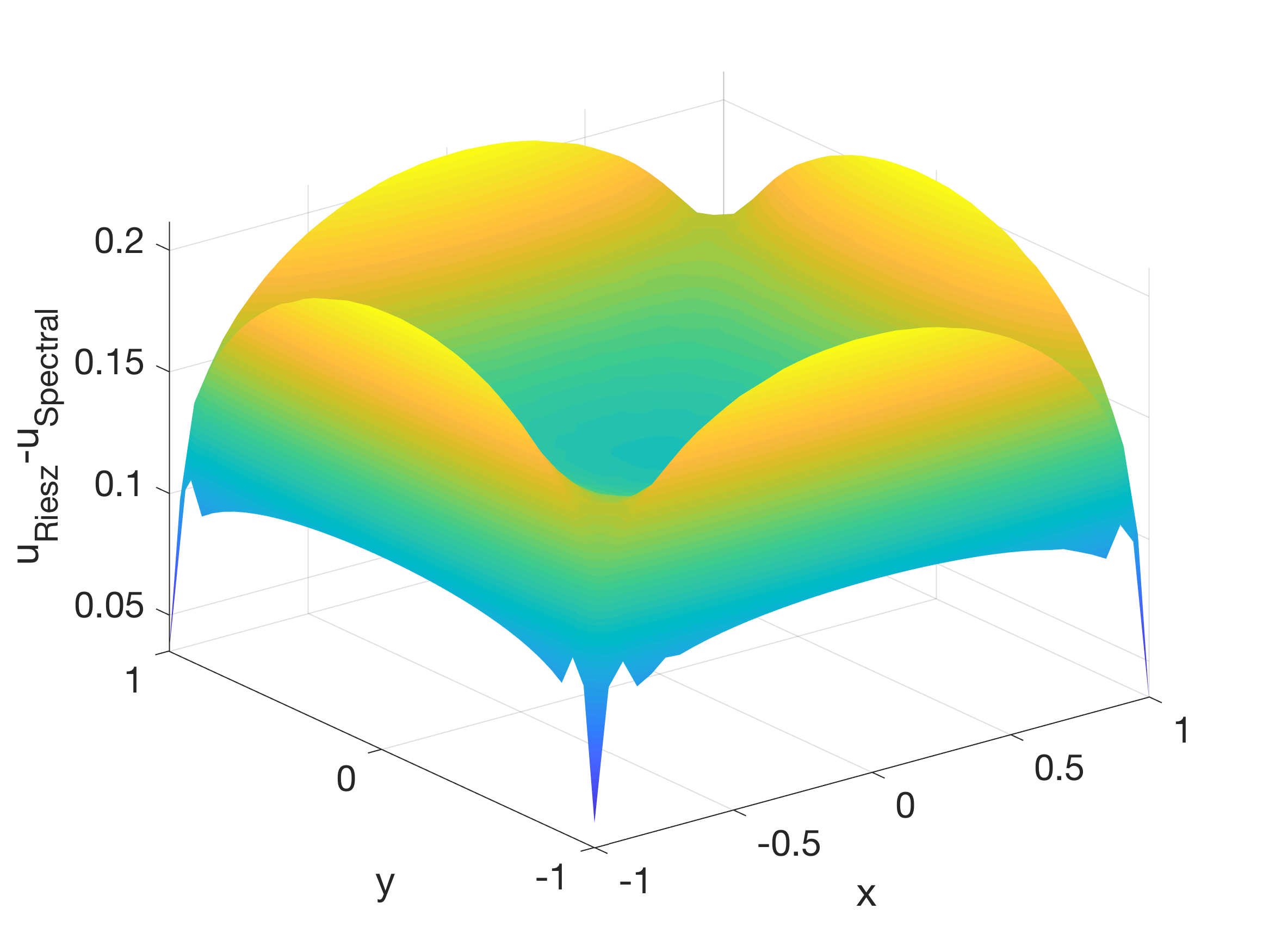}
\end{minipage}
 }\\
\subfloat[Solutions $u$ associated with $f=1$ and $\alpha = 1.5$ in the square domain using the spectral definition (using SEM) (\emph{left}) and the Riesz definition (using AFEM) (\emph{center}), and the difference between $u_{\text{Riesz}}$ and $u_{\text{spectral}}$ for this case (\emph{right}).]{
 \begin{minipage}[]{\textwidth}\centering
 \includegraphics[width=0.25\textwidth]{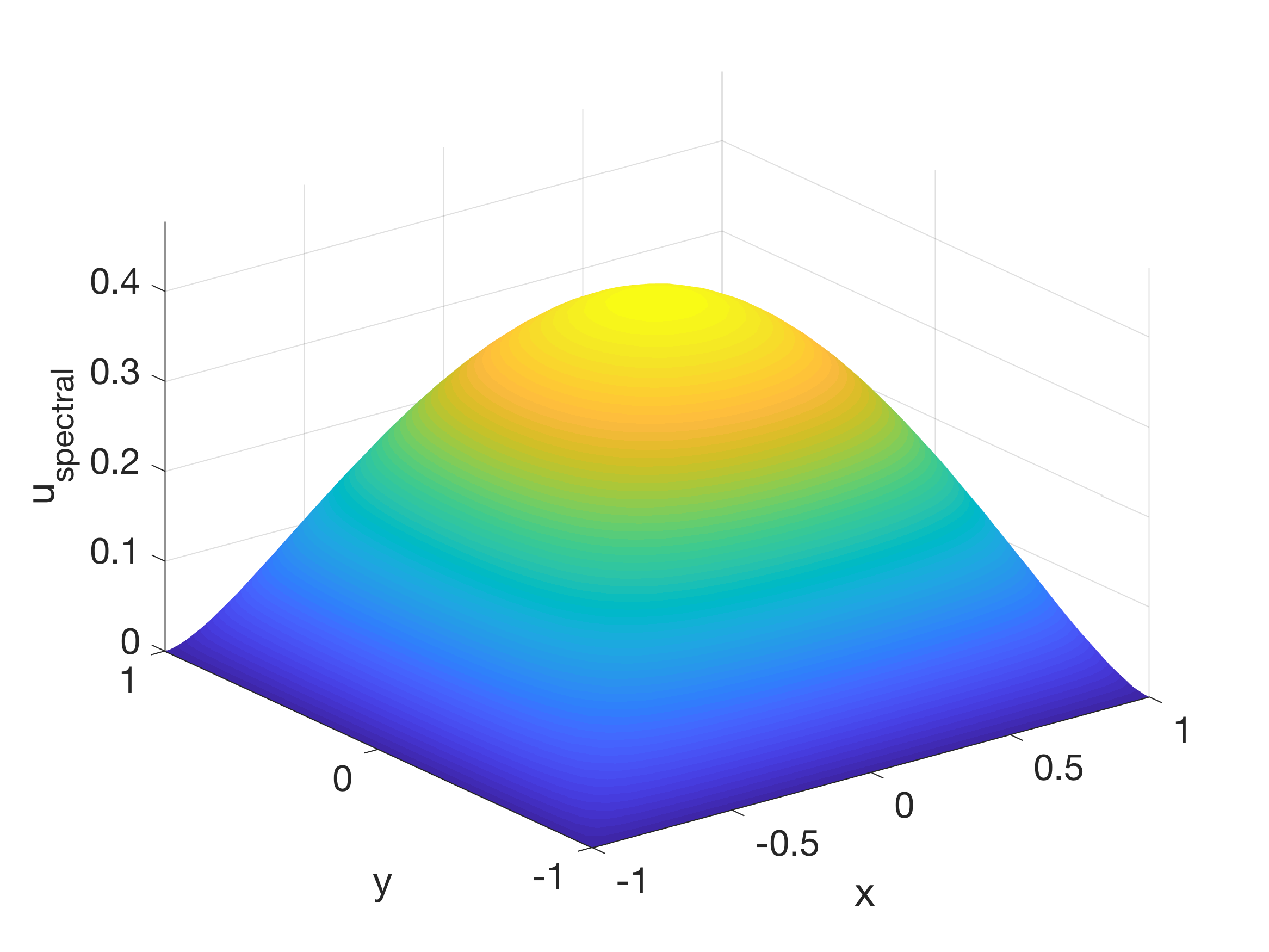}
  \includegraphics[width=0.25\textwidth]{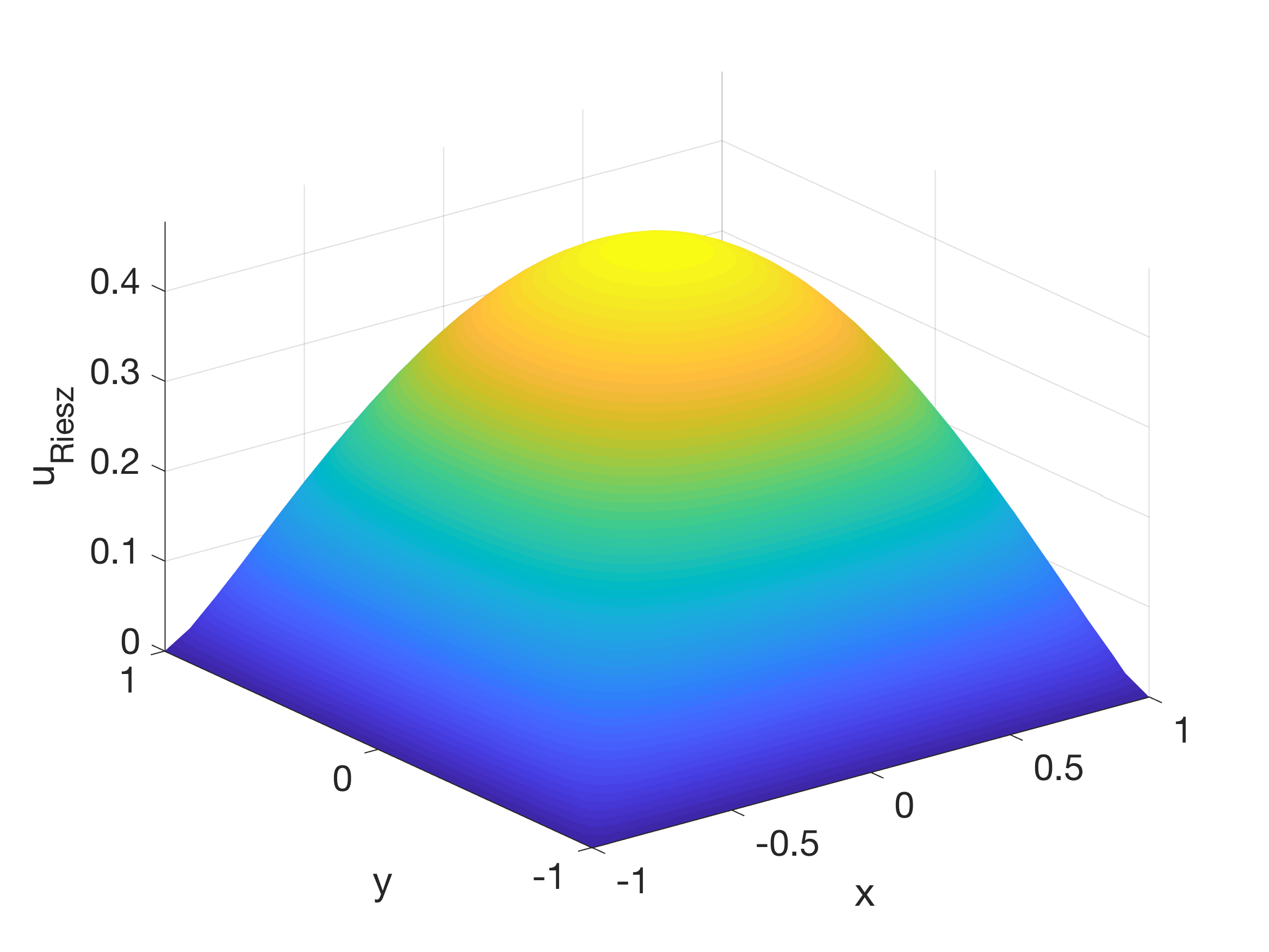}
    \includegraphics[width=0.25\textwidth]{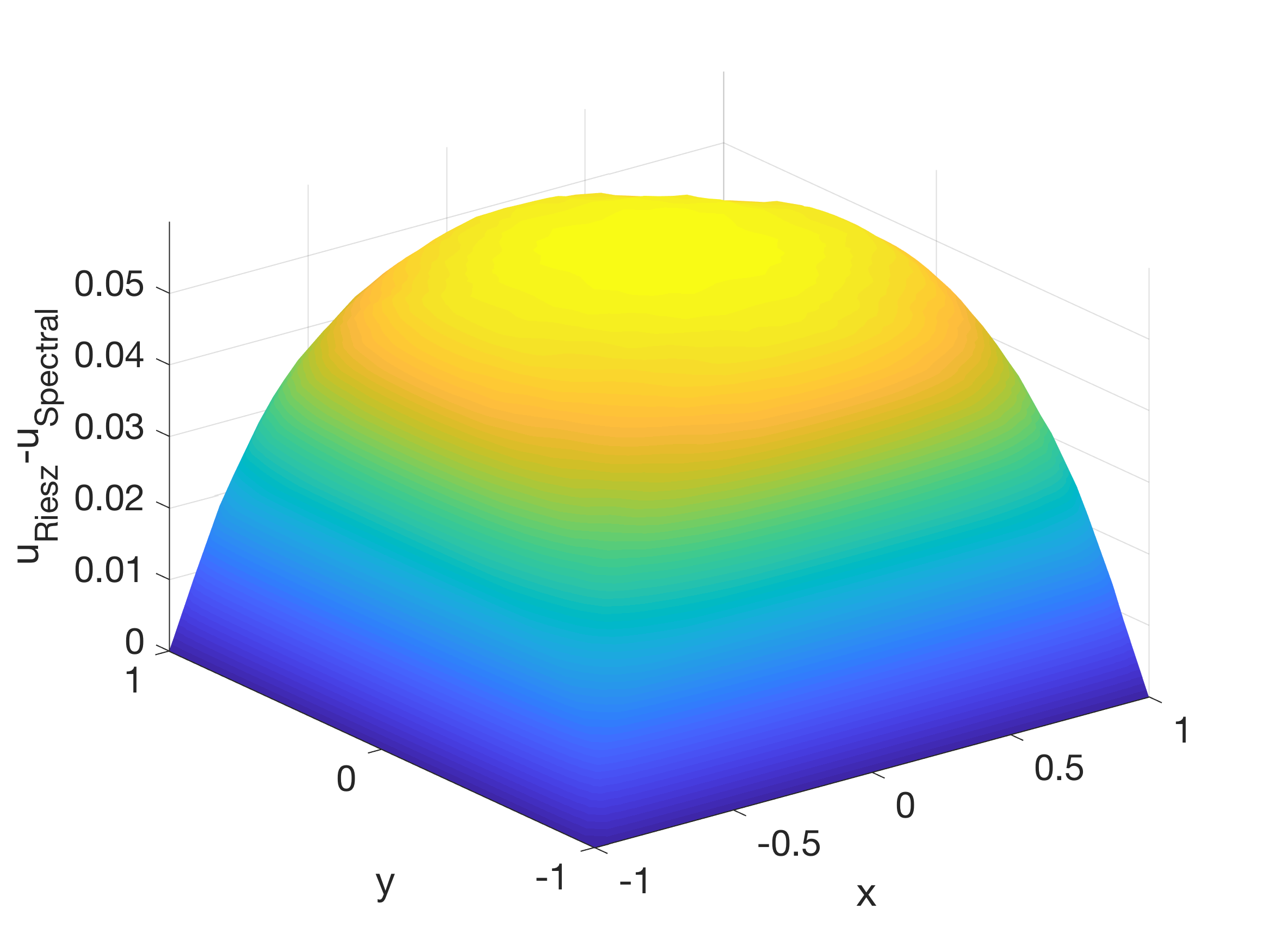}
\end{minipage}
 }
 \caption{ \label{square12} {Solutions and differences} between $u_{\text{Riesz}}$ and $u_{\text{spectral}}$ on the square domain for $\alpha = 0.5$ {\color{forest}and $1.5$.}}
 \end{figure}

\begin{figure}[ht!]
 \centering
 \subfloat[Solutions $u$ associated with $f=\sin(\pi x)\sin(\pi y)$ and $\alpha = 0.5$ in the square domain using the spectral definition (using SEM)  (\emph{left}) and the Riesz definition (using AFEM) (\emph{center}), and the difference between $u_{\text{Riesz}}$ and $u_{\text{spectral}}$ for this case (\emph{right}).]{
 \begin{minipage}[]{\textwidth}\centering
 \includegraphics[width=0.3\textwidth]{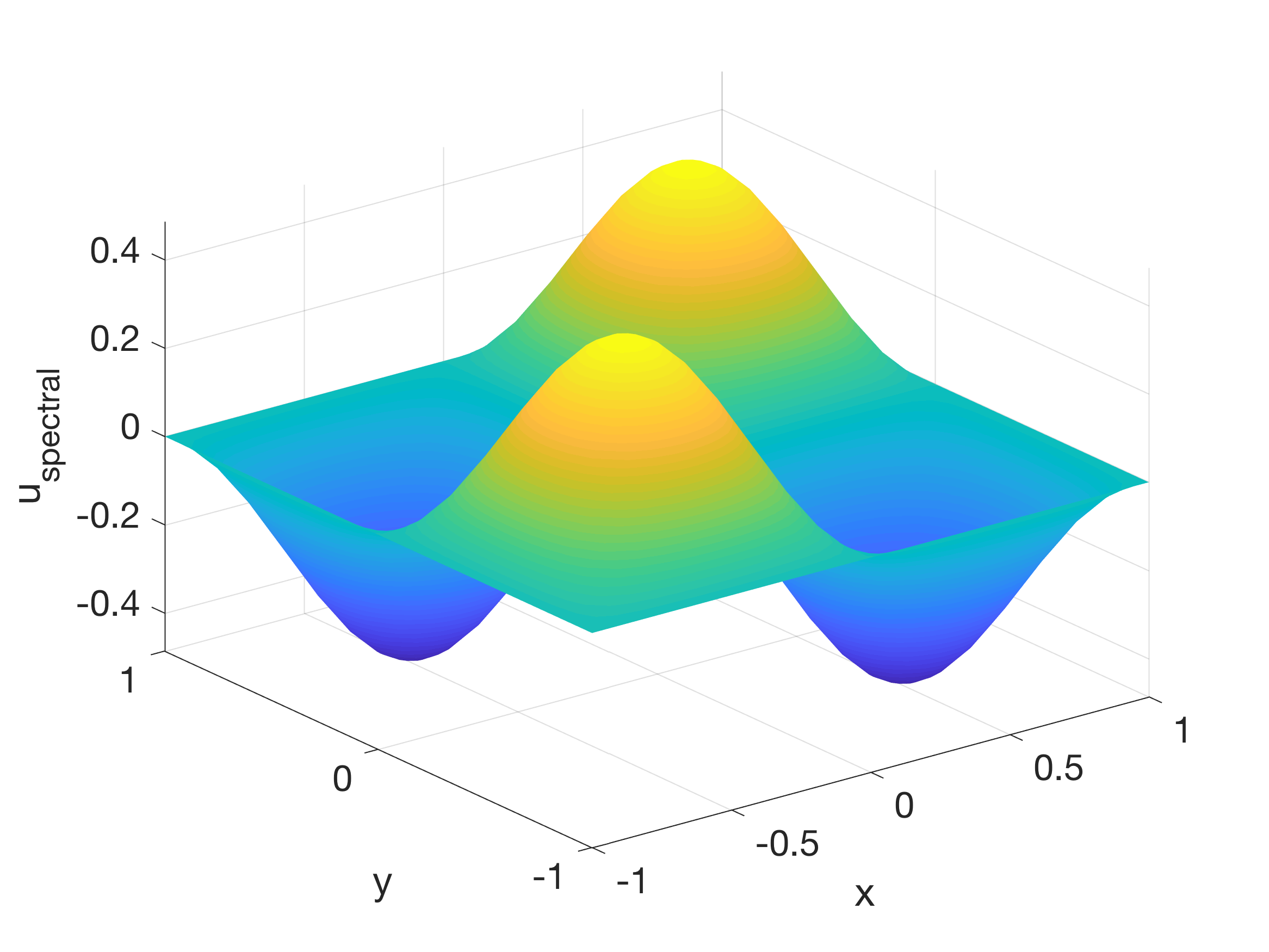}
  \includegraphics[width=0.3\textwidth]{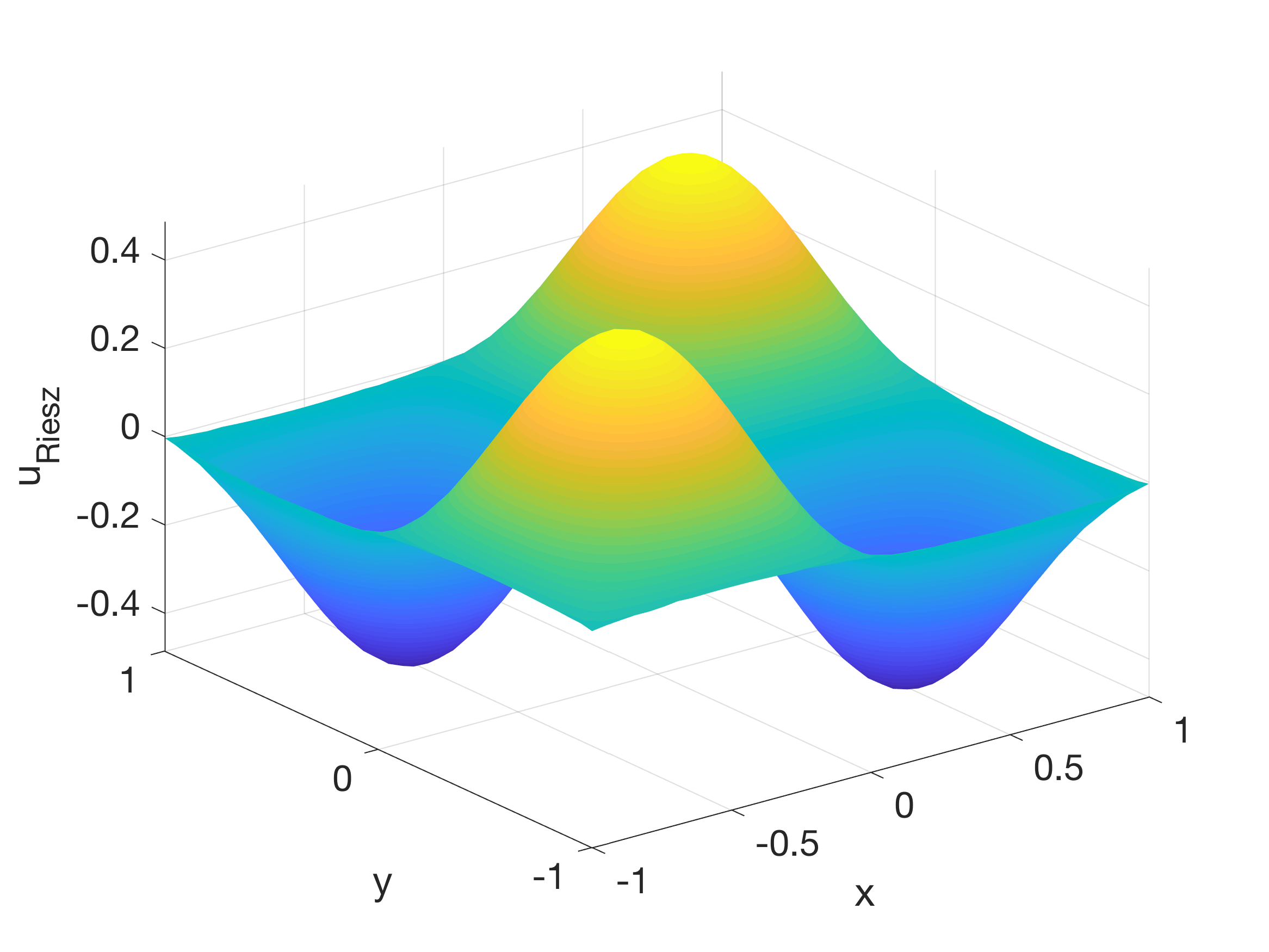}
  \includegraphics[width=0.3\textwidth]{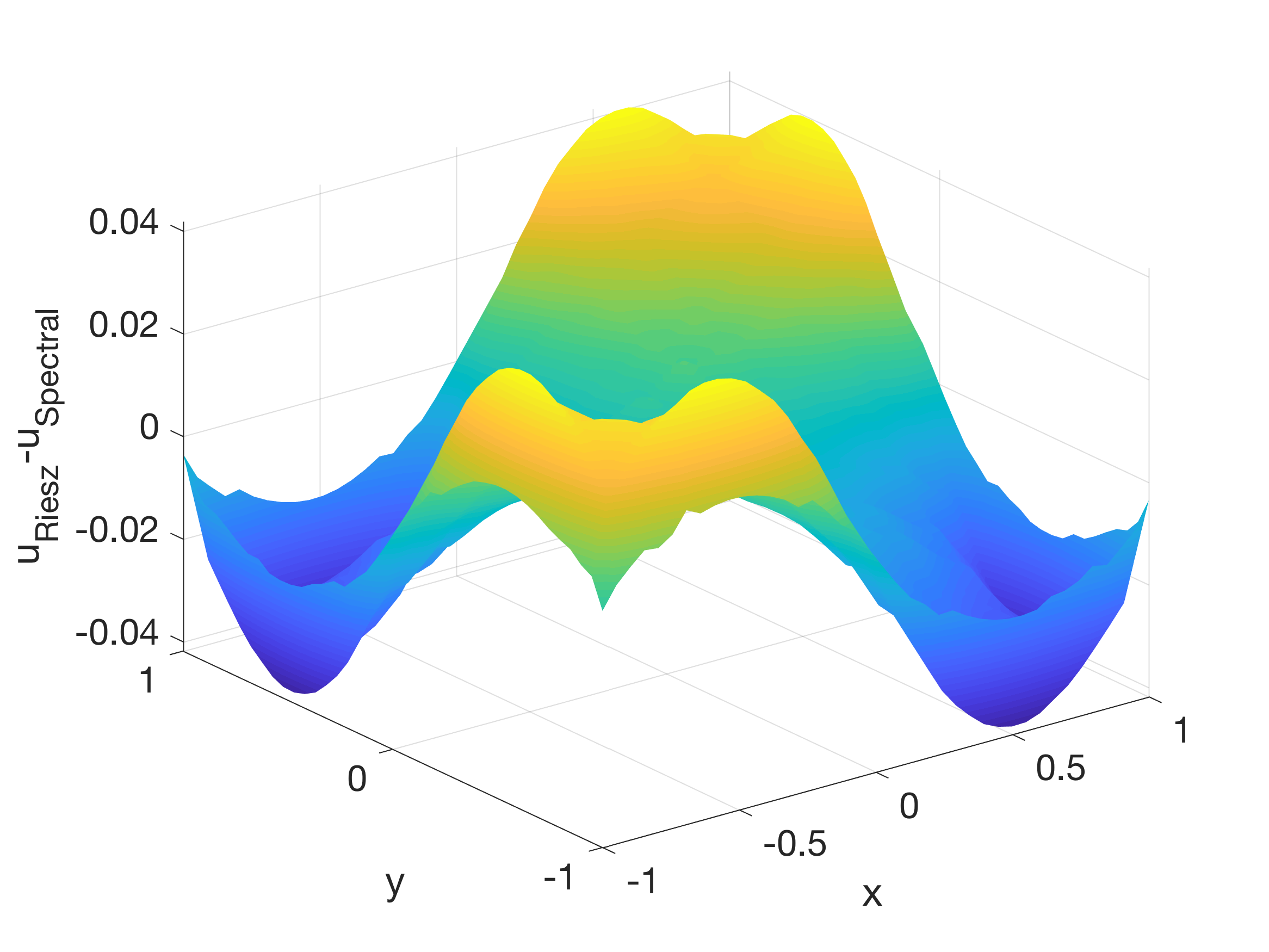}
\end{minipage}
 }\\
\subfloat[Solutions $u$ associated with $f=\sin(\pi x)\sin(\pi y)$ and $\alpha = 1.5$ in the square domain using the spectral definition (using SEM) (\emph{left}) and the Riesz definition (using AFEM) (\emph{center}), and the difference between $u_{\text{Riesz}}$ and $u_{\text{spectral}}$ for this case (\emph{right}).]{
 \begin{minipage}[]{\textwidth}\centering
 \includegraphics[width=0.3\textwidth]{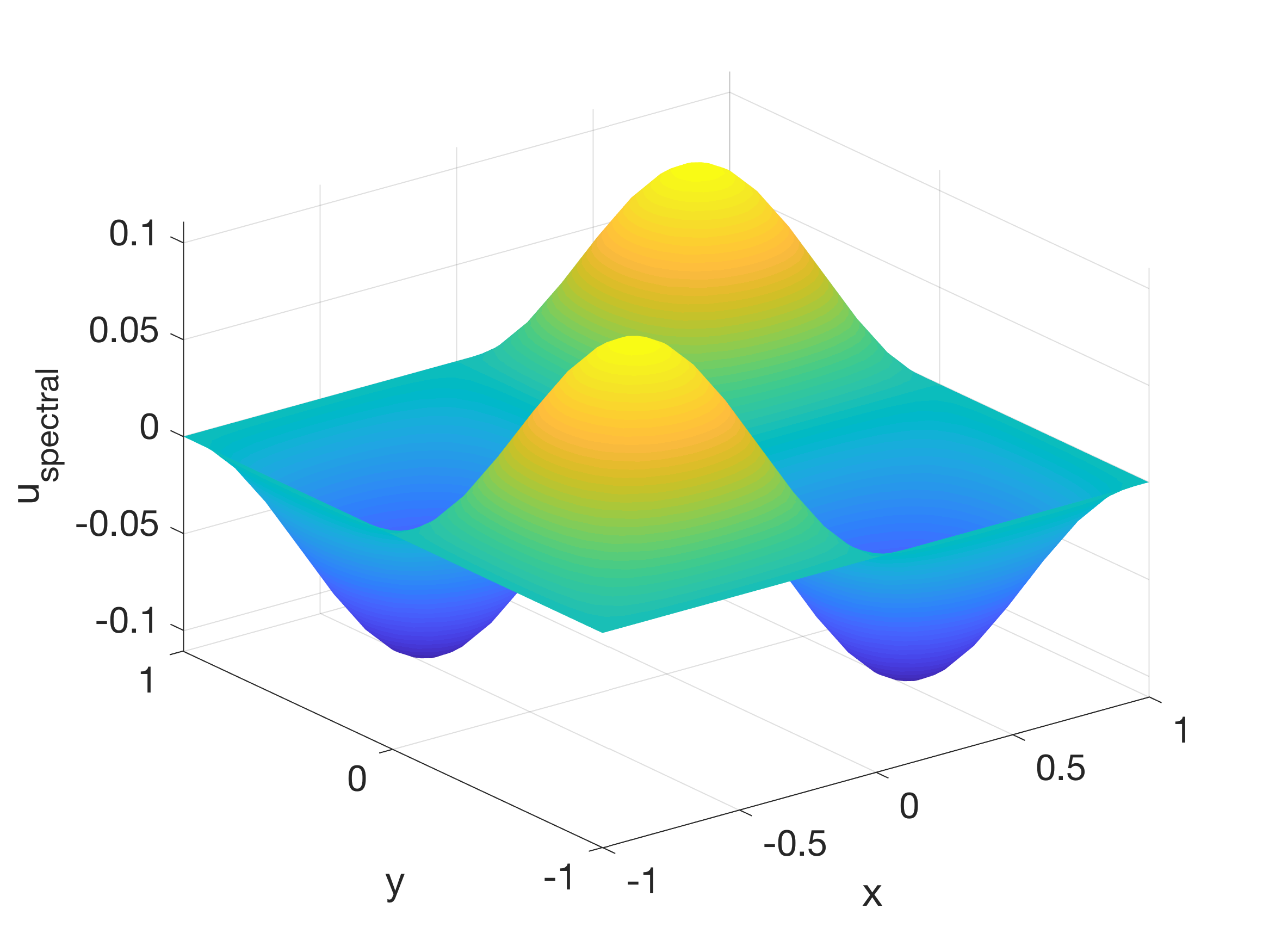}
  \includegraphics[width=0.3\textwidth]{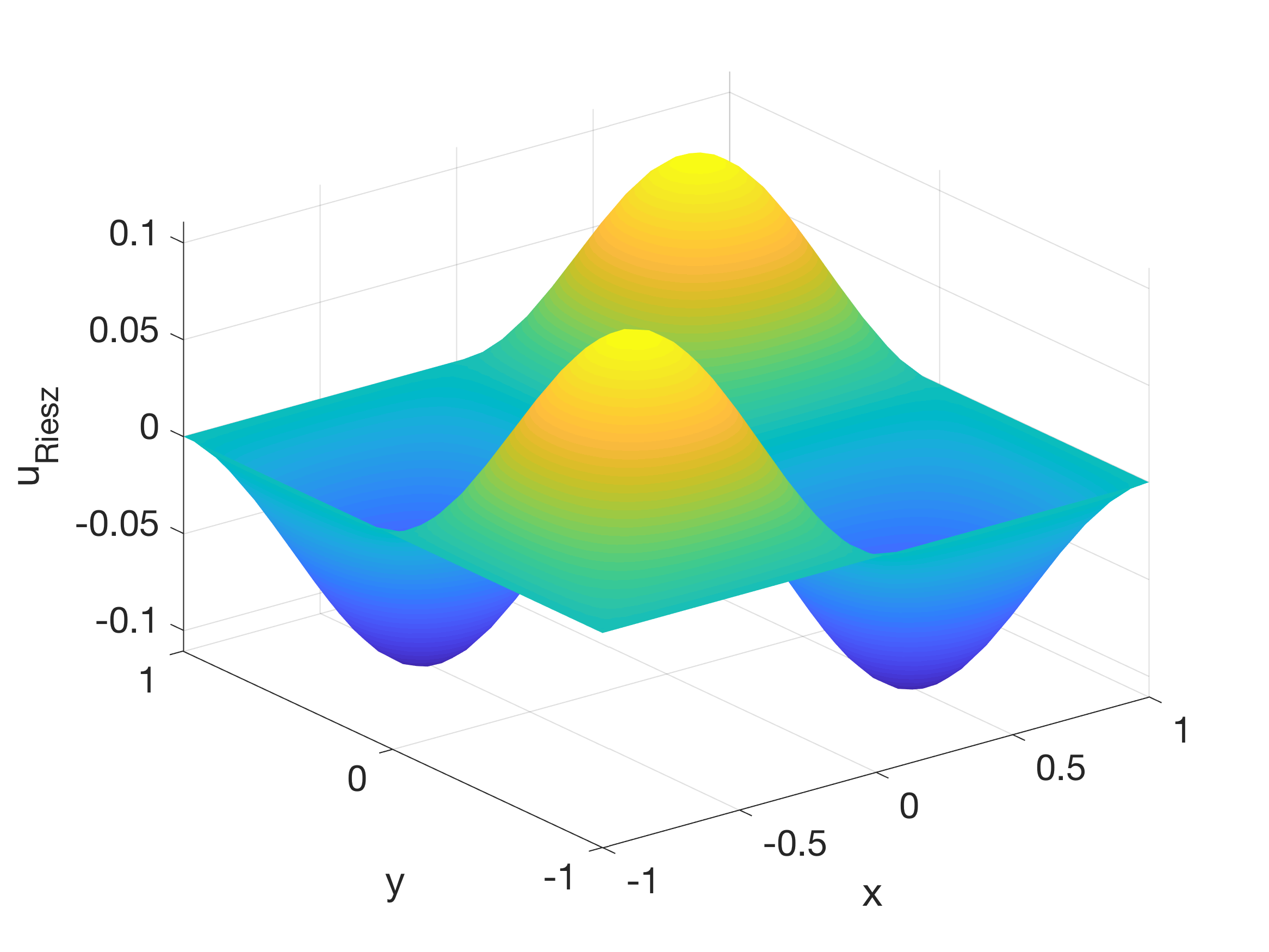}
    \includegraphics[width=0.3\textwidth]{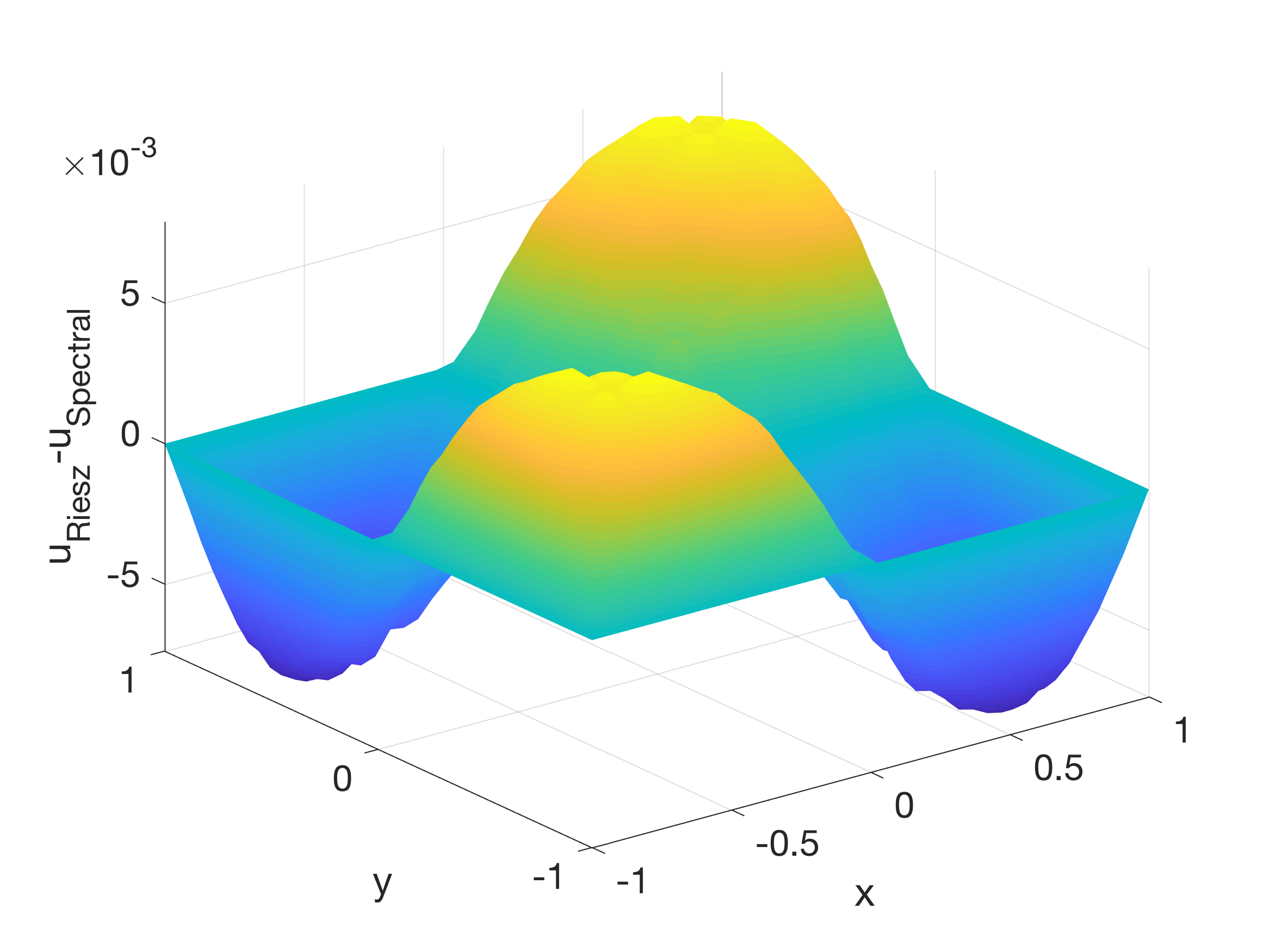}
\end{minipage}
 }
 \caption{\label{square34} {Solutions and differences} between $u_{\text{Riesz}}$ and $u_{\text{spectral}}$ on the square domain for $\alpha = 0.5$ {\color{forest} and $1.5$.}}
 \end{figure}
 
   \begin{figure}[ht!]
 \centering
\includegraphics[height=.2\textheight]{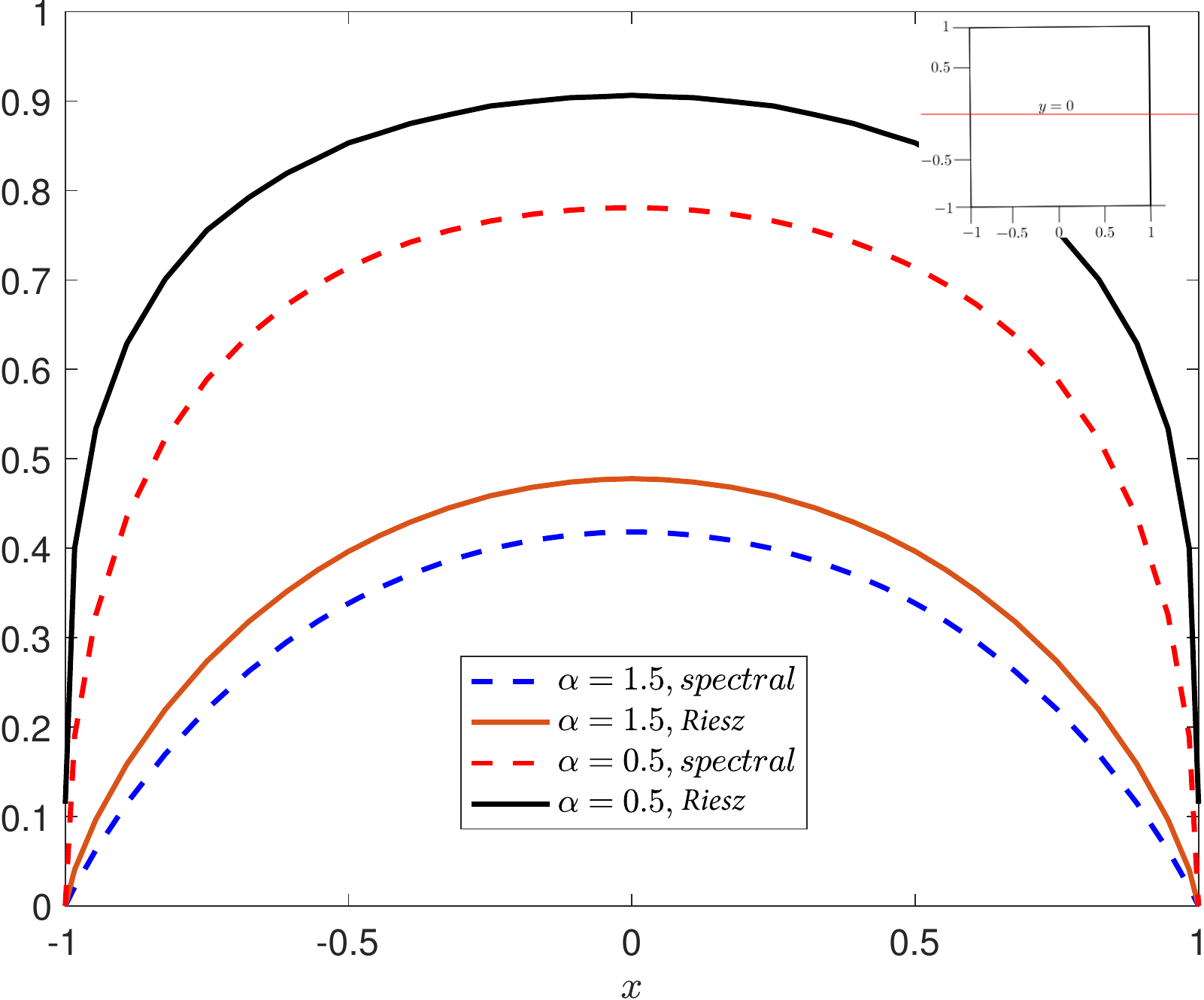}
\includegraphics[height=.2\textheight]{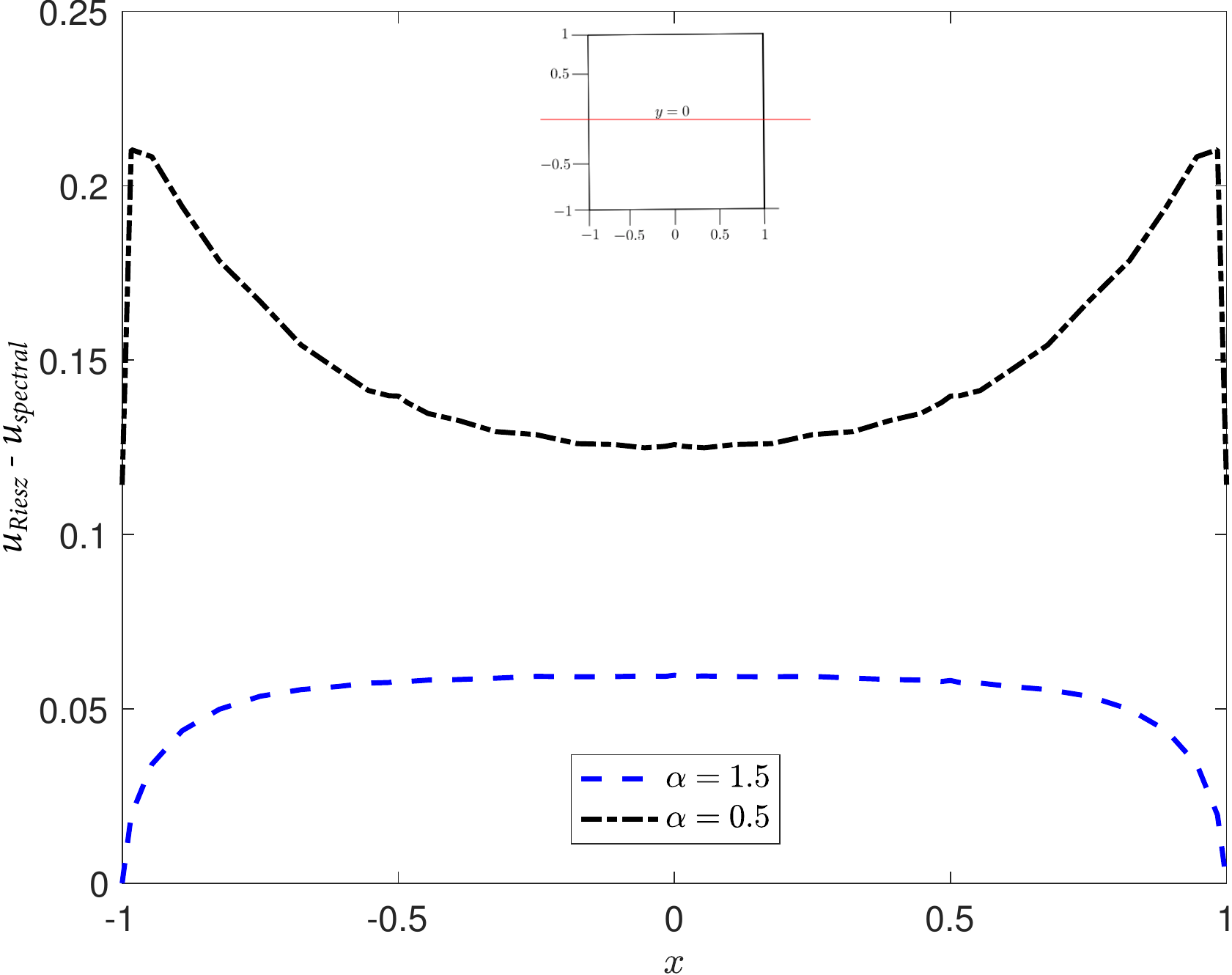}\\
\includegraphics[height=.2\textheight]{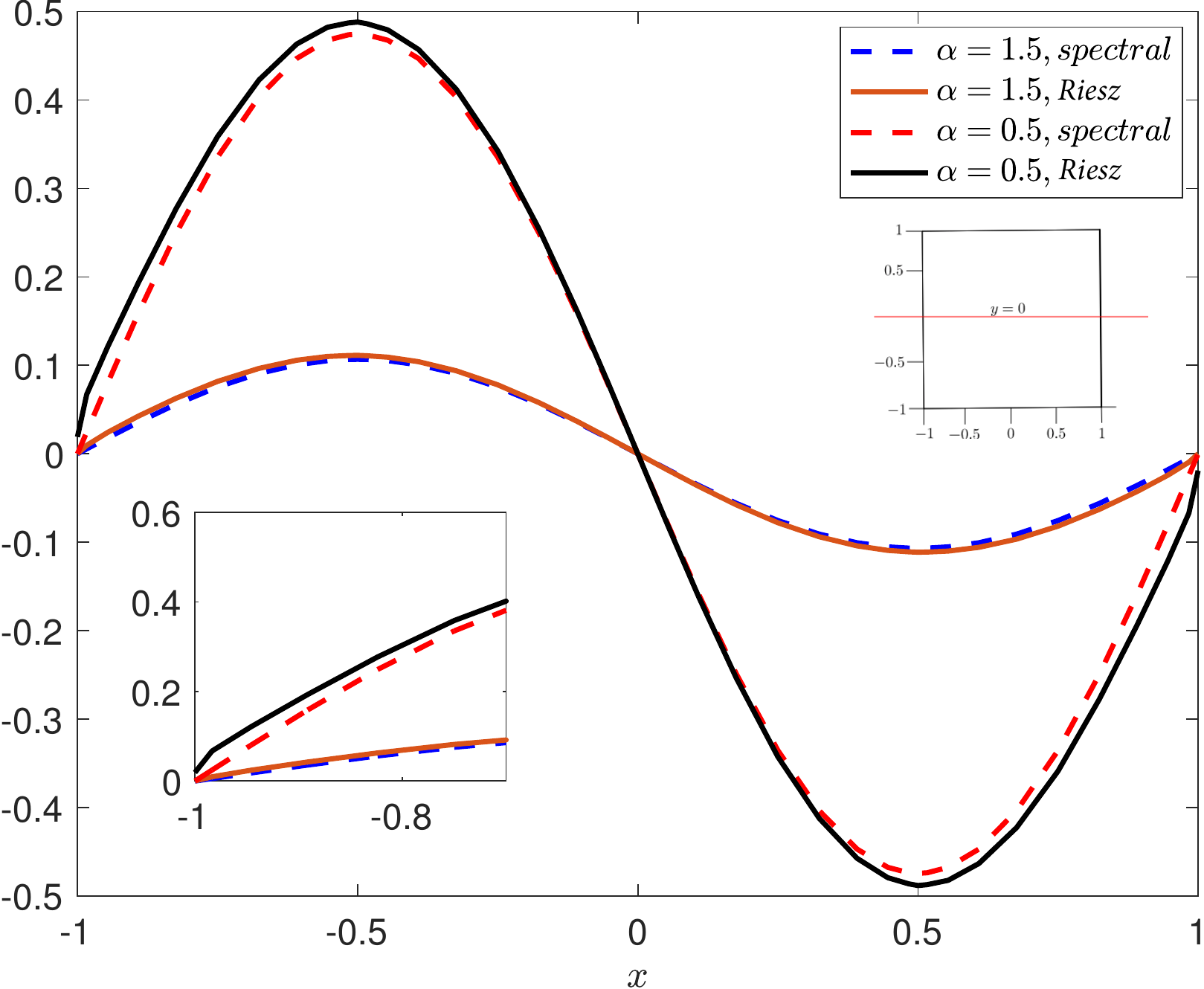}
\includegraphics[height=.2\textheight]{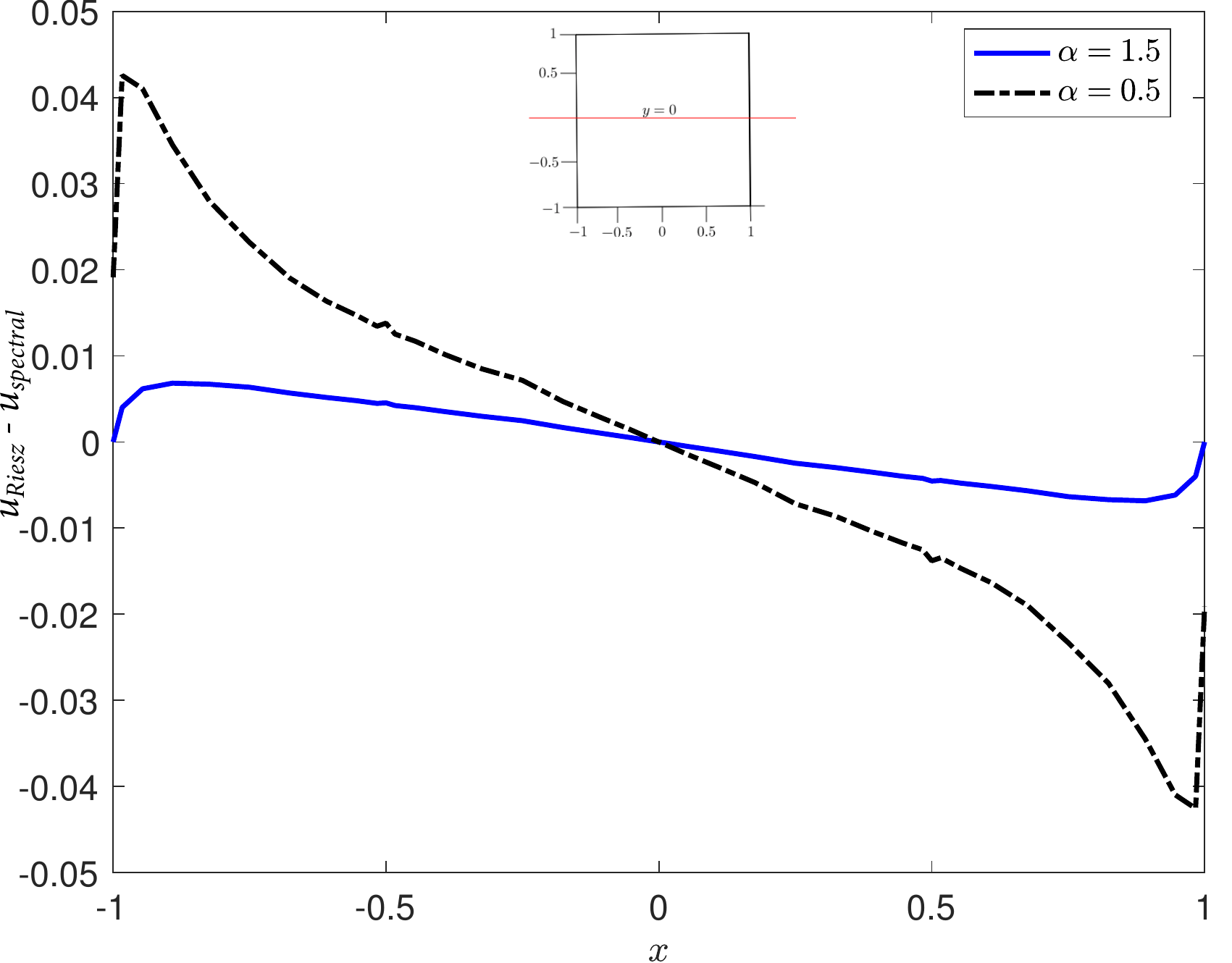}
\caption{\label{squareslicesf1y0} {Slices along the line $y=0$} in the square domain with both $\alpha = 0.5$ and $\alpha = 1.5$. (\emph{top left}) Plots of the solutions for the Riesz and spectral formulations with $f = 1$ and (\emph{top right}) plot of the differences $u_{\text{Riesz}}-u_{\text{spectral}}$ along the line $y = 0$ with $f = 1$. (\emph{bottom left}) Plots of the solutions for the Riesz and spectral formulations with $f = \sin(\pi x)\sin(\pi y)$ and (\emph{bottom right}) plot of the differences $u_{\text{Riesz}}-u_{\text{spectral}}$ along the line $y=0$ with $f = \sin(\pi x)\sin(\pi y)$. Some small wiggles that appear in the difference plots are due to the computation of the differences on two different meshes, even though the solutions are sufficiently converged and stable.}
\end{figure}

\begin{figure}[ht!]
\centering
\includegraphics[width=0.6\textwidth]{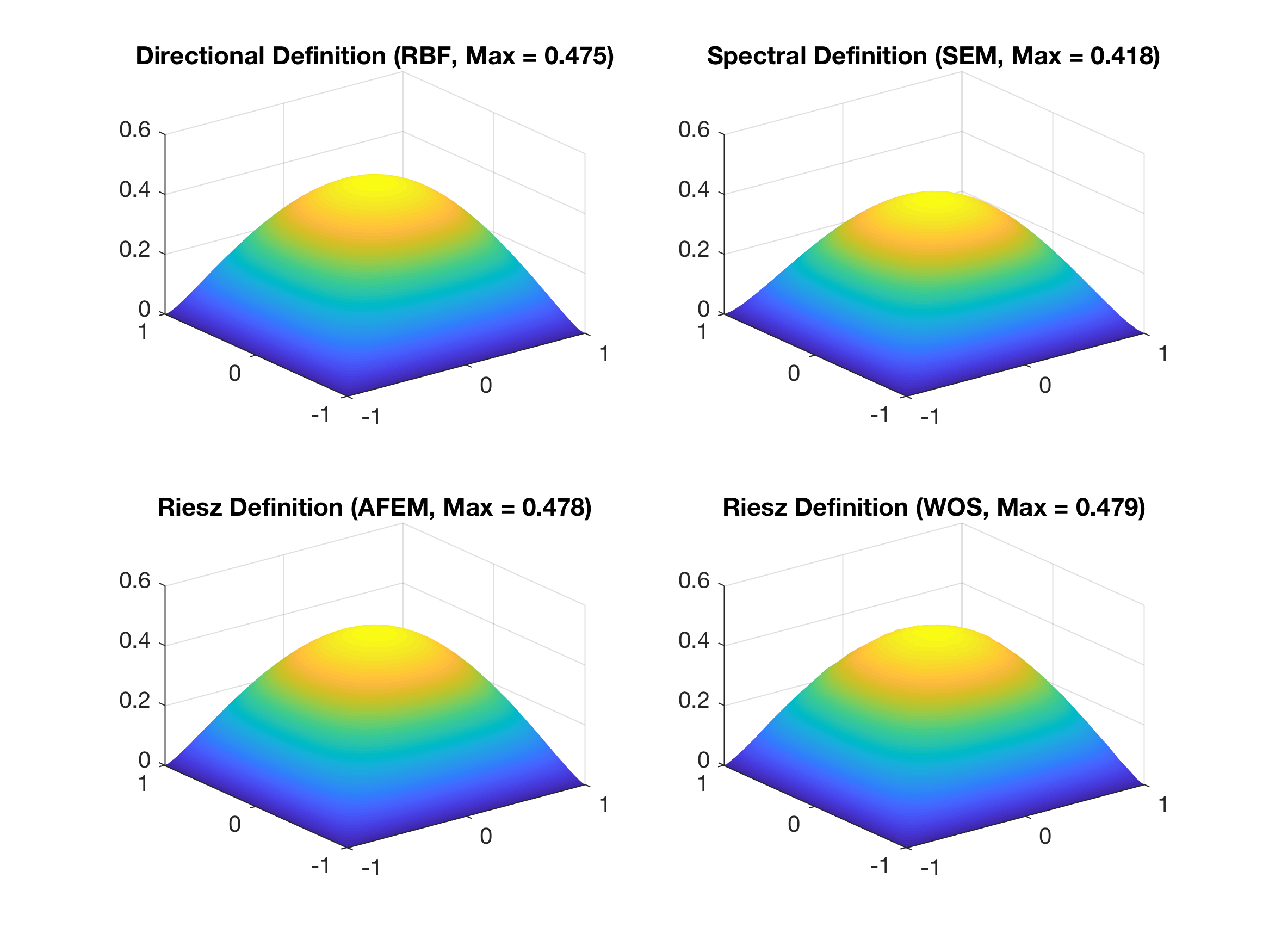}
\caption{
\color{blue}
\label{cmp-squ-const} Square, $f(x)=1$, and $g(x)=0$: 
Comparison, for $\alpha = 1.5$, of $u_{\text{Riesz}}$ and $u_{\text{spectral}}$, using three methods to compute the Riesz solution and one method to compute the spectral solution. 
\emph{Top left}:
The Riesz solution obtained using the RBF collocation method based on the directional representation
\emph{Top right}:
The spectral solution obtained using the SEM.
\emph{Bottom left}:
The Riesz solution obtained using the AFEM.
\emph{Bottom right}:
The Riesz solution obtained using the WOS method. 
The only solution with a significant difference is the spectral solution (\emph{top right}); all other solutions are equivalent up to numerical error.}
\end{figure}

\FloatBarrier

\subsection{Disk Domain}\label{2d:disk}

In this section, we solve the benchmark problems in Table \ref{2dbenchmarks} on the disk domain: $\Omega := \{ (x,y) | \sqrt{x^2 + y^2} \leq 1\}$.
We use the same numerical methods as in Section \ref{2d:square}, where we solved the benchmark problems on the square. The mesh  for the SEM used to compute the spectral solution was generated by transforming the square mesh onto the circular domain. The details of this transformation are included in \cite{SongXuKarniadakis2017}. The collocation points used for the RBF method and the adaptively refined meshes used to compute the Riesz solutions are included in Figure \ref{DiskMeshes}.

In the Figure \ref{disk12}, we again see that the Riesz solutions have oscillations near the boundary in the case with $\alpha = 0.5$, where the spectral solutions satisfy the boundary conditions. For Cases 3 and 4 with $f = \sin(\pi r^2)$, (see Table \ref{2dbenchmarks}), our observations are very similar to those of Figure \ref{square34}, so these cases are displayed in Appendix \ref{app:disk}, Figure \ref{disk34}. {\color{blue} As for the case $f \equiv 1$, the property $f \ge 0$ leads to the Riesz solution lying above the spectral solution \cite{MUSINA20161667}}. We also include one-dimensional slice plots of the Riesz and spectral solutions and their differences along the line $y = 0$ in Figures \ref{diskslicesf1y0}; the profiles are similar to those we saw in the square examples, where the Riesz solutions exhibit sharper boundary layers than the spectral solutions.

The comparisons with the directional definition solution using the RBF collocation method are similar to the square domain examples, and so we include the figure containing these plots in Appendix \ref{app:disk}.

\FloatBarrier

\begin{figure}[ht!]
 \centering
\subfloat[Solutions $u$ associated with $f=1$ and $\alpha = 0.5$ in the disk domain using the spectral definition (using SEM) (\emph{left}) and the Riesz definition (using AFEM) (\emph{center}), and the difference between $u_{\text{Riesz}}$ and $u_{\text{spectral}}$ for this case (\emph{right}).]{
 \begin{minipage}[]{\textwidth}\centering
 \includegraphics[width=0.3\textwidth]{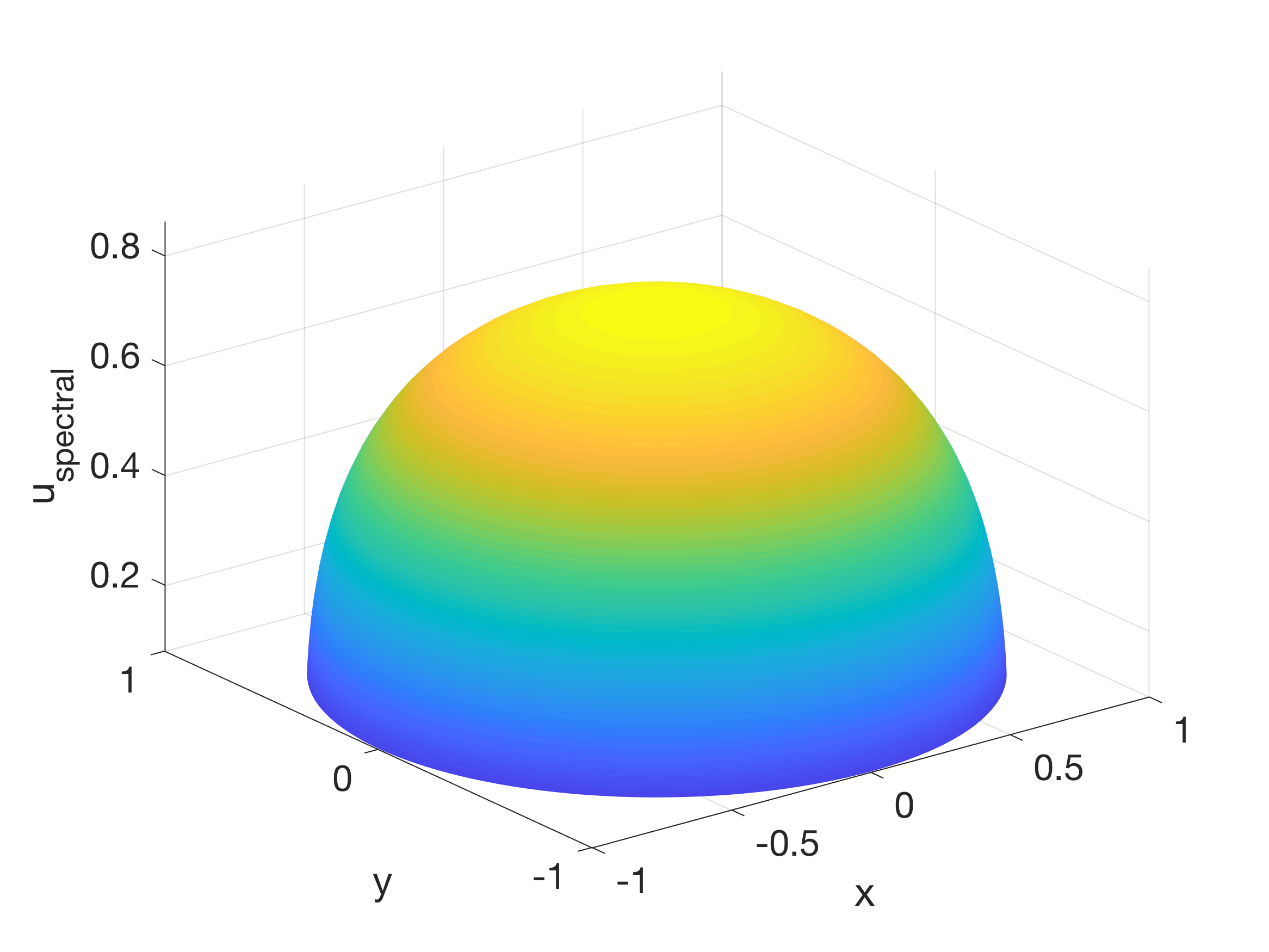}
  \includegraphics[width=0.3\textwidth]{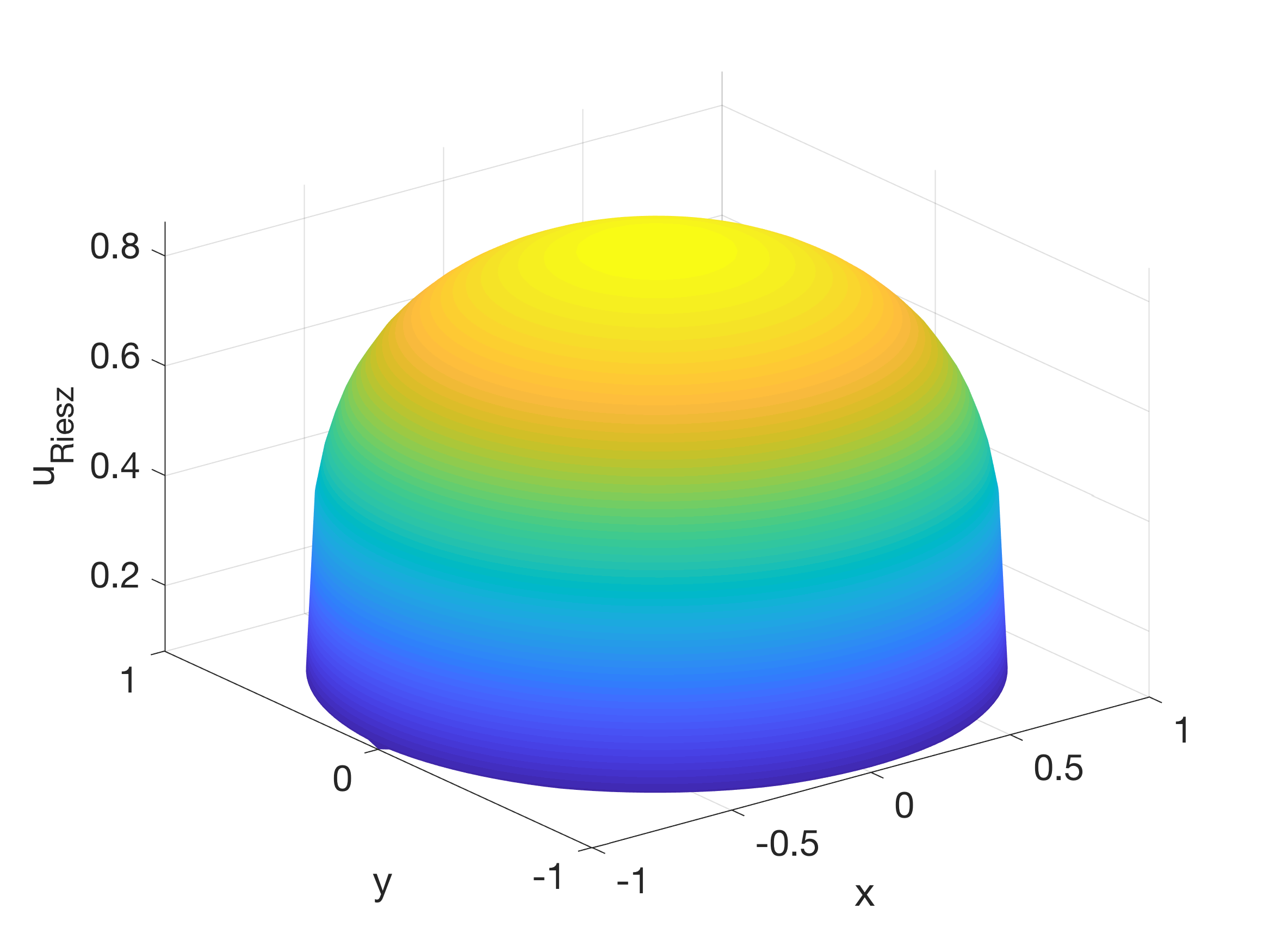}
  \includegraphics[width=0.3\textwidth]{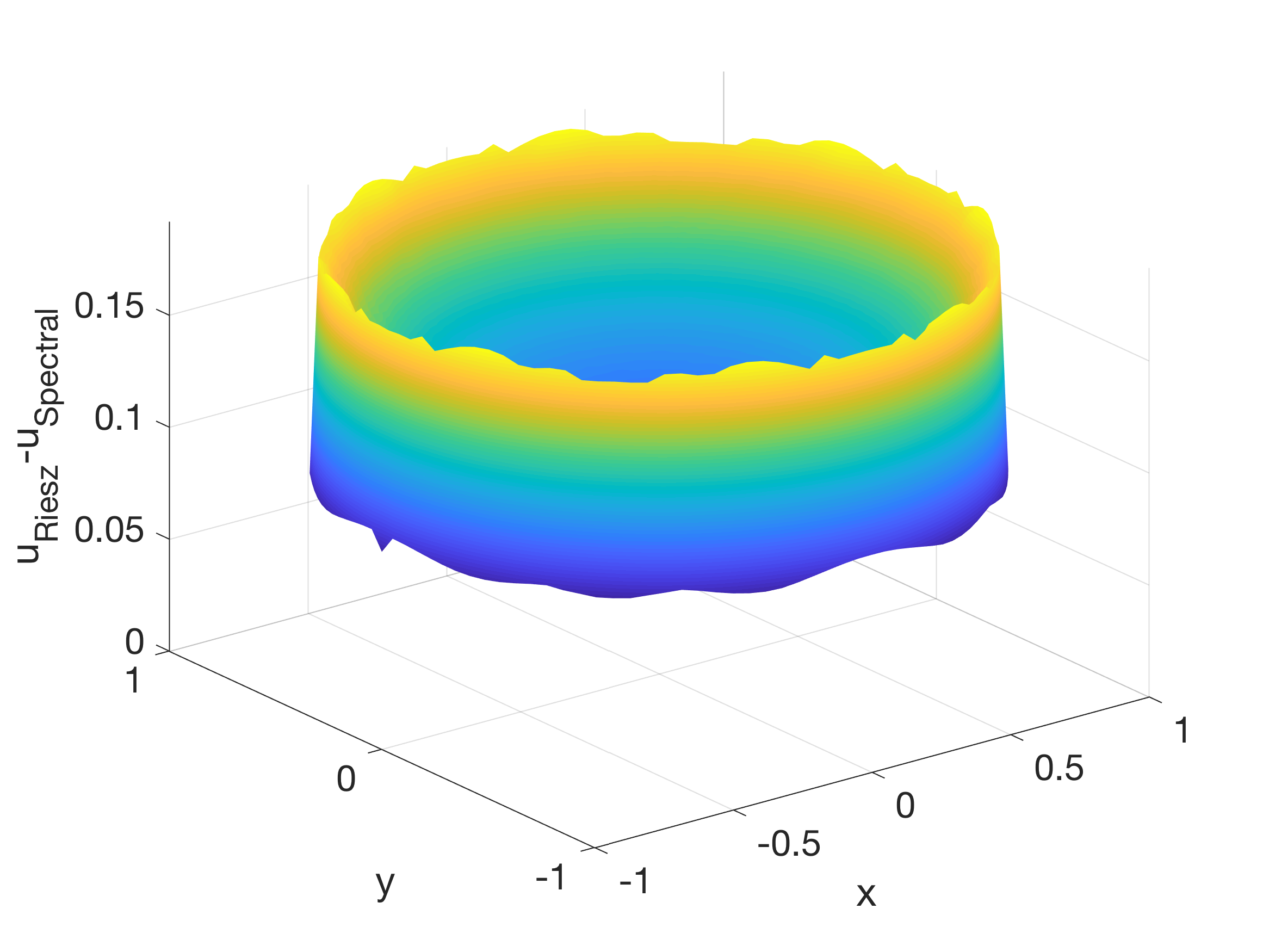}
\end{minipage}
 }\\
\subfloat[Solutions $u$ associated with $f=1$ and $\alpha = 1.5$ in the disk domain using the spectral definition (using SEM) (\emph{left}) and the Riesz definition (using AFEM) (\emph{center}), and the difference between $u_{\text{Riesz}}$ and $u_{\text{spectral}}$ for this case (\emph{right}).]{
 \begin{minipage}[]{\textwidth}\centering
 \includegraphics[width=0.3\textwidth]{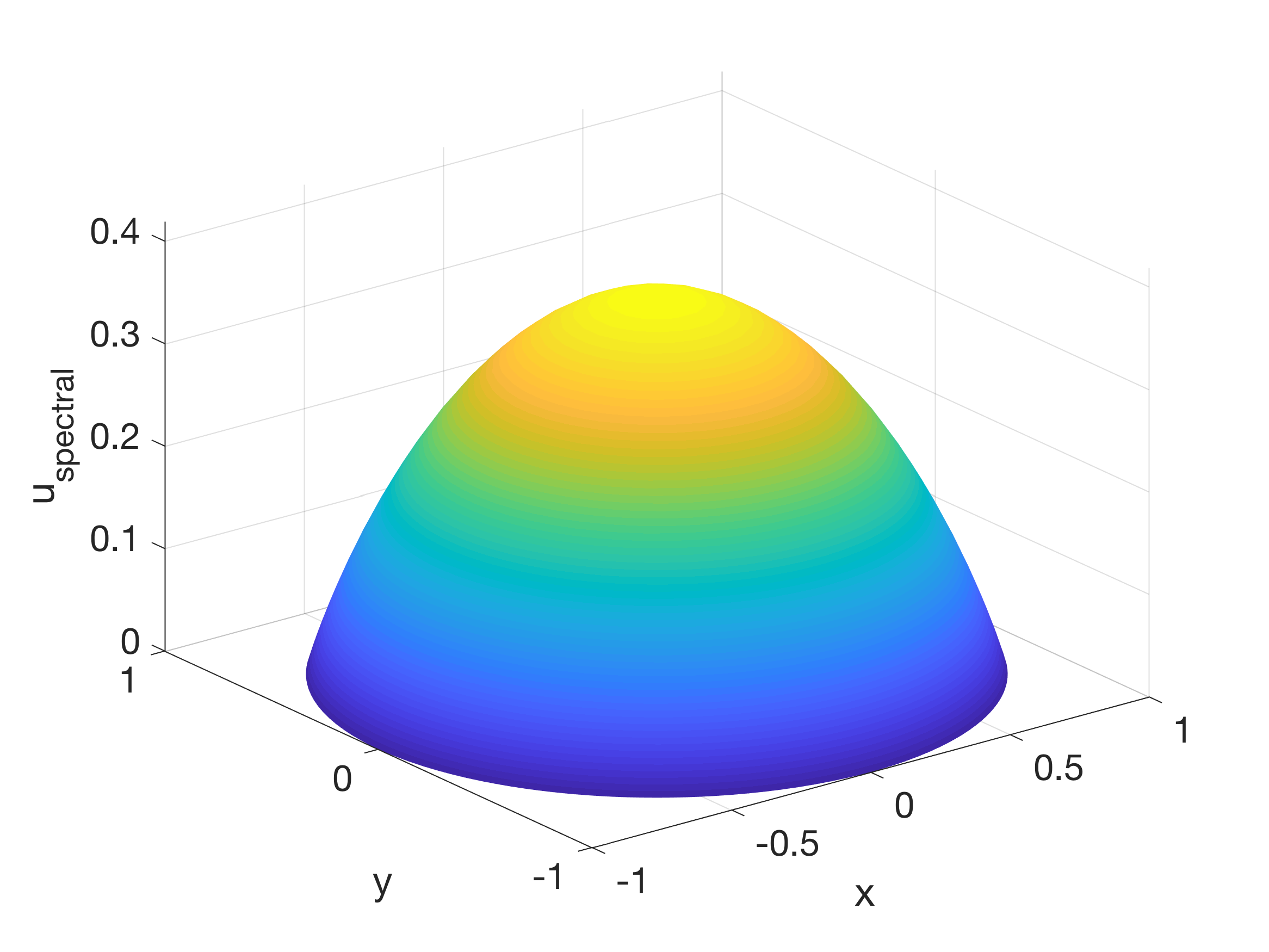}
  \includegraphics[width=0.3\textwidth]{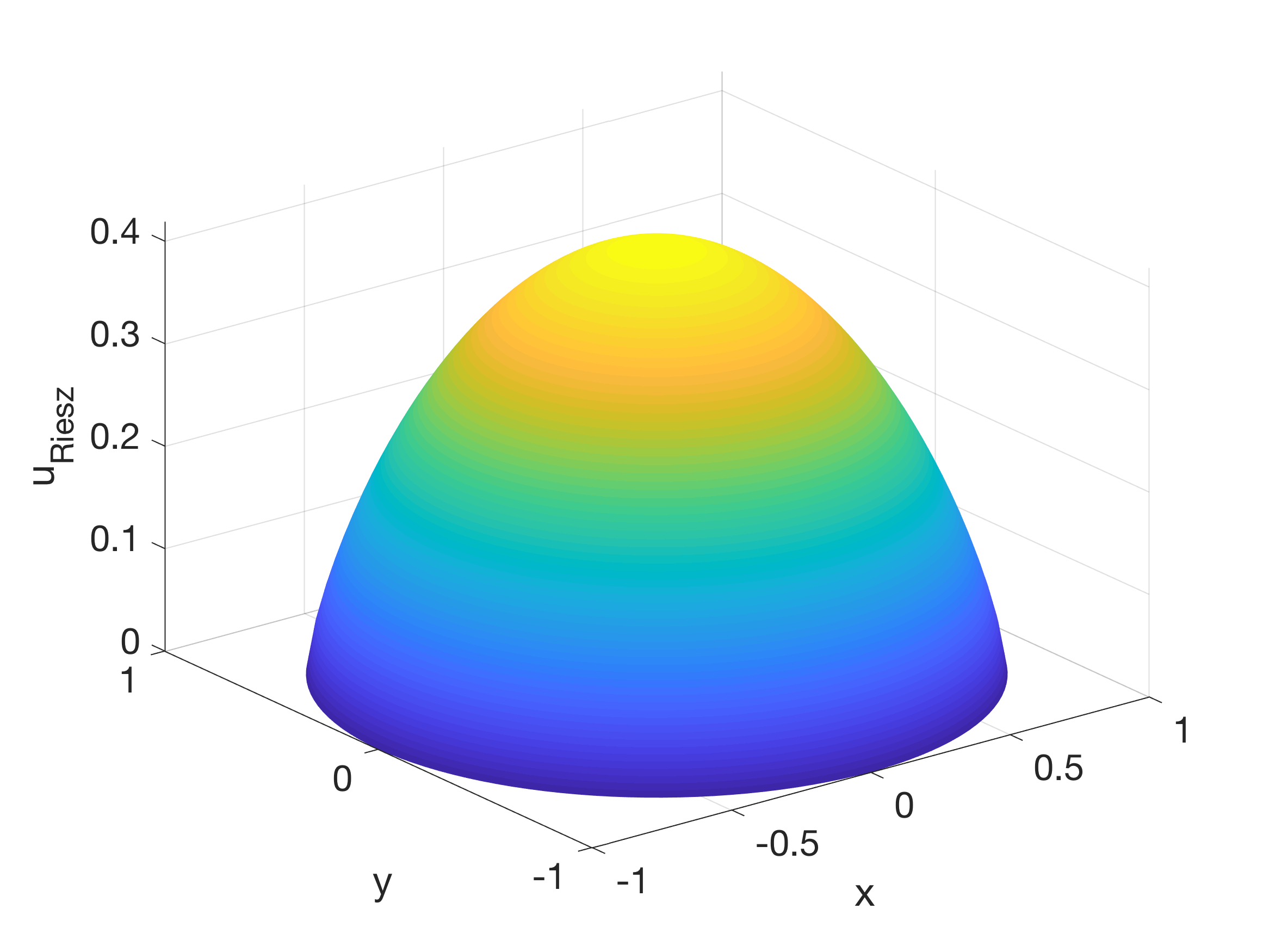}
  \includegraphics[width=0.3\textwidth]{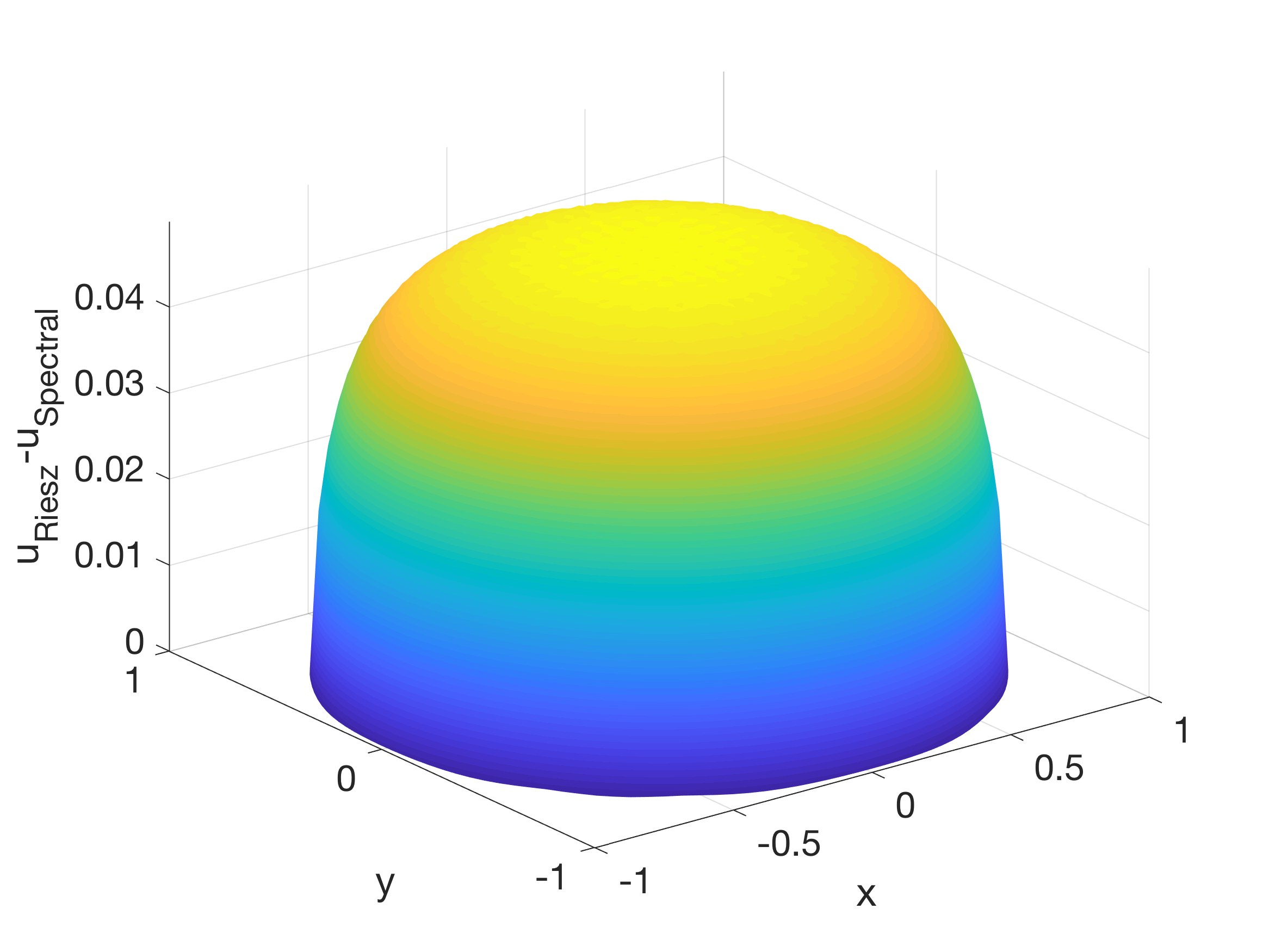}
\end{minipage}
 }
 \caption{ \label{disk12} {Solutions and differences} between $u_{\text{Riesz}}$ and $u_{\text{spectral}}$ on the disk domain for $f = 1$ and $\alpha = 0.5$ and $1.5$.}
 \end{figure}

  \begin{figure}[ht!]
 \centering
\includegraphics[height=.2\textheight]{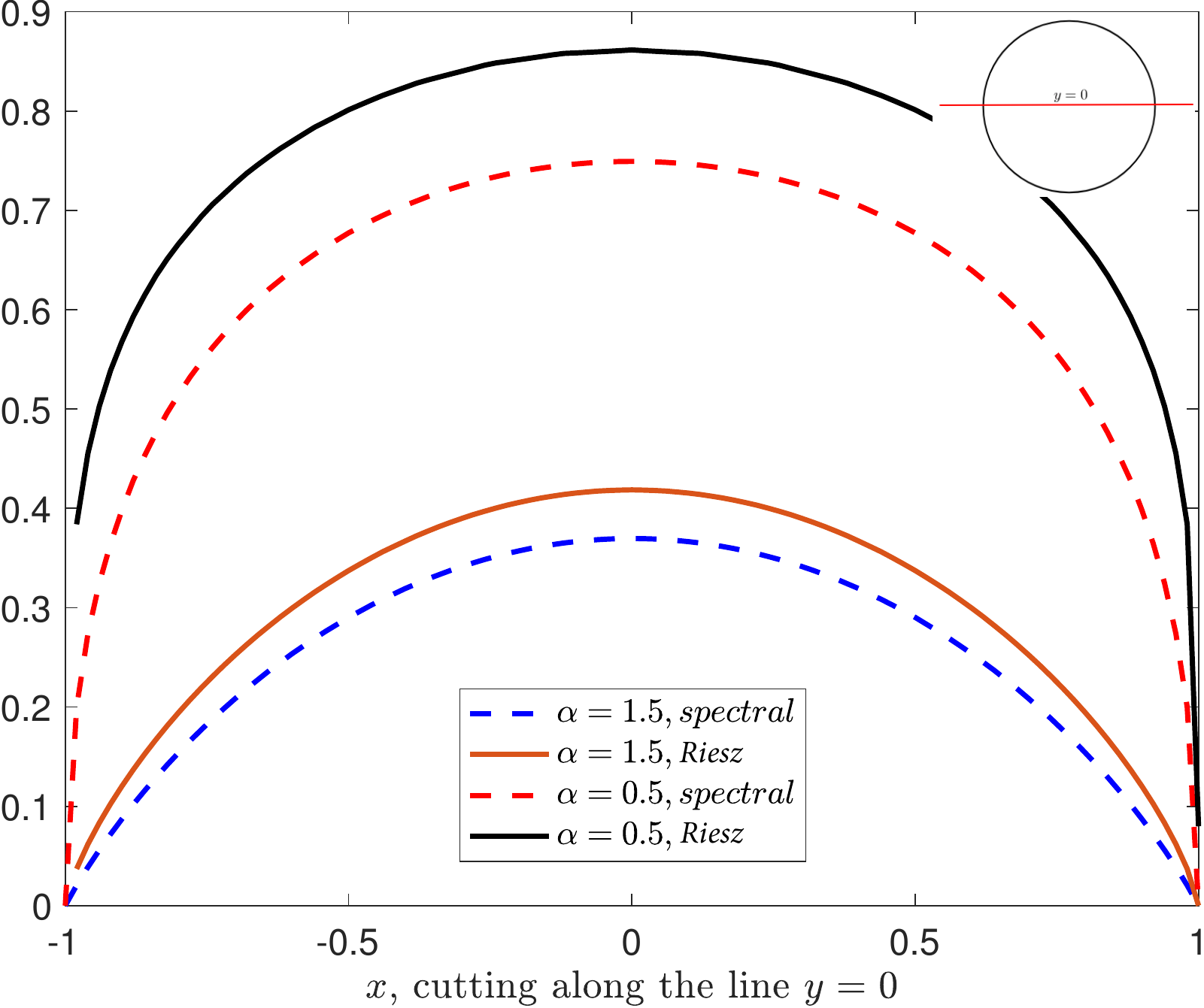}
\includegraphics[height=.2\textheight]{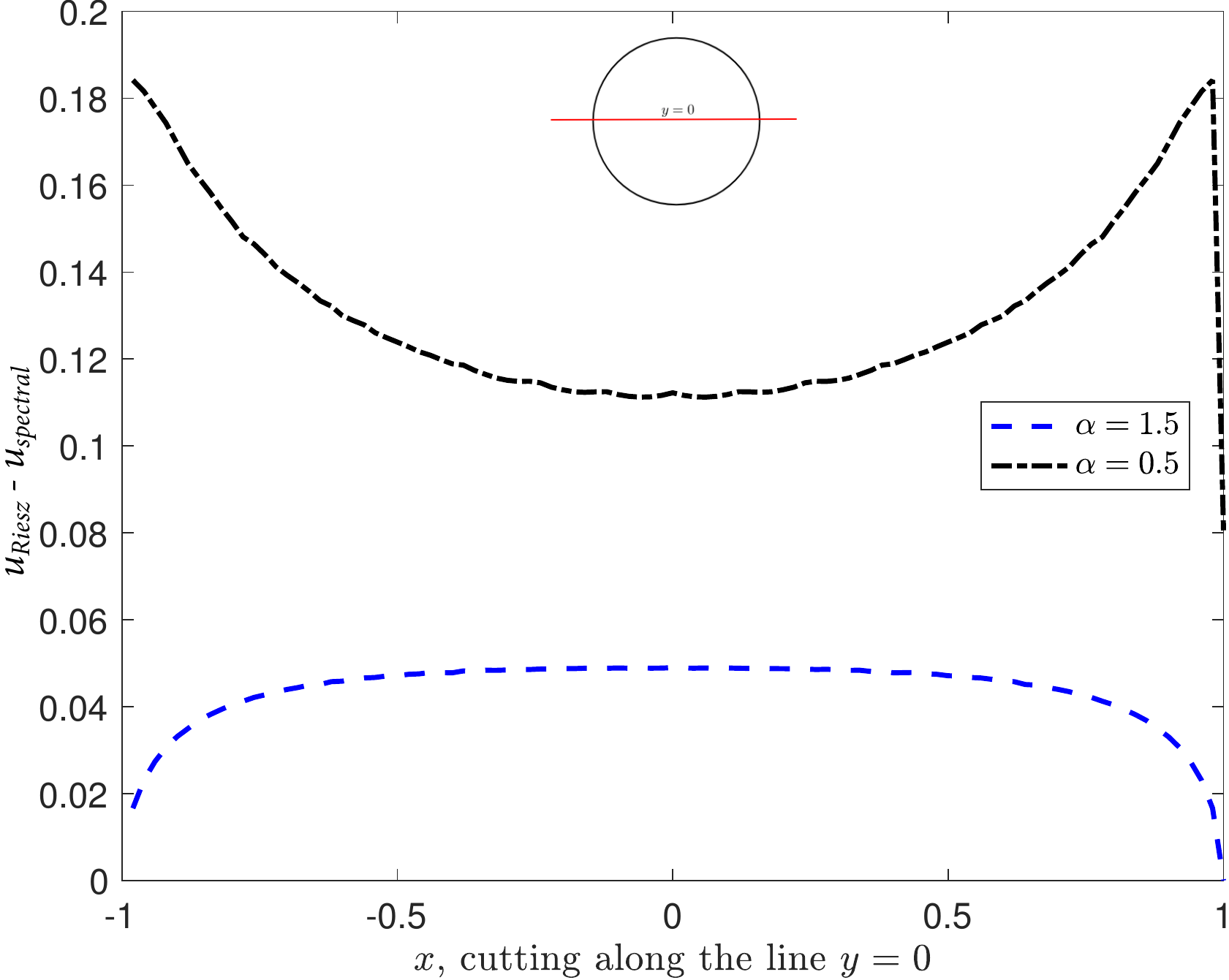}\\
\includegraphics[height=.2\textheight]{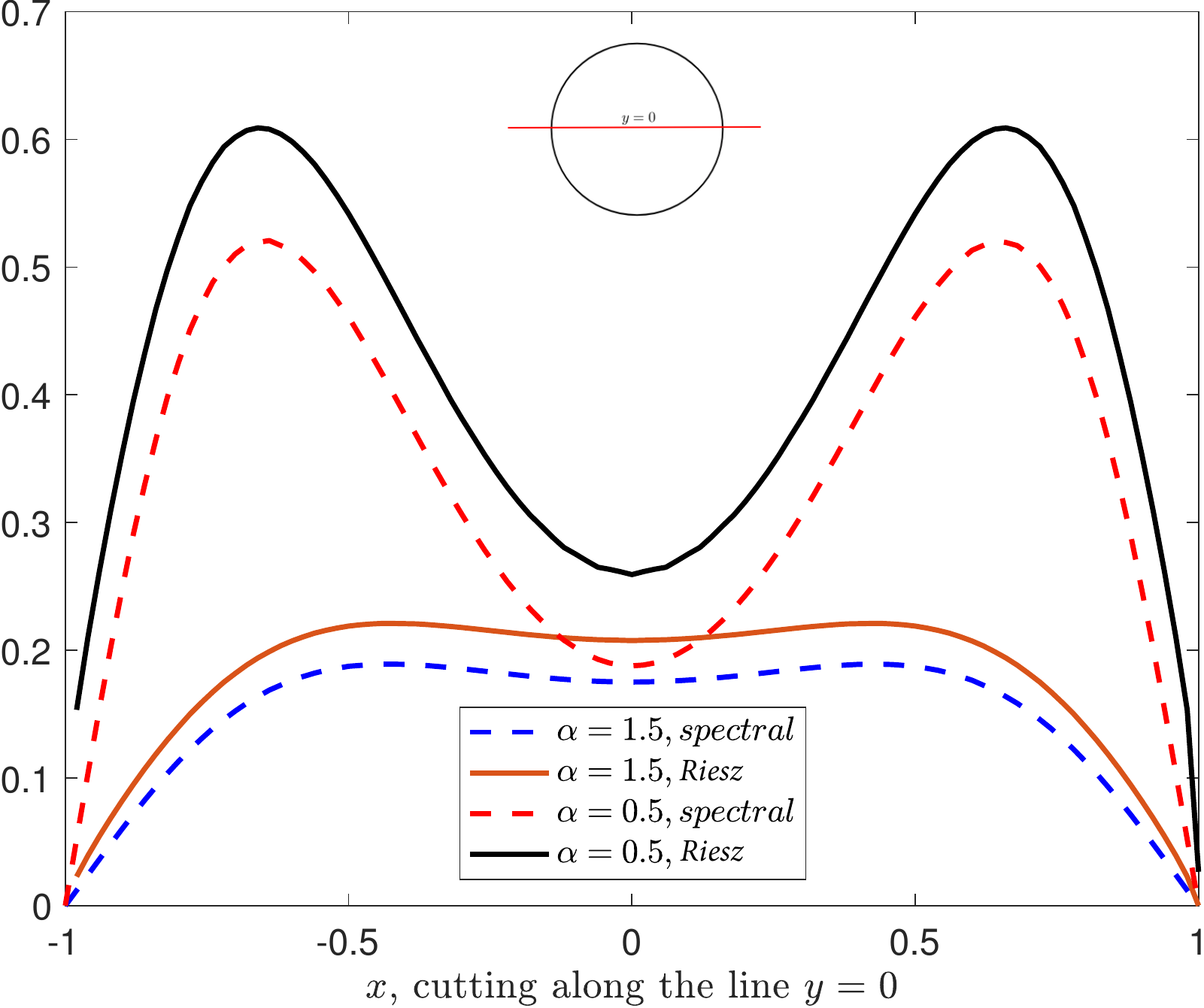}
\includegraphics[height=.2\textheight]{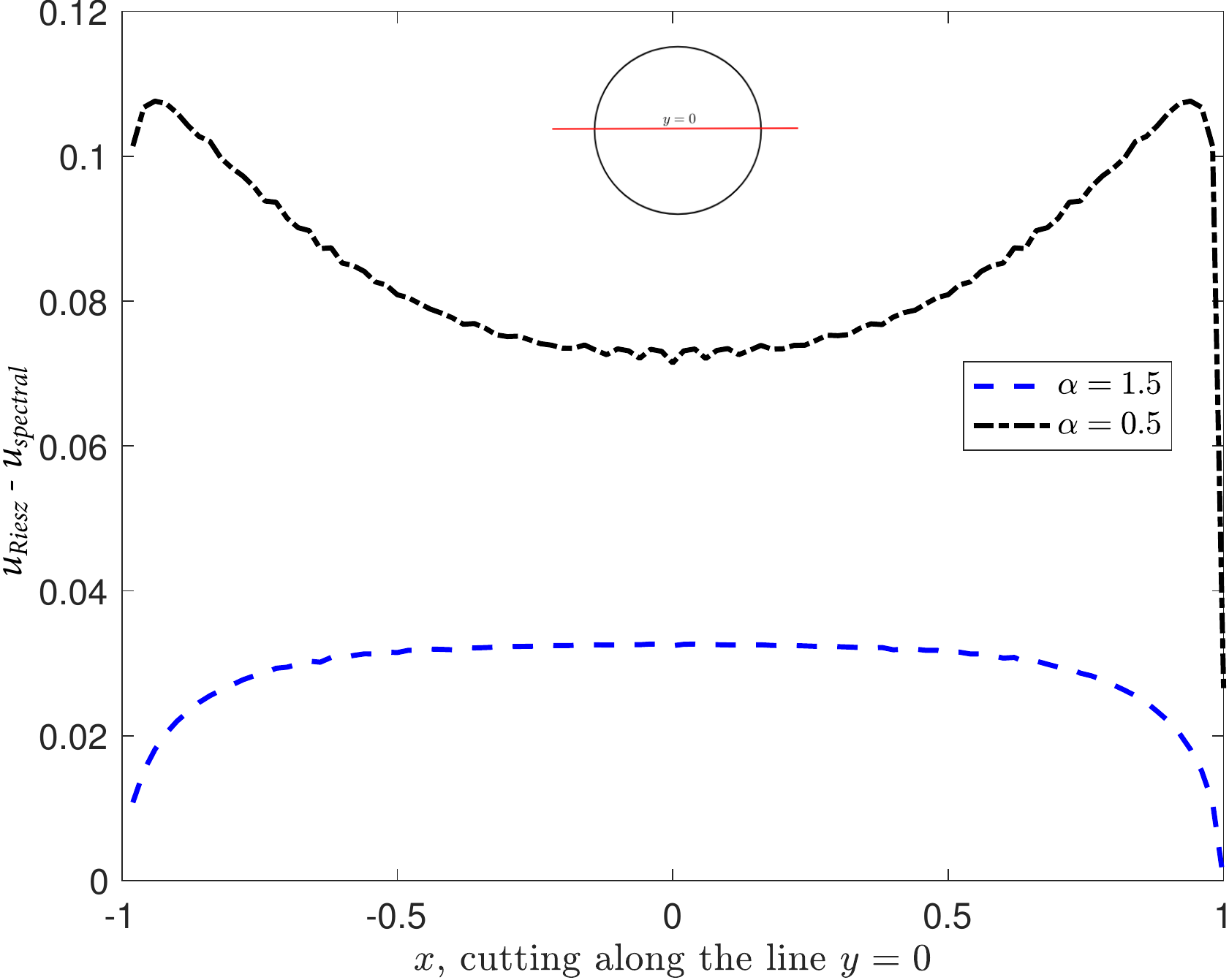}
\caption{\label{diskslicesf1y0} {Slices along the line $y=0$} in the unit disk domain with both $\alpha = 0.5$ and $\alpha = 1.5$. (\emph{top left}) Plots of the solutions for the Riesz and spectral formulations with $f = 1$ and (\emph{top right}) plot of the differences $u_{\text{Riesz}}-u_{\text{spectral}}$ along the line $y = 0$ and $f = 1$. (\emph{bottom left}) Plots of the solutions for the Riesz and spectral formulations with $f = \sin(\pi r^2)$ and (\emph{bottom right}) plot of the differences $u_{\text{Riesz}}-u_{\text{spectral}}$ along the line $y = 0$ and $f = \sin(\pi r^2)$. Some small wiggles that appear in the difference plots are due to the computation of the differences on two different meshes, even though the solutions are sufficiently converged and stable.}
\end{figure}

\subsection{L-Shaped Domain}\label{2d:Lshape}
In this section, we solve the four benchmark problems of Table \ref{2dbenchmarks} on the L-shaped domain: $\Omega := \{[-1,1]^2 \setminus [0,1)^2\}$, i.e., the square $[-1,1]^2$ with the upper right corner removed. We display the solutions from two viewing angles: from the points $(-1,-1)$ and $(1,1)$, so that all features are clearly visible.

We use the same numerical methods as in Section \ref{2d:square} and \ref{2d:disk}, where we solved the benchmark problems on the square and disk domains. The mesh for the SEM used to compute the spectral solution includes additional refinement near the boundary to ensure a converged numerical solution, and is plotted in Figure \ref{LshapeMeshes} in Appendix \ref{grids}. The collocation points used for the RBF method and the adaptively refined meshes used to compute the Riesz solutions are also included in Figure \ref{LshapeMeshes} in Appendix \ref{grids}.

Figure \ref{Lshape12} shows the Riesz and spectral solutions for Cases 1 and 2, where $\alpha = 0.5$ and $1.5$, respectively, and $f = 1$ at the two viewing angles, and Figure \ref{Lshape34} shows the solutions for Cases 3 and 4. We observe that the difference plots, particularly for the cases with $\alpha = 0.5$, exhibit a relatively minor spike near the inside corner of the domain. {\color{blue} Note that} this spike does not occur (or is at least much less pronounced) for the cases with $\alpha = 1.5$.

We also compute the directional definition solution on the L-shaped domain using the RBF collocation method. We show the comparison in Figure \ref{cmp-L-const} using the view from $(1,1)$, and the view from $(-1,-1)$ is included in Appendix \ref{app:Lshape}.

In this set of examples, we also plot the solutions and differences along the slices defined by the lines $y = x$ and $y = 1-x$ in Figure \ref{Lslicesf1yx} to further observe the behaviors near the corner. 
{\color{blue} Again, the fact that the Riesz solution lies above the spectral solution is consistent with $f \ge 0$ in the domain and the theoretical result of \cite{MUSINA20161667}. The spike} near the corner {\color{blue}is} intensified in the cases where $\alpha = 0.5$.

  \begin{figure}[ht!]
 \centering
\subfloat[Solutions $u$ associated with $f=1$ and $\alpha = 0.5$ in the L-shaped domain using the spectral definition (using SEM) (\emph{left}) and the Riesz definition (using AFEM) (\emph{center}), and the difference between $u_{\text{Riesz}}$ and $u_{\text{spectral}}$ for this case (\emph{right}).]{
 \begin{minipage}[]{\textwidth}\centering
 \includegraphics[width=0.25\textwidth]{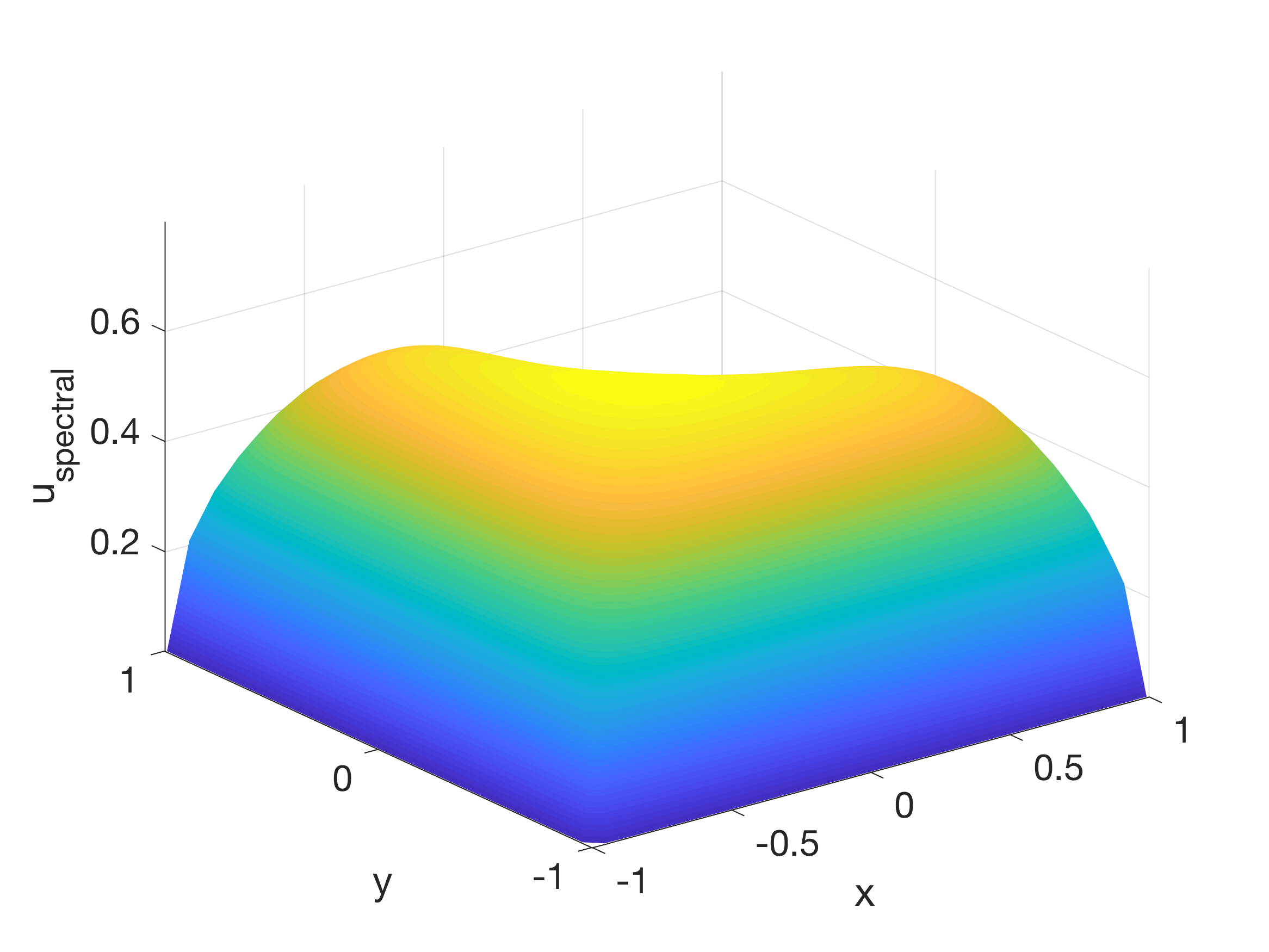}
  \includegraphics[width=0.25\textwidth]{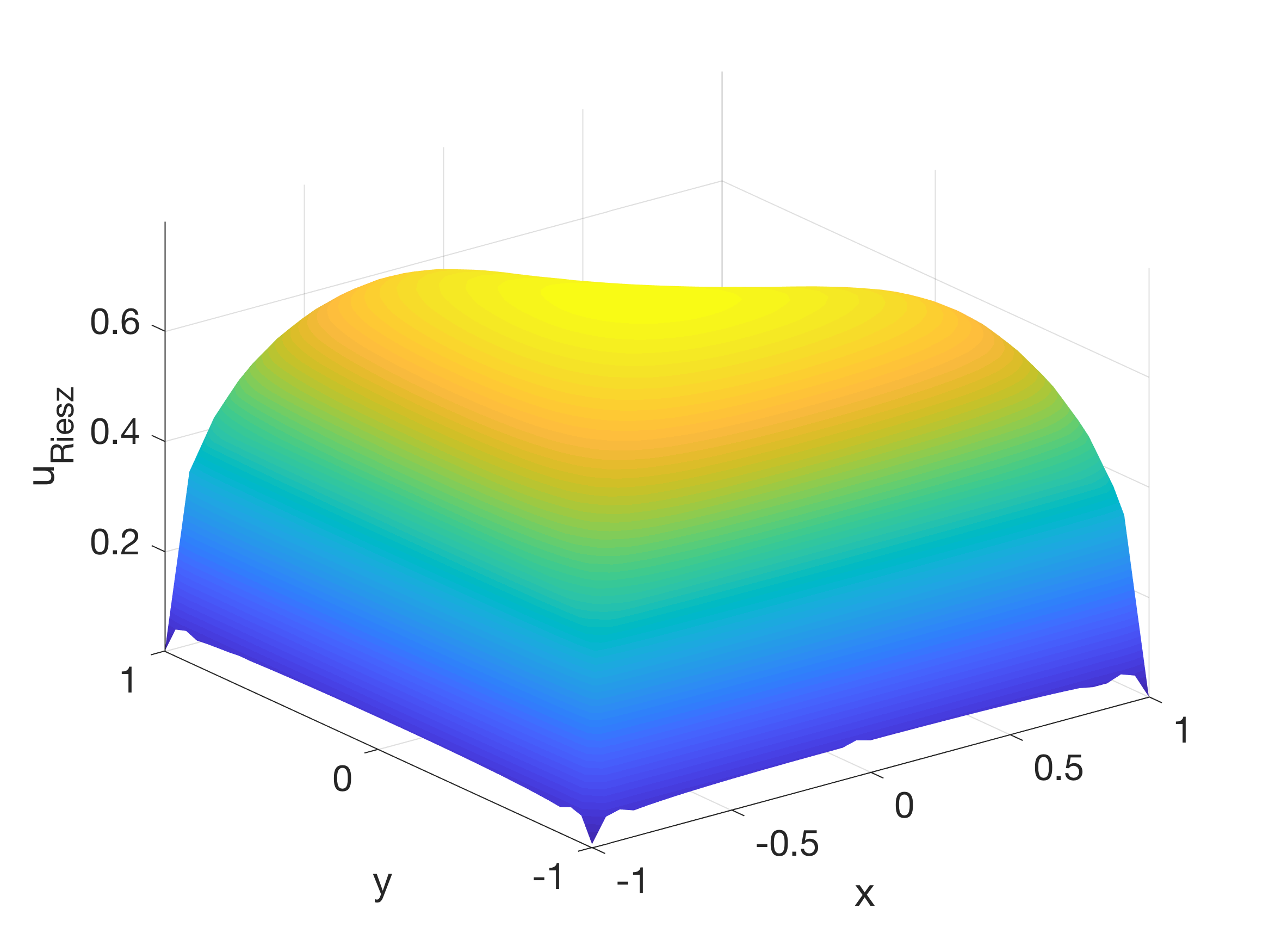}
  \includegraphics[width=0.25\textwidth]{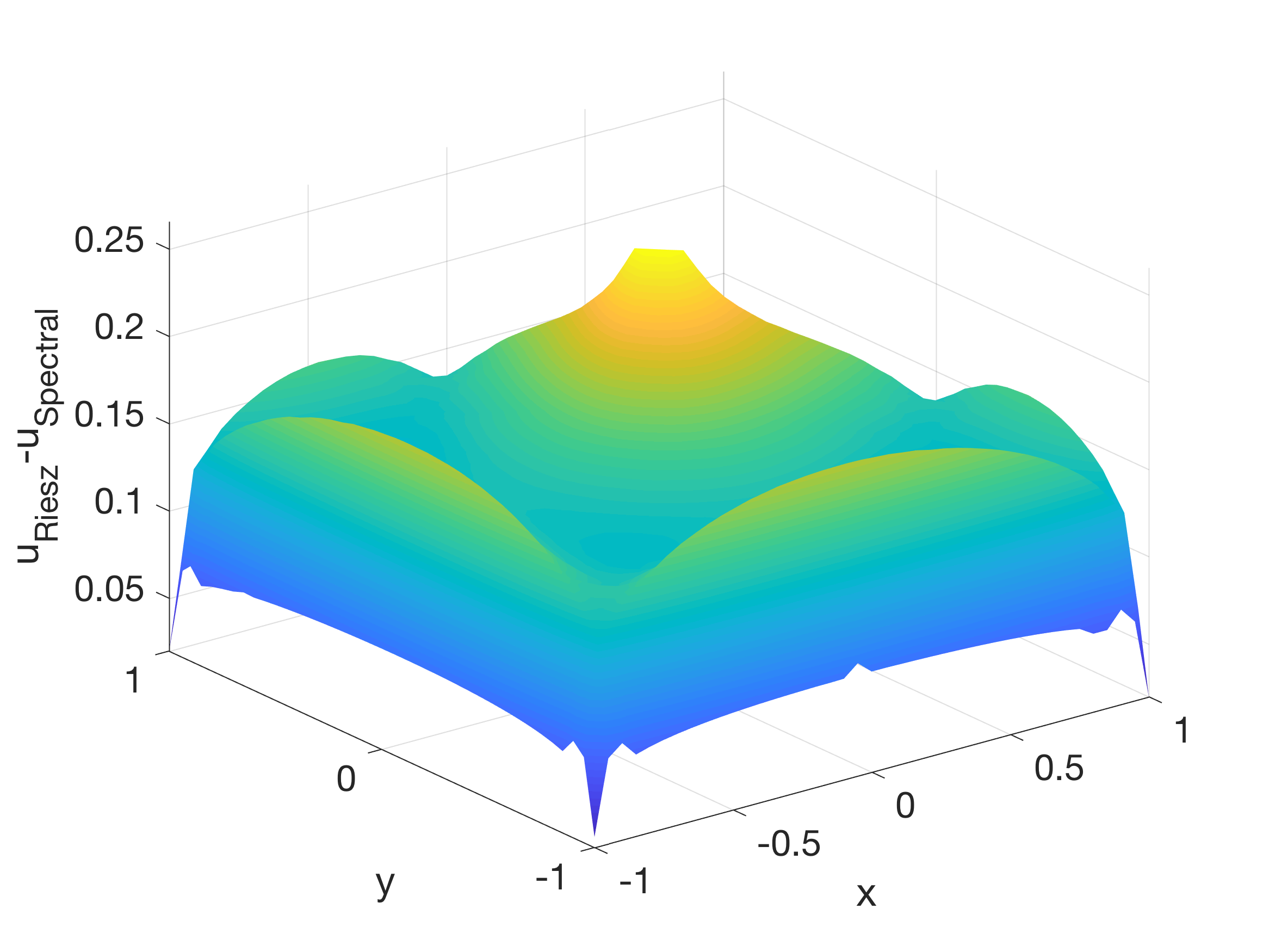}
\end{minipage}
 }\\
  \subfloat[Solutions $u$ associated with $f=1$ and $\alpha = 1.5$ in the L-shaped domain using the spectral definition (using SEM) (\emph{left}) and the Riesz definition (using AFEM) (\emph{center}), and the difference between $u_{\text{Riesz}}$ and $u_{\text{spectral}}$ for this case (\emph{right}).]{
 \begin{minipage}[]{\textwidth}\centering
 \includegraphics[width=0.25\textwidth]{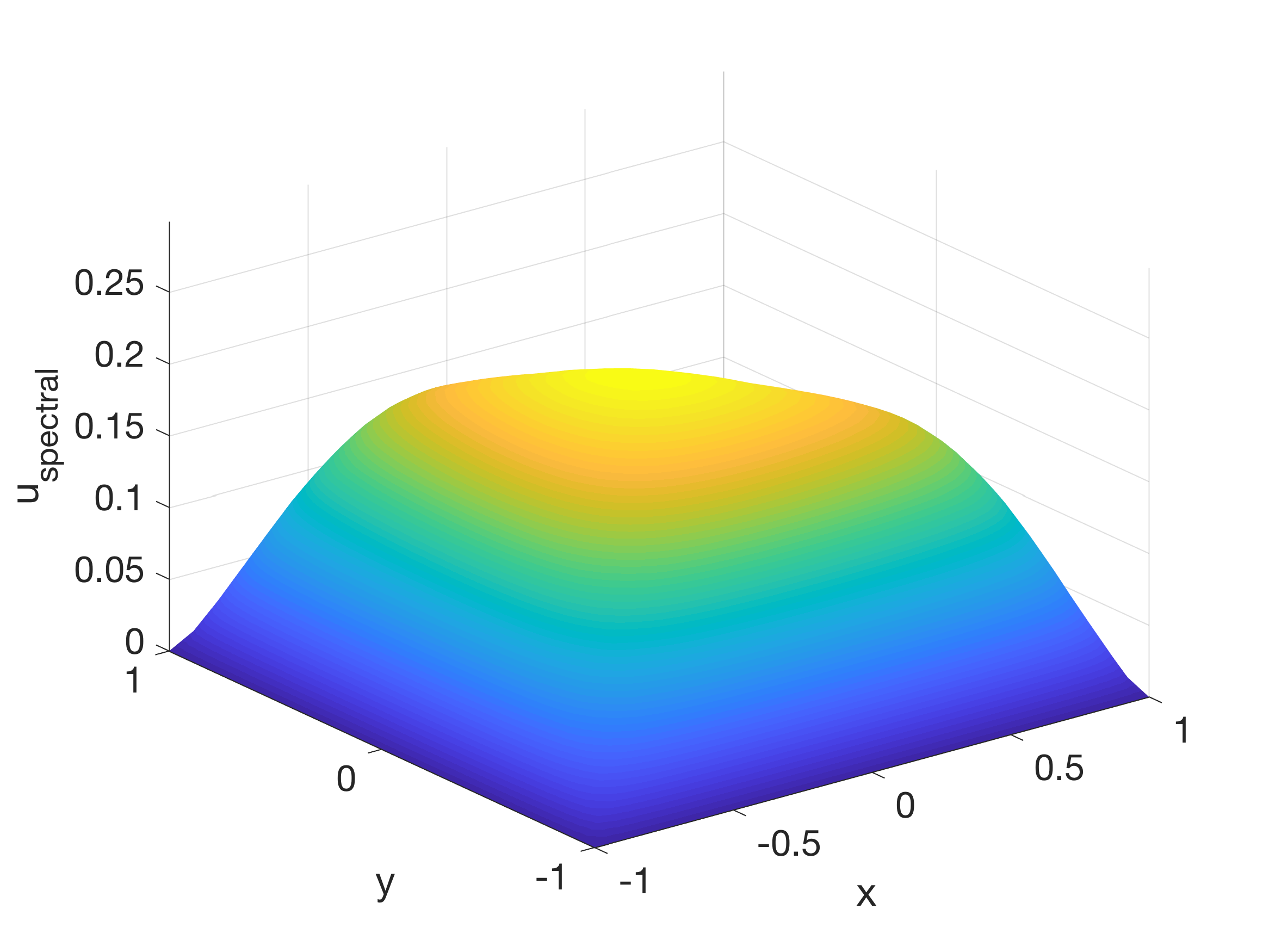}
  \includegraphics[width=0.25\textwidth]{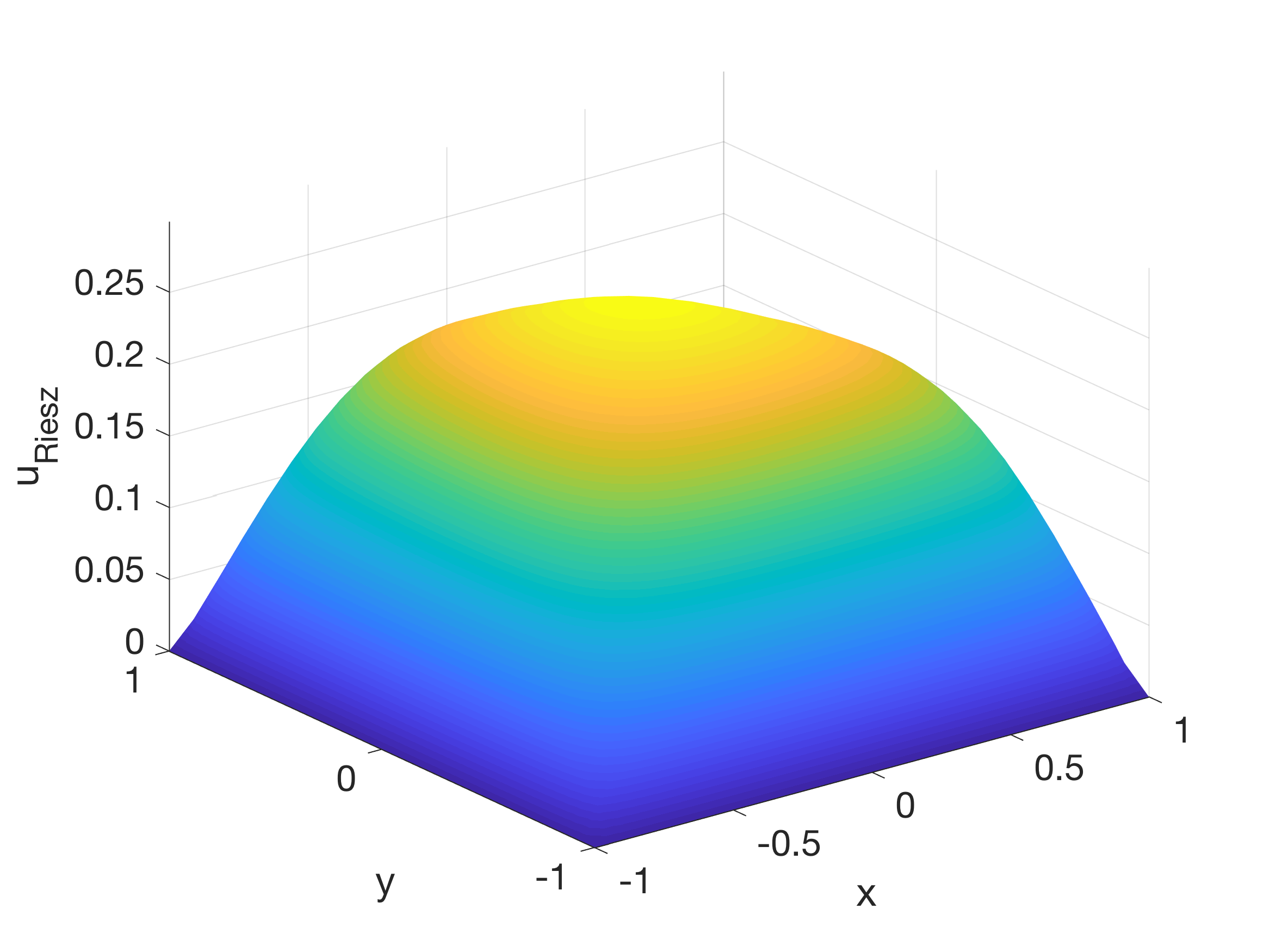}
  \includegraphics[width=0.25\textwidth]{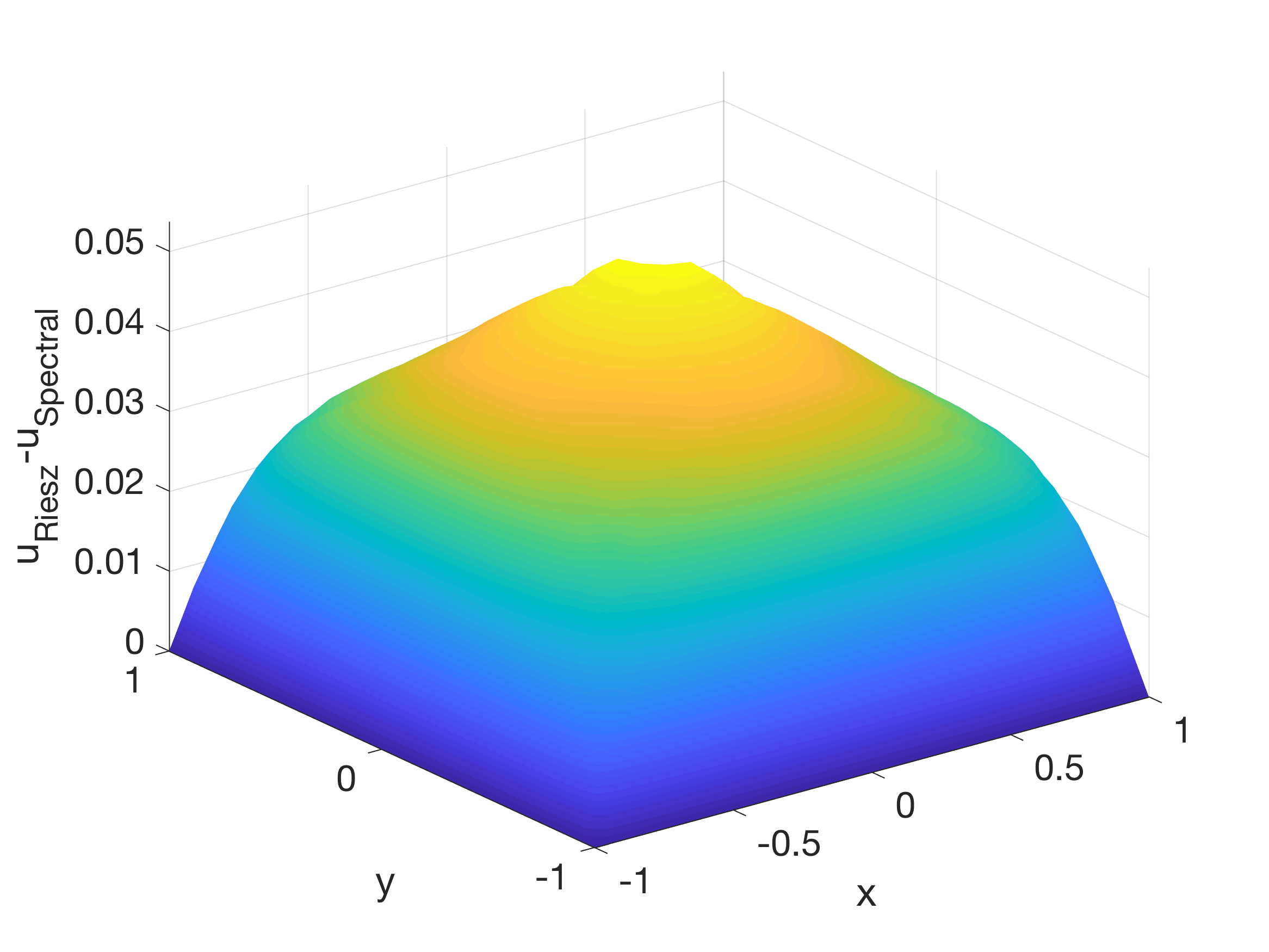}
\end{minipage}
 }\\
 \subfloat[Solutions $u$ associated with $f=1$ and $\alpha = 0.5$ in the L-shaped domain using the spectral definition (using SEM) (\emph{left}) and the Riesz definition (using AFEM) (\emph{center}), and the difference between $u_{\text{Riesz}}$ and $u_{\text{spectral}}$ for this case (\emph{right}).]{
 \begin{minipage}[]{\textwidth}\centering
 \includegraphics[width=0.25\textwidth]{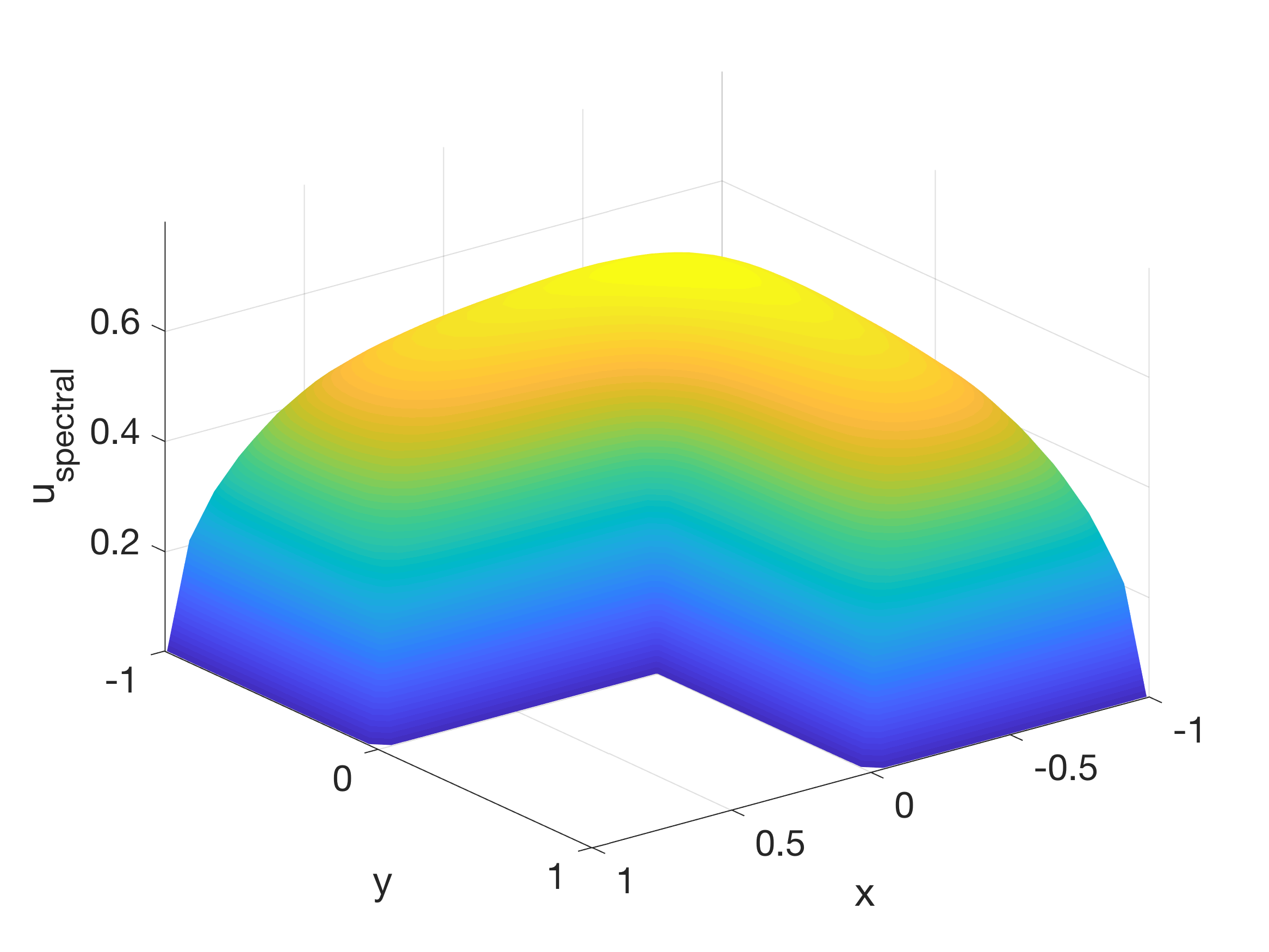}
  \includegraphics[width=0.25\textwidth]{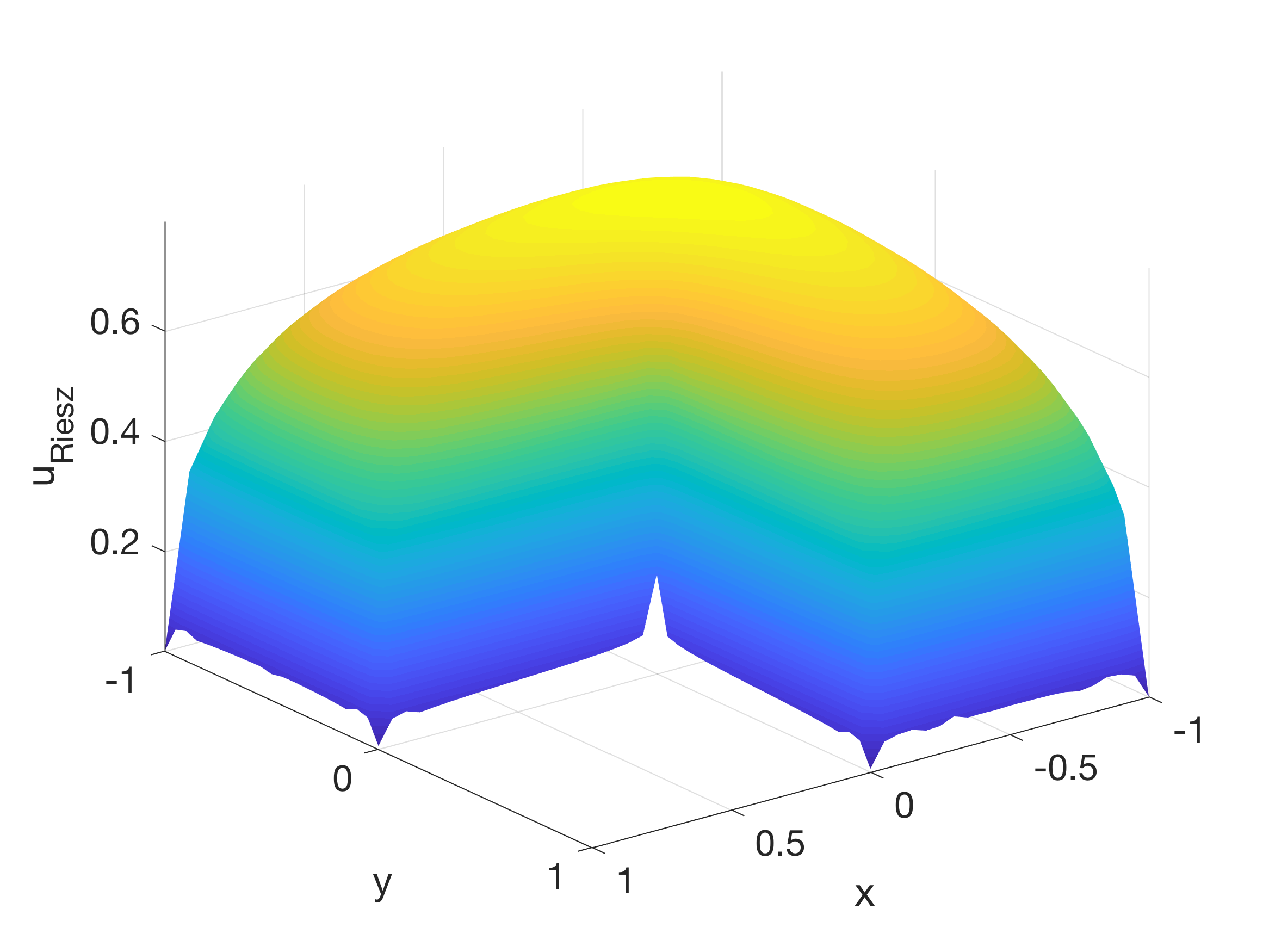}
  \includegraphics[width=0.25\textwidth]{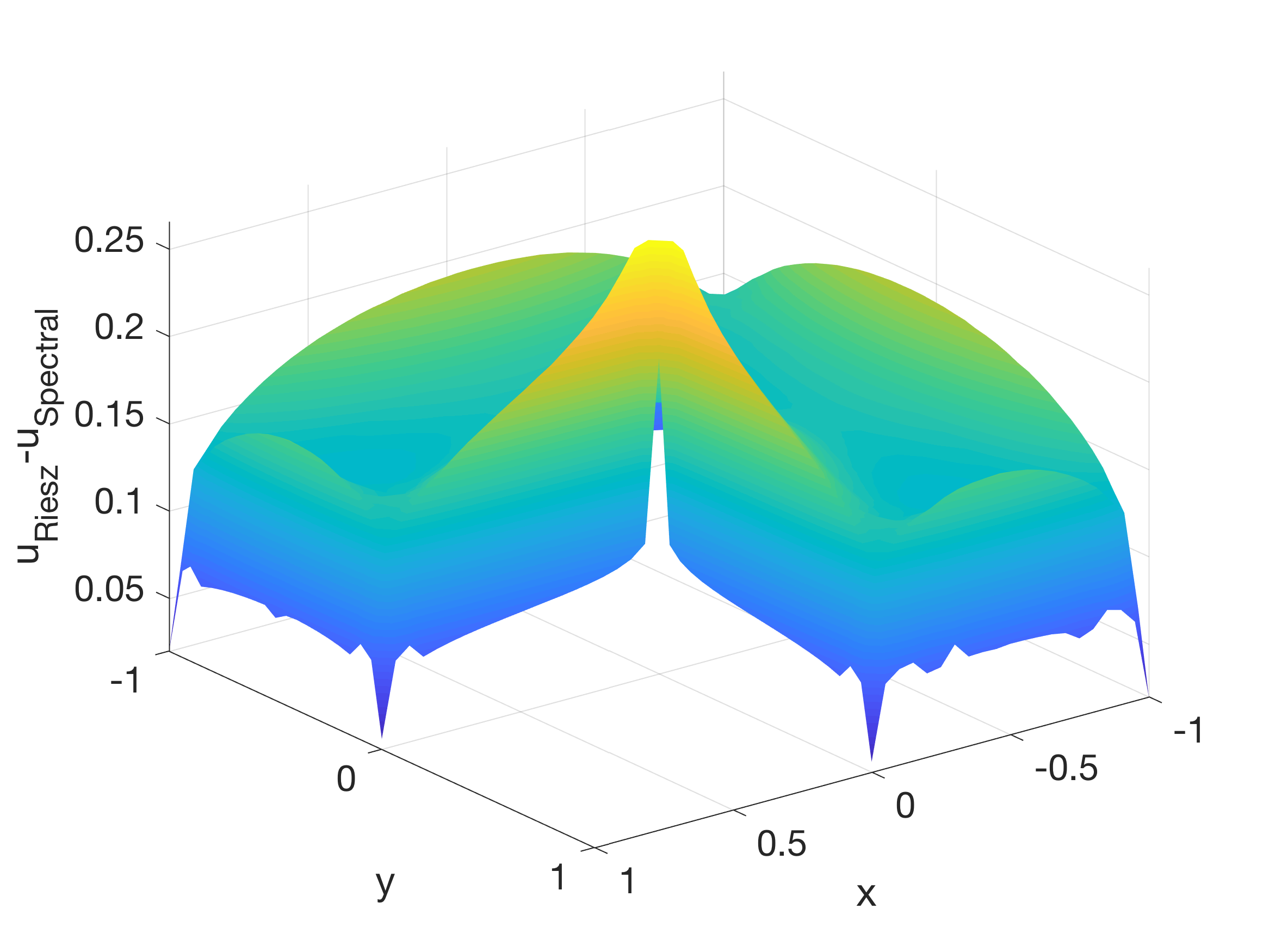}
\end{minipage}
 }\\
 \subfloat[Solutions $u$ associated with $f=1$ and $\alpha = 1.5$ in the L-shaped domain using the spectral definition (using SEM) (\emph{left}) and the Riesz definition (using AFEM) (\emph{center}), and the difference between $u_{\text{Riesz}}$ and $u_{\text{spectral}}$ for this case (\emph{right}).]{
 \begin{minipage}[]{\textwidth}\centering
 \includegraphics[width=0.25\textwidth]{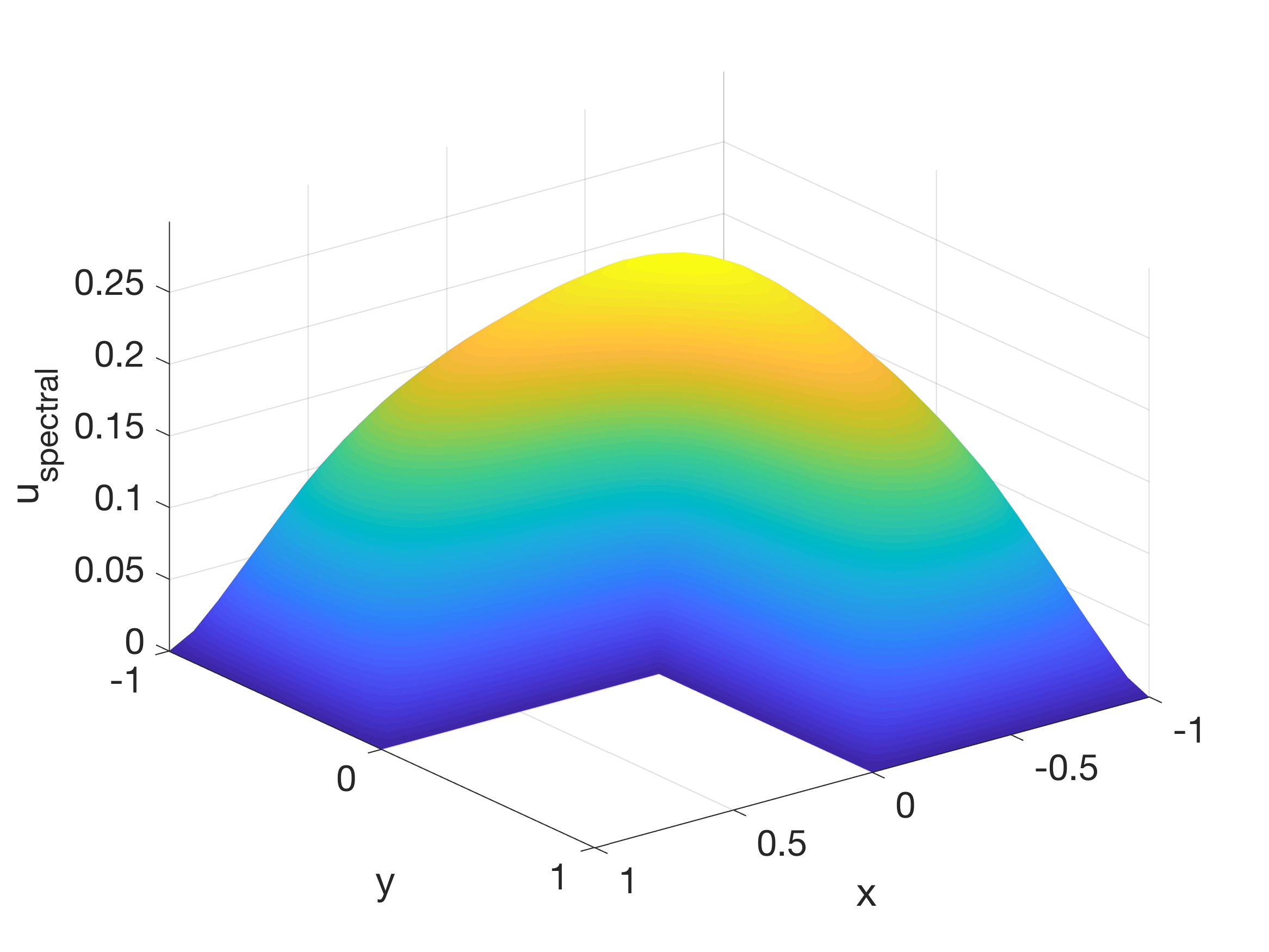}
  \includegraphics[width=0.25\textwidth]{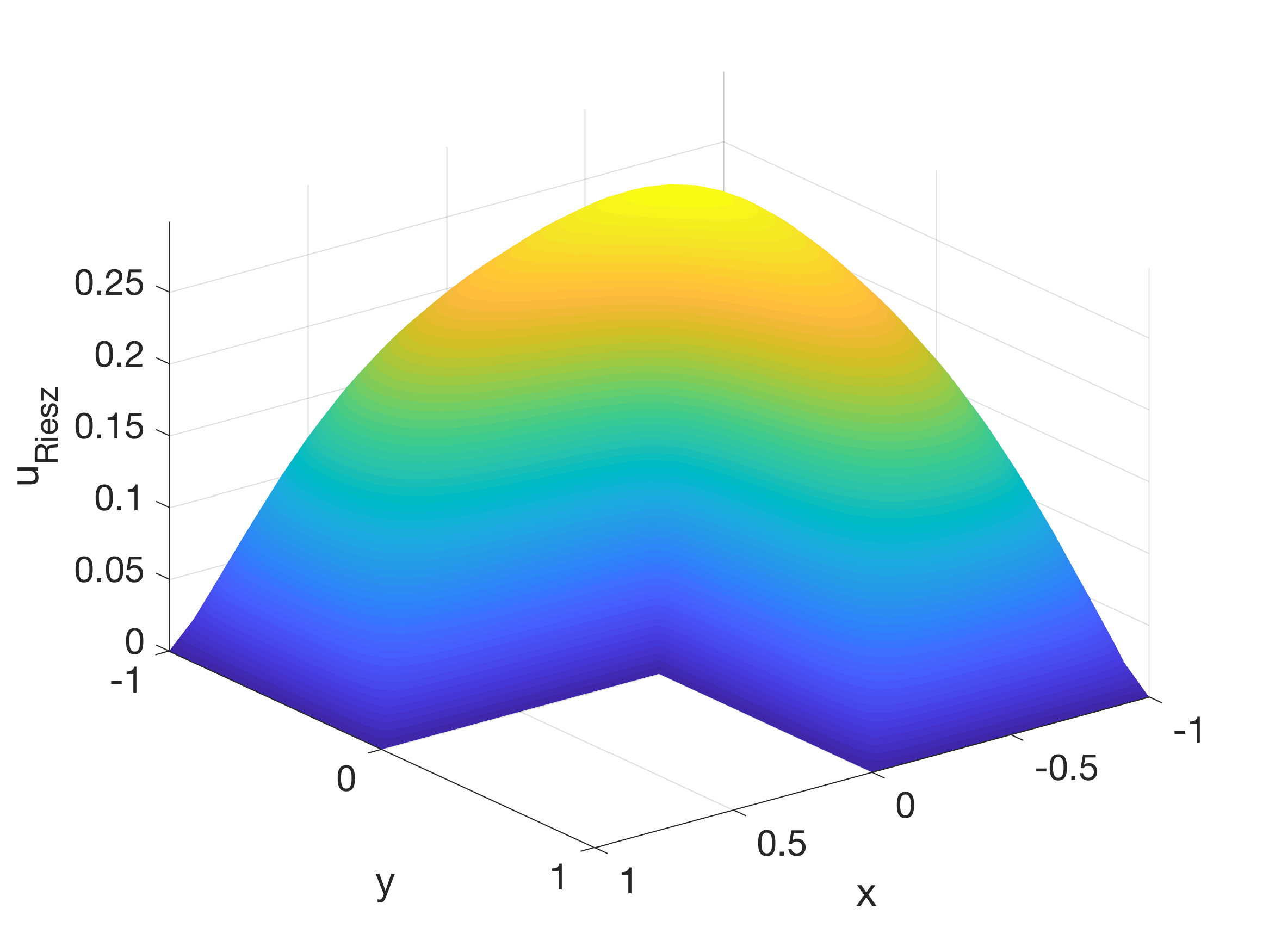}
  \includegraphics[width=0.25\textwidth]{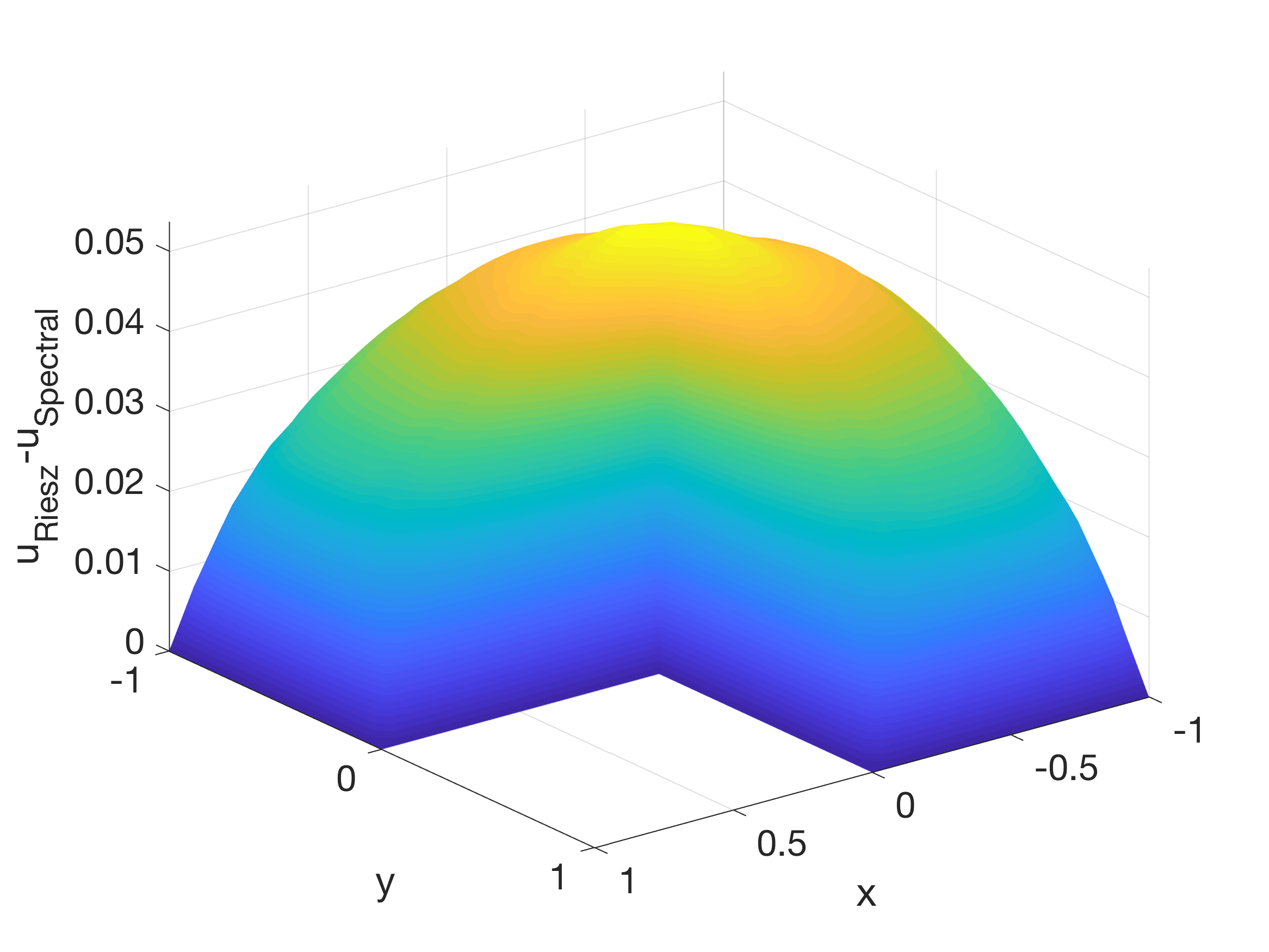}
\end{minipage}
 }
 \caption{\label{Lshape12} {Solutions and differences} between $u_{\text{Riesz}}$ and $u_{\text{spectral}}$ on the L-shaped domain for $\alpha = 0.5$ {\color{forest} and $1.5$.}}
 \end{figure}
 
   \begin{figure}[ht!]
 \centering
\subfloat[Solutions $u$ associated with $f=\sin(\pi x)\sin(\pi y)$ and $\alpha = 0.5$ in the L-shaped domain using the spectral definition (using SEM) (\emph{left}) and the Riesz definition (using AFEM) (\emph{center}), and the difference between $u_{\text{Riesz}}$ and $u_{\text{spectral}}$ for this case (\emph{right}).]{
 \begin{minipage}[]{\textwidth}\centering
 \includegraphics[width=0.3\textwidth]{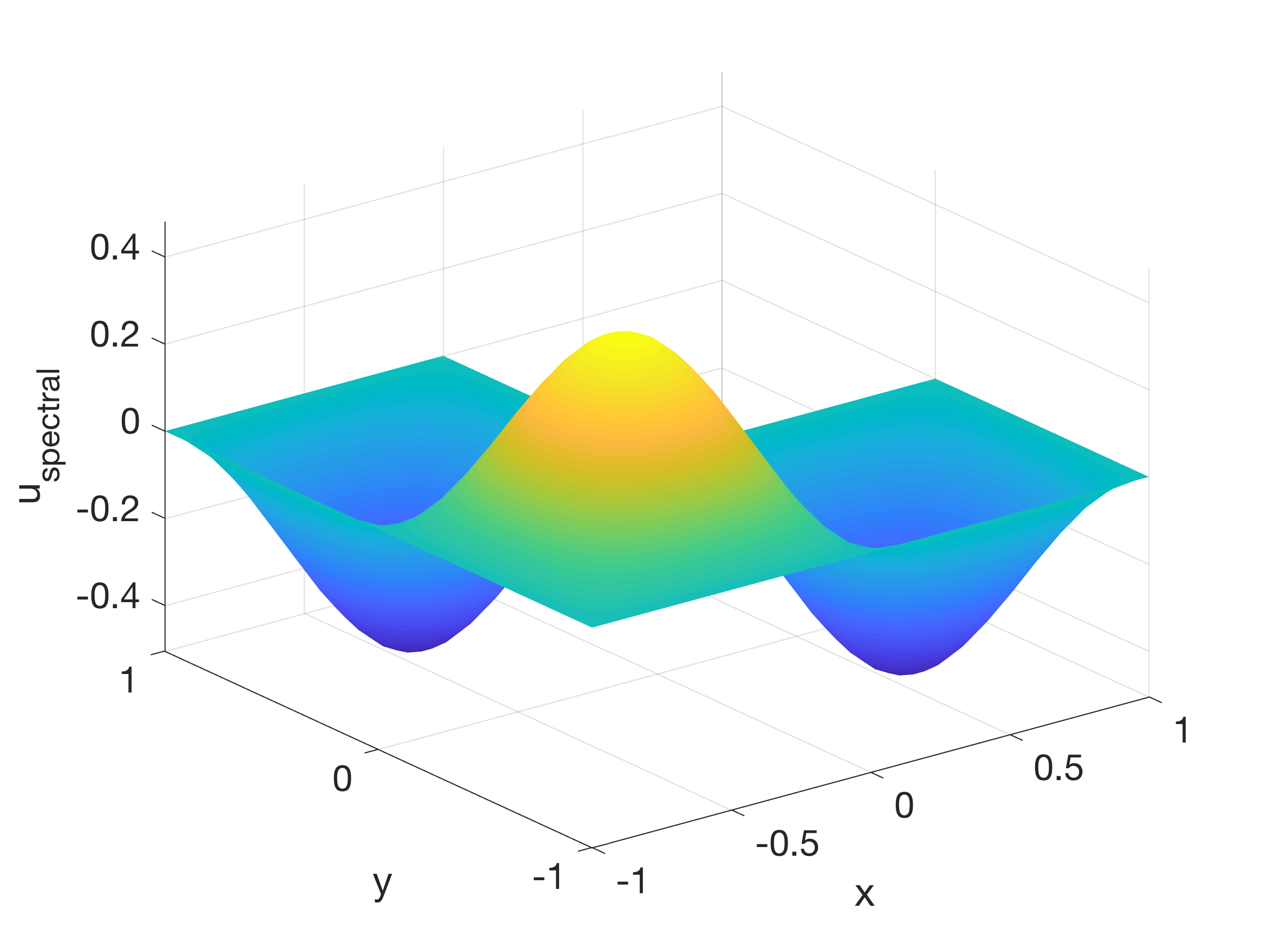}
  \includegraphics[width=0.3\textwidth]{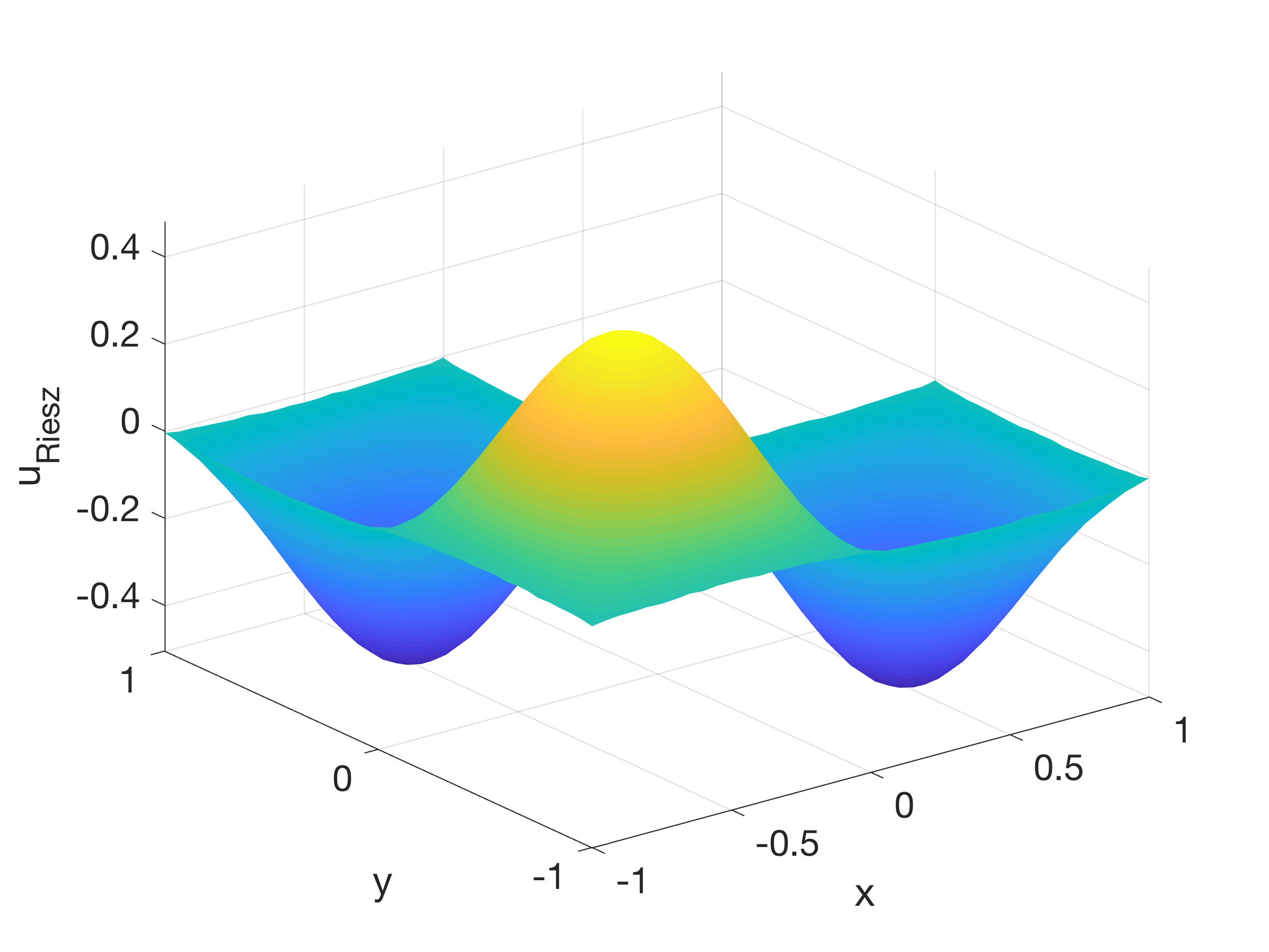}
  \includegraphics[width=0.3\textwidth]{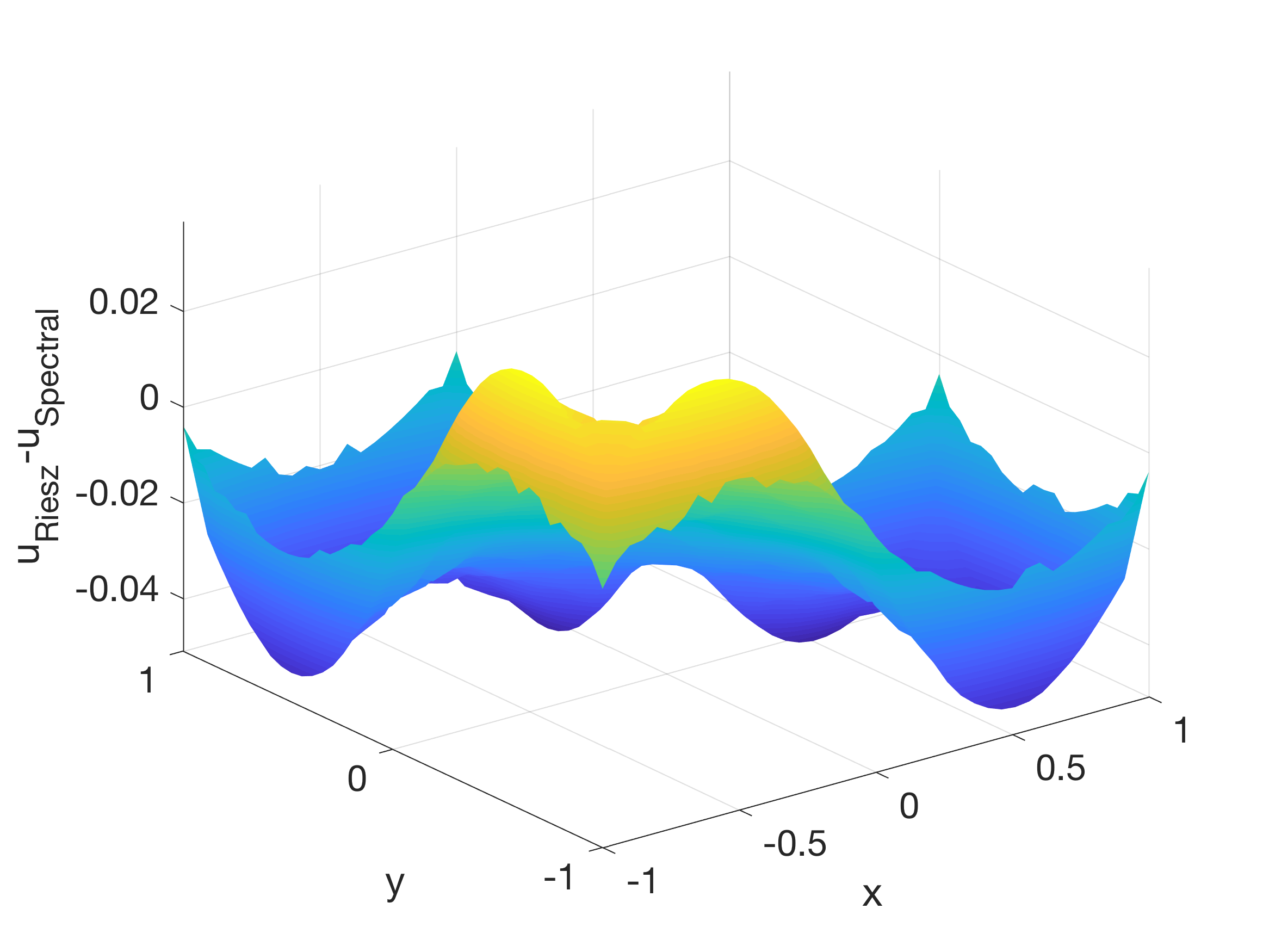}
\end{minipage}
 }\\
 \subfloat[Solutions $u$ associated with $f=\sin(\pi x)\sin(\pi y)$ and $\alpha = 1.5$ in the L-shaped domain using the spectral definition (using SEM) (\emph{left}) and the Riesz definition (using AFEM) (\emph{center}), and the difference between $u_{\text{Riesz}}$ and $u_{\text{spectral}}$ for this case (\emph{right}).]{
 \begin{minipage}[]{\textwidth}\centering
 \includegraphics[width=0.3\textwidth]{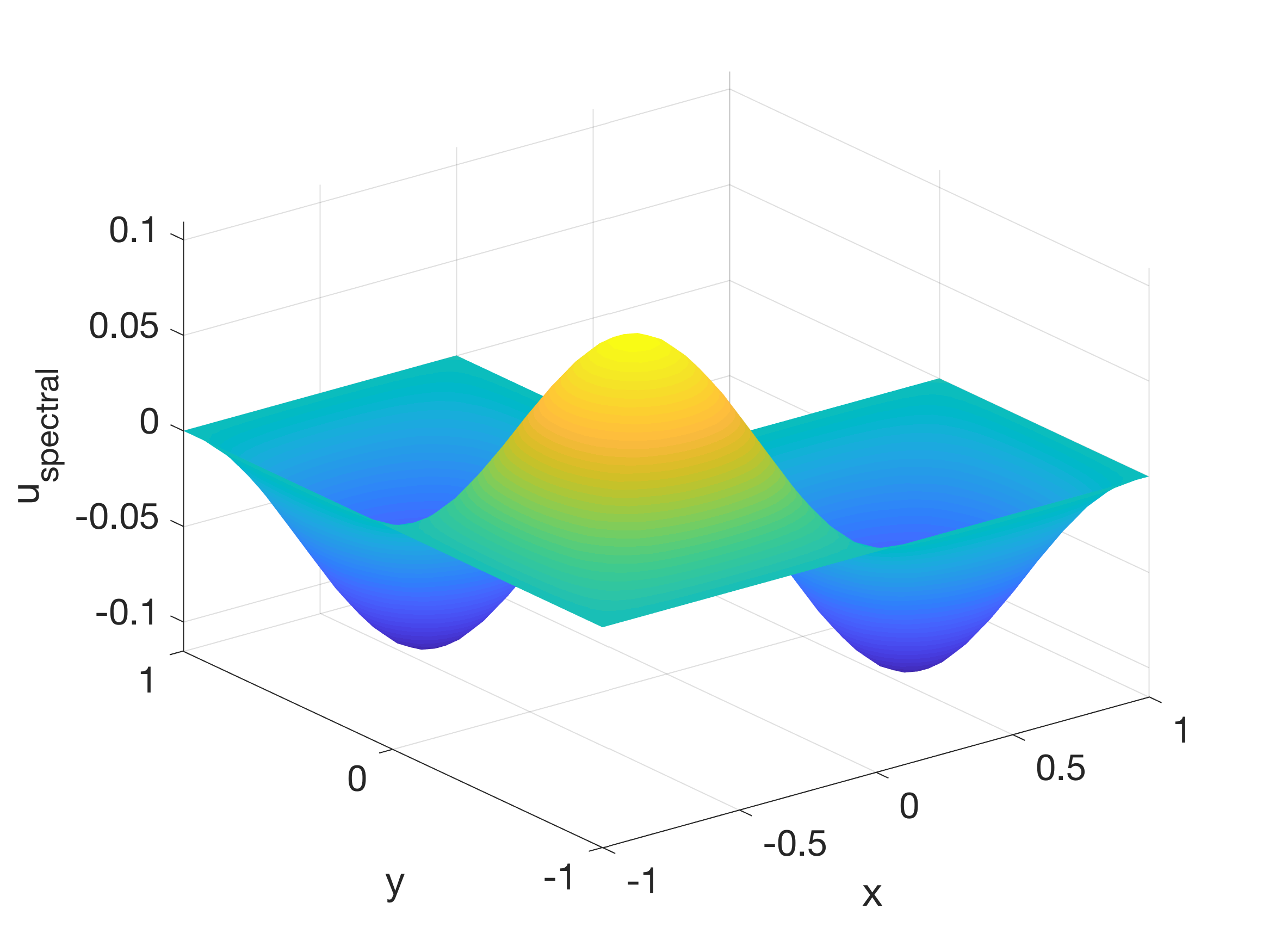}
  \includegraphics[width=0.3\textwidth]{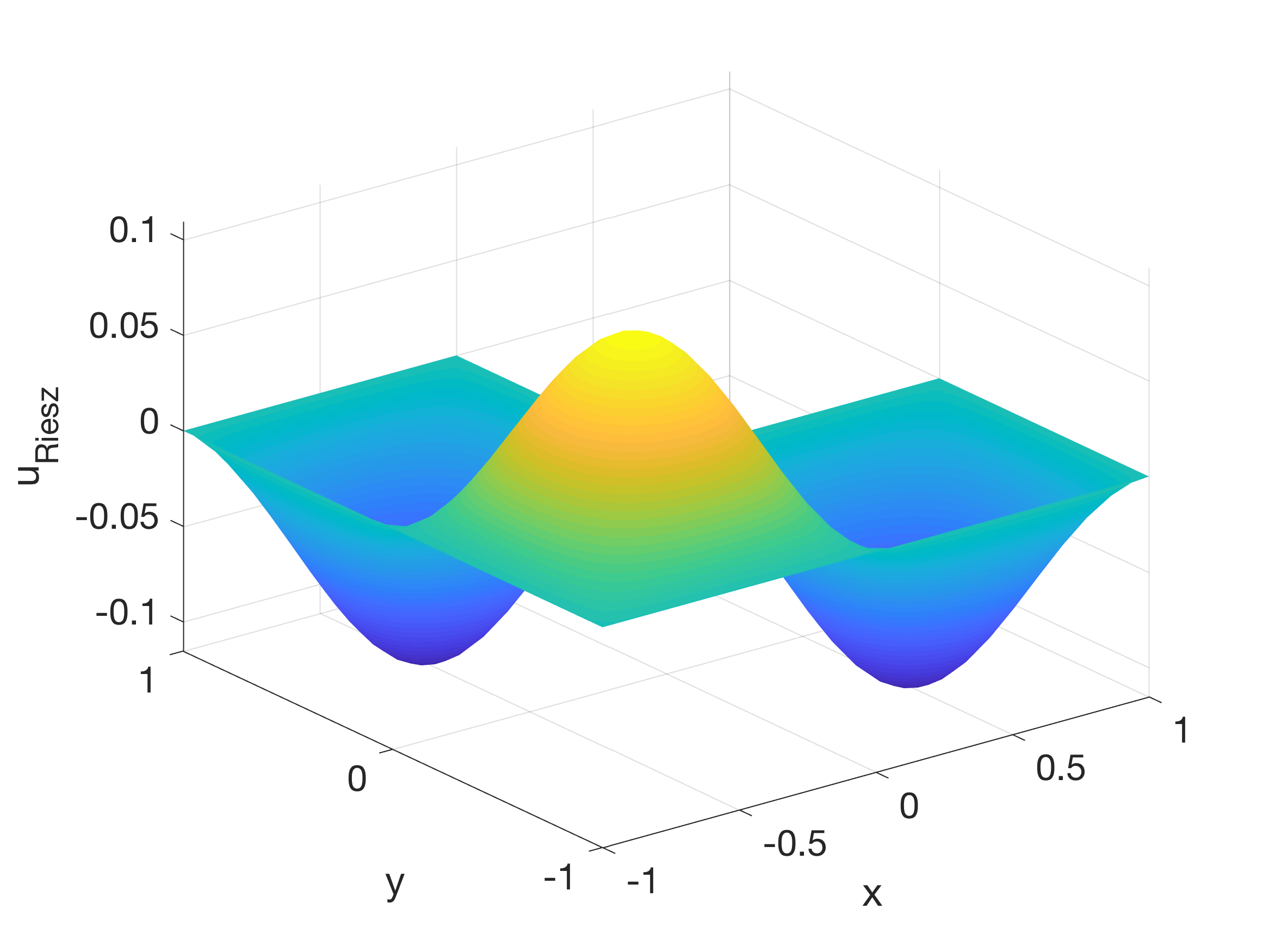}
  \includegraphics[width=0.3\textwidth]{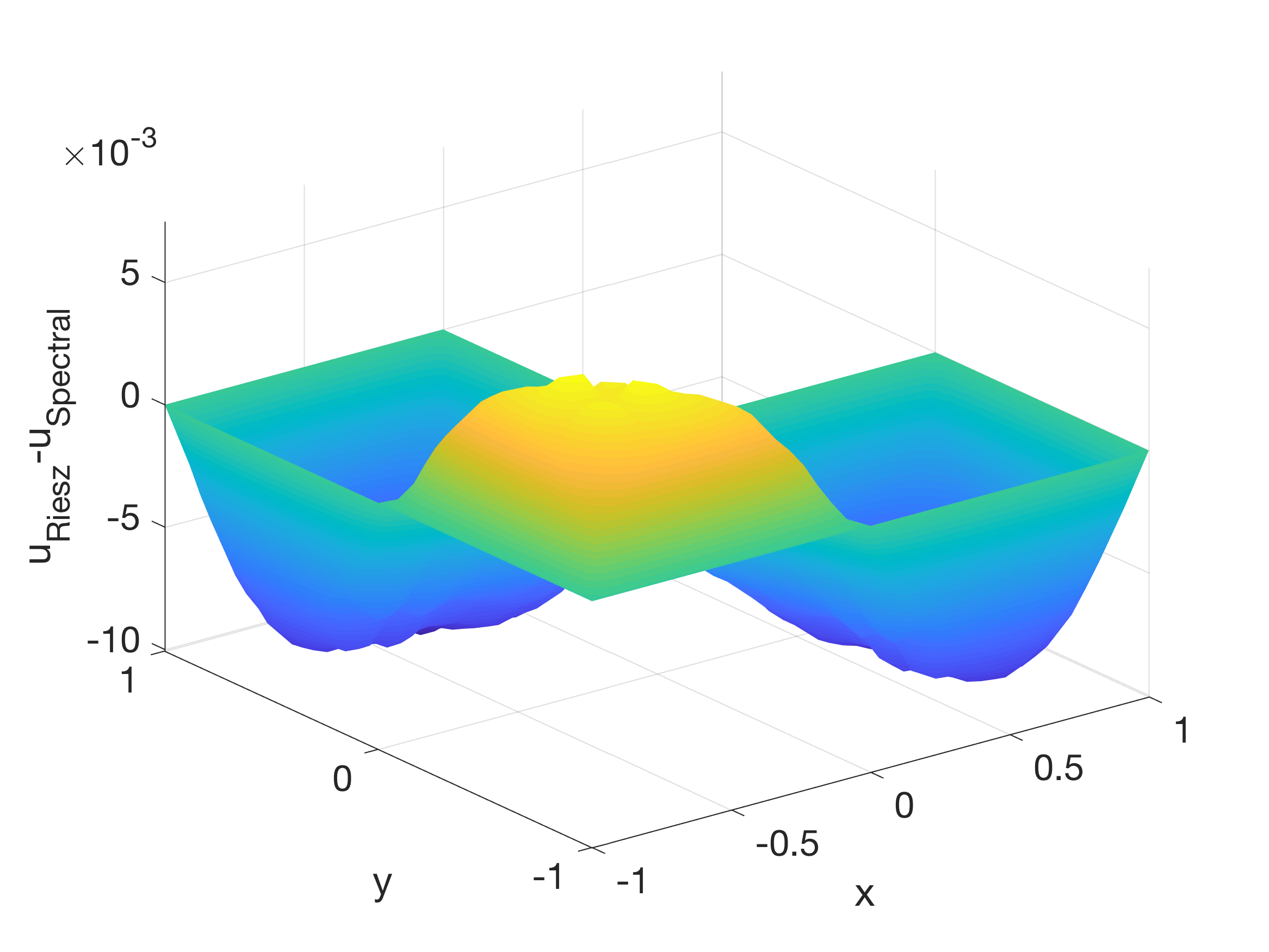}
\end{minipage}
 }\\
\subfloat[Solutions $u$ associated with $f=\sin(\pi x)\sin(\pi y)$ and $\alpha = 0.5$ in the L-shaped domain using the spectral definition (using SEM) (\emph{left}) and the Riesz definition (using AFEM) (\emph{center}), and the difference between $u_{\text{Riesz}}$ and $u_{\text{spectral}}$ for this case (\emph{right}).]{
 \begin{minipage}[]{\textwidth}\centering
 \includegraphics[width=0.3\textwidth]{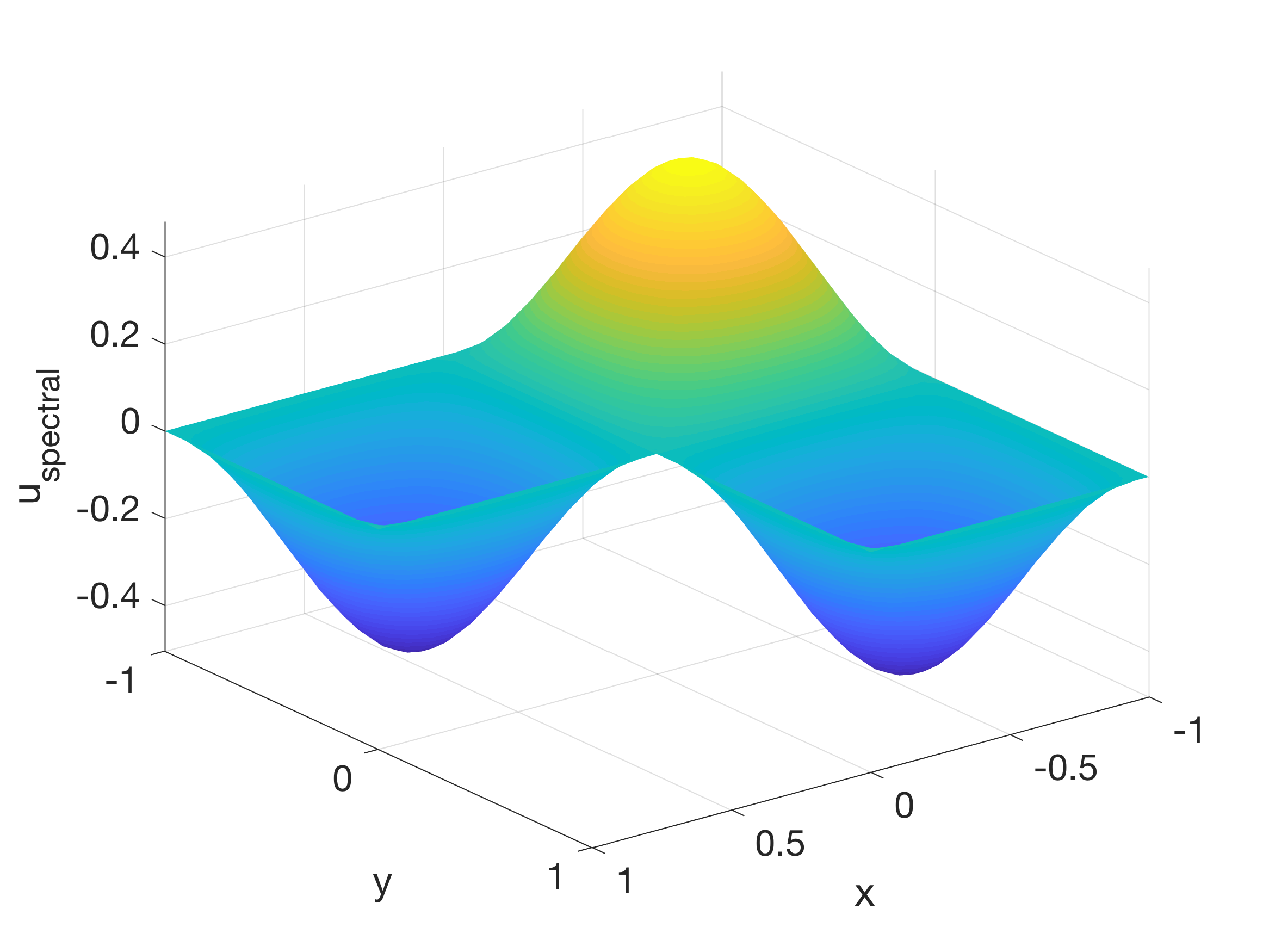}
  \includegraphics[width=0.3\textwidth]{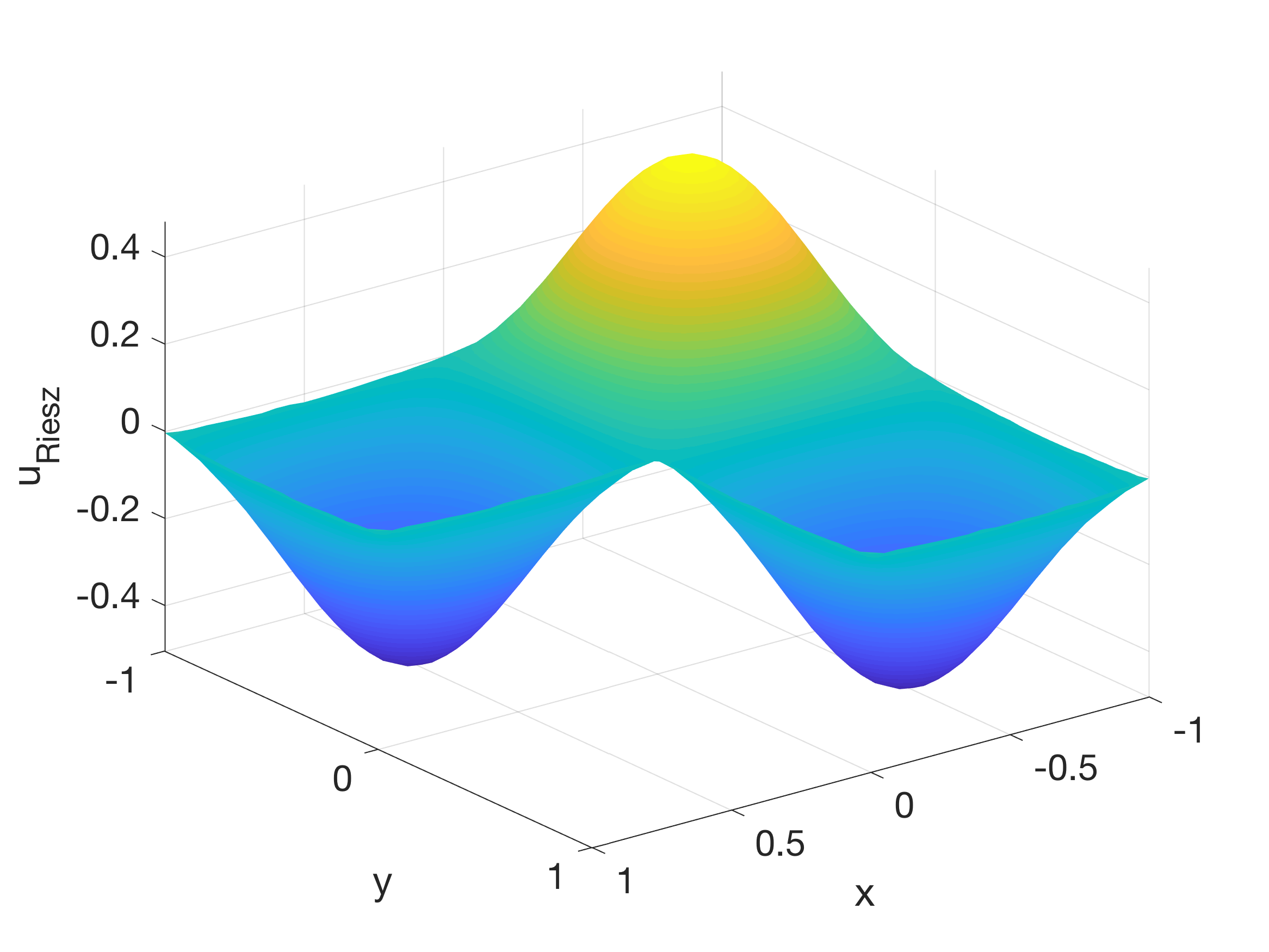}
  \includegraphics[width=0.3\textwidth]{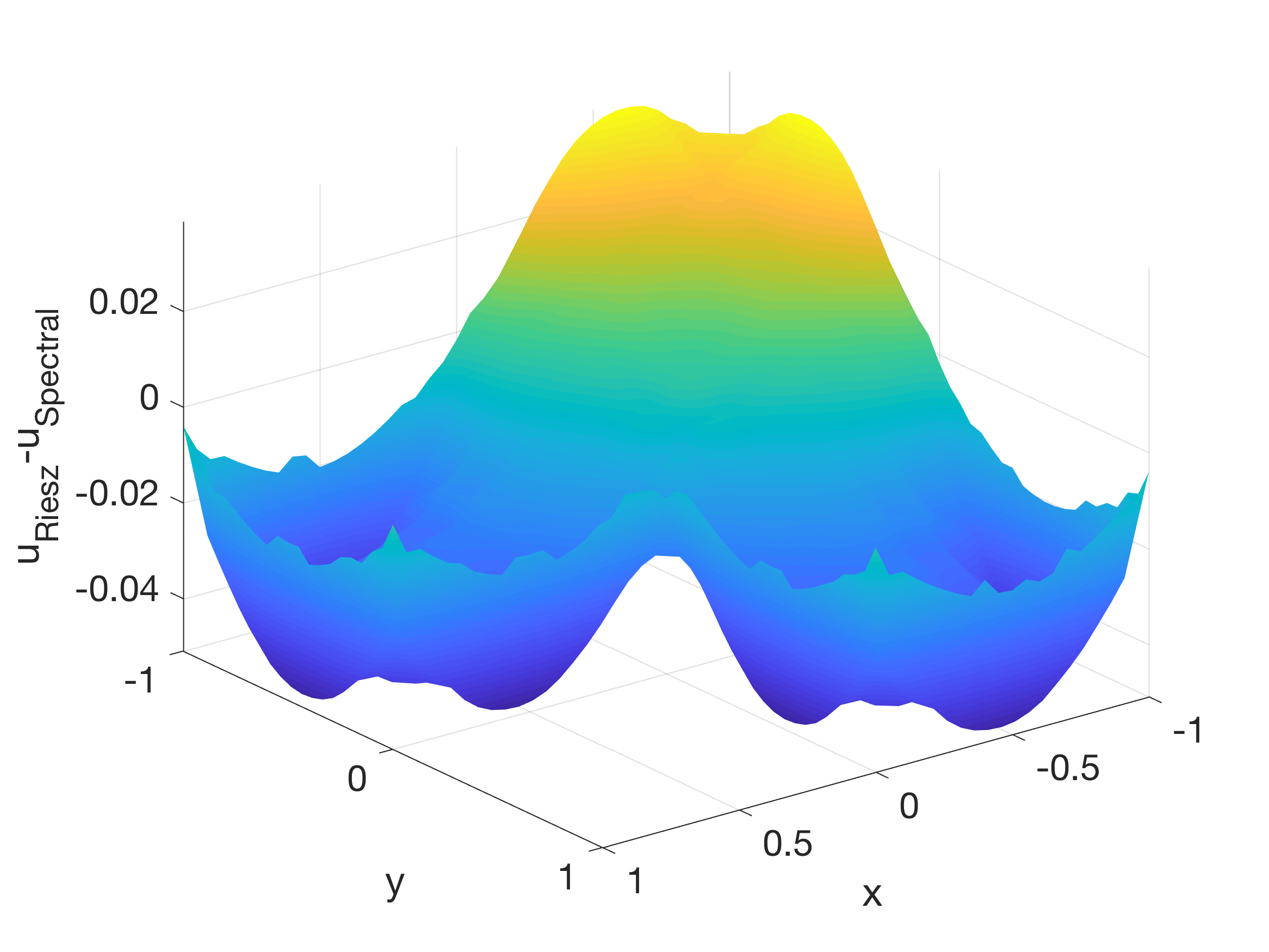}
\end{minipage}
 }\\
 \subfloat[Solutions $u$ associated with $f=\sin(\pi x)\sin(\pi y)$ and $\alpha = 1.5$ in the L-shaped domain using the spectral definition (using SEM) (\emph{left}) and the Riesz definition (using AFEM) (\emph{center}), and the difference between $u_{\text{Riesz}}$ and $u_{\text{spectral}}$ for this case (\emph{right}).]{
 \begin{minipage}[]{\textwidth}\centering
 \includegraphics[width=0.3\textwidth]{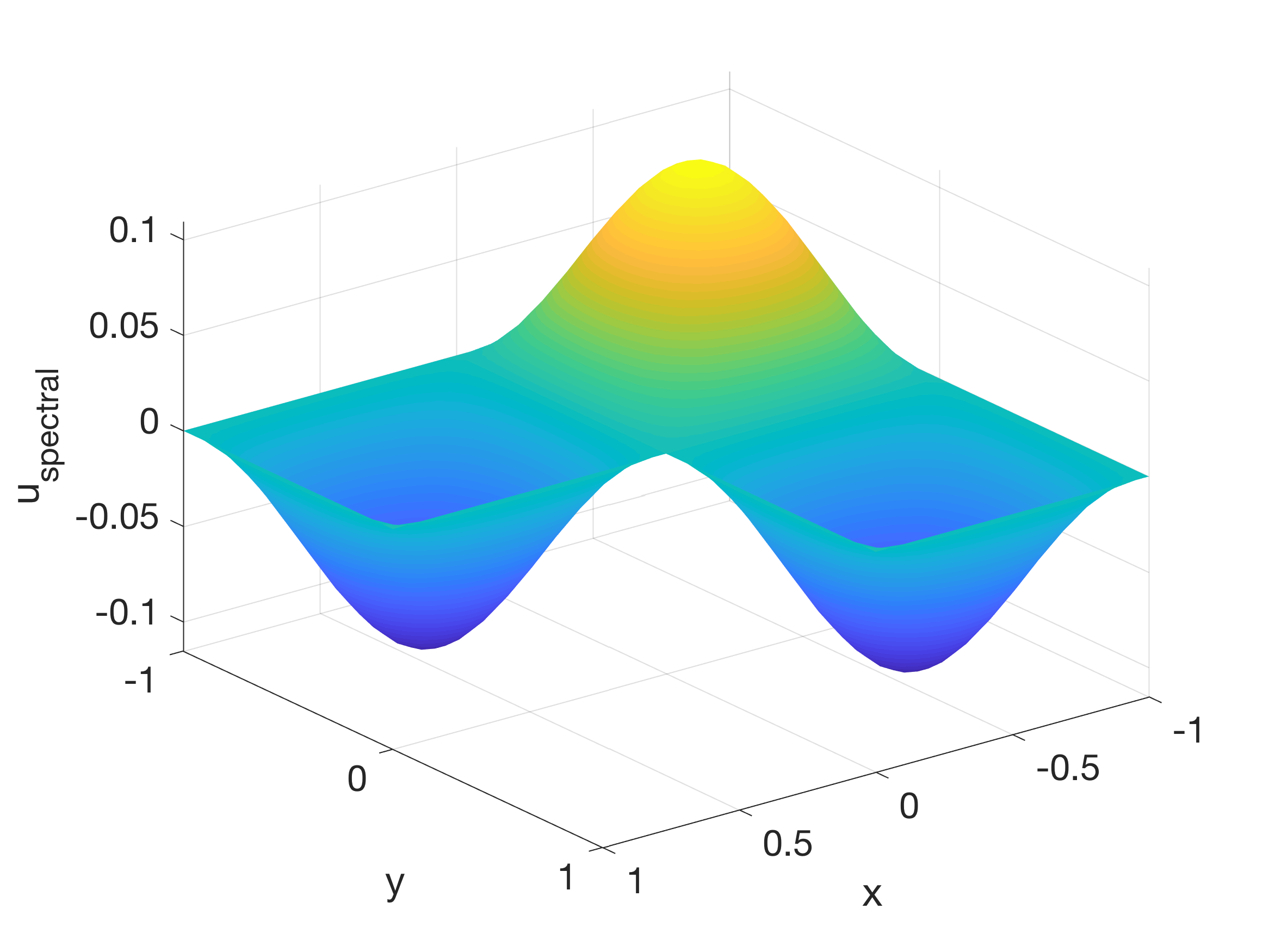}
  \includegraphics[width=0.3\textwidth]{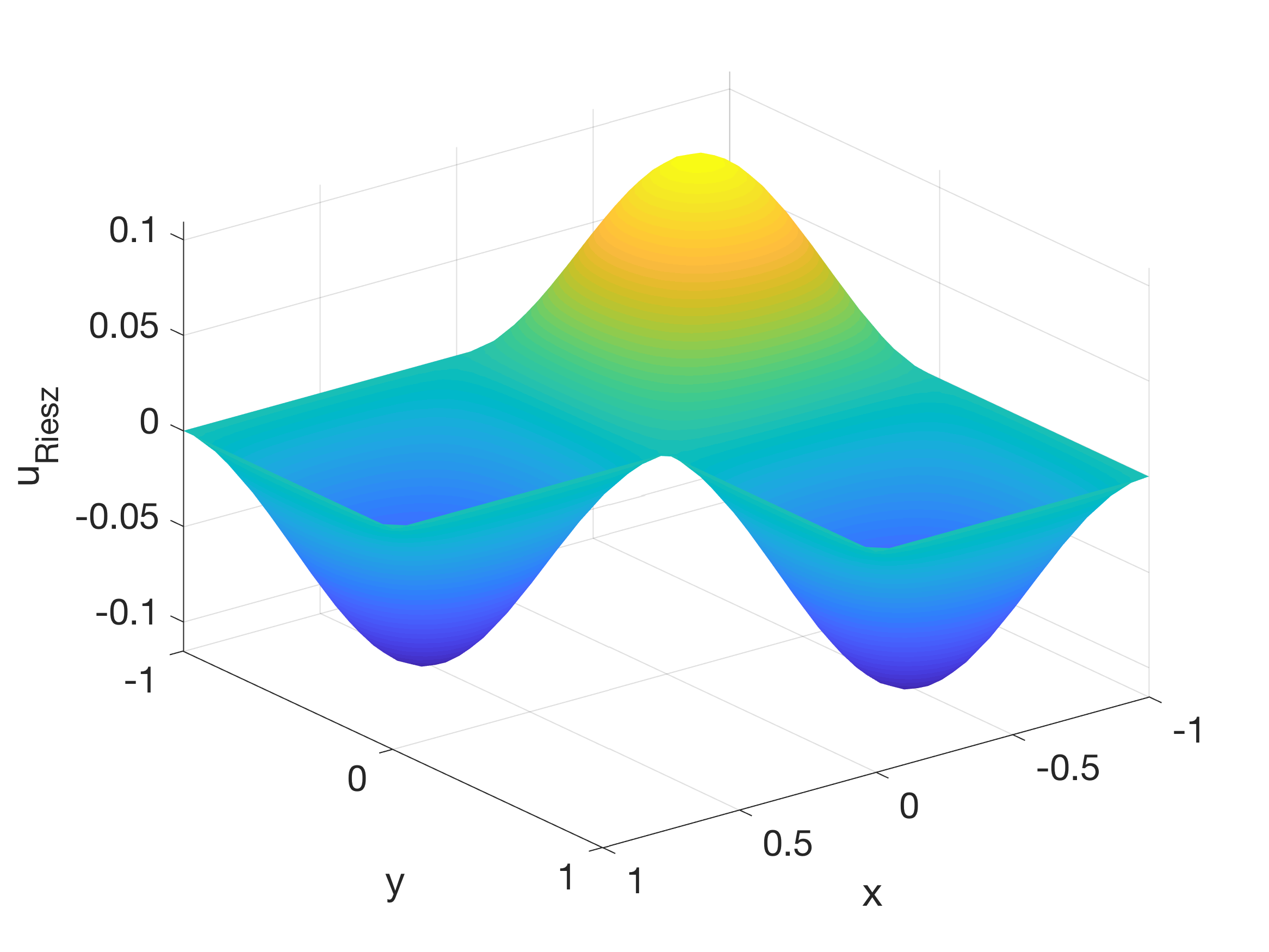}
  \includegraphics[width=0.3\textwidth]{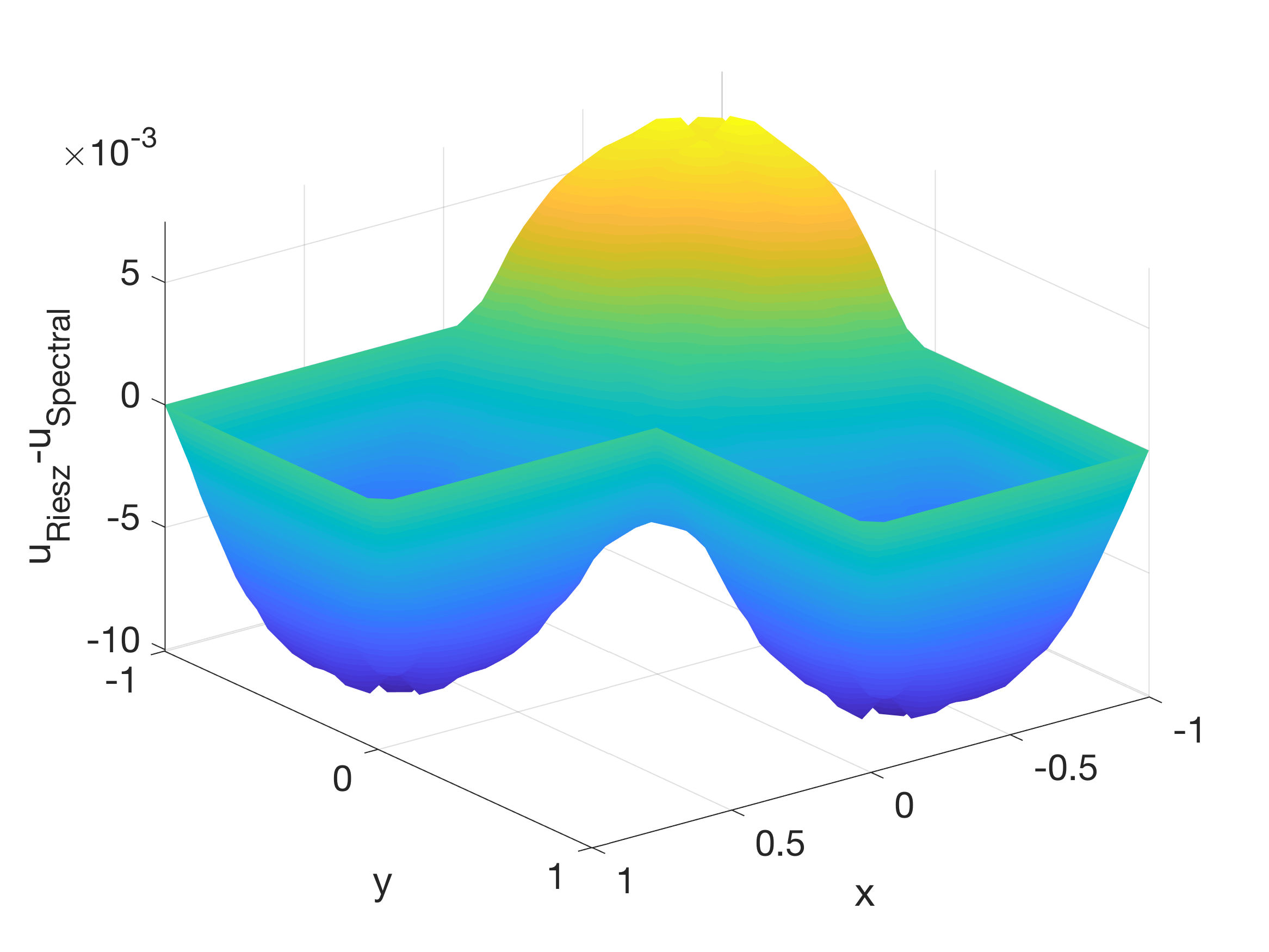}
\end{minipage}
 }
 \caption{\label{Lshape34} {Solutions and differences} between $u_{\text{Riesz}}$ and $u_{\text{spectral}}$ on the L-shaped domain for $\alpha = 0.5$ {\color{forest}and $1.5$.}}
 \end{figure}

    \begin{figure}[ht!]
 \centering
\includegraphics[height=.2\textheight]{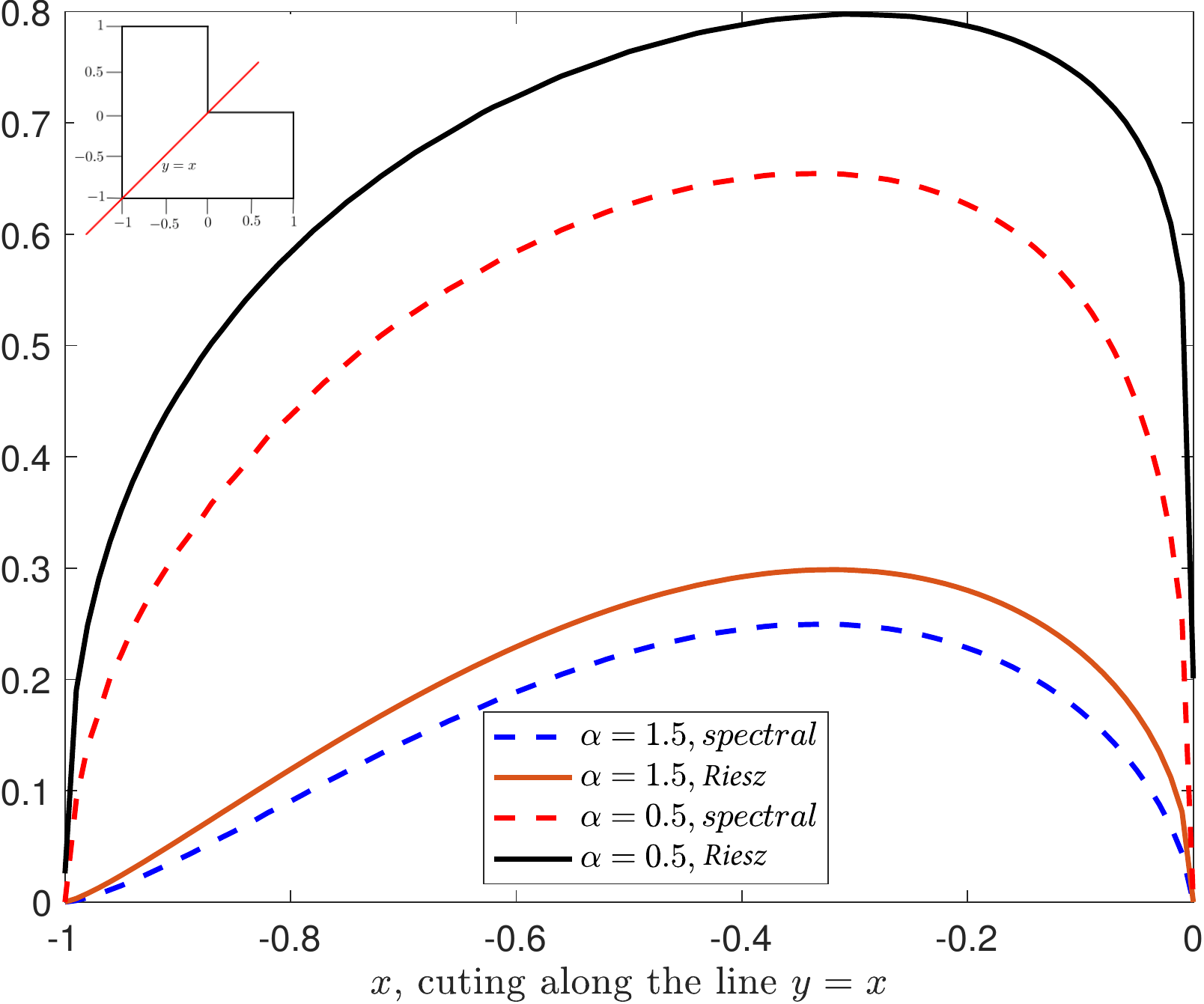}
\includegraphics[height=.2\textheight]{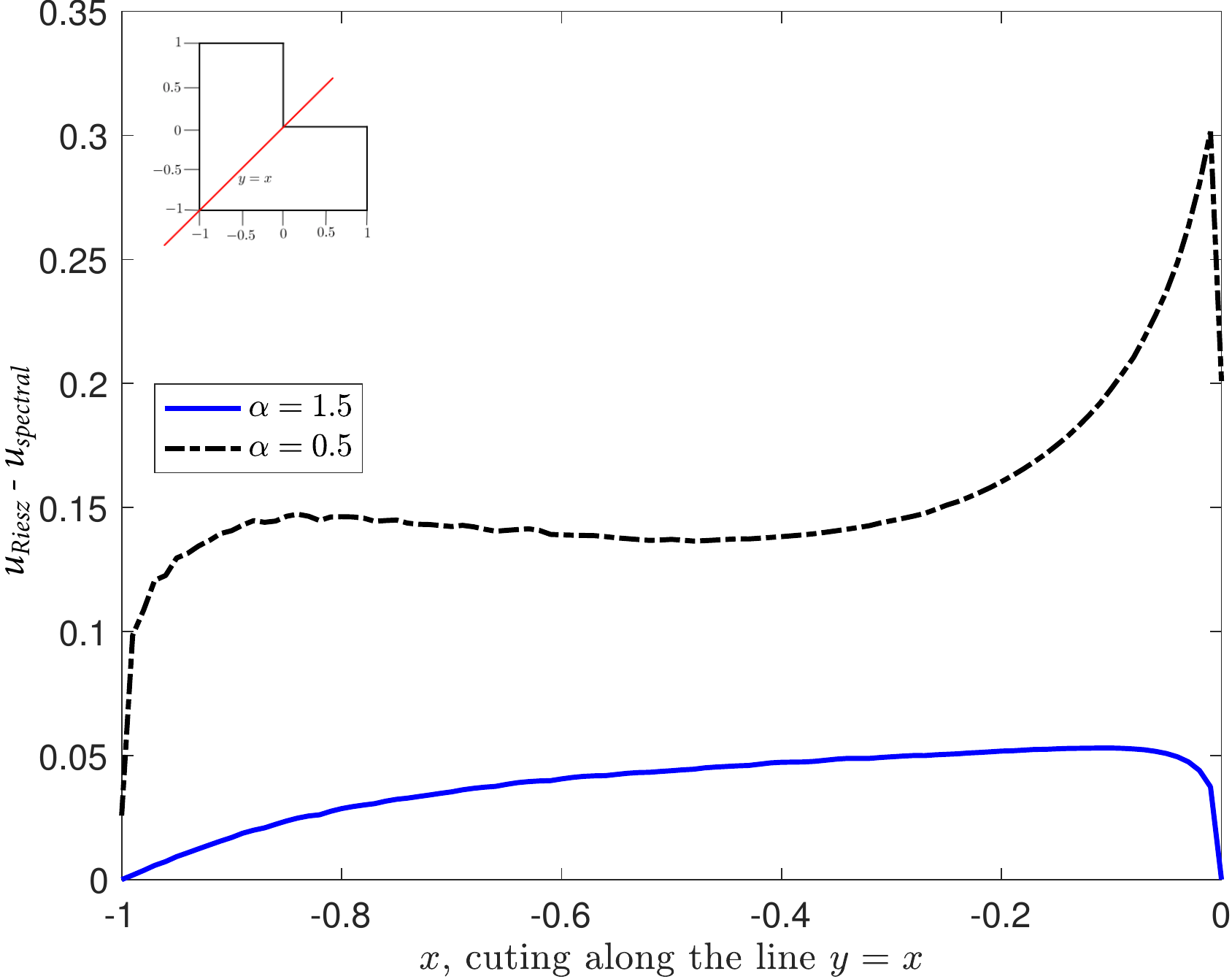}\\
\includegraphics[height=.2\textheight]{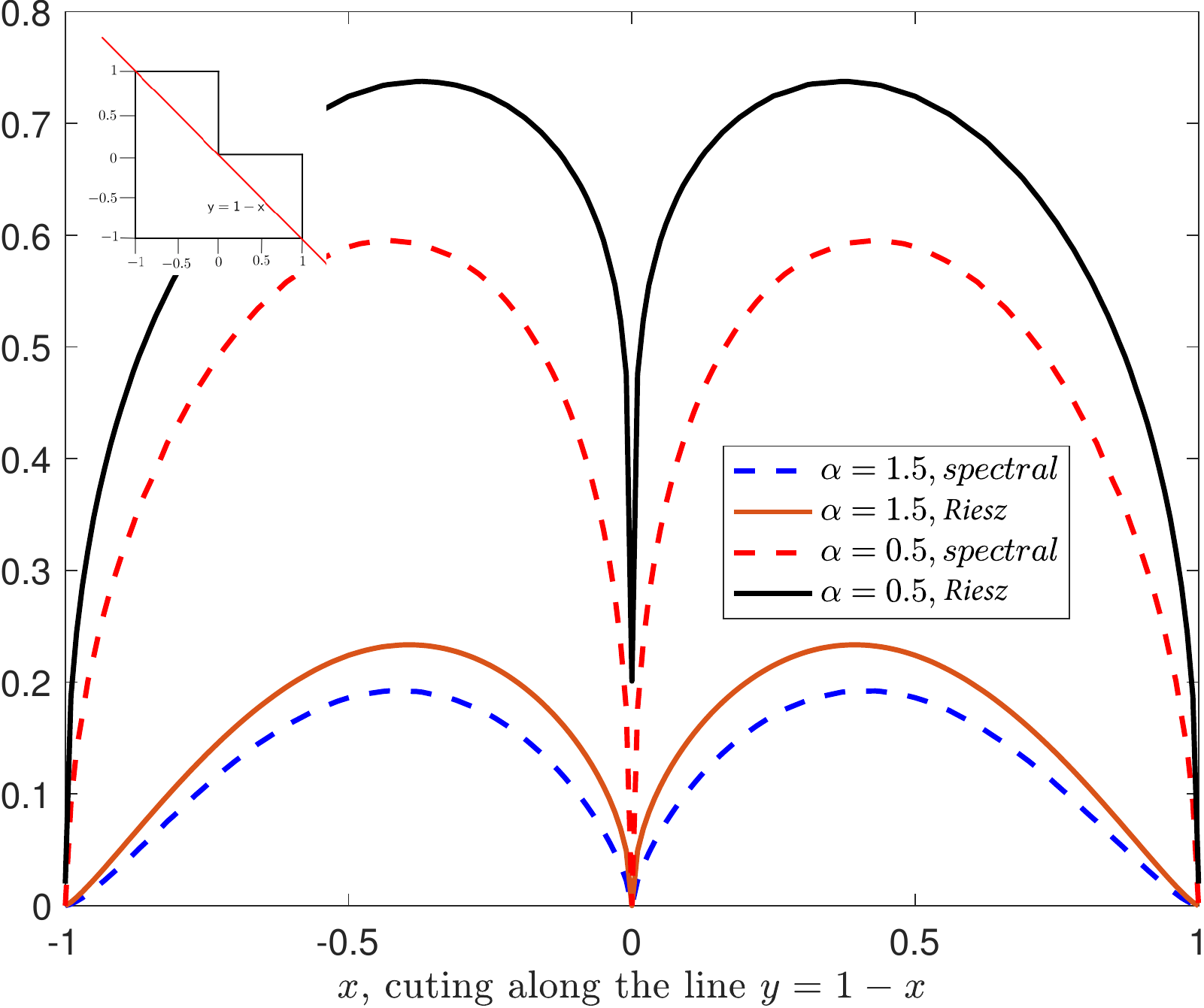}
\includegraphics[height=.2\textheight]{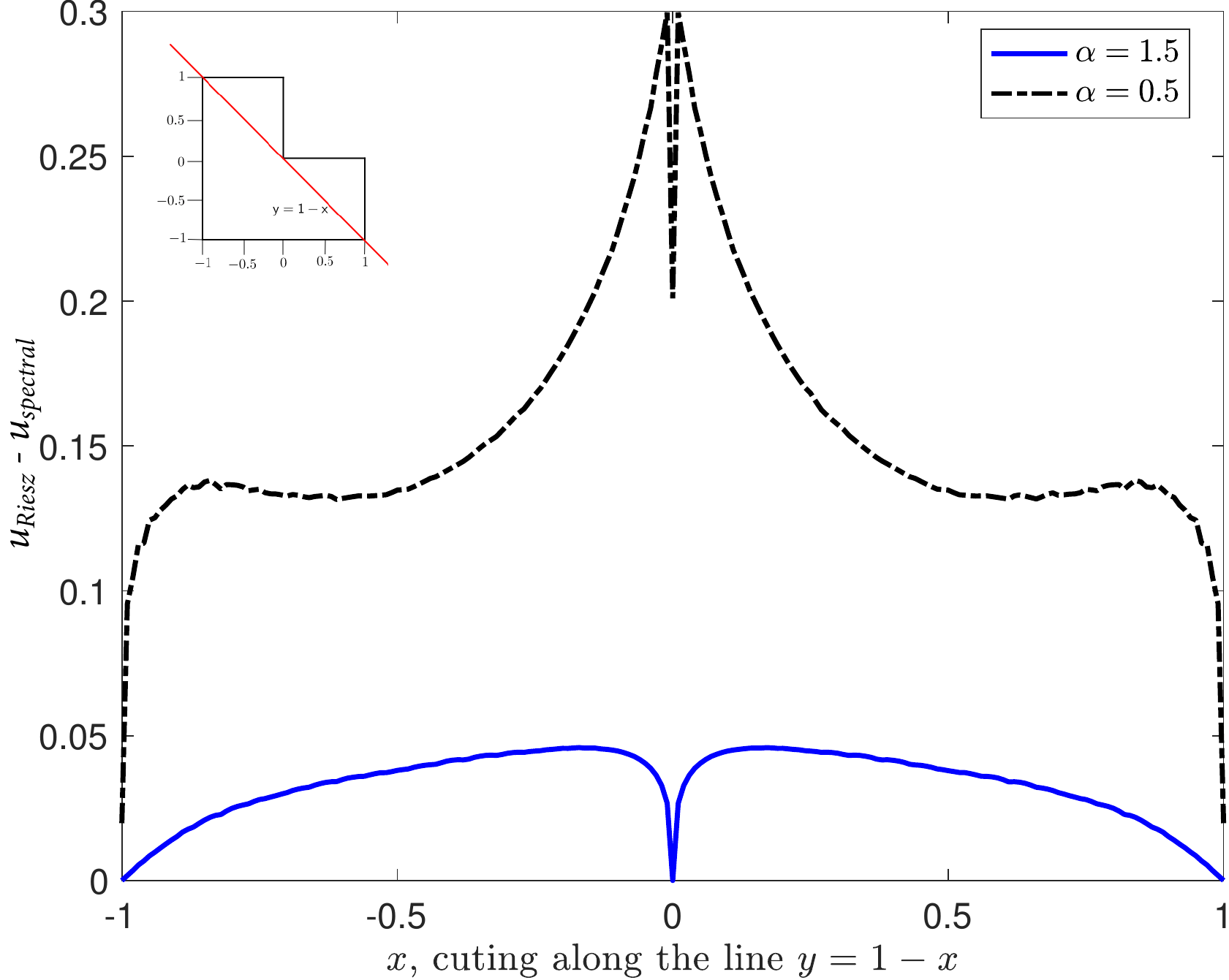}
\caption{\label{Lslicesf1yx} {Slices along the line $y=x$ and $y = 1-x$} in the L-shaped domain with $f = 1$ and both $\alpha = 0.5$ and $\alpha = 1.5$. (\emph{top left}) Plots of the solutions for the Riesz and spectral formulations along $y = x$ and (\emph{top right}) plot of the differences $u_{\text{Riesz}}-u_{\text{spectral}}$ along the line $y = x$. (\emph{bottom left}) Plots of the solutions for the Riesz and spectral formulations along $y = 1-x$ and (\emph{bottom right}) plot of the differences $u_{\text{Riesz}}-u_{\text{spectral}}$ along the line $y = 1-x$.}
\end{figure}

\FloatBarrier

\subsection{Computational Considerations}

Many of the numerical methods used earlier in this work have only recently been reported, and some require further development. Therefore, it is not reasonable to directly compare the computational complexities of these methods, since they are all at differing stages of development. Instead, in this section, we present an overview of the computational advantages and disadvantages of each method.

The AFEM of Ainsworth and Glusa \cite{AinsworthGlusa2017_TowardsEfficientFiniteElement,AinsworthGlusa2017_AspectsAdaptiveFiniteElement} for solving Riesz (or integral) fractional Poisson equations has been well-developed, and its overall computational complexity is $\mcO(n (\log n)^{2d})$ operations, including the assembly, computation of error indicators for the adaptive refinement, and solving the linear system. A detailed study of the complexity was included in \cite{AinsworthGlusa2017_AspectsAdaptiveFiniteElement}. At this time, the AFEM has been applied only to zero Dirichlet boundary conditions, although it may be possible to extend this method to nonzero Dirichlet conditions using a similar clustering approach as was mentioned in Section 3.1.1. As can be seen in the adaptively refined meshes shown in Appendix \ref{grids}, the AFEM captures the boundary singularity with a high degree of accuracy, and is a robust method for discretizing the Riesz fractional Laplacian.

Another approach we used to discretize the Riesz definition was the WOS algorithm (see Section \ref{sec:wos}), which can be applied to quasi-convex domains and zero or nonzero Dirichlet boundary conditions. In this method, the computational complexity is dominated by the Monte-Carlo simulation of many sample paths for each point in the bounded domain $\Omega$, although the complexity of our implementation scaled as the product of the number of sample paths multiplied by the number of points in the computational domain. Furthermore, the WOS computation of the solution at each point in $\Omega$ is independent of the computations for any other point, which means all of the walks-on-spheres can be done simultaneously. Therefore, WOS is an embarrassingly parallel algorithm. The WOS algorithm also has the advantage that no modification needs to be made when the Dirichlet boundary condition is nonzero other than inserting the appropriate function $g(x)$ into the Feynman Kac formula \eqref{FK-formula}. In \cite{kyprianou2016unbiasedwalk}, the convergence of the WOS algorithm is proved, and sample code for solving the fractional Poisson equation is provided.

To discretize the directional definition \eqref{def:directional_uniform}, we used the RBF collocation method of Section \ref{RBFM}. {\color{chocolate} An advantage of this approach is that it can be easily modified to handle the case where the measure $M(d\bm\theta)$ is non-uniform, as was done using another version of this method in \cite{pang2015space}. The RBF collocation method is applicable to problems with nonzero Dirichlet boundary conditions as long as the exterior boundary value decays at infinity. In addition to these advantages, there are some aspects in which further development is needed.} When collocation points are too close together, the method becomes unstable, which limits the level of accuracy this method {\color{blue} can achieve \cite{chen2014recent}}. Furthermore, for direct and iterative solvers, the linear system resulting from the RBF approach requires cubic and squared complexity to solve, respectively. {\color{chocolate} Currently, we} lack good preconditioners due to the complicated matrix structure formed by the Gr\"unwald-Letnikov scheme as well as Gaussian quadrature for the integral with respect to $\theta$.

For the spectral definition, there are many choices of numerical methods. In the comparisons above, we reported results computed using the SEM described in Section \ref{SEM}. Given a mesh of elements on a bounded domain $\Omega$, this method requires the computation of a large number of eigenvalues and eigenfunctions of the standard Laplacian. This is computationally expensive, and for larger eigenvalues, the computation of these eigenpairs becomes inaccurate. Furthermore, the resulting linear system is of cubic complexity. The advantage of this approach is its flexibility with respect to different types of boundary conditions and complex domains. Furthermore, if the SEM is applied to a time-dependent problem, e.g., the space-fractional heat equation, the eigenpairs need only be computed once at the beginning, and can be reused during each time step. To understand why this is the case, recall that the discretized form of the fractional Poisson equation is
\begin{align}
	A U_{\mcN} &= \hat{f},
\end{align}
where $\hat{f}$ is the load vector, $A_{\mcN}$ is the discrete counterpart of the spectral fractional Laplacian:
\begin{align}
	(A_{ij}) &= \sum_{n=1}^\mcN \lambda_n^{\alpha/2} (\tilde{\phi}_i,\tilde{\phi}_j),
\end{align}
and $\{\tilde{\phi}_i\}$ are the orthonormal Lagrange basis functions generated by the WGS procedure. Hence, $A$ is a diagonal matrix, so solving the linear system is simply a computation of the matrix-vector product $A^{-1} \hat{f}$. So this method is inexpensive in the context of time-dependent problems. For more details of the implementation and computational complexity for the SEM, see \cite{SongXuKarniadakis2017}.

\section{Nonzero Boundary Conditions} \label{sec:nonzerobcs}

\hrule
\vspace{1em}
\noindent \textbf{Section Overview}\\[0.2em]
\indent Depending on the fractional Laplacian definition, nonzero boundary conditions can lead to a significant increase in the computational cost of solving our benchmark equations. Most importantly, modifications to the definition itself may be necessary, particularly in the case of the spectral fractional Laplacian, as discussed in Section \ref{spectral}. We demonstrate how our numerical methods may be adapted for this case. We also present computationally feasible approaches to solving the Riesz fractional Poisson equation with nonzero Dirichlet boundary conditions, and we make comparisons between the solutions to some inhomogeneous benchmark problems. We note some qualitative differences in the solutions for the inhomogeneous case from our observations of the solutions to the homogeneous problems.

We are unable to compute solutions to the fractional Neumann problem involving the Riesz definition for two reasons. Firstly, because the formulation of the nonlocal Neumann condition is a subject of some controversy, as there are multiple formulations in the literature with no consensus \cite{Dipierro14,Song,Bogdan,GuanMa2006,barles_neumann_half_space}. Secondly, numerical approaches have not yet been developed for these fractional Neumann problems. It is more straightforward to define Neumann boundary conditions for the spectral fractional Laplacian, since this definition only requires local boundary conditions, as discussed above. However, the points we make in the following section can be understood fully using only nonzero Dirichlet boundary conditions.
\vspace{1em}
\hrule
\vspace{2em}

\subsection{Spectral Definition}
In this section, we follow up on the discussion of the different (equivalent) formulations of the inhomogeneous spectral fractional Laplacian in Section \ref{spectral} by presenting some numerical comparisons involving nonzero Dirichlet boundary conditions. We demonstrate numerically that the approach of Antil et al. \cite{AntilPfeffererRogovs}, the heat semigroup formulation of Cusimano et al. \cite{Cusimano2017}, and our nonharmonic lifting approach are all equivalent ways of describing the inhomogeneous spectral fractional Laplacian. As these formulations require distinct numerical approaches, we also compare these approaches in Section \ref{comparison_APR_HeatSG}.

\subsubsection{Numerical Comparison of \color{chocolate} Nonharmonic Lifting and 
{\color{chocolate} Harmonic Lifting Methods}}\label{lifting_weak_compare}
In Section \ref{sec:nonharmonic_lifting}, we described the nonharmonic lifting approach for discretizing the inhomogeneous spectral fractional Laplacian. In this section, we compare the results of {the \color{chocolate}nonharmonic} lifting method with the 
{\color{chocolate} harmonic lifting method of Antil, Pfefferer, and Rogovs\cite{AntilPfeffererRogovs}, described in Section \ref{sec:harmonic_lifting}. In our figures, we frequently refer to the latter as the ``APR'' method.}
{\color{chocolate} Recall that this method allows for lifting the boundary data by functions $v$ that solve $-\Delta v = 0$ in a \emph{very weak} variational form \eqref{variational_very_weak}; however, in all of the examples we consider, the boundary data is smooth enough that we can lift by standard harmonic functions. }

{
\color{chocolate}
For the harmonic lifting method, as explained in Section \ref{sec:inhomogeneous_spectral},  any method used to discretize the homogeneous spectral fractional Laplacian can be applied after the harmonic lifting. In our examples, we use the spectral element method described in Section \ref{SEM} in order to discretize $(-\Delta)^{\alpha/2}u(x)$. 
}
{\color{blue} {\color{chocolate} For the nonharmonic lifting method, we also apply} the spectral element method (or the discrete eigenfunction method in the case $d = 1$) of Section \ref{SEM}, where the numerical solution of the fractional Poisson problem is approximated by a function in the space $(\mathbb{P}_N \cap H_0^1)(\Omega)$. We consider the fractional Poisson equation \eqref{fracPoisson} with nonzero Dirichlet boundary $u(x)\big|_{\partial \Omega} = g(x)$. Let $u = w + v$, where $w$ is an unknown function and $w\big|_{\partial \Omega} = 0$, and $v$ is a  function chosen to satisfy $v \big|_{\partial \Omega}  = g$. Hence this function $v$ is not unique.  
Assuming that $\partial \Omega$ is smooth, $g$ must belong to the space $H^{1/2}(\partial \Omega)$, since the solution $u \in H^1_0(\Omega)$ and has trace equal to $g$. Then, following the approach of Section \ref{sec:nonharmonic_lifting}, the weak form of the inhomogeneous fractional Poisson problem can be written as follows: Find $w \in (\mathbb{P}_N \cap H_0^1)(\Omega)$, such that
\begin{align}
	((-\Delta)^{\alpha/2}w,\phi) &= (f,\phi) - \left(\nabla v, \nabla((-\Delta)^{\alpha/2-1}\phi)\right)
\end{align}
for all $\phi \in (\mathbb{P}_N \cap H_0^1)(\Omega)$.}

{\color{blue}We solve this problem in the domain $\Omega = (-1,1)$, where} we test two cases with different forcing terms, $f = x$ and $f = -x$ with boundary conditions $u(-1) = -1$ and $u(1) = +1$. We have exact solutions in the case $\alpha = 2$: $u = \frac{1}{6}x^3 + \frac{5}{6}x$ with the forcing term $f = -x$ (\emph{Case I}), and $u = -\frac{1}{6}x^3 + \frac{7}{6} x$ with the forcing term $f = x$ (\emph{Case II}). We choose the lifting functions $v = x$ and $v = x^3$ for our numerical tests. We plot the differences of the numerical solutions between the {\color{chocolate} nonharmonic lifting} method and the {\color{chocolate} harmonic lifting method} of \cite{AntilPfeffererRogovs} in Figure \ref{liftingdiff} for different values of $\alpha$ and different lifting functions $v(x)$.  
We can see from Figure \ref{liftingdiff} that the differences are on the order of machine precision, demonstrating that these approaches yield equivalent solutions.

\begin{figure}[ht!]
\centering
\includegraphics[width=.45\textwidth]{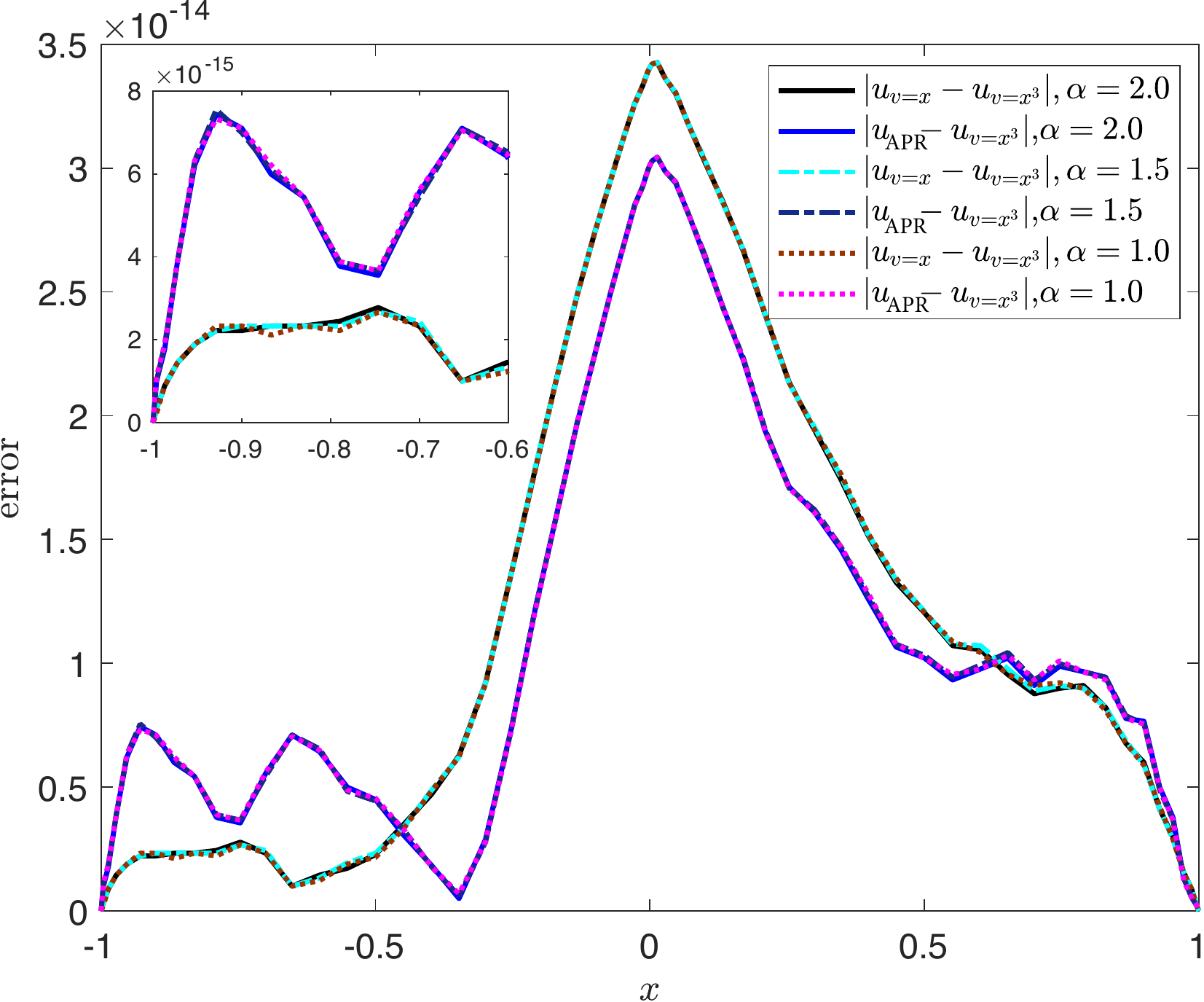}
\includegraphics[width=.45\textwidth]{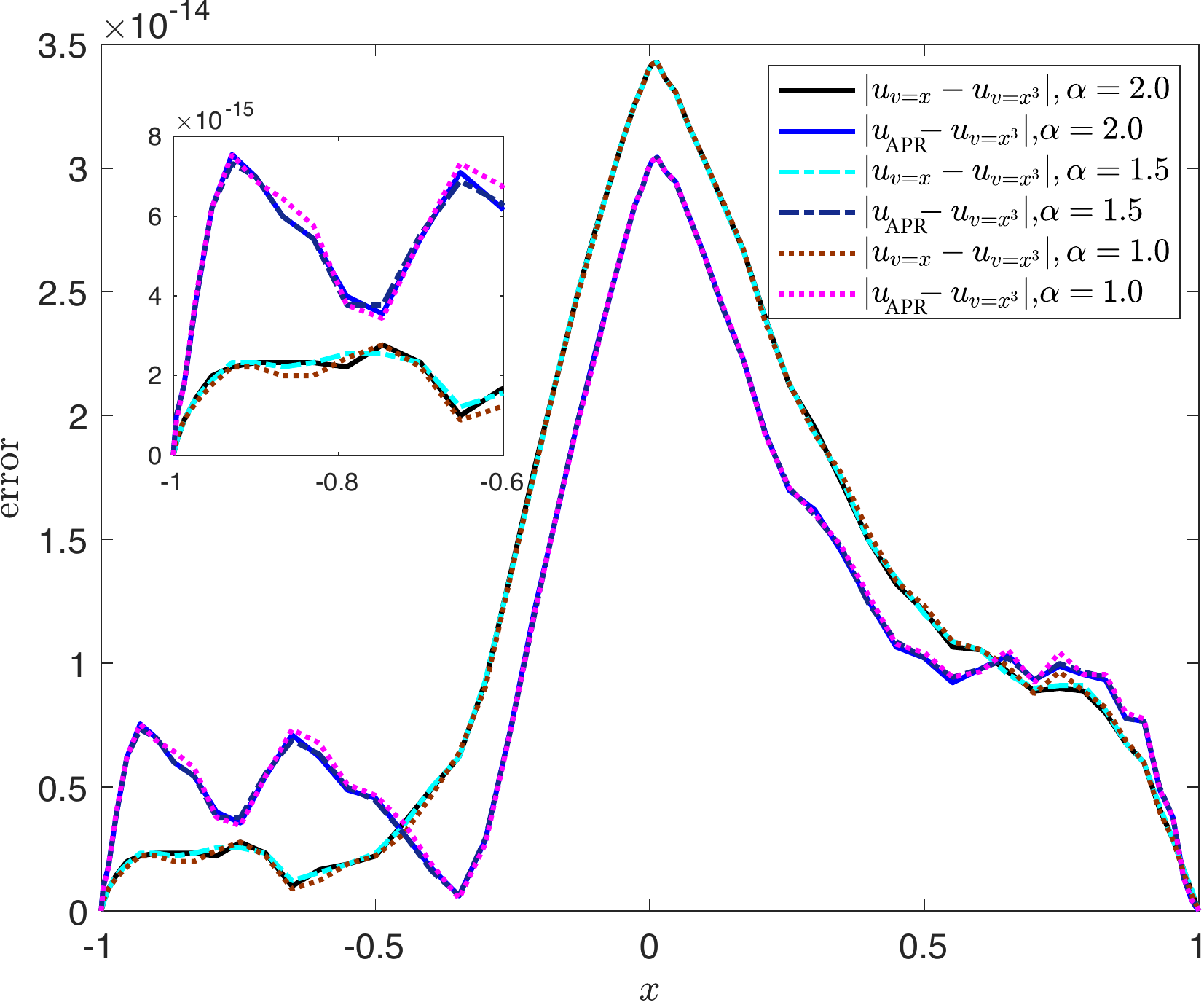}
\caption{\label{liftingdiff}Differences between the {APR \color{chocolate} harmonic lifting method} \cite{AntilPfeffererRogovs} and {\color{chocolate} nonharmonic} lifting method using (\emph{left}) $f = x$ and (\emph{right}) $f = -x$ with different fractional orders $\alpha$. The black solid lines, dashed cyan lines, and dashed brown lines represent the differences between the two versions of the lifting method with $v = x$ and $v = x^3$ for different values of $\alpha$. These differences turn out to be very similar regardless of the values of $\alpha$. We also plot the differences (for different $\alpha$'s) between the very weak solution, $u_A$, and the lifting solution with $v =x^3$ denoted $u_{v = x^3}$.}
\end{figure}

\FloatBarrier

\subsubsection{Numerical Comparison of the {\color{chocolate} Harmonic Lifting} and Heat Semigroup Approaches}\label{comparison_APR_HeatSG}

In this section, we compute the spectral fractional Laplacian of $u = \cos(\pi x)\sin(\pi y)$ on the domain $\Omega = [0,1]^2$ using both the {\color{chocolate} harmonic lifting method} \cite{AntilPfeffererRogovs} as well as the heat semigroup formulation of Cusimano et al. \cite{Cusimano2017}, {\color{chocolate} which we abbreviate as ``Heat SG" in our figures}. The idea of this section is to give an indication of the computational cost of computing $(-\Delta)^{\alpha/2} u(x)$ using the different (equivalent) formulations. 
{\color{chocolate} Once again,} we use the SEM described in Section \ref{SEM} in order to discretize the {\color{chocolate} harmonic lifting (APR)} formulation of $(-\Delta)^{\alpha/2}u(x)$. For the heat SG formulation, we use a quadrature rule to approximate the (truncated) integral \eqref{cusimano_def1}, and we use the same SEM of Section \ref{SEM} along with the Euler method to solve the heat equation \eqref{heatsg_heateqn}. We give more details of these computations below.

The heat semigroup approach uses the formulation of \eqref{cusimano_def1} and \eqref{heatsg_heateqn}, as discussed in Section \ref{sec:heat_semigroup}. The integration \eqref{cusimano_def1} over $(0,\infty)$ is truncated to the interval $(0,T)$ and is approximated using the quadrature formula
\begin{align}
	\mcL_{\Omega,0}^{\alpha/2} u(x_i) &:= \sum_{j=1}^{N_t} (e^{t_j \Delta_\mcB}u(x_i) - u(x_i)) \beta_j = \sum_{j=1}^{N_t} (w_h(x_i,t_j) - w_h(x_i,0))\beta_j,
\end{align}
where $\beta_j := \frac{1}{\Gamma(-\alpha/2)} \int_{t_j - \Delta t/2}^{t_j + \Delta t/2} \frac{dt}{t^{1+\alpha/2}}$. The $w_h$ represents the SEM approximation of the solution $w(x,t)$ to \eqref{heatsg_heateqn}. Equation \eqref{heatsg_heateqn} is discretized in space using the SEM of Section \ref{SEM}, and the time discretization is done using a first order Euler method. The solution $w_h(x_i,t_j)$ is computed by evaluating the polynomial expansion produced by the SEM at $x_i$ at time step $j$.
For the case with nonzero boundary condition $u\big|_{\partial \Omega} = g$, we use a lifting function $v$ that satisfies the integer order equation
\begin{align}
\begin{split}
	-\Delta v &= 0, \hspace{10pt} x \in \Omega, \\
	v &= g, \hspace{10pt} x \in \partial \Omega.
\end{split}
\end{align}
Then we subtract the harmonic lifting function $v$ from $u$ and proceed as above to calculate $\mcL_{\Omega,0}^{\alpha/2} (u-v)$.

In Figure \ref{apr_heatsg_compare}, we plot $(-\Delta)^{\alpha/2}u(x)$ with $\alpha = 0.5$ and $\alpha = 0.8$ for both the {\color{chocolate} harmonic lifting} and Heat SG formulations, both in the square $[0,1]^2$ and along the slice $y = 0.46$. We also plot the differences between these approximations along the same slice, $y = 0.46$. In these examples, we use the SEM with a single element and expansion order $N = 20$, both in the {\color{chocolate} harmonic lifting} discretization and in the space-discretization of the heat equation \eqref{heatsg_heateqn} associated with the heat SG formulation. For the time-discretization of the heat equation \eqref{heatsg_heateqn}, we used the forward Euler method with time step size 6e-5 and $200,000$ iterations. The integral \eqref{cusimano_def1} was approximated with 30 quadrature points and the integration was truncated to the interval $(0,T=1)$.

\begin{figure}[ht!]
\centering
\includegraphics[width=\textwidth]{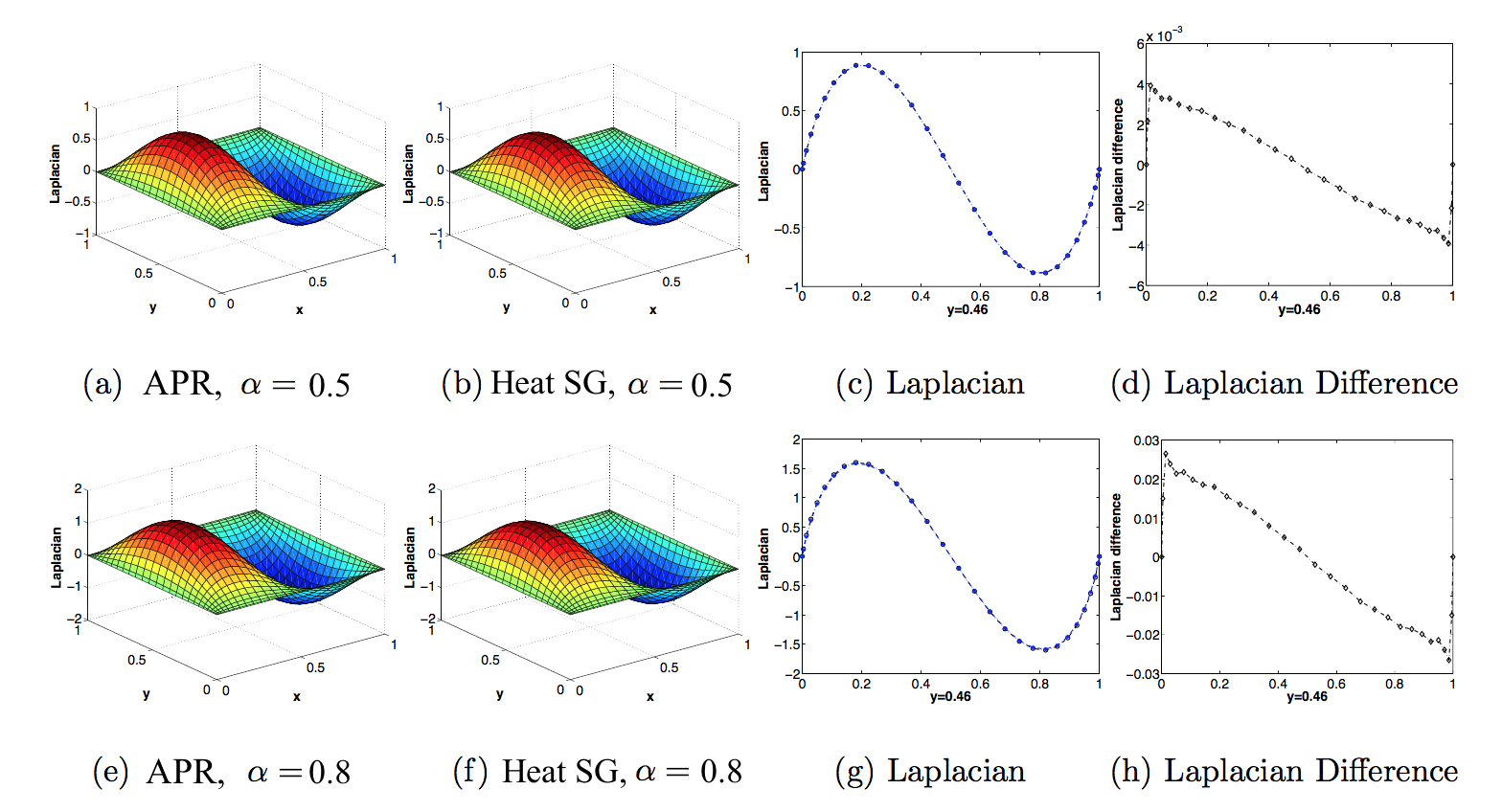}
\caption{$\bm{\alpha = 0.5}$: Plots of $(-\Delta)^{\alpha/2}u(x)$ for $u(x) = \cos(\pi x)\sin(\pi y)$ using (a) the {\color{chocolate} harmonic lifting} (APR) formulation \cite{AntilPfeffererRogovs} and (b) the heat SG formulation \cite{Cusimano2017} on the square domain $[0,1]^2$. (c) Plots of the solutions along the slice $y = 0.46$, where the discretizations appear on top of each other. (d) Difference of the two discretizations plotted along the slice $y = 0.46$. $\bm{\alpha = 0.8}$: Plots of $(-\Delta)^{\alpha/2}u(x)$ for $u(x) = \cos(\pi x)\sin(\pi y)$ using (e) the {\color{chocolate} harmonic lifting}  (APR) formulation and (f) the heat SG formulation on the square domain $[0,1]^2$. (g) Plots of the solutions along the slice $y = 0.46$, where the discretizations appear on top of each other. (h) Difference of the two discretizations plotted along the slice $y = 0.46$. 
{\color{magenta}In subplots (d) and (h), the difference between the two numerical solutions arrives at a maximum near the boundaries, but is zero exactly at the boundaries because the boundary condition is strongly enforced. This difference will decrease and converge to zero as the accuracy each method is increased, by increasing the integration time and decreasing the time step size in the Heat SG method and increasing the number of eigenfunctions and fineness of the mesh in the SEM discretization used in the APR method.}
\label{apr_heatsg_compare}}
\end{figure}

{
\color{blue}
Although theoretical estimates for the convergence of the SEM in section \ref{SEM} are currently under development, it is still possible to have a heuristic discussion comparing the efficiency of using the SEM together with the APR/harmonic lifting formulation of \cite{AntilPfeffererRogovs} vs. using the heat semigroup formulation of \cite{Cusimano2017} to discretize the inhomogeneous spectral fractional Laplacian.
We remark for the heat semigroup formulation, one must solve the (standard) heat equation for many time steps and use the solution to compute the integral \eqref{cusimano_def1} over $t \in (0,\infty)$. The advantage of this approach is that many robust methods exist to discretize the heat equation with arbitrary boundary conditions in high dimensions. On the other hand, discretizing according to the APR formulation \cite{AntilPfeffererRogovs} using the SEM requires the computation of $N$ eigenvalues and eigenfunctions on each element of the domain, which has complexity $N^3$.  Therefore, for time-\emph{independent} fractional equations, the heat SG formulation may offer a faster discretization. However, in order to solve a time-\emph{dependent} equation, such as a fractional heat equation, the heat semigroup formulation would require repeatedly solving a standard heat equation for long time \emph{at each time step}, leading to significantly higher computational complexity. {\color{chocolate} In contrast,} the eigenpairs of the SEM need only be computed once, and can be re-used at each time step. Therefore, the APR with SEM formulation can be expected to be more efficient for time-dependent problems. 
}

\subsection{Comparison of Spectral, Directional, and Riesz Solutions}\label{NonzeroComparison}

Now we can compare the solutions of the inhomogeneous spectral, directional, and Riesz fractional Laplacians. Consider the following equation with nonzero Dirichlet boundary conditions.
\begin{align}
\label{inhom_poisson}
\begin{split}
	(-\Delta)^{\alpha/2} u &= 1, \hspace{10pt} x \in \Omega := [-1,1] \times [-1,1], \\
	u(x) = g(x) :&= \exp(-|x|^2), \hspace{10pt} x \in \partial \Omega \ \text{or} \ \mbbR^d \setminus \Omega,
\end{split}
\end{align}
where $\alpha$ is chosen to be $1.5$. The boundary condition is posed on $\partial \Omega$ for the spectral fractional Laplacian, and it is posed on $\mbbR^d \setminus \Omega$ for the Riesz and directional definitions.

We use the method of \cite{AntilPfeffererRogovs} to solve \eqref{inhom_poisson} using the spectral fractional Laplacian, the WOS method for the Riesz definition, and the RBF collocation method for the directional definition. We again find that the directional and Riesz solutions are equivalent up to numerical error, as shown in Figure \ref{inhom_solns}. We also observe that the spectral solution is of greater magnitude than the other solutions in this example, contrary to the observations we made in the zero boundary condition case. This indicates that the relative magnitudes of the solutions for different definitions relies on the boundary condition. In Figure \ref{inhom_slices}, we plot the solutions and differences along the line $y = 0$, where this property can be more easily observed.
 
\begin{figure}[ht!]
\centering
\includegraphics[width=.3\textwidth]{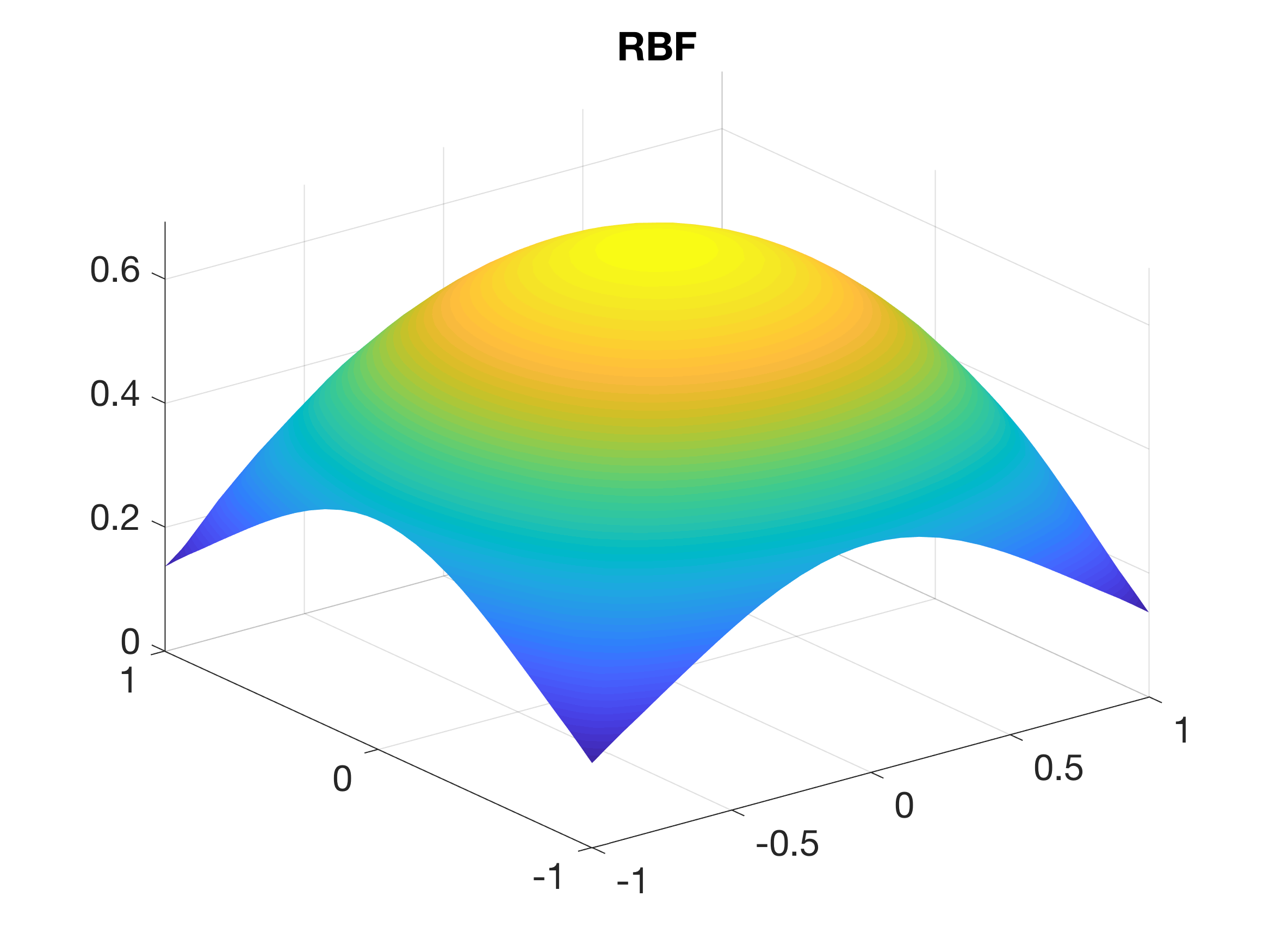}
\includegraphics[width=.3\textwidth]{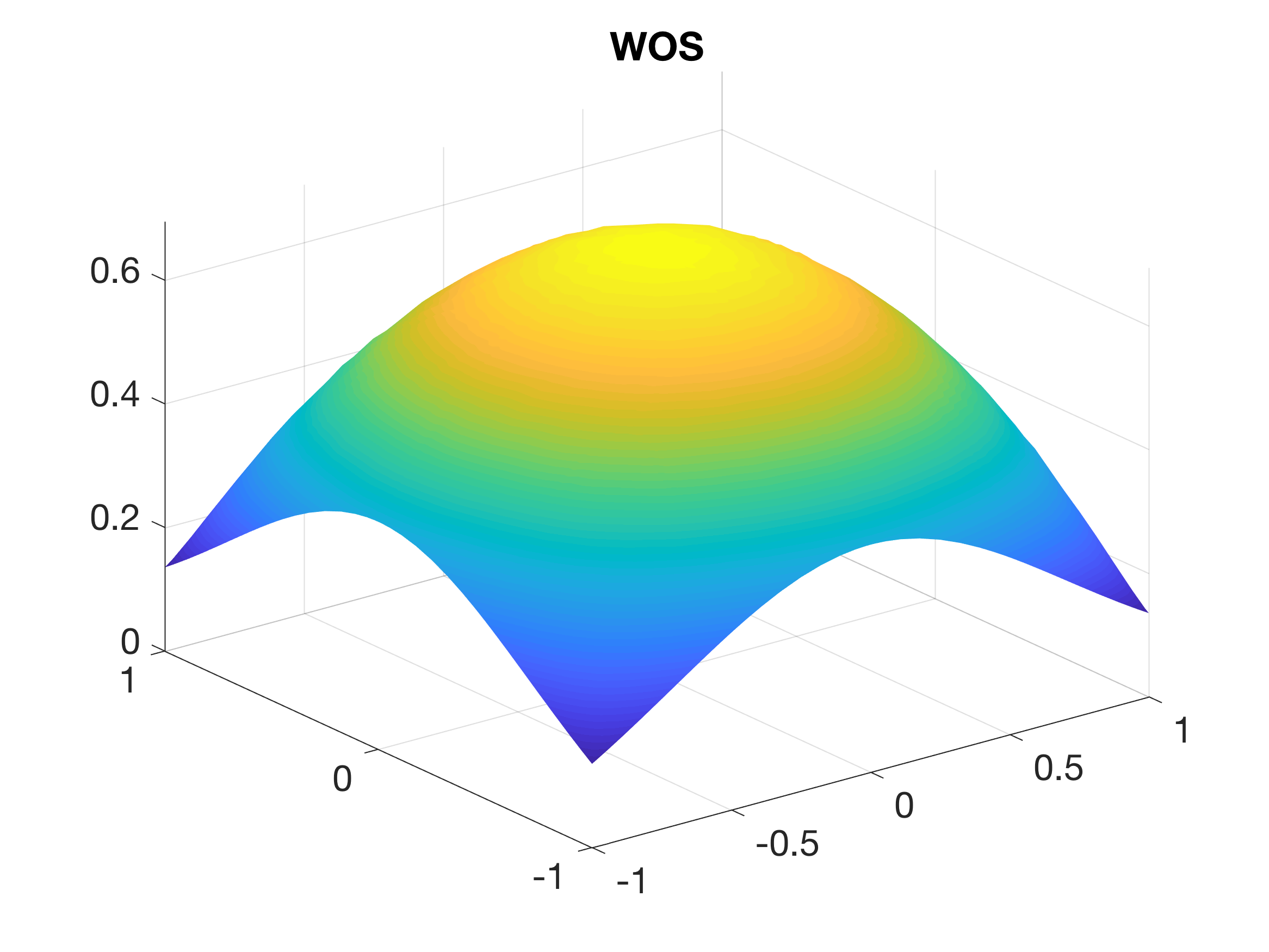}
\includegraphics[width=.3\textwidth]{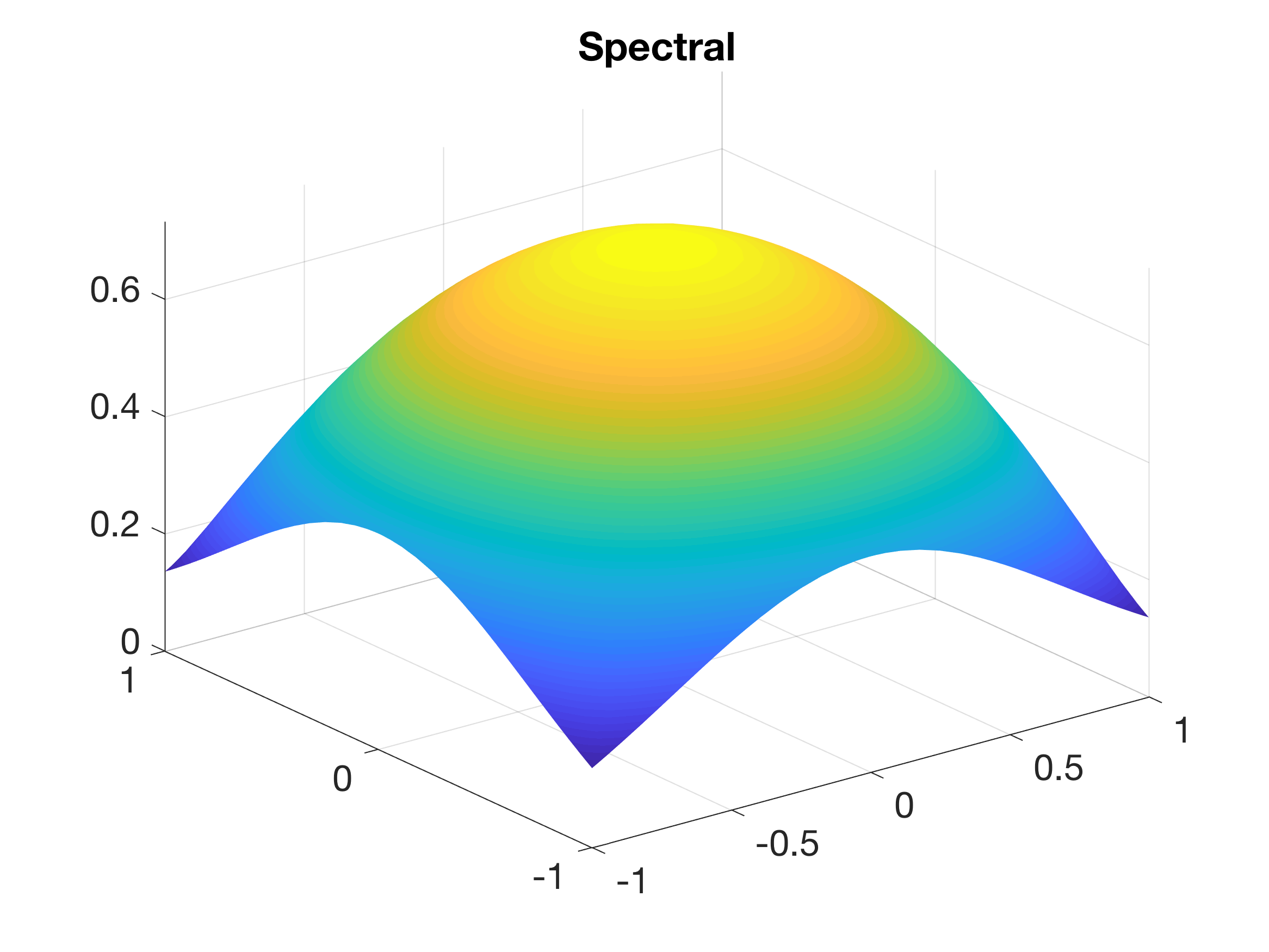}
\caption{\label{inhom_solns} {Solutions to the inhomogeneous fractional Poisson equation}: (\emph{left}) RBF solution corresponding to the directional definition, (\emph{center}) WOS solution corresponding to the Riesz definition, and (\emph{right}) SEM solution corresponding to the spectral definition.}
\end{figure}

\begin{figure}[ht!]
\centering
\includegraphics[width=.3\textwidth]{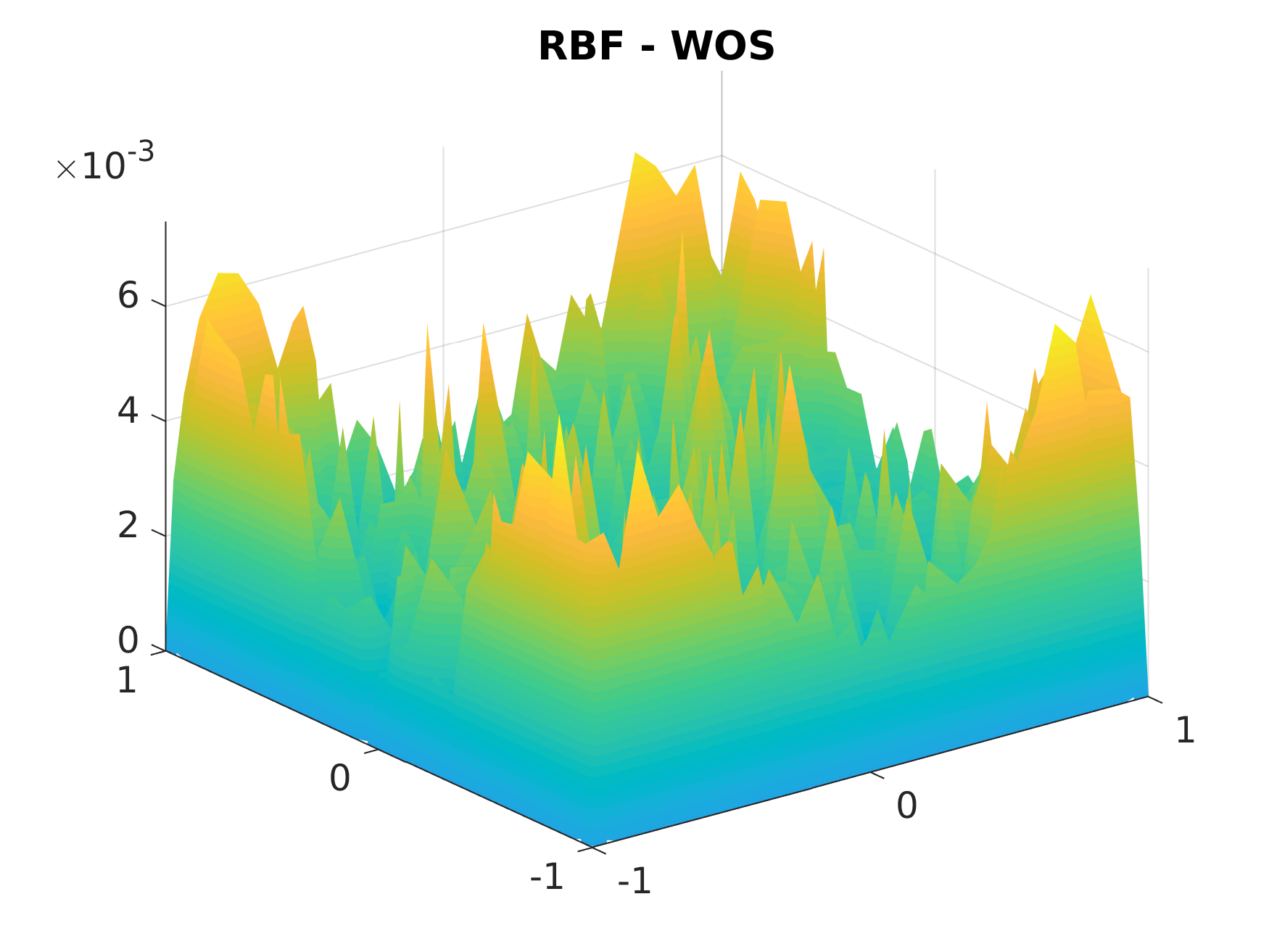}
\includegraphics[width=.3\textwidth]{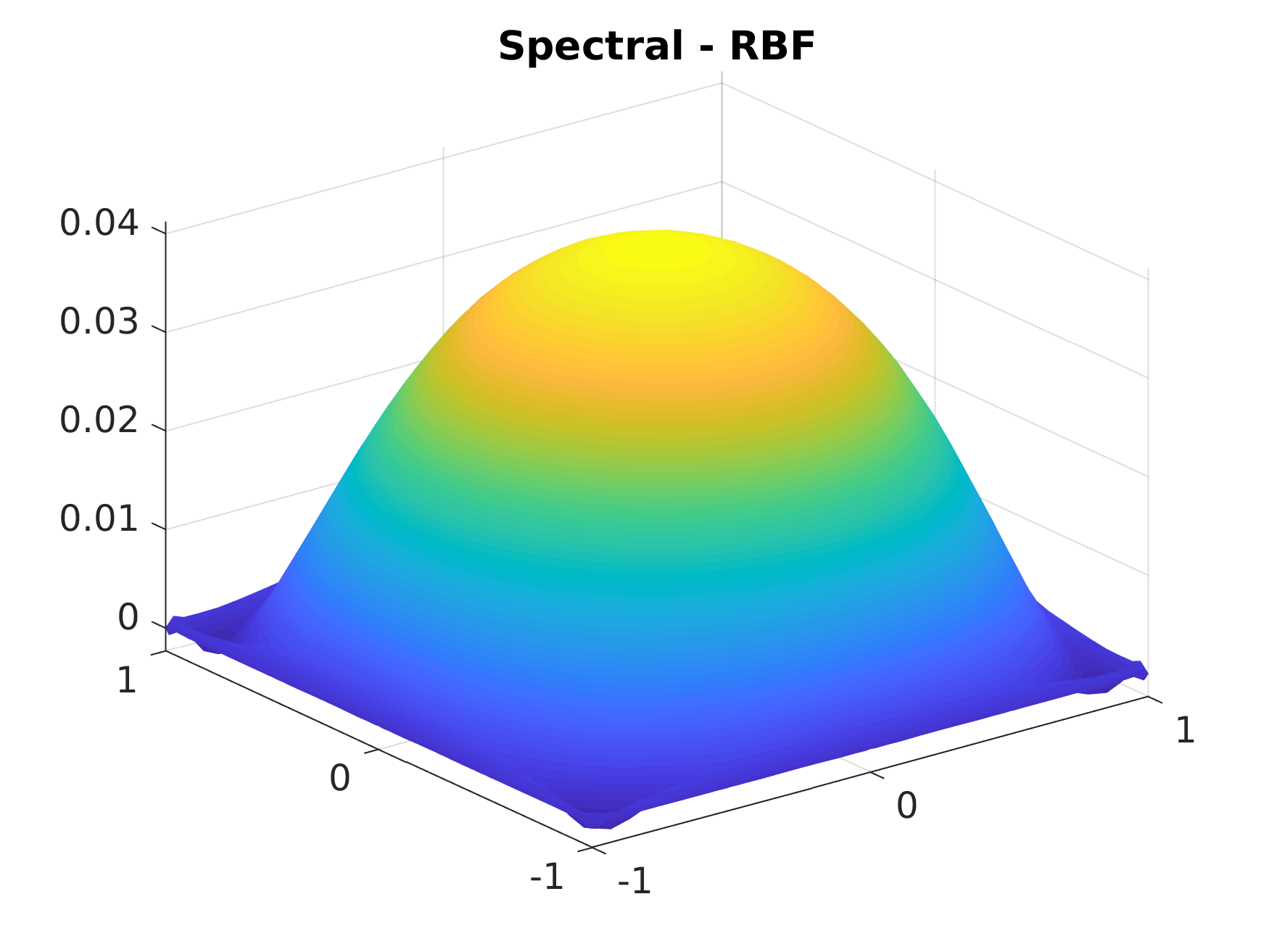}
\includegraphics[width=.3\textwidth]{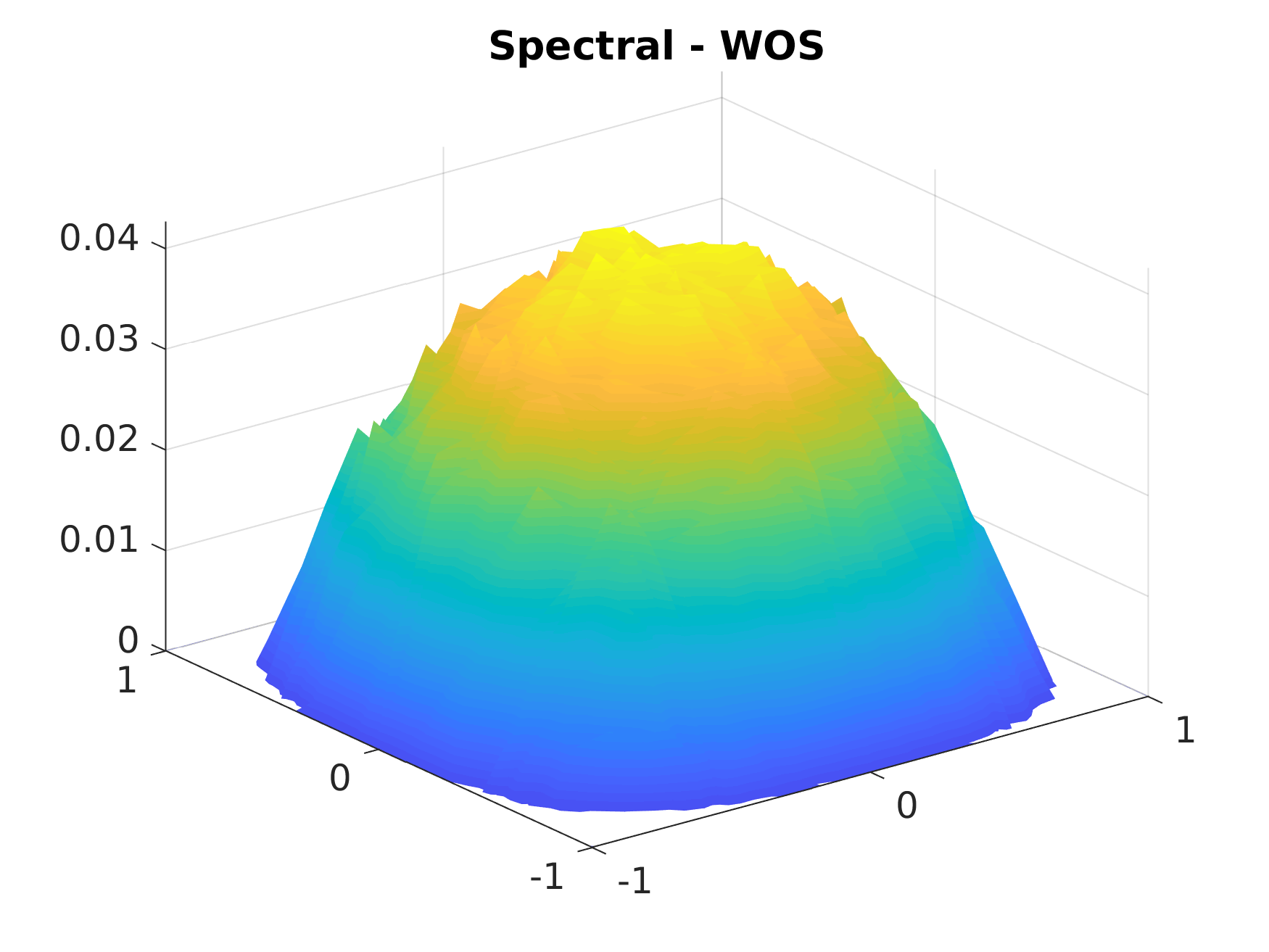}
\caption{\label{inhom_solns} {Differences of the solutions to the inhomogeneous fractional Poisson equation}: (\emph{left}) the difference between the RBF and WOS solution demonstrates the close similarity between these solutions, as the directional and Riesz definitions are equivalent; (\emph{center}) the difference between the spectral and RBF solutions; (\emph{right}) the difference between the spectral and WOS solutions. In these examples, the source function is $f = 1$ and the BC is $g = \exp(-|x|^2)$, and $\alpha = 1.5$. Furthermore, the spectral solution has greater magnitude than the directional and Riesz solutions, which is in contrast with the results for zero boundary conditions.}
\end{figure}

\begin{figure}[ht!]
\centering
\includegraphics[width=.45\textwidth]{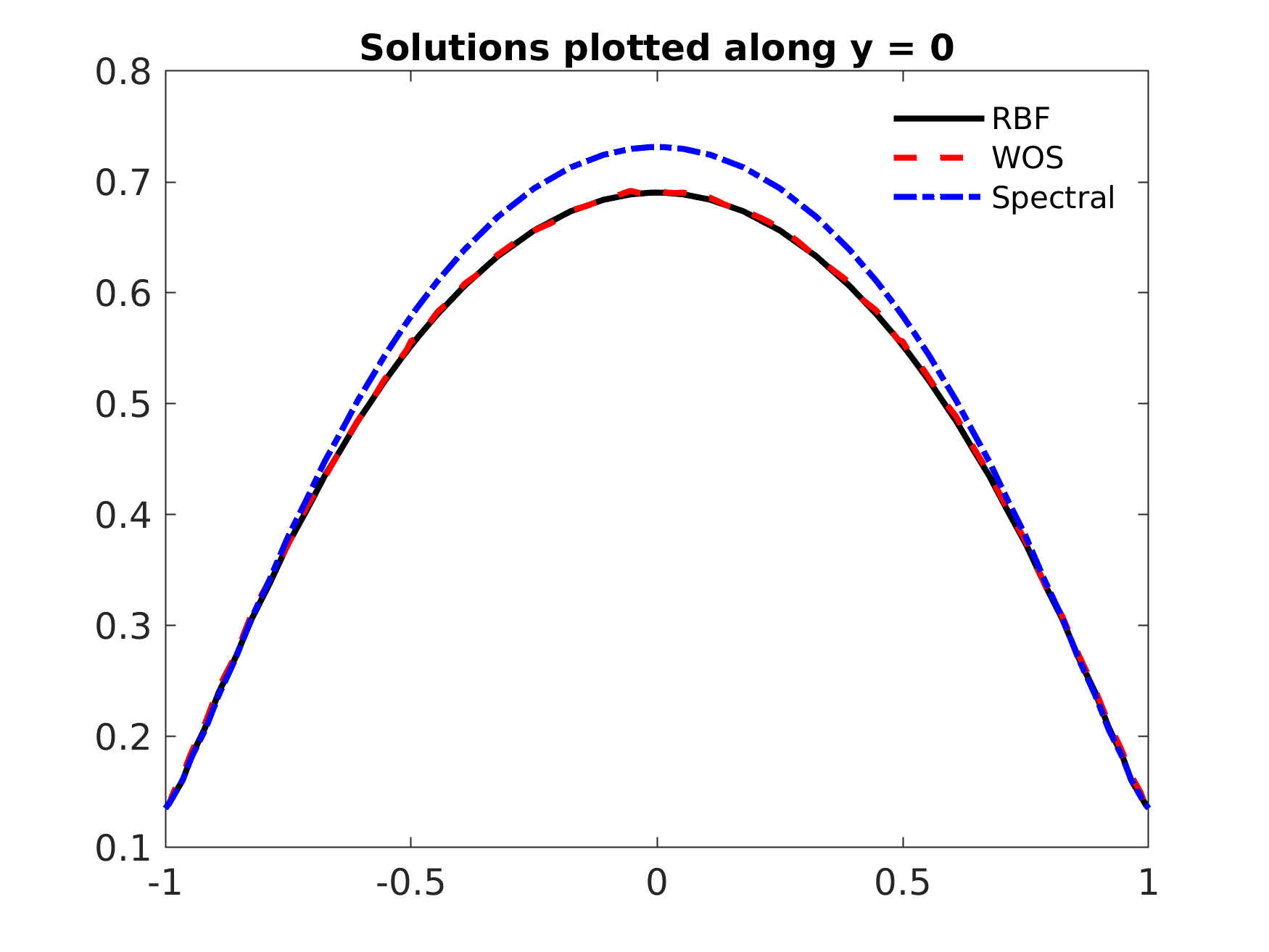}
\includegraphics[width=.45\textwidth]{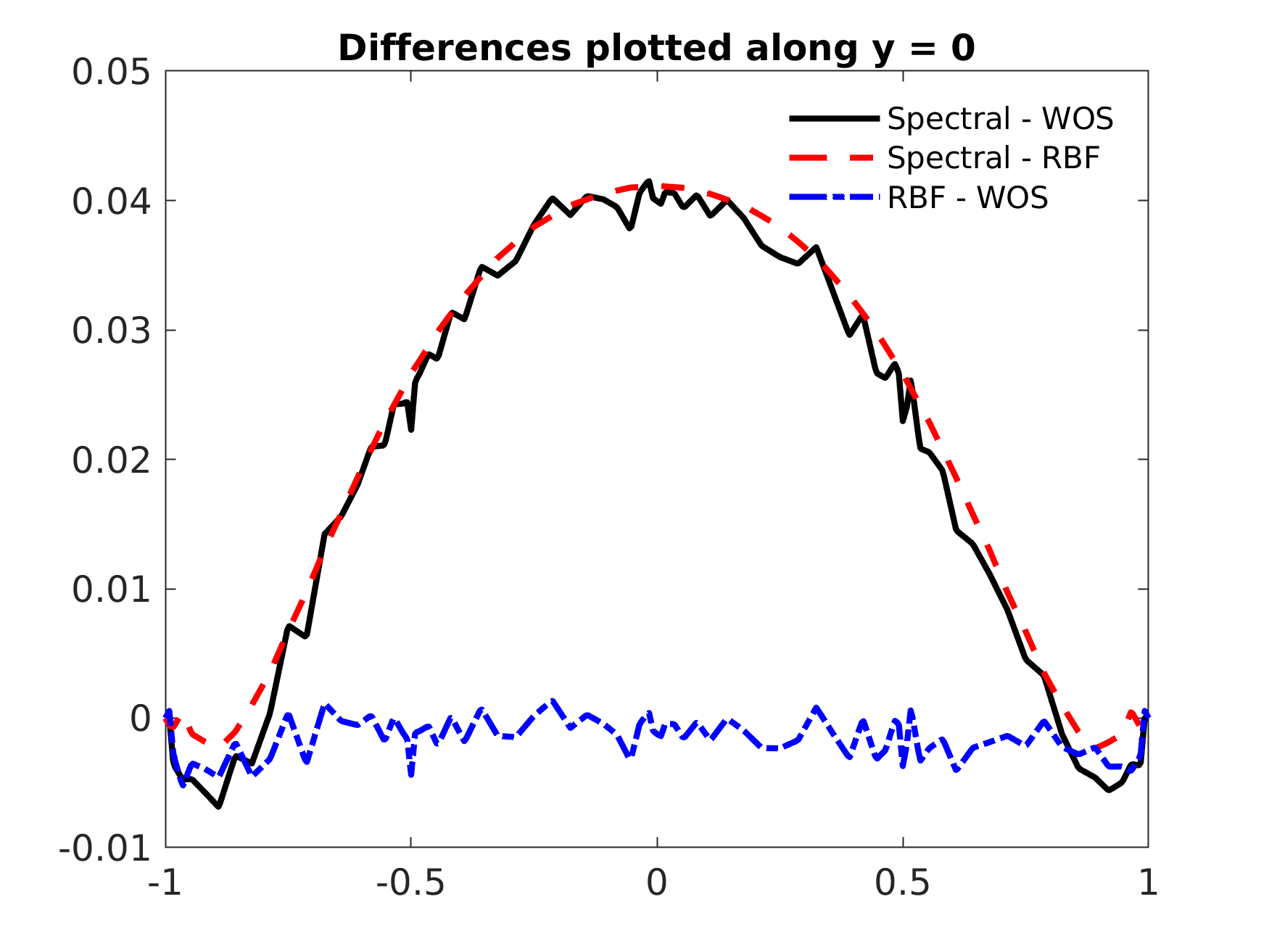}
\caption{\label{inhom_slices} {Plots of solutions and differences along the line $y = 0$}: (\emph{left}) We plot the RBF, WOS, and spectral solutions along the line $y = 0$, and (\emph{right}) the differences between the solutions. The oscillations are due to the WOS method and are expected since this method is based on the Feynman-Kac approach.}
\end{figure}

The influence of the truncation parameter on the solution accuracy for the RBF collocation method
is shown in Table \ref{inhomogeneous_truncation_square_table}. 
The solution in this case is converged when the truncation parameter $K_2 = 6000$, 
due to the rapid decay of the exterior condition $g(x)=\exp(-|x|^2)$

\begin{table}[htbp!]
\begin{tabular}{|p{2in}|p{1.5in}|p{1.5in}|}
\hline
\textbf{Truncation parameter $K_2$} & \textbf{$L_2$ difference from previous solution} & \textbf{$L_2$ difference from reference solution} \\ \hline
1000                                & N/A                                              & 1.06e+0                                     \\ \hline
1200                                & 5.41e-1                                          & 5.19e-1                                     \\ \hline
1600                                & 4.05e-1                                          & 1.14e-1                                     \\ \hline
2000                                & 9.27e-2                                          & 2.15e-2                                     \\ \hline
3000                                & 2.13e-2                                          & 1.50e-4                                     \\ \hline
4000                                & 1.50e-4                                          & 2.63e-7                                     \\ \hline
6000 (reference solution)                    & 2.63e-7                                          & 0                                           \\ \hline
\end{tabular}
\caption{\textit{Influence of truncation parameter $K_2$ on the RBF collocation numerical solution for inhomogeneous boundary condition on the unit square $[-1,1]\times[-1,1]$.} We consider the 2D fractional Poisson problem with the forcing term $f(x)=1$ and the boundary condition $g(x)=\exp(-|x|^2)$. 
The collocation points are on a 41 $\times$ 41 regular grid, including 1521 domain points and 160 boundary points. The RBF shape parameter is 0.05 and the spatial step size $h$ is set to be 0.001. We increase the truncation parameter from $K_2=1000$ to $K_2=6000$ until the numerical solution becomes sufficiently consistent with respect to $K_2$.}
\label{inhomogeneous_truncation_square_table}
\end{table}

\section{Summary and Discussion \label{conclusion}}

In this work, we {\color{chocolate} followed} a{ \color{chocolate} joint theoretical and computational} approach to examining the different characteristics of {\color{chocolate} different fractional Laplacians} and solutions of related fractional Poisson equations formulated on bounded domains. {\color{chocolate}This included} the spectral and horizon-based nonlocal definitions of the fractional Laplacian as well as various formulations of the Riesz fractional Laplacian. {\color{chocolate}We made numerical comparisons using different methods and high levels of refinement in order to compare solutions to both one- and two-dimensional benchmark problems formulated with different fractional Laplacians. We surveyed relevant numerical methods for performing these computations and detailed the implementation of the methods used in this work, and a new radial basis function collocation method was presented. We discussed relative advantages of the computational approaches we implemented and identified directions for future development. We also outlined the theoretical derivation of each fractional Laplacian definition in $\mbbR^d$ and contrasted the different approaches to restricting these definitions to bounded domains. Connections were made between fractional Laplacians with different types of boundary conditions with their associated stochastic processes. We surveyed relevant regularity results for the fractional Poisson problems, and discussed the well-posedness of the considered equations.}

The {\color{chocolate}nonlocality of the fractional Poisson problems} presented significant computational challenges even when posed with zero Dirichlet boundary conditions, and nonzero boundary conditions further compounded these difficulties. We discussed recently-proposed approaches for approximating the inhomogeneous fractional Laplacians \cite{AntilPfeffererRogovs,Cusimano2017} and compared these {\color{chocolate}approaches} to nonharmonic lifting of the equation to a homogeneous reformulation. We found that {\color{chocolate}all considered methods} resulted in equivalent solutions to the inhomogeneous benchmark problems, and we proved this equivalence analytically {\color{chocolate}for the first time.}

{\color{chocolate}The numerical methods used in this work are not representative of the breadth of research on this topic, and the types of methods considered reflect the desire of the authors to maintain a focus on fundamental questions (such as nonzero boundary conditions) related to the fractional Laplacian rather than a comprehensive survey of numerical methods. For example, finite-difference approaches to fractional Laplacians and fractional diffusion, such as 
\cite{huang2014numerical, li2018numerical, meerschaert_bcs, kelly2018anomalous}, have not been discussed at length despite being of classical importance and an area of active development. However, the discussion of the varied types of methods implemented and developed for this work, which we believe to be state-of-the art at the time of this writing, has significant value for aiding researchers simulating fractional models in choosing an appropriate numerical method.}

{\color{chocolate}Using these numerical methods, we were able to make several interesting observations.} We found that the size of the domain has a significant effect on the evolution of the numerical solutions as the fractional order $\alpha$ was changed, as discussed in Sec. \ref{intro}. We also emphasized the singular behavior of the Riesz fractional Poisson equation solutions near the boundary, which sharply contrasted with the smooth behavior of the spectral solutions in the same locations. We observed in our examples that given a source function $f$ with no boundary singularity, we can expect that the solution to the Riesz fractional Poisson equation will have a boundary singularity, and the solution to the spectral fractional Poisson equation will be smooth at the boundaries. This observation should be considered when choosing which definition to use to fit data most appropriately. The analytical perspective on this boundary regularity issue was discussed in detail in Sections \ref{background} and \ref{sec:regularity}.

Although the study of the fractional Laplacian is far from complete, this work can serve as a starting point for researchers using these operators to model systems {\color{chocolate}exhibiting anomalous transport phenomena.}

\pagebreak
\global\pdfpageattr\expandafter{\the\pdfpageattr/Rotate 90}
\newgeometry{left=0.5in, bottom=0.5in, top=0.5in, right=0.5in}

\begin{sidewaysfigure}[ht!]
\centering
\resizebox{\textheight}{!}{
\begin{tikzpicture}[node distance = 5mm and 5mm,
  double/.style={very thick, draw, anchor=text, rectangle split,rectangle split parts=2},
  triple/.style={very thick, draw, anchor=text, rectangle split,rectangle split parts=3},
  terminal/.style={very thick, draw=black, top color=white, bottom color=black!20, anchor=text, rectangle, minimum size=6mm},
  myarrow/.style={very thick, -latex},
	]
  \node (spectlabel) [terminal,align=center] at (10,2) {\textbf{Spectral Definition} \\ \textbf{on} \\ \textbf{Bounded Domains}};
  
      \node(homoneu)[double,align=center,below=3cm,at=(spectlabel.south)] {\textbf{Zero Neumann}
    \nodepart{second}
      Balakrishnan formula \cite{bonito2017,bonito2015}, Sec. \ref{homo_representations} \\
      {\color{forest}Harmonic Lifting +} FEM \cite{AntilPfeffererRogovs}, Sec. \ref{sec:harmonic_lifting} \\
      Heat Semigroup \cite{Cusimano2017}, Sec. \ref{Spectral} \\
      SEM \cite{SongXuKarniadakis2017}, Sec. \ref{SEM}
  };
  
    \node(nond)[double,align=center,left=3mm,at=(homoneu.west)] {\textbf{Nonzero Dirichlet}
   	\nodepart{second}
	      {\color{forest}Harmonic Lifting +} FEM \cite{AntilPfeffererRogovs}, Sec. \ref{sec:harmonic_lifting} \\
		Heat Semigroup \cite{Cusimano2017}, Sec. \ref{Spectral} \\
	      {\color{forest}Nonharmonic Lifting}, Sec. \ref{sec:nonharmonic_lifting} \& \ref{SEM}
	};

  \node(homod)[double,align=center,left=3mm,at=(nond.west)] {\textbf{Zero Dirichlet}
  	\nodepart{second}
	   Balakrishnan formula \cite{bonito2017,bonito2015}, Sec. \ref{homo_representations} \\
	   {\color{forest}Harmonic Lifting +} FEM \cite{AntilPfeffererRogovs}, Sec. \ref{sec:harmonic_lifting} \\
	   Heat Semigroup \cite{Cusimano2017}, Sec. \ref{Spectral} \\
	   SEM \cite{SongXuKarniadakis2017}, Sec. \ref{SEM} \\
	   \tikz{\node(silvcaff)[double,align=center] {Extension Method \cite{CaffarelliSilvestre2007_ExtensionProblemRelatedToFractionalLaplacian}, Sec. \ref{Spectral}
	     \nodepart{second}
	     Hybrid FE-SM \cite{AinsworthGlusa2017_HybridFiniteElementSpectral}, Sec. \ref{Spectral} \\
	     FEM \cite{Nochetto2015} \\
	     H-P FEM \cite{Gatto2015}};} 
  };
  
  \node(nonneu)[double,align=center,right=3mm,at=(homoneu.east)] {\textbf{Nonzero Neumann} 
        \nodepart{second}
        {\color{forest}Harmonic Lifting + } FEM \cite{AntilPfeffererRogovs}, Sec. \ref{sec:harmonic_lifting} \\
        Heat Semigroup \cite{Cusimano2017}, Sec. \ref{Spectral}};
  
      \node(robin)[double,align=center,right=3mm,at=(nonneu.east)] {\textbf{Robin}
    \nodepart{second}
      Heat Semigroup \cite{Cusimano2017}, Sec. \ref{Spectral}
  };

  \draw[myarrow,black] (spectlabel) -- (homod);
  \draw[myarrow,black] (spectlabel) -- (homoneu);
  \draw[myarrow,black] (spectlabel) -- (robin);
  \draw[myarrow,black] (spectlabel) -- (nond);
  \draw[myarrow,black] (spectlabel) -- (nonneu);

\end{tikzpicture}
}
\caption{\label{spectral_methods} Numerical methods {\color{chocolate} that are focused on} or developed in this work for solving equations involving the spectral fractional Laplacian. {\color{chocolate} This is by no means an exhaustive list of all possible methods for the spectral fractional Laplacian.}}
\end{sidewaysfigure}
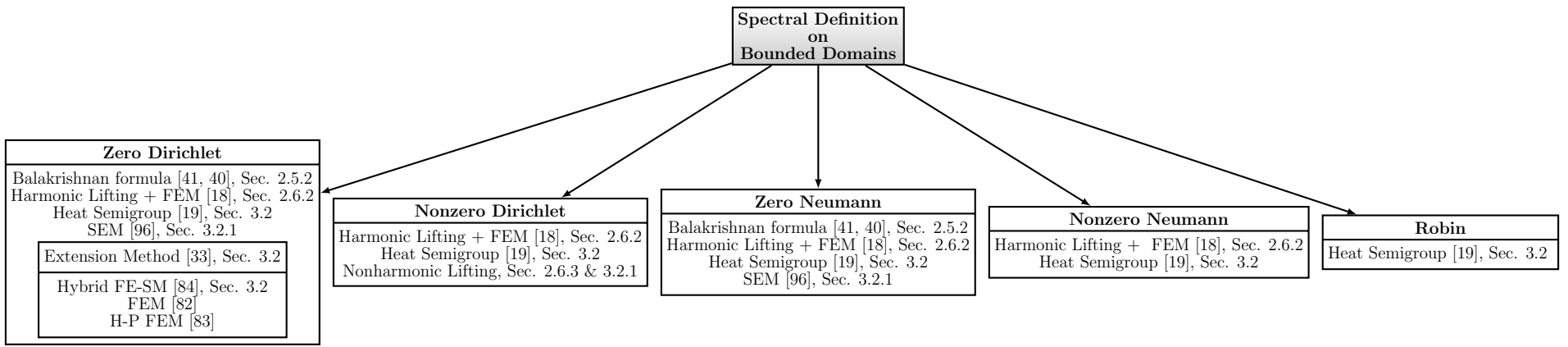

\restoregeometry
%{\pagebreak[4]\global\pdfpageattr\expandafter{\the\pdfpageattr/Rotate 0}}%
\pagebreak

\begin{figure}
\centering
\begin{tikzpicture}[node distance = 5mm and 5mm,
  double/.style={very thick, draw, anchor=text, rectangle split,rectangle split parts=2},
  triple/.style={very thick, draw, anchor=text, rectangle split,rectangle split parts=3},
  terminal/.style={very thick, draw=black, top color=white, bottom color=black!20, anchor=text, rectangle, minimum size=6mm},
  myarrow/.style={very thick, -latex},
	]
  \node (rieszlabel) [terminal,align=center] at (10,2) {\textbf{Riesz Definition} \\ \textbf{on} \\ \textbf{Bounded Domains}};

    \node(nond)[double,align=center,below=3.15cm,right=0.2cm,at=(rieszlabel.east)] {\textbf{Nonzero Dirichlet}
   	\nodepart{second}
	      RBF Collocation, Sec. \ref{RBFM} \\
	      Walk on Spheres \cite{kyprianou2016unbiasedwalk}, Sec. \ref{sec:wos}
	      };

  \node(homod)[double,align=center,below=3.15cm,left=0.2cm,at=(rieszlabel.west)] {\textbf{Zero Dirichlet}
  	\nodepart{second}
	AFEM \cite{AinsworthGlusa2017_TowardsEfficientFiniteElement}, Sec. \ref{sec:afem} \\
      	RBF Collocation, Sec. \ref{RBFM} \\
      	Walk on Spheres \cite{kyprianou2016unbiasedwalk}, Sec. \ref{sec:wos}
  };

  \draw[myarrow,black] (rieszlabel) -- (nond);
  \draw[myarrow,black] (rieszlabel) -- (homod);

\end{tikzpicture}
\caption{\label{Riesz_methods} Numerical methods {\color{chocolate} that are focused on} or developed in this work for solving equations involving the Riesz fractional Laplacian. {\color{chocolate} This is by no means an exhaustive list of all possible methods for the Riesz fractional Laplacian.}}
\end{figure}
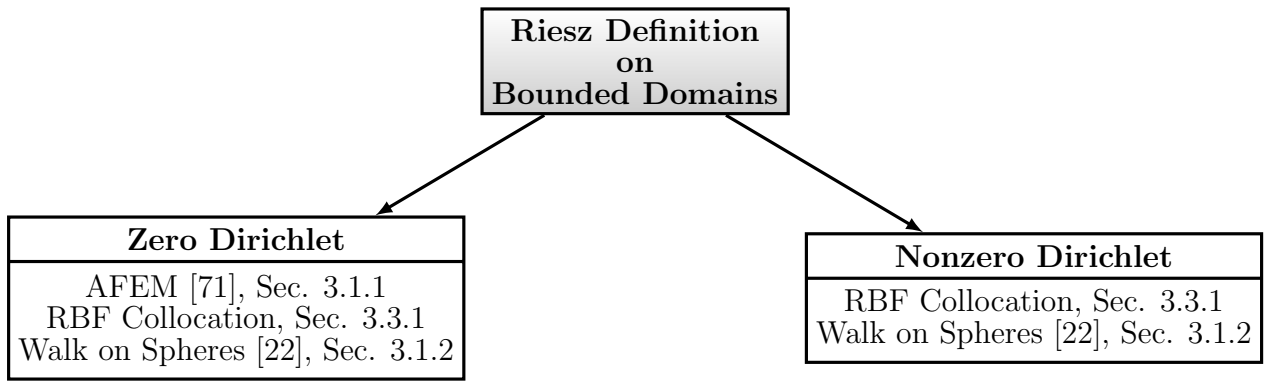

%\restoregeometry
{\pagebreak[4]\global\pdfpageattr\expandafter{\the\pdfpageattr/Rotate 0}}%
\pagebreak

\section{Acknowledgements}
This work was supported by the OSD/ARO/MURI on ``Fractional PDEs for Conservation Laws and Beyond: Theory, Numerics and Applications" (W911NF-15-1-0562).

W. Cai was supported by an NSF grant (DMS-1764187).

G. Pang was supported by the postdoctoral fellowship at Beijing Computational Science Research Center and also by the National Natural Science Foundation of China (11701025).

{\color{chocolate} M. Gulian was supported by an NSF Graduate Research Fellowship. }

Sandia National Laboratories is a multimission laboratory managed and operated by National Technology and Engineering Solutions of Sandia, LLC, a wholly owned subsidiary of Honeywell International, Inc., for the U.S. Department of Energy’s National Nuclear Security Administration under contract {DE-NA0003525}. This paper describes objective technical results and analysis. Any subjective views or opinions that might be expressed in the paper do not necessarily represent the views of the U.S. Department of Energy or the United States Government.

\section{References}
\bibliographystyle{elsarticle-num}
\bibliography{ref}

\begin{appendices}
%\renewcommand{\processdelayedfloats}{}
%\restoregeometry

\section{Sobolev Spaces and the Trace Theorem}\label{sobolev_spaces}
In this appendix, we assemble the commonly-used Sobolev spaces for the fractional Laplacian. These spaces are discussed in \cite{Grubb2015_FractionalLaplaciansDomainsDevelopment,grubb_spectral}, but we have more closely followed the exposition and notation of \cite{AntilPfeffererRogovs}.
\begin{definition}[\cite{AntilPfeffererRogovs}]
\label{sobdef:integral}
For $0 < s < 1$, we define the fractional Sobolev space
\begin{align}
	H^{s}(\Omega) &:= \left\{ u \in L^2(\Omega) : \int_\Omega \int_\Omega \frac{|u(x) - u(y)|^2}{|x-y|^{n+2s}} dx dy < \infty \right\},
\end{align}
{\color{blue}with the semi-norm and norm
\begin{align}
	|u|_{H^{s}(\Omega)}^2 &:= \int_\Omega \int_\Omega \frac{|u(x) - u(y)|^2}{|x-y|^{n+2s}} dx dy, \label{seminorm}\\
	\|u\|_{H^{s}(\Omega)} &:= \left(\|u\|^2_{L^2(\Omega)} + |u|_{H^{s}(\Omega)}^2 \right)^{1/2}.
\end{align}}
\end{definition}
\begin{definition}[\cite{ding_trace_theorem}]
\label{trace_definition}
For any $u \in C^\infty(\Omega)$, define the trace operator $\gamma\big|_{\partial \Omega}$ by
\begin{align}
	\gamma\big|_{\partial \Omega}u(x) &= u(x), \hspace{10pt} x \in \partial \Omega.
\end{align}
\end{definition}

\begin{theorem}{Trace Theorem \cite{ding_trace_theorem}.}\label{trace_theorem}
Let $\Omega$ be a bounded, simply connected Lipschitz domain and $1/2 < s < 3/2$. Then the trace operator $\gamma\big|_{\partial \Omega}$ is a bounded linear operator from $H^s(\Omega)$ to $H^{s-\frac{1}{2}}(\partial \Omega)$.
\end{theorem}
This range for $s$, $1/2 < s < 3/2$, is sufficient for the discussion of traces in this article. Now we can define the subspace $H_0^s(\Omega)$ of $H^s(\Omega)$.
\begin{definition}\label{sobdef:H0}
For $s > 1/2$, 
\begin{align}
	H_0^s(\Omega) &:= \left\{ u \in H^s(\Omega) : \gamma\big|_{\partial\Omega}u(x) = 0\right\}.
\end{align}
\end{definition}
In the case {\color{blue}$s = 1/2$}, we must define another fractional Sobolev space, denoted $H_{00}^{1/2}(\Omega)$.
\begin{definition}[\cite{AntilPfeffererRogovs}]
\label{sobdef:00}
\begin{align}
	H_{00}^{1/2}(\Omega) &:= \left\{ u \in H^{1/2}(\Omega) : \int_\Omega \frac{u^2(x)}{\text{dist}(x,\partial\Omega)} dx < \infty \right\}.
\end{align}
The associated norm is defined
\begin{align}
	\|u\|_{H_{00}^{1/2}(\Omega)} &:= \left(\|u\|_{H^{1/2}(\Omega)}^2 + \int_\Omega \frac{u^2(x)}{\text{dist}(x,\partial \Omega)} dx \right)^{1/2}.
\end{align}
\end{definition}

We define the space {\color{blue}$\mbbH^{s}(\Omega)$}, which is used to characterize the regularity properties of the spectral fractional Laplacian.
\begin{definition}[\cite{AntilPfeffererRogovs}]
\label{sobdef:boldh}
{\color{blue}For any $s \geq 0$,
\begin{align}
	\mbbH^{s}(\Omega) &:= \left\{ u = \sum_{k=1}^\infty u_k \phi_k \in L^2(\Omega) : \|u\|_{H^{s}(\Omega)}^2 := \sum_{k=1}^\infty \lambda_k^{s} u_k^2 < \infty \right\}
\end{align}}
where $(\lambda_k,\phi_k)$ are the eigenpairs of the integer Laplacian $-\Delta$ on $\Omega$ with zero Dirichlet boundary conditions.
\end{definition}
As noted in Ref. \cite{AntilPfeffererRogovs}, the following relationship exists between the fractional Sobolev spaces presented in Defs. \ref{sobdef:integral}, \ref{sobdef:H0}, \ref{sobdef:00}, and \ref{sobdef:boldh}{\color{blue}
\begin{align}
	\mbbH^{s}(\Omega) = \begin{cases} H^{s}(\Omega) = H_0^{s}(\Omega) & \text{if} \ 0 < s < 1/2, \\
		H_{00}^{1/2}(\Omega) & \text{if} \ s = 1/2, \\
		H_0^{s}(\Omega) & \text{if} \ 1/2 < s < 1.
	\end{cases}
\end{align}}

{\color{blue}
\begin{definition}
The space $\mathbb{H}^{-s}(\Omega)$, for $s \geq 0$, is defined to be the dual space of $\mathbb{H}^s(\Omega)$.
\end{definition}
}

\begin{definition}
$H_{loc}^{s}(\Omega) = \{ u \in H^{s}(K) \ \text{for all compact sets} \ K \subseteq \overline{\Omega}\}.$
\end{definition}

\vfill
\break

\section{Grids}\label{grids}
Here, we include the meshes used for the two-dimensional numerical comparisons of Sec. \ref{sec:num_comp}.

 \begin{figure}[ht!]
 \centering
\subfloat[1600 RBF collocation points with shape parameter $0.3$ used for RBF collocation method (\emph{left}) and mesh used for SEM method (\emph{right}).]{
 \begin{minipage}[]{\textwidth}\centering
 \includegraphics[width=0.25\textwidth]{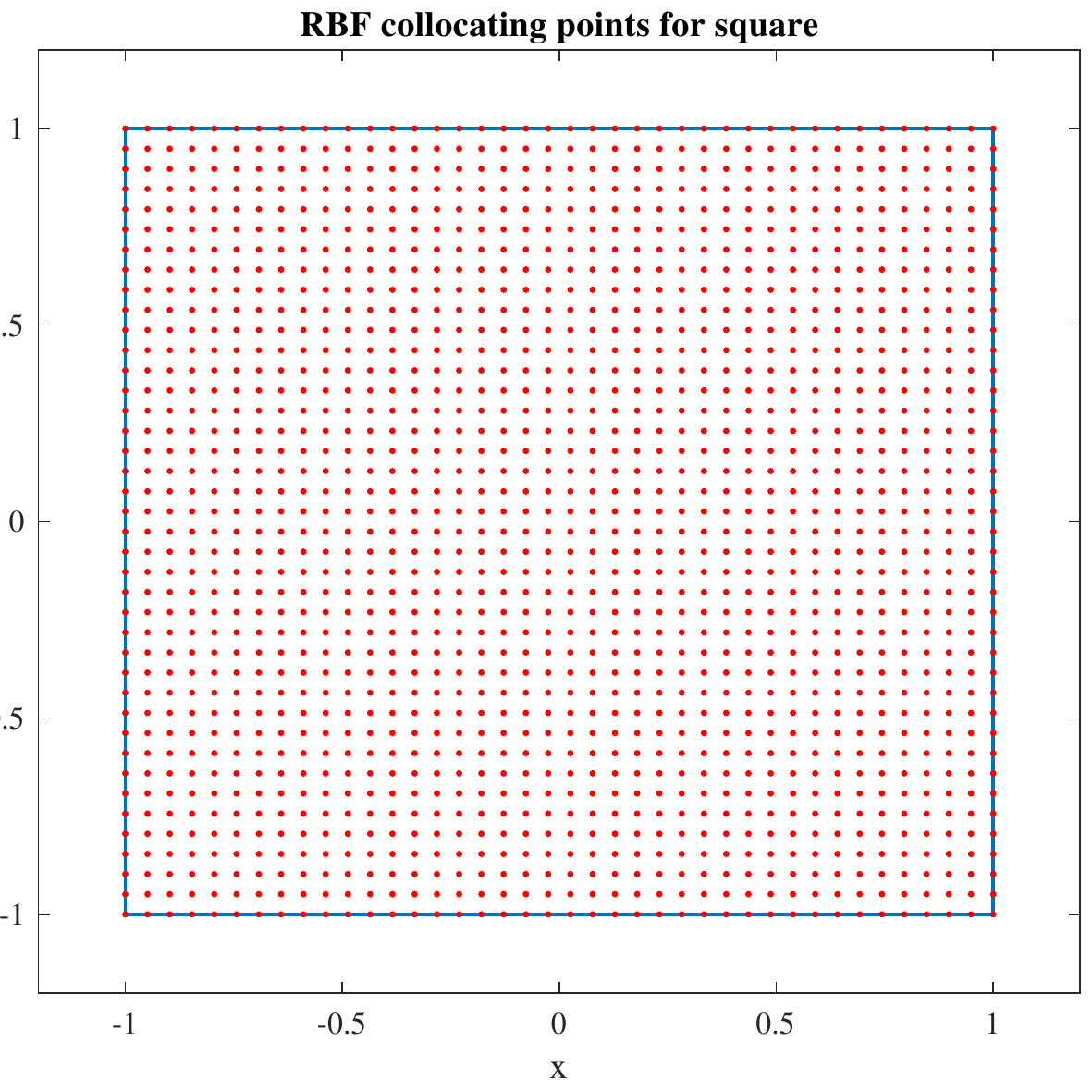} \hspace{0.3cm}
  \includegraphics[width=0.25\textwidth]{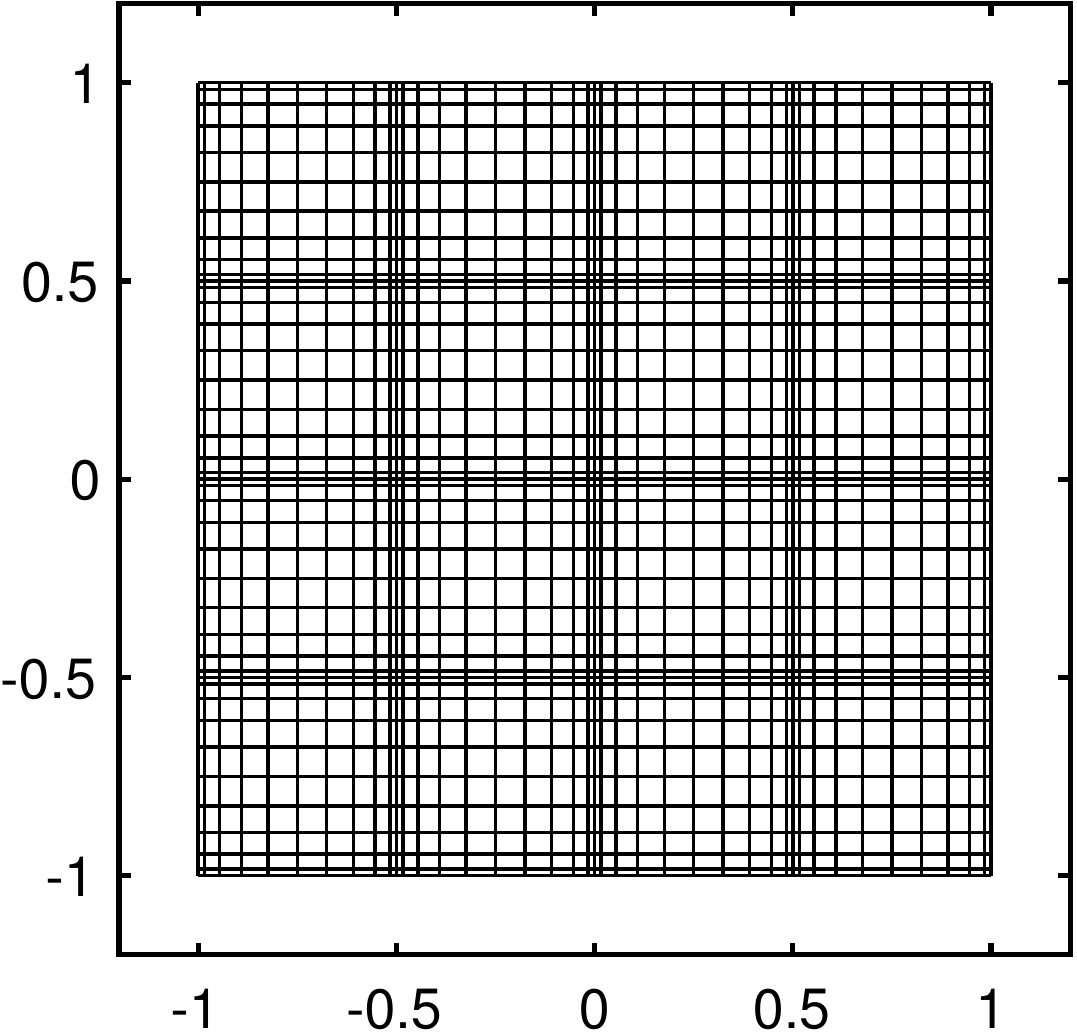}
\end{minipage}
 }\\
\subfloat[Adaptively refined meshes for finite element solution to Riesz fractional Poisson equation corresponding to the constant source term \(f=1\) for  \(\alpha=0.5\) \emph{(left)} and for \(\alpha=1.5\) \emph{(right)}.]{
 \begin{minipage}[]{\textwidth}\centering
 \includegraphics[scale=1]{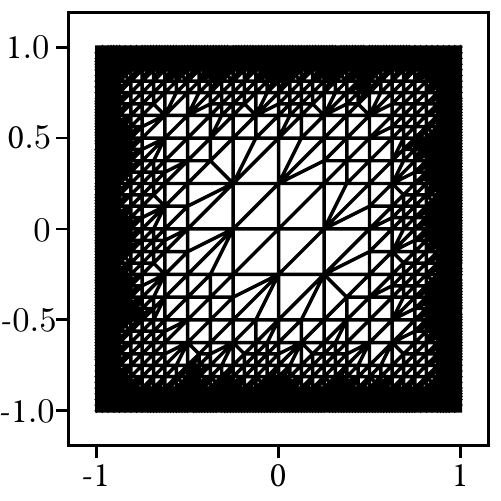}
 \includegraphics[scale=1]{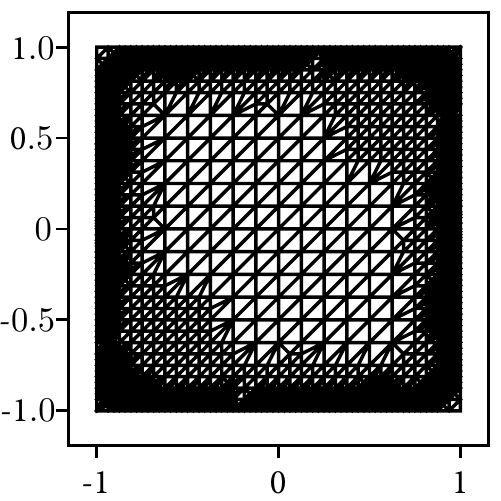}
\end{minipage}
 }\\
 \subfloat[Adaptively refined meshes for finite element solution to Riesz fractional Poisson equation corresponding to the source term $f = \sin(\pi x)\sin(\pi y)$ for  \(\alpha=0.5\) \emph{(left)} and for \(\alpha=1.5\) \emph{(right)}.]{
 \begin{minipage}[]{\textwidth}\centering
\includegraphics[scale=1]{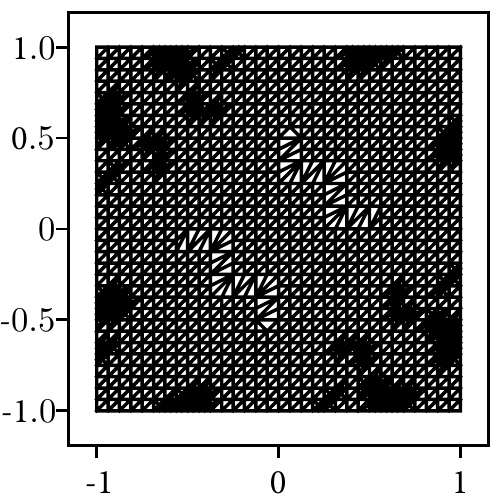}
  \includegraphics[scale=1]{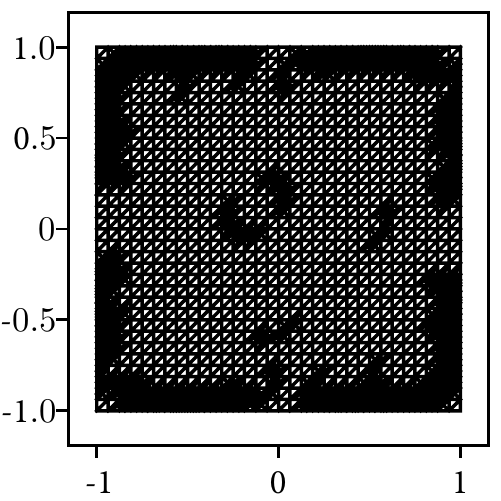}
\end{minipage}
 }
 \caption{\label{squaremeshes}  {Comparisons on the square}: Meshes and collocation points used for each numerical method in computing the solutions of the directional (\emph{top left}), spectral (\emph{top right}), and Riesz (\emph{center and bottom}) fractional Poisson equations.}
 \end{figure}

%\subsection{Disk Meshes and Collocation Points}

  \begin{figure}[ht!]
 \centering
\subfloat[{Comparison on the unit disk}: distribution of 1965 RBF collocation points with shape parameter $0.3$ used to compute the solution to the directional fractional Poisson equation (\emph{left}) and mesh for SEM used to compute the solution to the spectral fractional Poisson equation (\emph{right}).]{
 \begin{minipage}[]{\textwidth}\centering
 \includegraphics[width=0.3\textwidth]{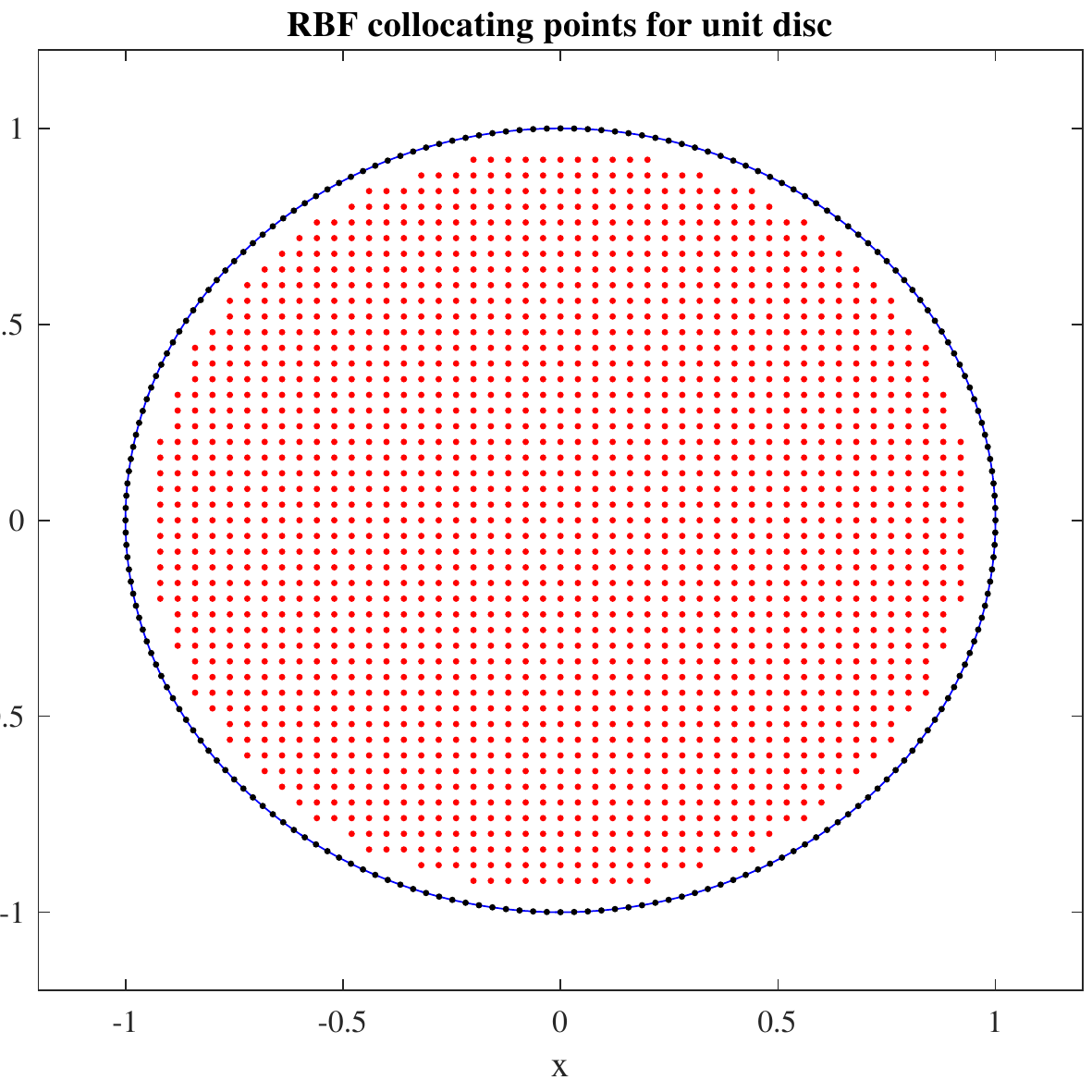} \hspace{0.5cm}
  \includegraphics[width=0.3\textwidth]{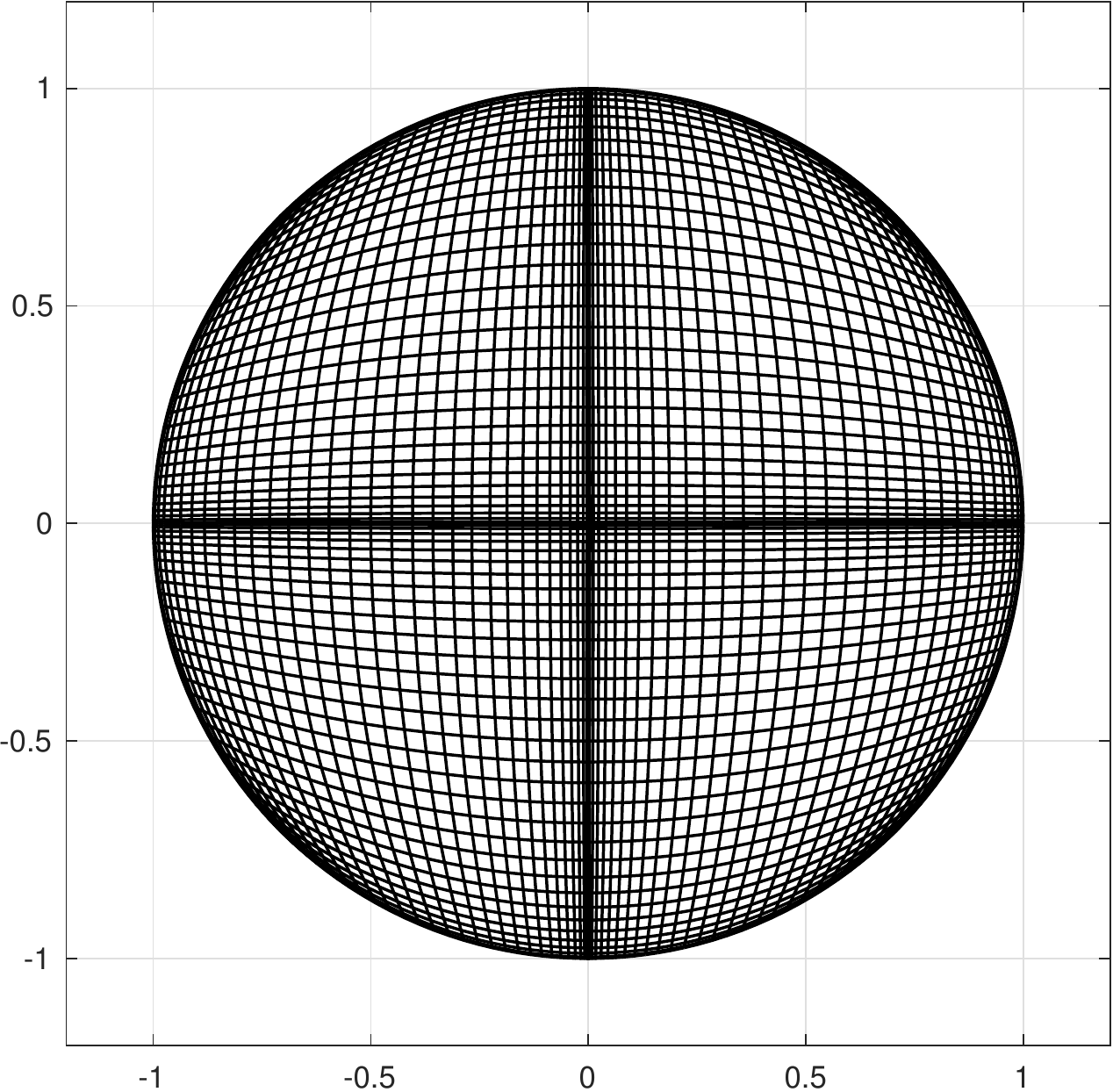}
\end{minipage}
 }\\
\subfloat[Adaptively refined meshes for finite element solution to Riesz fractional Poisson equation corresponding to the constant source term $f = 1$ for  \(\alpha=0.5\) \emph{(left)} and for \(\alpha=1.5\) \emph{(right)}.]{
 \begin{minipage}[]{\textwidth}\centering
  \includegraphics[scale=1.2]{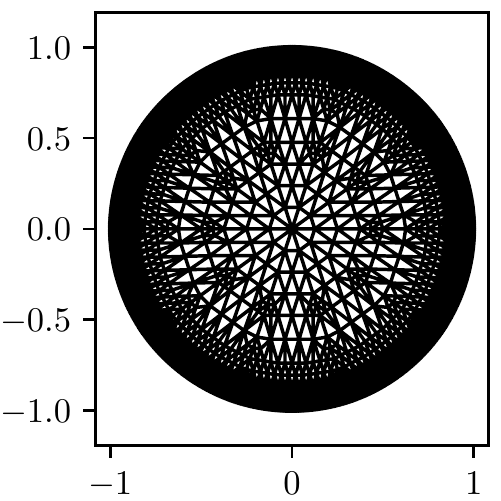} \hspace{0.3cm}
  \includegraphics[scale=1.2]{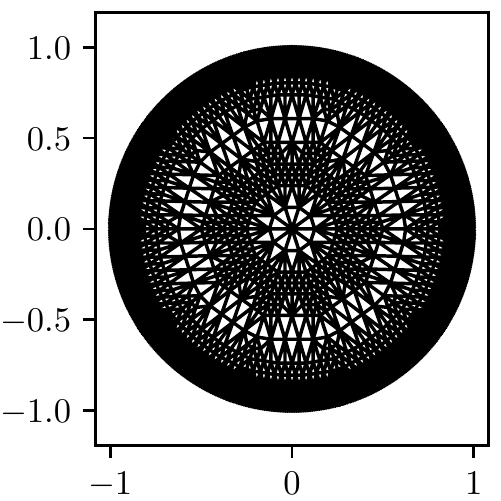}
\end{minipage}
 }\\
 \subfloat[Adaptively refined meshes for finite element solution to Riesz fractional Poisson equation corresponding to the source term $f = \sin(\pi x)\sin(\pi y)$ for  \(\alpha=0.5\) \emph{(left)} and for \(\alpha=1.5\) \emph{(right)}.]{
 \begin{minipage}[]{\textwidth}\centering
 \includegraphics[scale=1.2]{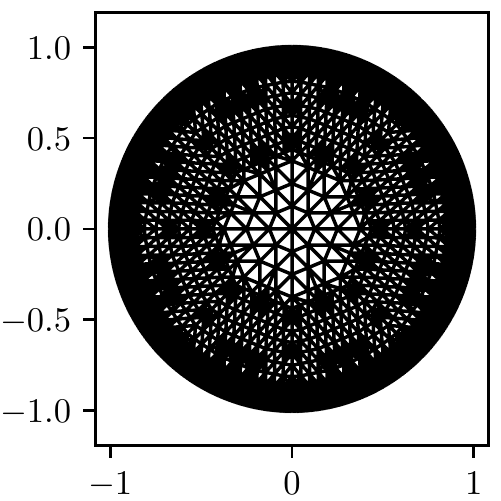}
  \includegraphics[scale=1.2]{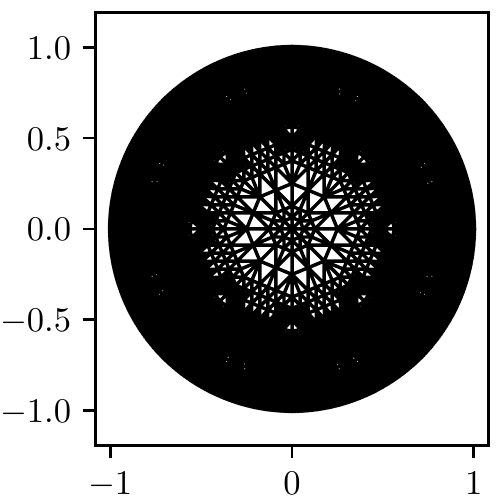}
\end{minipage}
 }
 \caption{\label{DiskMeshes}  {Comparisons on the disk}: Meshes and collocation points used for each numerical method in computing the solutions of the directional (\emph{top left}), spectral (\emph{top right}), and Riesz (\emph{center and bottom}) fractional Poisson equations.}
 \end{figure}

% \subsection{L-shape Meshes and Collocation Points}
 
  \begin{figure}[ht!]
 \centering
\subfloat[Distribution of 1976 RBF collocating points using shape parameter $0.2$ for the RBF collocation, which is used to compute the solution to the directional fractional Poisson equatio (\emph{left}), and the mesh used in the SEM to compute the solution to the spectral fractional Poisson equation for both $\alpha = 0.5$ and $1.5$ (\emph{right}). The same points and mesh are used regardless of the definition of the source function $f$.]{
 \begin{minipage}[]{\textwidth}\centering
 \includegraphics[width=0.3\textwidth]{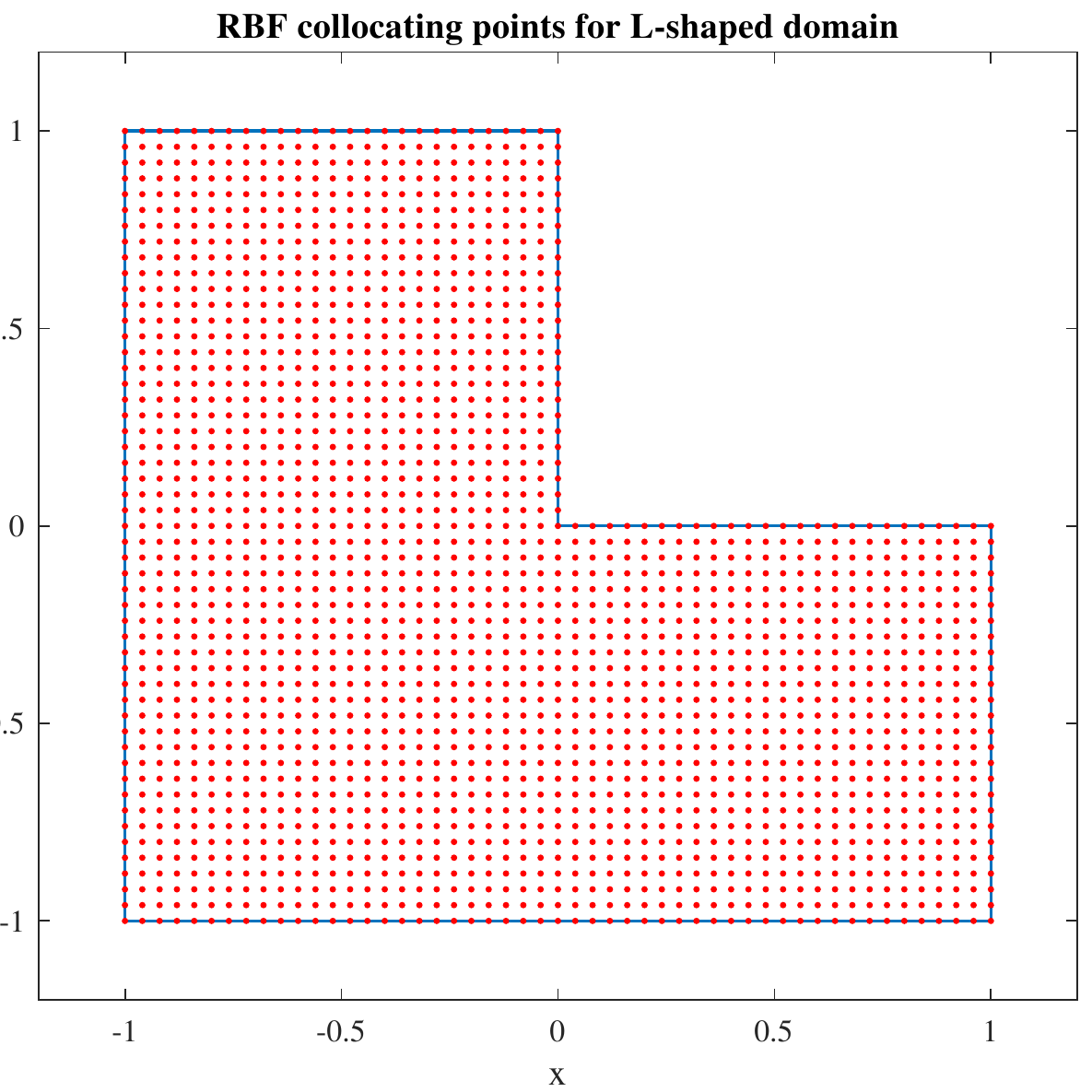}
  \includegraphics[width=0.3\textwidth]{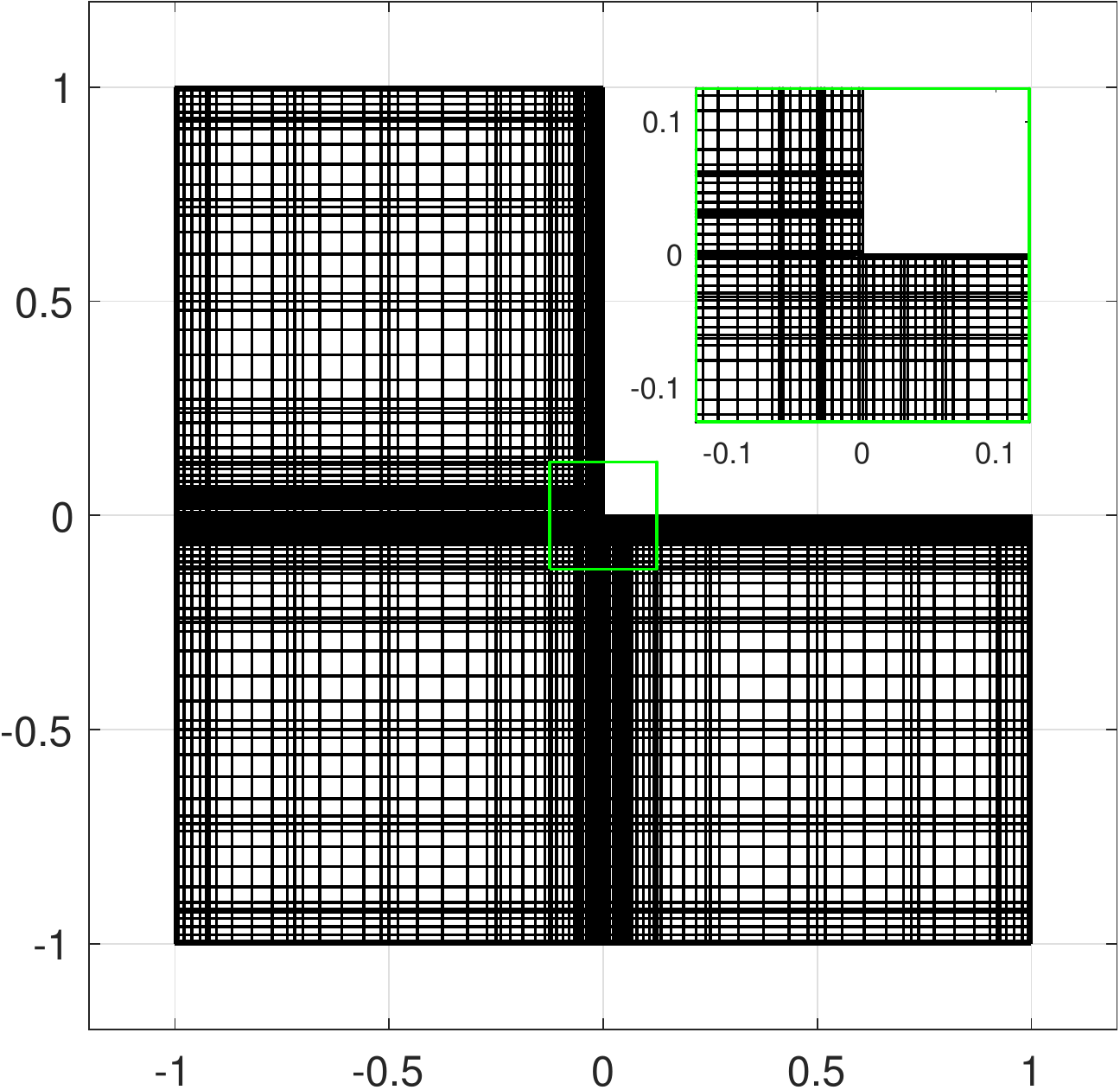}
\end{minipage}
 }\\
\subfloat[Adaptively refined meshes for finite element solution to Riesz fractional Poisson equation corresponding to the constant source term \(f=1\) for  \(\alpha=0.5\) \emph{(left)} and for \(\alpha=1.5\) \emph{(right)}.]{
 \begin{minipage}[]{\textwidth}\centering
  \includegraphics[scale=1.2]{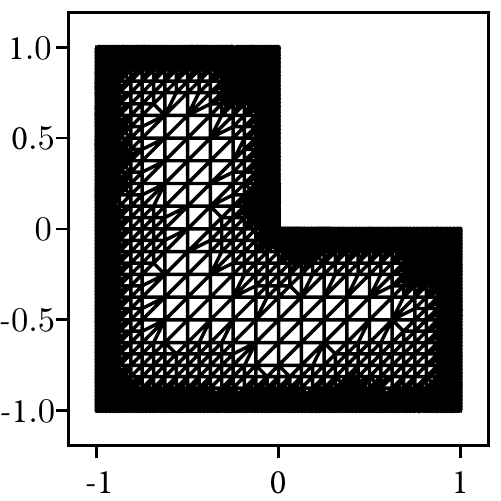}
  \includegraphics[scale=1.2]{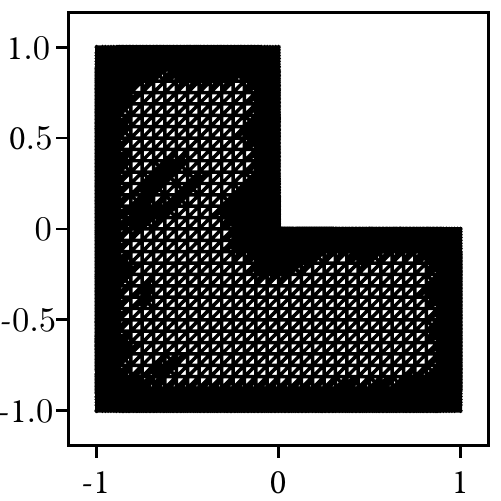}
\end{minipage}
 }\\
 \subfloat[Adaptively refined meshes for finite element solution to Riesz fractional Poisson equation corresponding to the source term $f = \sin(\pi x)\sin(\pi y)$ for  \(\alpha=0.5\) \emph{(left)} and for \(\alpha=1.5\) \emph{(right)}.]{
 \begin{minipage}[]{\textwidth}\centering
  \includegraphics[scale=1.2]{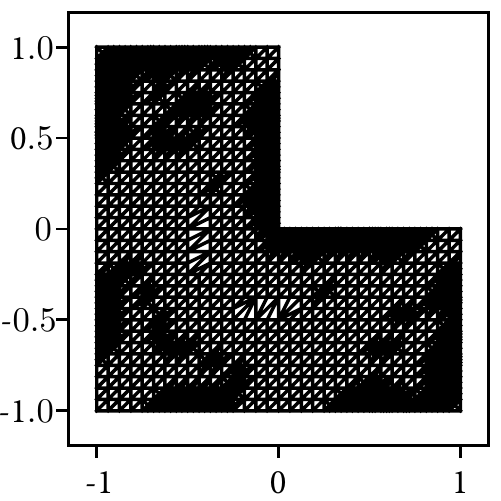}
  \includegraphics[scale=1.2]{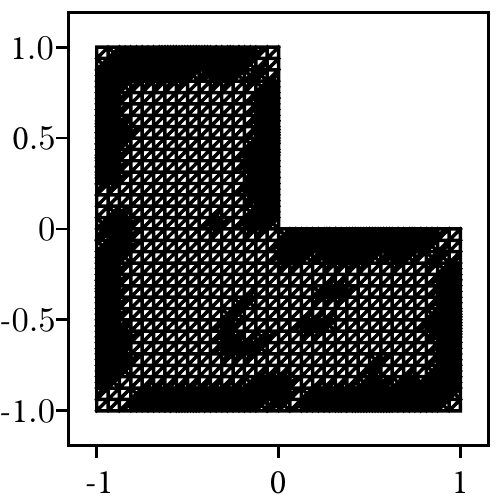}
\end{minipage}
 }
 \caption{\label{LshapeMeshes}  {Comparisons on the L-shaped domain}: Meshes and collocation points used for each numerical method in computing the solutions of the directional (\emph{top left}), spectral (\emph{top right}), and Riesz (\emph{center and bottom}) fractional Poisson equations.}
 \end{figure}
 
 \FloatBarrier

\section{Additional Disk Comparisons\label{app:disk}}

 \begin{figure}[ht!]
 \centering
 \subfloat[Solutions $u$ associated with $f=\sin(\pi r^2)$ and $\alpha = 0.5$ in the disk domain using the spectral definition (using SEM) (\emph{left}) and the Riesz definition (using AFEM) (\emph{center}), and the difference between $u_{Riesz}$ and $u_{spectral}$ for this case (\emph{right}).]{
 \begin{minipage}[]{\textwidth}\centering
 \includegraphics[width=0.25\textwidth]{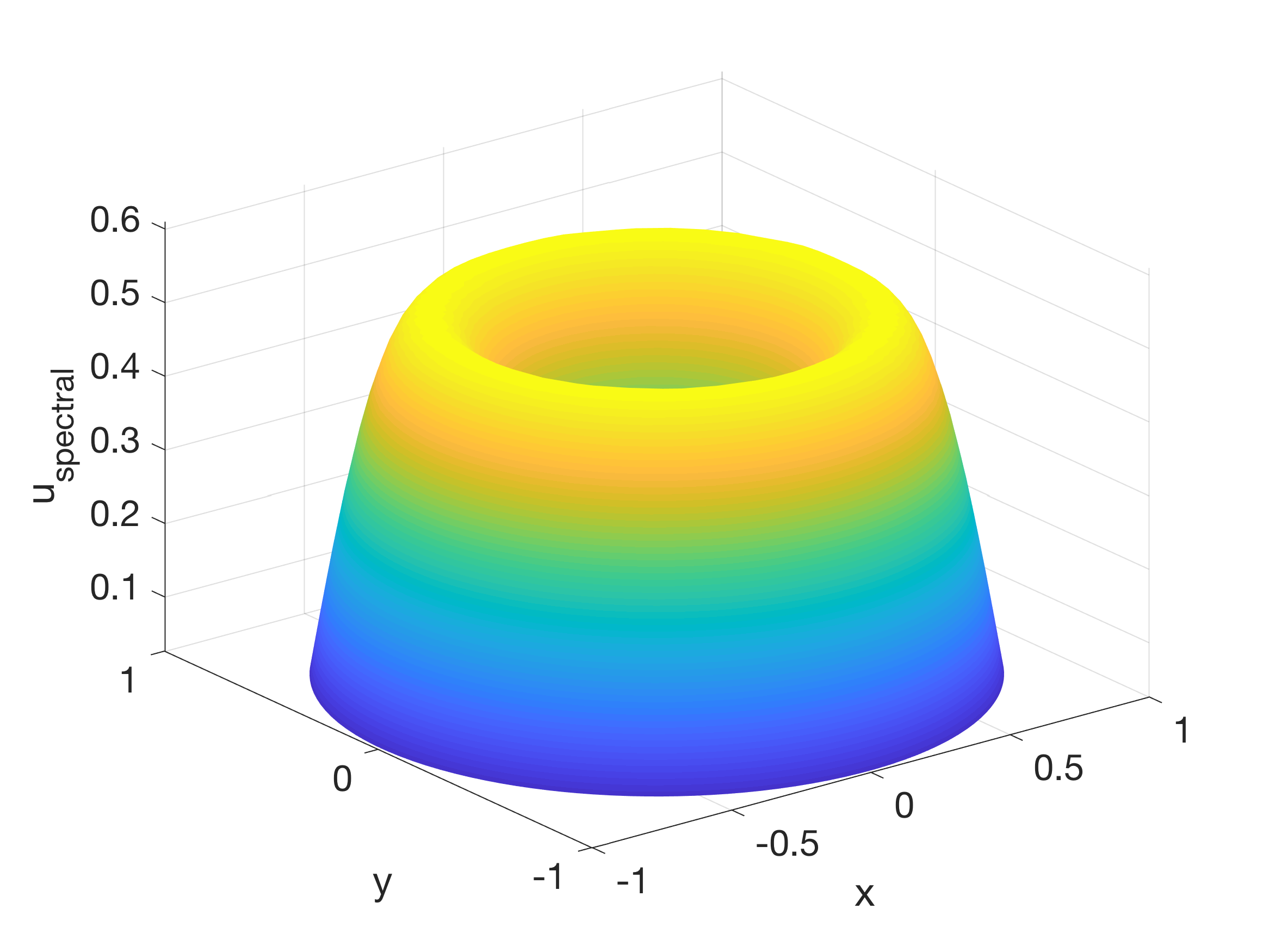}
  \includegraphics[width=0.25\textwidth]{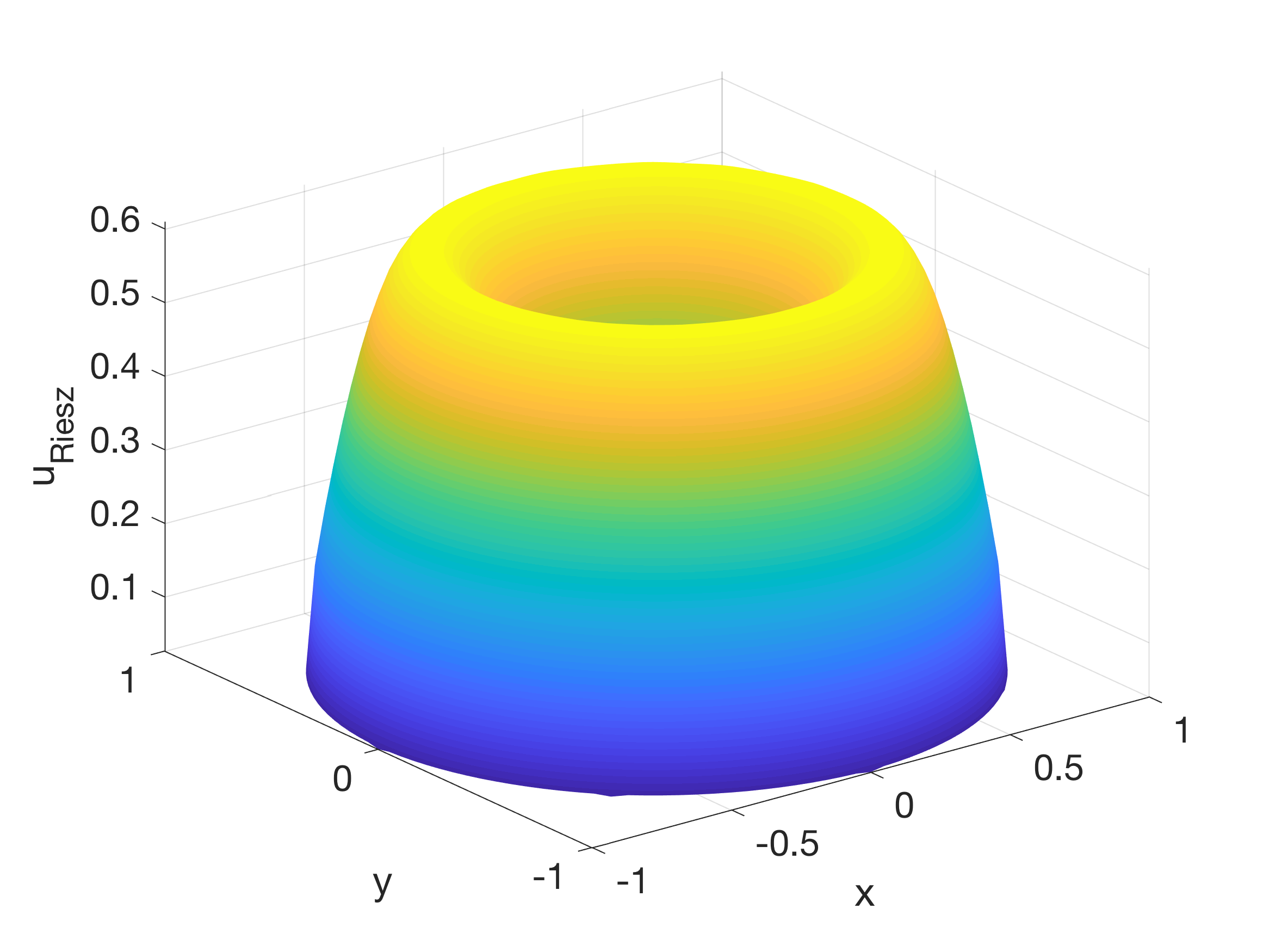}
  \includegraphics[width=0.25\textwidth]{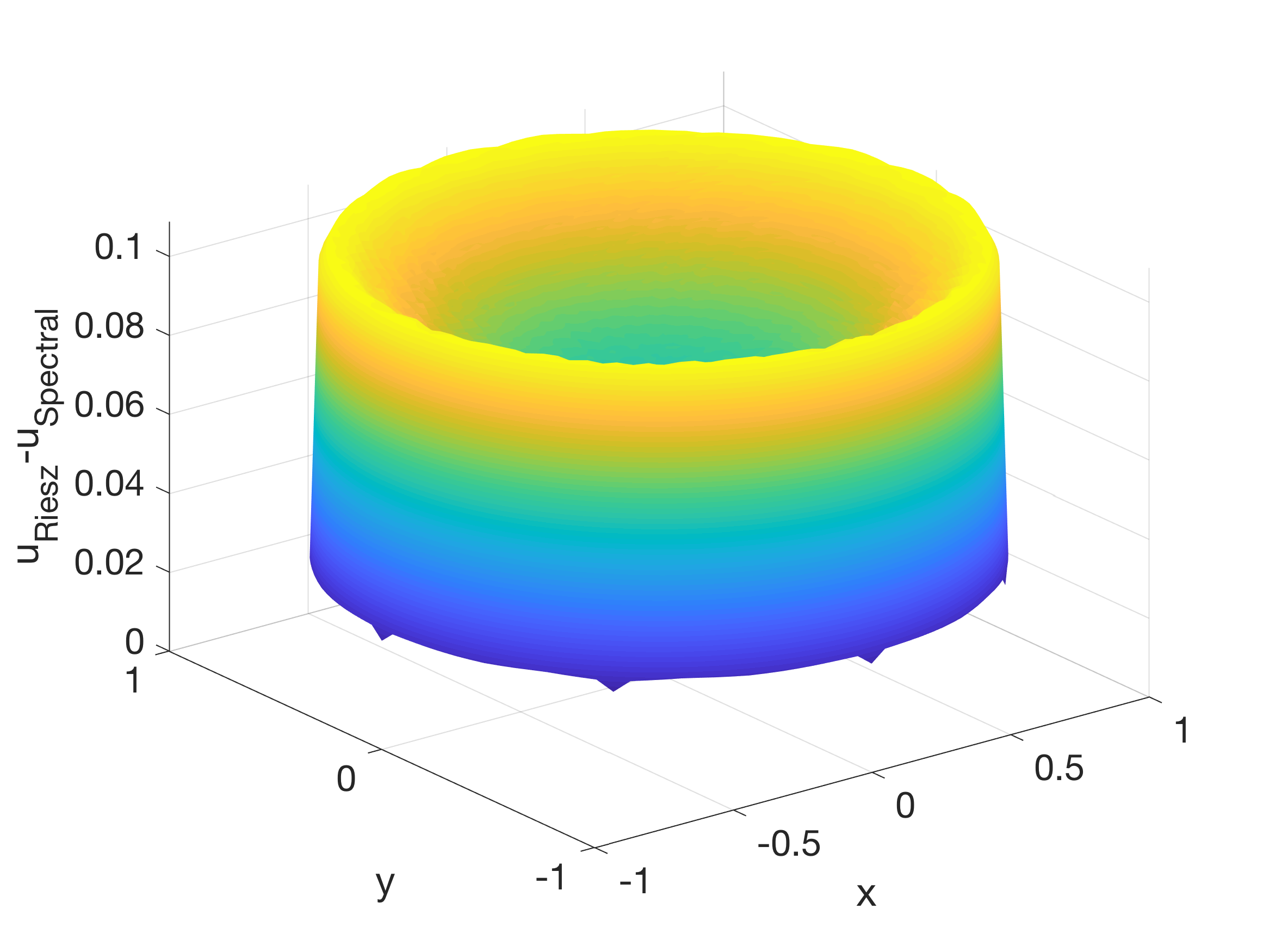}
\end{minipage}
 }\\
\subfloat[Solutions $u$ associated with $f=\sin(\pi r^2)$ and $\alpha = 1.5$ in the disk domain using the spectral definition (using SEM) (\emph{left}) and the Riesz definition (using AFEM) (\emph{center}), and the difference between $u_{Riesz}$ and $u_{spectral}$ for this case (\emph{right}).]{
 \begin{minipage}[]{\textwidth}\centering
 \includegraphics[width=0.25\textwidth]{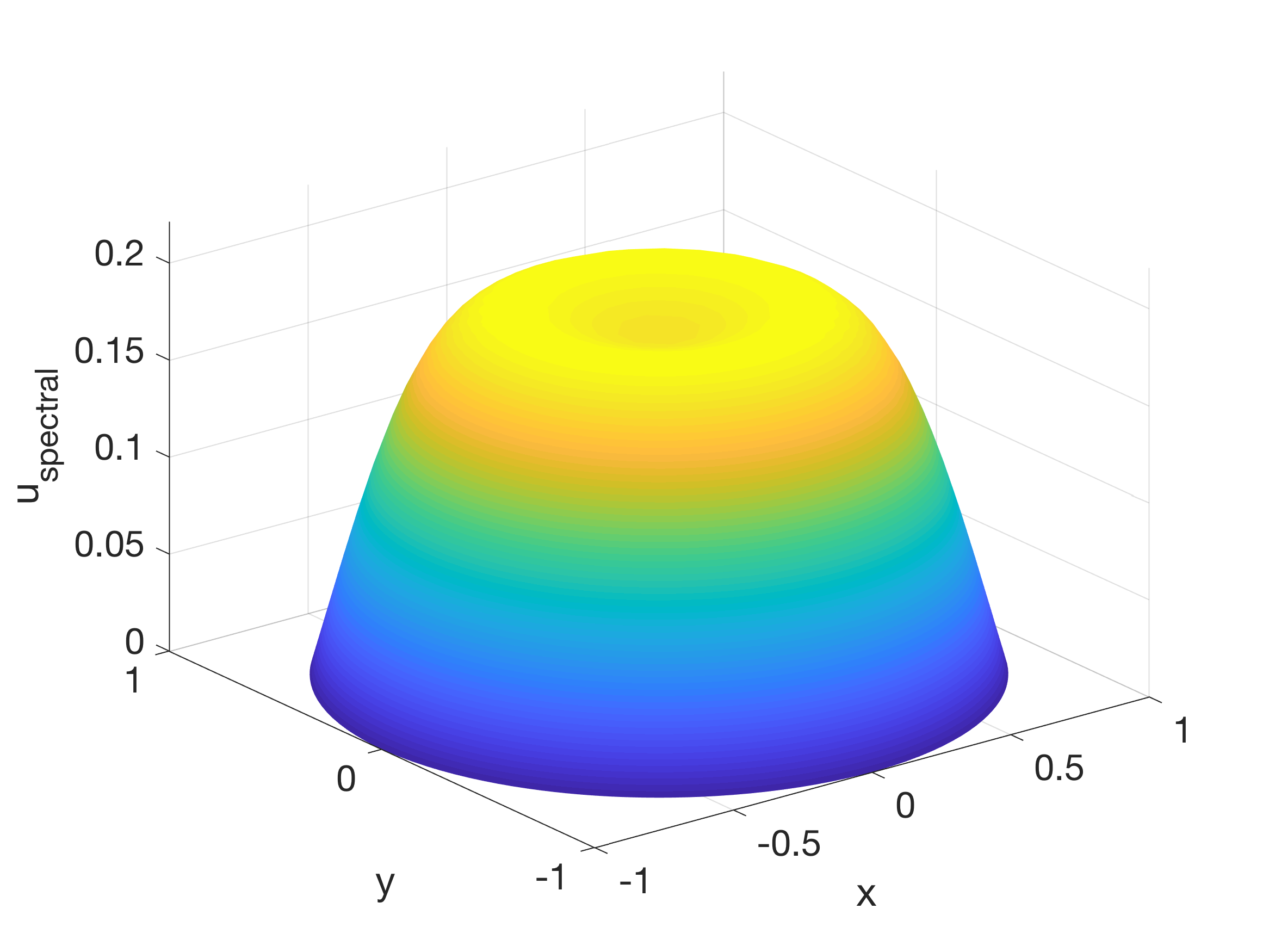}
  \includegraphics[width=0.25\textwidth]{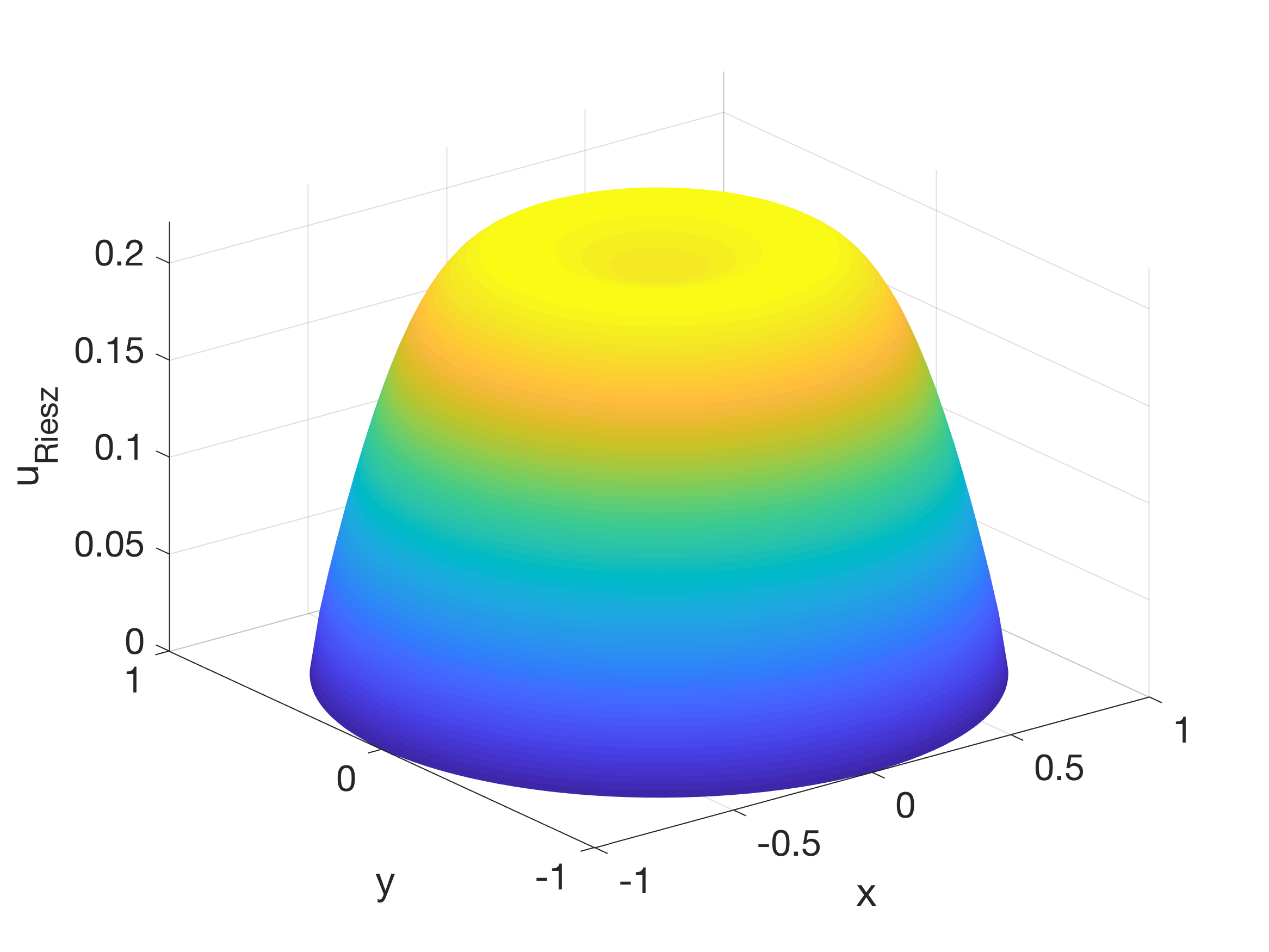}
  \includegraphics[width=0.25\textwidth]{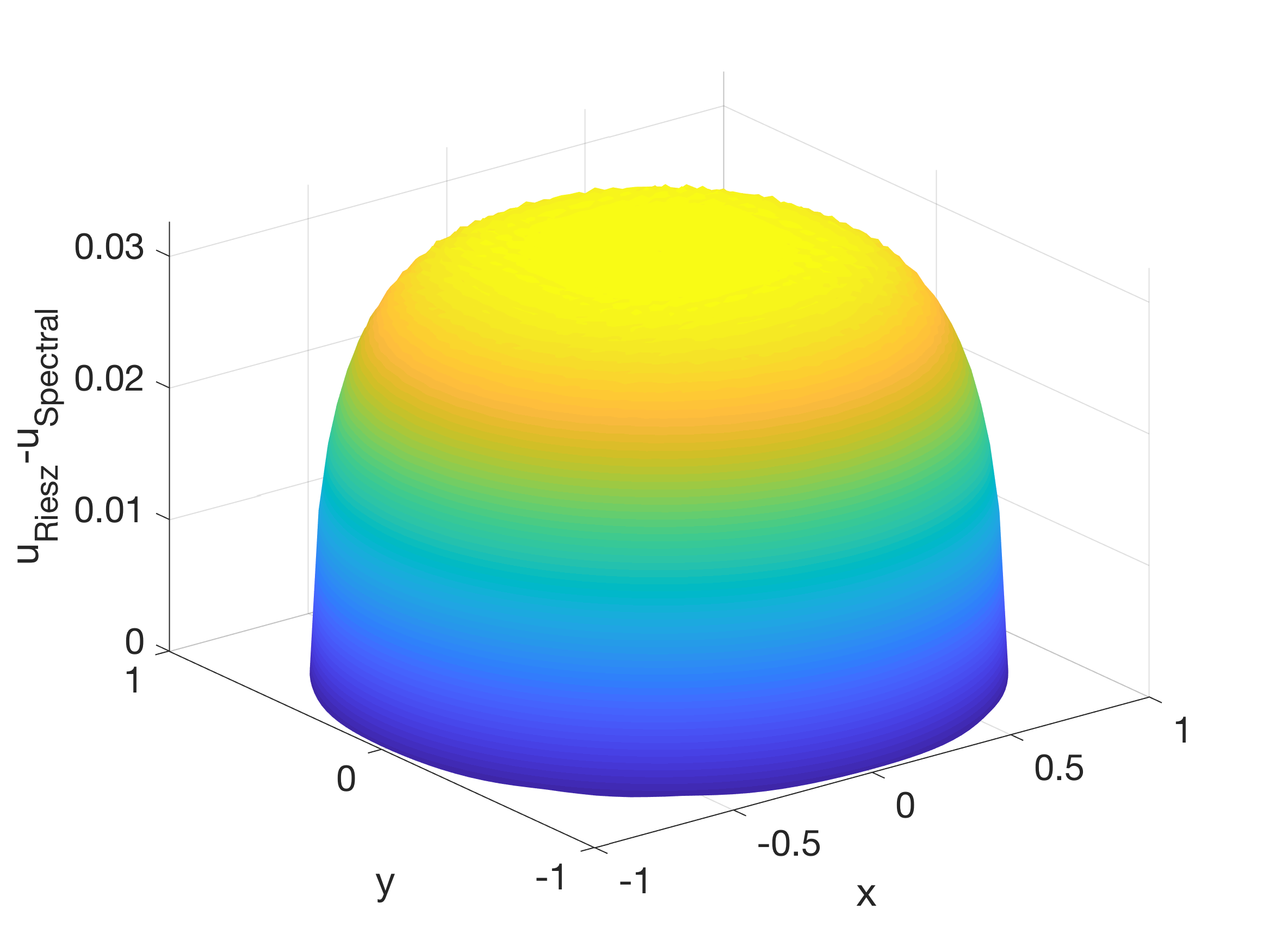}
\end{minipage}
 }

 \caption{ \label{disk34} {Solutions and differences} between $u_{Riesz}$ and $u_{spectral}$ on the disk domain for $\alpha = 1.5$.}
 \end{figure}
 
 \FloatBarrier

\begin{figure}[ht!]
\centering
\includegraphics[width=0.6\textwidth]{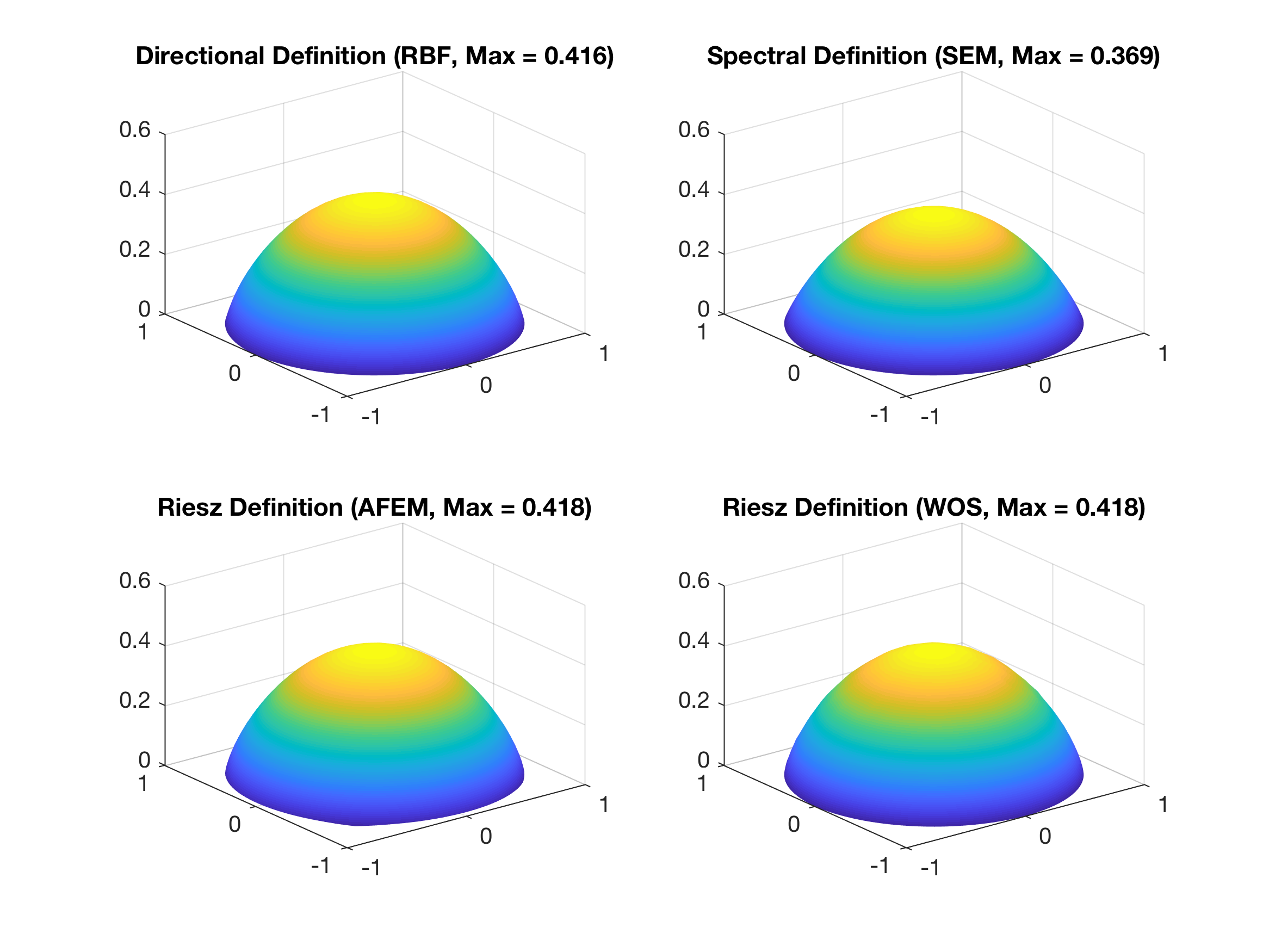}
\caption{\label{cmp-cir-const}
\color{blue}
Unit disk, $f=1$, and $g(x)=0$:
Comparison, for $\alpha = 1.5$, of $u_{\text{Riesz}}$ and $u_{\text{spectral}}$, using three methods to compute the Riesz solution and one method to compute the spectral solution. 
\emph{Top left}:
The Riesz solution obtained using the RBF collocation method based on the directional representation
\emph{Top right}:
The spectral solution obtained using the SEM.
\emph{Bottom left}:
The Riesz solution obtained using the AFEM.
\emph{Bottom right}:
The Riesz solution obtained using the WOS method. 
The only solution with a significant difference is the spectral solution (\emph{top right}); all other solutions are equivalent up to numerical error.}
\end{figure}

\FloatBarrier

\vfill
\break

\section{Additional L-shape Comparisons\label{app:Lshape}}

\begin{figure}[ht!]
\centering
\includegraphics[width=0.6\textwidth]{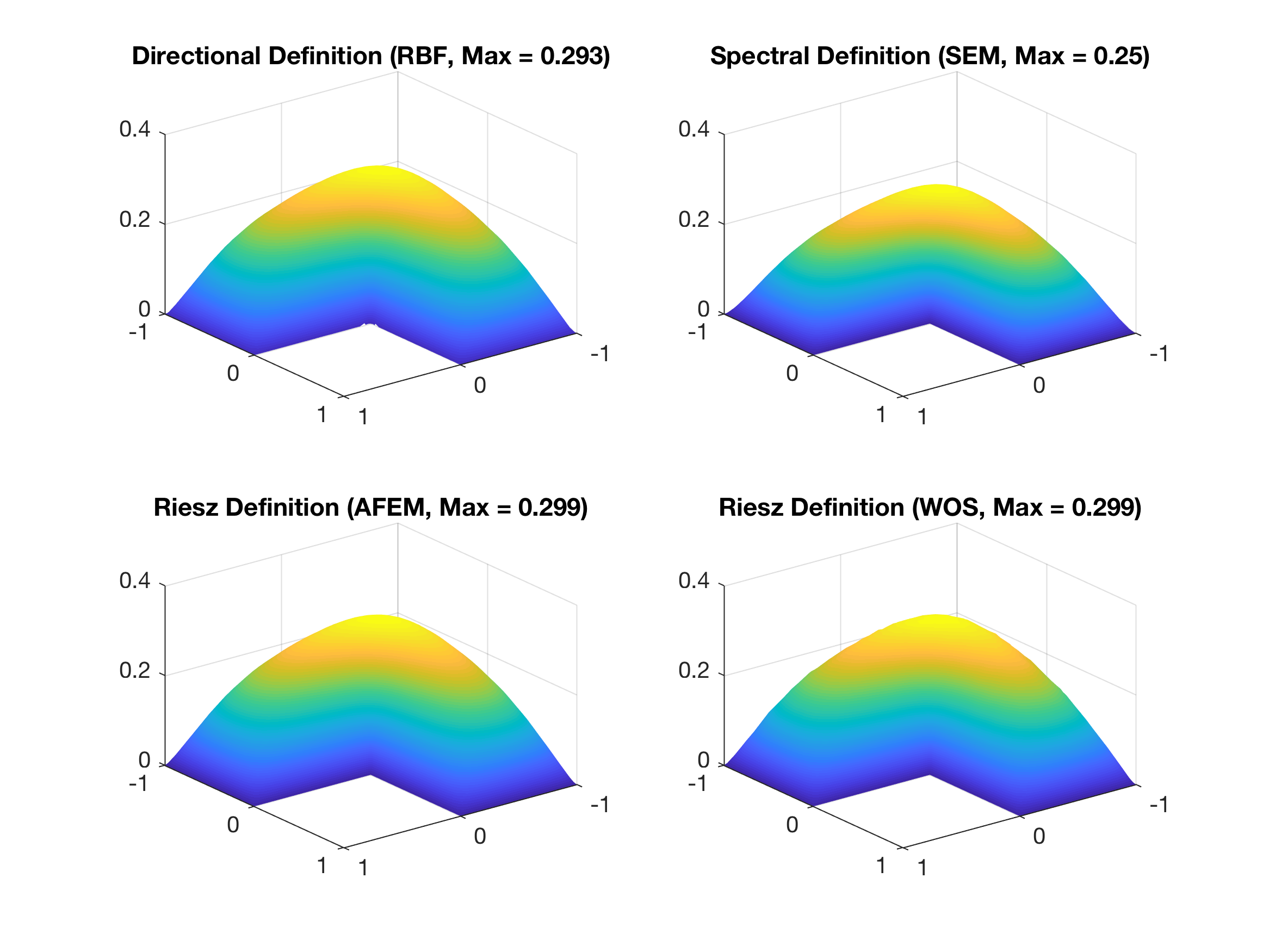}
\caption{\label{cmp-L-const} 
\color{blue}
L-shaped domain, $f(x)=1$, and $g(x)=0$: 
Comparison, for $\alpha = 1.5$, of $u_{\text{Riesz}}$ and $u_{\text{spectral}}$, using three methods to compute the Riesz solution and one method to compute the spectral solution. 
\emph{Top left}:
The Riesz solution obtained using the RBF collocation method based on the directional representation
\emph{Top right}:
The spectral solution obtained using the SEM.
\emph{Bottom left}:
The Riesz solution obtained using the AFEM.
\emph{Bottom right}:
The Riesz solution obtained using the WOS method. 
The only solution with a significant difference is the spectral solution (\emph{top right}); all other solutions are equivalent up to numerical error.}
\end{figure}

\begin{figure}[ht!]
\centering
\includegraphics[width=0.6\textwidth]{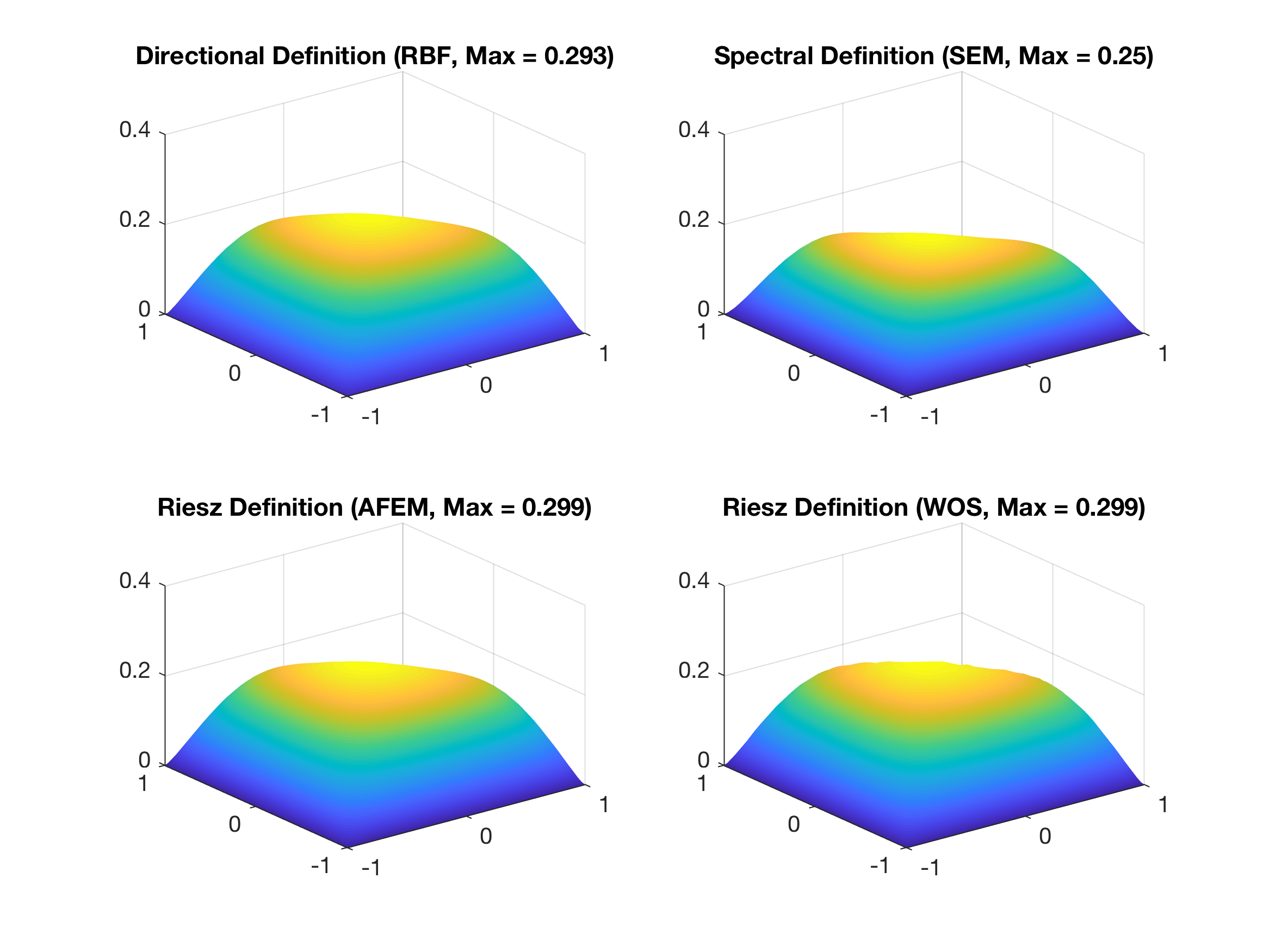}
\caption{\label{cmp-L-const-2}
\color{chocolate}The same as Fig. \ref{cmp-L-const} , but with a view facing the outside corner. 
}
\end{figure}

\FloatBarrier

\end{appendices}

\end{document}